\documentclass[11pt]{amsart} % reduces 10 pages

\usepackage[usenames,dvipsnames,svgnames,table]{xcolor}
\usepackage[pagebackref,linktocpage=true,colorlinks=true,linkcolor=Blue,citecolor=BrickRed,urlcolor=RoyalBlue]{hyperref}
\usepackage[alphabetic]{amsrefs}
\usepackage{skak} % provides with
\usepackage{xcolor}
\usepackage{subcaption}

\usepackage{stmaryrd}
\usepackage{mathrsfs}
\usepackage[bb=libus]{mathalpha}
\usepackage{amsmath}
\usepackage{amssymb}
\usepackage{amsthm}
\usepackage{mathtools}
\usepackage{graphicx}
\usepackage{enumitem}
\usepackage[colorinlistoftodos,prependcaption,textsize=tiny,
textwidth=0.75in]{todonotes}%, color=cyan
\usepackage[alphabetic]{amsrefs}

\newtheorem{theorem}{Theorem}[section]
\newtheorem{lemma}[theorem]{Lemma}
\newtheorem{proposition}[theorem]{Proposition}
\newtheorem{definition}[theorem]{Definition}
\newtheorem{corollary}[theorem]{Corollary}

\theoremstyle{remark}
\newtheorem{remark}[theorem]{Remark}

\theoremstyle{definition}

\numberwithin{equation}{subsection} % numbering over sections
\let\oldsection\section% Store \section
\renewcommand{\section}{% Update \section
	\renewcommand{\theequation}{\thesection.\arabic{equation}}% Update equation number
	\oldsection}% Regular \section
\let\oldsubsection\subsection% Store \subsection
\renewcommand{\subsection}{% Update \subsection
	\renewcommand{\theequation}{\thesubsection.\arabic{equation}}% Update equation number
	\oldsubsection}% Regular \subsection

\def\R{\mathbb{R}}

\def\d{\partial}
\def\curl{\textnormal{\textrm{curl}}}

\def\supp{\textnormal{\textrm{supp}}}
\def\sgn{\textnormal{\textrm{sgn}}}
\def\dif{{\mathrm d}}
\def\ep{\varepsilon}
\def\vp{\varphi}

\def\x{\vec{x}}
\def\y{\vec{y}}

\def\X{\vec{X}}
\def\vv{\vec{\mathrm{v}}}

\def\normal{\vec{N}}
\def\un{\hat{n}}
\def\hx{\hat{x}}
\def\hy{\hat{y}}
\def\hz{\hat{z}}
\def\hr{\hat{r}}

\def\hth{\hat{\theta}}

\def\mA{\mathcal{A}}
\def\mB{\mathcal{B}}
\def\mC{\mathcal{C}}

\def\mF{\mathcal{F}}
\def\mG{\mathcal{G}}
\def\mH{\mathcal{H}}
\def\mI{\mathcal{I}}

\def\mL{\mathcal{L}}
\def\mM{\mathcal{M}}
\def\mR{\mathcal{R}}
\def\mS{\mathcal{S}}
\def\mT{\mathcal{T}}

\def\mX{\mathcal{X}}

\def\fb{\mathfrak{b}}

\def\ff{\mathfrak{f}}
\def\fg{\mathfrak{g}}
\def\fh{\mathfrak{h}}
\def\fl{\mathfrak{l}}

\def\fp{\mathfrak{p}}
\def\fq{\mathfrak{q}}
\def\fr{\mathfrak{r}}

\def\fC{\mathfrak{C}}
\def\fF{\mathfrak{F}}
\def\fD{\mathfrak{D}}

\def\fR{\mathfrak{R}}
\def\fU{\mathfrak{U}}

\def\Sym{\mathfrak{S}}

\def\tA{\tilde{A}}
\def\teta{\tilde{\eta}}
\def\talpha{\tilde{\alpha}}
\def\tgamma{\tilde{\gamma}}

\def\sM{\mathscr M}

\def\pbl{
\pmb{
\left\{
\right.}} % Poisson bracket left

\def\pbr{\pmb{
\left.
\right\}
}}%Poisson bracket right

\renewcommand{\div}{\textnormal{\textrm{div}}}
\newcommand\absm[1]{\langle #1\rangle}

\newcommand\wh[1]{\widehat{#1}}

\usepackage{tikz}
\usepackage{tikz-3dplot}
\usetikzlibrary{math}

\usepackage{pgfplots}

\pgfplotsset{compat=1.10}
\usepgfplotslibrary{fillbetween}
\usetikzlibrary{patterns}

\begin{document}

\title{
The Well-posedness of Cylindrical Jets with Surface Tension %Free Boundary Problems% Stability of Axial Jets with surface tension
}
\author{
	Yucong {\sc Huang}
	and
	Aram {\sc Karakhanyan} 
}

\date{}
\thanks{The research was partially supported by EPSRC grant   
EP/S03157X/1 "Mean curvature measure of free boundary".}
\subjclass[2000]{Primary: 35Bxx; Secondary: 35Jxx, 35Lxx, 35Sxx.}
\keywords{Dirichlet-Neumann operator, jet flow, paralinearization.}

\begin{abstract}
In 1879 Rayleigh \cite{Rayleigh} studied the stability of infinite cylindrical jets, 
inspired by the experiments of Plateau \cite{Plateau}. The principal question that 
Rayleigh asked is: under what circumstances the jet is stable, for small displacements.
%and proposed a mathematical explanation of the experimental finding of Plateau \cite{Plateau}.  

In this paper we show that the jet flow is well-posed in short time
if the initial condition belongs to some Sobolev space, and the initial  jet boundary remains uniformly bounded away from the axis of symmetry. This will be proved by the method of paradifferential calculus and paralinearization. 
The salient feature of these results is that no smallness assumption is imposed on the 
initial condition. 

\end{abstract}

\maketitle

{\hypersetup{linkcolor=blue}
\setcounter{tocdepth}{1}
	\tableofcontents
}

\section{Introduction}
In 1879 Rayleigh \cite{Rayleigh} formulated the problem of stability of cylindrical jets:   
"Let us conceive, then,  an infinitely long  circular  cylinder  of  liquid, at rest, and inquire under what circumstances it is stable, or unstable, for small displacements, symmetrical about the axis of figure."

Earlier Plateau \cite{Plateau} studied the effect of surface tension on the stability. 
Motivated by the experiments of Plateau,  Rayleigh's gave a heuristic 
justification of Plateau's stability theory based on finite length  energy considerations and Fourier expansion, which amounts to 
smallness of the perturbation compared to the size of the cross section of the cylinder at rest.

In this paper we completely solve  this problem and show that the jet flow is well-posed in short  time when the initial condition is in some Sobolev space, and the initial jet boundary remains uniformly bounded away from the axis of symmetry.

%%%%%%%%%%%% Begin Figure 1 %%%%%%%%%%%%%%%%

%\iffalse
\begin{figure}[!ht]
	\begin{subfigure}[b]{0.5\textwidth}
		\centering
		\tdplotsetmaincoords{60}{130}
		\begin{tikzpicture}[tdplot_main_coords, thick, scale=0.9]
			% The function that is rotated
			\tikzmath{function f(\x) {return 1.2 - 0.35*sin(\x r);};}
			\pgfmathsetmacro{\dominio}{4.5}
			\pgfmathsetmacro{\xi}{-\dominio}
			\pgfmathsetmacro{\step}{(\dominio-\xi)/70.0} %density of cross circles
			\pgfmathsetmacro{\xs}{\xi+\step}
			\pgfmathsetmacro{\max}{5}%size of axis 6
			% Circumferences (behind the coordiante axis)
			\foreach \x in {\xi,\xs,...,\dominio}{
				\pgfmathsetmacro{\radio}{f(\x)}	% radius of the circumference of the solid of revolution
				\draw[cyan,very thick,opacity=0.35] plot[domain=0.5*pi:2.0*pi,smooth,variable=\t] ({\radio*cos(\t r)},{\x},{\radio*sin(\t r});	
			}
			% Part of the solid of revolution behind the coordinate axis
			\foreach \angulo in {358,356,...,90}{
				\draw[cyan,very thick,rotate around y=\angulo,opacity=0.35] plot[domain=-\dominio:\dominio,smooth,variable=\t] ({0},{\t},{f(\t)});
			}
			% Graph of the function rotated about the $y$ axis
			\draw[red,ultra thick] plot[domain=-\dominio:\dominio,smooth,variable=\t] ({0},{\t},{f(\t)}) node [above left=0.6cm and -1cm] {$\sqrt{x^2+y^2} = \eta(z, t)$};
			% Coordinate axis
			\draw[thick,->] (0,0,0) -- (0,\max,0) node [right] {$z$};
			\draw[thick,->] (0,0,0) -- (\max-1,0,0) node [left] {$y$};
			\draw[thick,->] (0,0,0) -- (0,0,\max-2) node [above] {$x$};
			% Circumferences (in front of the coordiante axis)
			\foreach \x in {\xi,\xs,...,\dominio}{
				\pgfmathsetmacro{\radio}{f(\x)}	% Radio del círculo al inicio del sólido de revolución
				\draw[cyan,very thick,opacity=0.35] plot[domain=0.0:0.5*pi,smooth,variable=\t] ({\radio*cos(\t r)},{\x},{\radio*sin(\t r});	
			}
			% The solid of revolution (in front of the coordinate axis)
			\foreach \angulo in {0,2,...,89}{
				\draw[cyan,very thick,rotate around y=\angulo,opacity=0.35] plot[domain=-\dominio:\dominio,smooth,variable=\t] ({0},{\t},{f(\t)});
			}
		\end{tikzpicture}
		\caption{The axisymmetric cylindrical jet in $\R^3$.}
		\label{fig:jet-3D}
	\end{subfigure}
	\hspace{0.5cm}
	\begin{subfigure}[b]{0.4\textwidth}
		\centering
		\begin{tikzpicture}[scale=0.9]
			\begin{axis}[axis lines=middle,
				xlabel=$z$,
				ylabel=$r$,
				enlargelimits,
				ytick=\empty,
				xtick=\empty, %{1,4},
				]
				\addplot[name path=F,blue,domain={-10:10}] {0.05*x^2*sin x-0.02*x+1} node[pos=.8, above]{$\eta$}; % 0.5*x^2*sin x-2*x+5
				
				\addplot[name path=G,blue,domain={-10:10}] {-0.05*x^2*sin x+0.02*x-1}node[pos=.3, below]{$-\eta$};
				
				\addplot[pattern=north west lines, pattern color=brown!50]fill between[of=F and G, soft clip={domain=-10:10}]
				;
				\node[coordinate,label=above:{$\quad \Omega_\eta$}] at (axis cs:-2.9,0.2){};
				
			\end{axis}
		\end{tikzpicture}
		\caption{The axisymmetric jet in the cylindrical coordinate system.}
		\label{fig:D-eta}
	\end{subfigure}
	\caption{}
\end{figure}
%\fi
%%%%%%%%%%%% End Figure 1 %%%%%%%%%%%%%%%%

The mathematical formulation of the jet problem 
can be stated as follows. 
The fluid in $\R^3$ occupies a cylindrical domain, see Figure  \ref{fig:jet-3D}.
It is convenient to write the Euler equations in cylindrical coordinates $r, \theta, z$ 
which under assumption of axial symmetry 
simplify further
due to the independence of the physical quantities from $\theta$.
Thus the axisymmetric motion of the irrotational ideal fluid is described by the following system of equations
\begin{equation}\label{000-intro-1}
\left\{
\begin{aligned}
 &\d_z(r \d_z \Psi) + \d_r(r \d_r \Psi) = 0 && \text{in } \ (z,r)\in \Omega_{\eta},\\
&\d_t \eta - \big\{ \d_r \Psi - \d_z \eta \d_z \Psi \big\}  = 0 &&\mbox{on} \ \d \Omega_\eta,\\
& \partial_t \Psi + \dfrac{|\d_z \Psi|^2}{2}  + \dfrac{|\d_r\Psi|^2}{2}  + P  = 0 &&\mbox{on} \ \d \Omega_\eta,\\
&\d_r\Psi(z,r)=0 && \mbox{on}\  \{r=0\} ,  \\
\end{aligned}
\right.
\end{equation}
where $\Psi=\Psi(t, z, r)$ is the axially-symmetric velocity potential in the cylindrical coordinates $(r,z)$, 
$\Omega_{\eta}\vcentcolon= \{ (z,r)\in \R^2\vcentcolon 0 < r < \eta(z, t) \}$ 
is the physical domain occupied by the fluid, 
and the free boundary is given by the equation 
$r=\eta(z,t)$, see Figure \ref{fig:D-eta}.   

The first 
equation states that the flow is incompressible in cylindrical 
coordinate system, the second and third equations are  the kinematic and dynamic conditions, respectively. The last one is the compatibility condition for the axisymmetric flow.
The detailed derivation of this system is contained in Appendix \ref{appen:zak}. 

Set $\psi(t,x)\!=\!\Psi(t,z,\eta(t,z))$. Then \eqref{000-intro-1} can be reduced to 
the following system 
\begin{equation}\label{000-intro-2}
\left\{
\begin{aligned}
		&\d_t \eta - G[\eta](\psi) = 0,\\
		%%%%%%%
		&\d_t\psi + \dfrac{\big|\d_z\psi|^2}{2}-\dfrac{| \d_z \eta \d_z \psi+G[\eta](\psi)\big|^2}{2 (1+|\d_z \eta|^2)}-\kappa \Big( \mH(\eta) + \frac{1}{2R} \Big) = 0, 
	\end{aligned}
\right.
\end{equation}
where $G[\eta](\cdot)$ is the Dirichlet-Neumann operator associated with the surface $r=\eta(t,z)$ and the elliptic operator $L := \d_z(r \d_z ) + \d_r(r \d_r )$. 
%two  equations with two unknowns $(\eta, \psi)$, following the formulation of Zakharov \cite{Z} for water wave problem: 

More specifically, this is defined as:
\begin{gather}
G[\eta](\psi) \vcentcolon= \big\{ \d_r \Psi - \d_z\eta \d_z \Psi \big\}\big\vert_{r=\eta(t,z)} \quad \text{where } \ \Psi \ \text{ is the solution to:}\nonumber\\
\Biggl\{\begin{aligned}
&L \Psi = \d_z(r \d_z \Psi) + \d_r(r \d_r \Psi) = 0 && \text{for } \ z\in\R \ \text{ and } \ 0<r\le \eta(t,z),\\
& \Psi\vert_{r=\eta(t,z)} = \psi, \quad \d_r \Psi\vert_{r=0} = 0 && \text{for } \ z\in\R.
\end{aligned} \label{cyl-harmonic}
\end{gather}
Moreover, $\kappa>0$ is the surface tension coefficient and 
\begin{equation}\label{tension}
    \mH(\eta)=\d_z \Big( \dfrac{\d_z \eta}{2\sqrt{1+|\d_z\eta|^2}} \Big) - \dfrac{1}{2\eta \sqrt{1+|\d_z\eta|^2}},
\end{equation}
is the mean curvature of the cylindrical surface. We set $\kappa\equiv 1$ for the rest of this paper. The reformulation of (\ref{000-intro-1}) into system (\ref{000-intro-2})--(\ref{tension}) is illustrated in Appendices \ref{ssec:KinDyn}--\ref{ssec:Equi}. In addition, the Hamiltonian for system (\ref{000-intro-2}) can be found in Appendix \ref{append:hamiltonian}. We also note that when a gravitational force $-gz$ is present in the $-\hat{z}$ direction, it can be shown that the corresponding problem is equivalent to (\ref{000-intro-2})--(\ref{tension}) via a transformation in the reference frame for the unknown $(\eta,\psi)$. The detailed analysis for this is presented in Appendix \ref{append:gravity}.

\subsection{Main results}
Throughout the paper the initial data
$(\eta_0, \psi_0)$ satisfies the following two assumptions:

\begin{enumerate}[label=\textbf{\textrm{(A\arabic*)}},ref=\textbf{\textrm{(A\arabic*)}}]
    \item\label{item:A1} $(\eta_0-R, \psi_0)$ are in $H^{s+\frac12}(\R)\times H^s(\R)$ with $s>\frac{5}{2}$,
	\item\label{item:A2} There are fixed constants $R,\,C_0>0$ such that $\eta_0-R\ge C_0^{-1}$, $\|\eta_0-R\|_{H^{s+\frac12}(\R)}\le C_0$. 
\end{enumerate}

\begin{theorem}[Existence]\label{thm:main1}
There exists $T\!>\!0$ which depends only on $R, C_0>0$ such that for all $(\eta_0 - R, \psi_0)$ satisfying \textnormal{\ref{item:A1}}--\textnormal{\ref{item:A2}},
the Cauchy problem \textnormal{(\ref{000-intro-2})}--\textnormal{(\ref{tension})} with initial data $(\eta_0, \psi_0)$ has a solution in $(\eta-R, \psi)\in \mC^0\big([0, T], H^{s+\frac12}(\R)\times H^s(\R)\big)$ 
and there exists a constant $C=C(R,C_0,T)>0$ such that \textnormal{\ref{item:A1}}--\textnormal{\ref{item:A2}} hold for $t\in [0, T].$
\end{theorem}

\begin{remark}
Observe that $(\eta,\psi)\equiv(R,0)$ for constant $R>0$ is a trivial solution to the system (\ref{000-intro-2})--(\ref{tension}). Thus the solution $(\eta-R,\psi)$ obtained in Theorem \ref{thm:main1} is considered as a perturbation from the equilibrium state $(\eta,\psi)\equiv(R,0)$ in the Sobolev space $H^{s+\frac{1}{2}}(\R)\times H^{s}(\R)$. The far-field condition $(\eta,\psi)\to (R,0)$ as $|z|\to \infty$ is consistent with the surface tension: $\mH(\eta)+\frac{1}{2R}$ given by (\ref{tension}). For more details on the modelling of surface tension and equilibrium solution, we refer readers to Appendices \ref{ssec:PH}--\ref{ssec:Equi}.
\end{remark}

\begin{theorem}[Uniqueness and short-time stability]
Suppose in addition that $s>4$. For two initial data $(\eta_1^{0},\psi_1^{0})$ and $(\eta_2^{0},\psi_2^{0})$ satisfying \textnormal{\ref{item:A1}--\ref{item:A2}} with $s>4$, let $(\eta_1,\psi_1)$ and $(\eta_2,\psi_2)$ be the corresponding solutions in $t\in[0,T]$, which are guaranteed by Theorem \ref{thm:main1}. Then there exists a constant $C=C(R,C_0,T)>0$ such that
\begin{equation*}
    \sup\limits_{0 \le t\le T}\|(\eta_1-\eta_2,\psi_1-\psi_2)(t,\cdot)\|_{H^{s-1}(\R)\times H^{s-\frac{3}{2}}(\R)}\! \le\! C \|(\eta_1^0-\eta_2^0,\psi_1^0-\psi_2^0)\|_{H^{s-1}(\R)\times H^{s-\frac{3}{2}}(\R)}.
\end{equation*}
\end{theorem}

%%%%%%%%%%%%%%%%%%%%%%%%%%%%%%%%%%%
The astute reader will observe that the system \eqref{000-intro-2} involves the 
nonlinear and nonlocal operator $G[\eta](\psi)$, which is the main object in the 
analysis. If we blackbox the operator $ G[\eta](\psi)$, then at symbolic level  the system  \eqref{000-intro-2} looks similar to the one appearing in the water wave problem. However, the 
operator $G[\eta](\psi)$ differs drastically from the one in the water  wave system.
In this context our main contributions in this paper are 

\begin{itemize}
\item[${\bf 1^\circ}$] The DN operator $G[\eta](\psi)$ is generated by a completely different degenerate elliptic operator $L$. Due to this, its paralinearization has different coefficients, leading to a more complicated version of a decomposition of the operator and G$\mathring{\mbox{a}}$rding's inequality.
\item[$\bf 2^\circ$]
The analysis we developed in ${\bf 1^\circ}$ enables us to apply the paradigm 
of diagnonalization from \cite{ABZ} and give a solution to one of the 
classical problems in unsteady free boundary flows for cylindrical jets originally studied by Rayleigh in 1879 \cite{Rayleigh}, see also  \cite{Gilbarg}, \cite{Birkhoff}.
\item[${\bf 3^\circ}$] 
In \cite{ABZ} the authors conjectured  that 
all dispersive estimates for the 
water wave problems can be obtained using the 
paralinearization technique. 
One of the main technical points of this paper 
is to show that this conjecture is valid for 
more general DN operator produced by the degenerate elliptic operator $L$. 
${\bf 1^\circ}$ and ${\bf 2^\circ}$ provide further evidence and confirmation of this claim.
\end{itemize}

In order to recognize that the complexity of our problem is at a different level, let us 
recall the paradigm of paradifferential calculus in the water wave problems.

\subsection{Water waves}
For $d\ge 2$, the Dirichlet-Neumann operator (DN operator for short) for a slab domain with finite depth $h>0$:
$\Omega=\{(x, y)\in \R^{d-1}\times\R\,\vert\, -h<y<0 \}$
is defined as the mapping $G[0]:\psi\mapsto \d_\nu \phi \vert_{y=0}$, where $\d_\nu$ denotes the normal derivative at $y=0$, and 
\begin{equation}\label{000-strip-prob}
\left\{
\begin{array}{lll}
\Delta_x \phi + \d_y^2 \phi=0\quad \mbox{in}\ \Omega, \\
\phi(x,y)\vert_{y=0}=\psi(x), \quad \d_y \phi(x,y) \vert_{y=-h}=0,
\end{array}
\right.
\end{equation}
and $\psi(x)\in \mathscr{S}(\R^{d-1})$ is the Dirichlet data belonging to Schwartz space. Following the analysis of \cite{Z}, the solution to 
\eqref{000-strip-prob} is given by 
\begin{equation}\label{wwphi}
\phi(x,y)=\frac{\cosh\big((y+h)|D|\big)}{\cosh (h|D|)}\psi(x), \quad \text{where} \quad |D|\vcentcolon= \bigg(\sum_{j=1}^{d-1}|\d_{x_j}|^2\bigg)^{\frac{1}{2}}.
\end{equation}
Consequently, we can write $G[0](\psi)=\d_\nu \phi|_{y=0}$ as a pseudodifferential operator:  
\begin{equation}\label{wwDN}
G[0](\psi)=a(D)\psi \quad  \text{with symbol } \ a(\xi)\vcentcolon=|\xi|\tanh(h|\xi|) \ \text{ for } \ \xi\in\R^{d-1}.
\end{equation}
If the domain is described by $\Omega=\{(x,y)\in \R^{d-1}\times \R\,\vert\, -h< y < \eta(x) \}$ for an unknown function $\eta(x)$, then it is evident that the corresponding DN operator $G[\eta](\cdot)$ is strongly nonlinear. Thus linearization with respect to $\eta\equiv 0$ is needed to understand its behavior for more general domains $\Omega$ and elliptic operators $L$, indicating the importance of optimal estimates for $G[0](\psi)$. For instance, it is easy to check that for $s\in\R$,
\begin{equation}
\|G[0](\psi)\|_{H^{s-1}(\R^{d-1})}\le C(h)\|\psi\|_{H^{s}(\R^{d-1})}, 
\quad \text{for } \ \psi\in H^{s}(\R^{d-1}).
\end{equation}

When the problem is posed in the time dependent domain: 
$\Omega(t)=\{(x, y)\in \R^{d} \,\vert\, -h<y<\eta(t,x)\}$, the 
kinematic and dynamic conditions on 
$\eta$ and $\psi$ form the Zakharov system, which involves 
the DN operator. The main technical point is then to find the appropriate linearization so that the Zakharov system can be reformulated into a $2$-by-$2$ matrix form 
with symmetric leading order differential operators. This is where the paralinearization of 
DN operator is used to avoid loss of regularity.

The unsteady flows of an incompressible ideal fluid in two dimension when the free boundary is governed by surface tension was studied by Yosihara \cite{Yosihara}. The local 
time existence is proved under the assumption that the initial condition is  small and the bottom is nearly flat.

In \cite{lannes} Nash-Moser scheme is used to prove the short time existence due to the 
loss of regularity that one encounters while symmetrizing the water wave system. 
The chief difficulty in this approach is to obtain precise estimates for the DN 
operator. For the water wave problem this was achieved by using the paradifferential calculus
\cite{ABZ}, resulting better symmetrization result and avoiding the loss of regularity.

The hidden symmetry  of the water-waves is elucidated by 
Zakharov \cite{Z}, who showed that 
the water wave system has a Hamiltonian $\mathrm{H}=\mathrm{H}_0+\sum_{k=1}^\infty \mathrm{H}_k$, 
where $\mathrm{H}_0$ gives the  linearized term, and the system can be rewritten as a 
single complex equation. 
In this context the diagonalization of Hamiltonian means that  
the nonlinear Zakharov system can be written in the vector form 
with twisted diagonal matrix
\begin{equation*}
\begin{pmatrix}
\d_t \eta\\ \d_t\psi
\end{pmatrix}
=
\begin{pmatrix}
0 & 1\\
-1& 0
\end{pmatrix}
\begin{pmatrix}
    \frac{\delta \mathrm{H} }{\delta \eta}\\[1ex]
    \frac{\delta \mathrm{H}}{\delta \psi}
\end{pmatrix}.
\end{equation*}
The main technical point in \cite{ABZ} is to 
obtain a similar reduction via paralinearization, namely 
there are two paradifferential operators $T_0, T_1$ such that Zakharov's system for the water wave problem is reformulated into the form 
\begin{equation}\label{000-intro-3}
(\d_t+T_0)\Phi +
\begin{pmatrix}
0 & -T_1\\
T_1 & 0
\end{pmatrix}
\Phi = \mbox{low order terms}, \quad \Phi=(\Phi_1, \Phi_2)^{\top}\!\in\! \Big(H^s(\R^{d-1})\Big)^2, 
\end{equation}
where $\Phi=(\Phi_1,\Phi_2)^{\top}$ is a new pair of unknowns obtained via a transformation from $(\eta,\psi)$, and no loss of regularity occurs for \eqref{000-intro-3}.

Our key technical results on the cylindrical DN operator are summarized as follows:
\begin{enumerate}[label=\arabic*.]
    \item The construction of DN operator for cylindrical surface $G[\eta](\psi)$, and its Sobolev estimates stated in Lemma \ref{lemma:GSob};
    \item The shape differential of $\eta\mapsto G[\eta](\psi)$, which is Theorem \ref{thm:shape};
    \item The paralinearization of $G[\eta](\psi)$ obtained in Theorem \ref{thm:paraG}.
\end{enumerate}
In fact, we show in Section \ref{ssec:flat} that, if the domain in (\ref{cyl-harmonic}) is a flat cylinder described by $\eta(t,z)\equiv R>0$, then the solution to (\ref{cyl-harmonic}), denoted as $\Psi\equiv\Upsilon_R$, and its cylindrical DN operator takes the form 
\begin{align}\label{cyl-0th}
\Upsilon_R(z,r) = \dfrac{\mI_0(r \d_z)}{\mI_0(R \d_z )} \psi(z) \ \ \text{ and } \ \ G[R](\psi) = \d_r \Upsilon_R \vert_{r=R} = \dfrac{\mI_0^{\,\prime}(R \d_z)}{\mI_0(R \d_z )} \d_z \psi(z),
\end{align}
where $\mI_0(x)$ is the modified Bessel function of the first kind with eigenvalue $0$. This provides a contrast to the operators (\ref{wwphi})--(\ref{wwDN}) for the water-wave problem. Clearly the kernel may have singularities and some extra care is needed in order to obtain optimal estimates for the DN operator even in the flat cylinder case.

\subsection{ Historic remarks}
%The flows with free streamlines are one  of the classical problems in hydrodynamics \cite{Gilbarg}, \cite{Birkhoff}.
%In early stages this problem for the steady flows was solved for perfect fluids in 2D using potential theory and 
%conformal mappings. The time dependent 3D problem is extremely complicated \cite{Gilbarg} and remained largely open.  

%The aim of this paper is to settle it for the axisymmetric flows  of infinite jets.  

The jet flows are a classical topic in hydrodynamics \cite{Gilbarg}.
They naturally form a class of  free streamline problems \cite{Birkhoff}.
%\todo{More references}
At the early stages the problem was studied in the steady state regime. 
Garabedian and Spencer \cite{Garabedian} considered the case of axisymmetric case with partially free streamline with physical motivation coming from a 
mathematical model of a missile moving through water at high speed.
Some explicit solutions of jet flows are constructed in  \cite{Andreev}.

In contrast to the jet problem, the water wave problem has attracted lots of attention in the recent years, see for instance \cite{ABZ}, \cite{ABZ-inv}, \cite{lannes},  \cite{GMSh}, \cite{Wu}.
For a recent account of the theory we also refer the reader to \cite{lannes2} and \cite{DI}.

In 1952 John  \cite{John} constructed the first examples of axially symmetric unsteady flows in three dimensions, 
a paper that has been largely overlooked. 
Later, Osvyannikov  \cite{Ovsyannikov}  constructed another class of explicit potential free boundary flows
of ideal fluid jets in the form of an  infinite round cylinder. 
In all these examples the surface tension on the free boundary is not present. 

The construction in \cite{Ovsyannikov} was subsequently generalized to flows with angular momentum and surface tension in \cite{Andreev}. 
In particular they modelled a flow which appears in applications as a jet of fluid flowing between 
two coaxial cylinders with surface tension. 
%The internal cylinder may  disappear in finite 
%time resulting a hydraulic shock after which the flow cannot continue 
%in Ovsjannikov regime due to the jump of total kinetic energy. (the nature of such flows can be extremely complex)

\subsection{Organization of paper}
The paper is organized as follows: In Section \ref{sec:DN-cyl} we construct the DN operator for cylindrical surface of revolution. One of the most important non-trivial results is the formula for shape differential given in Theorem \ref{thm:shape}. Then the paralinearization of cylindrical DN operator $G[\eta](\psi)$ is carried out in Section \ref{sec:DN-paralinearization}. The main theorem there  is Theorem \ref{thm:paraG} which later implies the existence of symmetrizer $\mathbf{S}$. Section \ref{sec:para-Zakh} contains the paralinearization of Zakharov's system, Proposition \ref{prop:JohnWick-4}, in terms of the Alinhac's good unknowns $(\eta, U)$, where $U=\psi-T_{\mB}\eta$. The main aim of Section \ref{sec:symbolic-five}  is to seek a suitable symmetrizer $\mathbf{S}=\big(\begin{smallmatrix} T_{p} & 0 \\ 0 & T_{q}  \end{smallmatrix}\big)$ for the paralinearised system (\ref{zak-para}). Here, $T_{p}$ and $T_q$ are two paradifferential operators chosen such that $\mathscr{M}_{\lambda, \ell}$ is transformed into a skew-symmetric matrix upon conjugation by $\mathbf{S}$. This then guarantees that Gr\"onwall's inequality can be employed without any loss of regularity. The key argument is to use the symbolic calculus proposed in \cite{ABZ}. The main results of this section is Proposition \ref{prop:qpgamma} and Theorem \ref{thm:symPHI}. Section \ref{sec:mollified-system} is devoted to the construction of mollified system (\ref{eq:the-L-system-ep}), which guarantees the existence of smooth approximate solutions. Moreover the proof of uniform a-priori estimate, stated in Theorem \ref{thm:InvEst}, is also obtained in this section. Finally, in Section \ref{sec:cauchy}, we prove the short time existence of jet flow, which is Theorem \ref{thm:exist}. Then in Lemma \ref{lemma:cit} and Theorem \ref{thm:uniq}, the continuity in time for the solution, the uniqueness, and stability with respect to initial data are established. 

%Probably the most difficult of the free boundary problems to treat exactly are those involving unsteady flows!\cite{Gilbarg} p356

%---------------------
%     Section
%---------------------
\section{DN Operator for Surface of Revolution}\label{sec:DN-cyl}
%\todo{DN not D-N, unify the notation}
The main purpose of this section is to construct the Dirichlet-Neumann operator for surfaces of revolution. For simplicity, we will suppress the time variable $t$, and denote $\eta(z)\equiv \eta(t,z)$, $\psi(z)\equiv \psi(t,z)$, $\Psi(z,r)\equiv \Psi(t,z,r)$ throughout the rest of this section.

Let $\eta(z)\vcentcolon \R \to (0,\infty)$ be such that $\eta(z)\ge c$ for some $c>0$. We set the domain:
\begin{equation}\label{Zdomain}
\Omega_{\eta}\vcentcolon= \{ (z,r)\in \R^2\vcentcolon 0 < r \le \eta(z) \}.
\end{equation}
For a given function $\psi(z)\vcentcolon \R \to \R $, let $\Psi(z,r)$ be the solution to the elliptic boundary-value problem:
\begin{subequations}\label{ellip}
\begin{align}
& -L\Psi \vcentcolon= - \d_z(r \d_z \Psi) - \d_r(r \d_r \Psi) = 0 \quad && \text{ in } \ (z,r)\in \Omega_{\eta},\label{Eq-cyl}\\
& \Psi(z,\eta(z)) = \psi(z), \qquad \d_r\Psi(z,0) = 0 && \text{ for } \ z\in\R. \label{Bd-D}
\end{align}
\end{subequations}
%where the Neumann boundary condition $\d_r\Psi(z,0) = 0$ is implied from the facts that harmonic function $\phi$ is axially symmetric and analytic at $\{z=0\}$. Suppose that for given $\eta(z)$ and $\psi(z)$, there exists a unique $\Psi\in \mC^1(\overline{\Omega}_\eta)$.
Then according to (\ref{DN-cyl}), the Dirichlet-Neumann operator $G[\eta](\cdot)$ associated with the surface $\{r=\eta(z)\}$ is given by:
\begin{equation}\label{GDN}
G[\eta] (\psi) %\vcentcolon= \big( \nabla F \cdot \nabla \phi\big) \big\vert_{\x\in \d \Omega_{\eta}} 
\vcentcolon= \big( \d_r \Psi - \d_z \eta \d_z \Psi \big)\big\vert_{r=\eta(z)}.
\end{equation}

\subsection{The equation for cylinder posed in flat strip}\label{ssec:flat}
\subsubsection{The Poisson kernel}\label{sssec:poisson-flat}
First, we consider the flat case where $\eta(z)=R$ for some $R>0$ which indicates that the surface is a cylinder of radius $R$. In this case, the $(z,r)$-domain is a flat strip $\Omega_{\eta}=\Omega_R\vcentcolon= \R\times[0,R]$. Given a boundary data $\psi(z)$, we set $\Upsilon_R(z,r)$ to be solution to the elliptic problem:
\begin{subequations}\label{flat}
	\begin{align}
		&-L\Upsilon_R = - \d_z(r\d_z \Upsilon_R) - \d_r(r \d_r \Upsilon_R)  = 0 && \text{in } \ (z,r)\in\Omega_R,\label{flata}\\
		&\Upsilon_R(z,R)=\psi(z), \qquad \d_r\Upsilon_R(z,0)=0 && \text{for } \ z\in\R.\label{flatb}
	\end{align}
\end{subequations} 
Let $\wh{\Upsilon}_R(\xi,r)$ be the Fourier transform of $\Upsilon_R$ in the variable $z$:
\begin{equation*}
	\wh{\Upsilon}_R(\xi,r) = \mF\{ \Upsilon_R(\cdot,r) \}(\xi) \vcentcolon= \dfrac{1}{\sqrt{2\pi}} \int_{\R} \Upsilon_R(z,r) e^{-i z \xi }\,\dif z. 
\end{equation*}
Taking Fourier transform on (\ref{flat}), then multiplying both sides by $-r$, we obtain that
\begin{subequations}\label{FPhi}
	\begin{align}
		&r^2 \d_r^2 \wh{\Upsilon}_R(\xi,r)+ r \d_r \wh{\Upsilon}_R(\xi,r)-\xi^2 r^2 \wh{\Upsilon}_R(\xi,r) = 0 && \text{in } \ (\xi,r)\in \Omega_R,\label{FEq}\\
		& \wh{\Upsilon}_R(\xi,R) = \wh{\psi}(\xi), \quad \d_r \wh{\Upsilon}_R(\xi,0) = 0 && \text{for } \ \xi\in\R.\label{FBd-D}
	\end{align}
\end{subequations}
%%%%%%%%%%%%%%%%%%%%%%%%%%%%%%%%%%%%%%%%
%%%%%%%%%% START OF COMMENT %%%%%%%%%%%%
%%%%%%%%%%%%%%%%%%%%%%%%%%%%%%%%%%%%%%%%
\iffalse
If we set $f_{\xi}(\sigma)\vcentcolon= \wh{\Upsilon}_R(\xi,\frac{\sigma}{\xi})$, then $\sigma\mapsto f_{\xi}(\sigma)$ satisfies 
\begin{subequations}
	\begin{align}
		&|r \xi|^2 f_{\xi}^{\prime\prime}(r\xi) +  r \xi f_{\xi}^{\prime}(r\xi) - |r\xi|^2 f_{\xi}(r\xi) = 0 && \text{for } \ (\xi,r)\in\Omega_R,\\
		& f_{\xi}(\xi) = \wh{\psi}(\xi), \quad f^{\prime}_{\xi}(0)=0 && \text{for } \ \xi\in\R.
	\end{align}
\end{subequations}
This implies that $\sigma\mapsto f_{\xi}(\sigma)$ solves the modified Bessel's equation with eigenvalue $0$. Thus if we set $\sigma\equiv r\xi$, then one has
\begin{align*}
	\wh{\Upsilon}_R(\xi,r) =& f_{\xi}(\sigma) = \tilde{c}_1(\xi) \mathcal{J}_0(i\sigma) + \tilde{c}_2(\xi) \mathcal{Y}_0(i\sigma) =  \tilde{c}_1(\xi) \mathcal{J}_0(ir\xi) + \tilde{c}_2(\xi) \mathcal{Y}_0(ir\xi)\\
	=&c_1(\xi) \mI_0(r\xi) + c_2(\xi) \mathcal{K}_0(r\xi). 
\end{align*}
\fi
%%%%%%%%%%%%%%%%%%%%%%%%%%%%%%%%%%%%%%
%%%%%%%%%% END OF COMMENT %%%%%%%%%%%%
%%%%%%%%%%%%%%%%%%%%%%%%%%%%%%%%%%%%%% 
It follows that $\wh{\Upsilon}_R$ is given by
\begin{equation*}
	\wh{\Upsilon}_R(\xi,r) = c_1(\xi) \mI_0(r\xi) + c_2(\xi) \mathcal{K}_0(r\xi),
\end{equation*}
where $\mI_0(\cdot)$ and $\mathcal{K}_0(\cdot)$ are respectively the modified Bessel's functions of first and second kind with eigenvalue $0$. Since $\mathcal{K}_0^{\,\prime}(x)\to -\infty$ and $\mI_0^{\,\prime}(x)\to 0$ as $x\to 0^{+}$, the Neumann boundary condition in (\ref{FBd-D}) implies that $c_2\equiv 0$. On the other hand, by the Dirichlet boundary condition in (\ref{FBd-D}) and $\mI_0(0)=1$, we have $c_1(\xi)=\frac{\wh{\psi}(\xi)}{\mI_0(R\xi)}$. Thus
\begin{equation}\label{hatPsi}
	\wh{\Upsilon}_R(\xi,r)= \wh{\psi}(\xi) \dfrac{\mI_0(r\xi)}{\mI_0(R\xi)}.
\end{equation}
Note that for real number $x\in\R$, $\mI_0(x)$ can be expressed as follows
\begin{equation}\label{I0}
	\mI_0(x)= \sum_{k=0}^{\infty} \dfrac{x^{2k}}{4^{k}(k!)^2} = \dfrac{1}{\pi} \int_{0}^{\pi} e^{x\cos\theta}\, \dif \theta.
\end{equation}
By the integral representation in (\ref{I0}), it follows that
\begin{align}
	\mI_0^{\,\prime}(x) =& \dfrac{1}{\pi} \int_{0}^{\pi} e^{x\cos\theta} \cos\theta \, \dif \theta = \dfrac{x}{2}\sum_{k=0}^{\infty}\dfrac{x^{2k}}{4^k (k+1)!k!} = \mI_1(x),\label{I0'}\\
	\mI_0^{(k)}(x) =& \dfrac{1}{\pi} \int_{0}^{\pi} e^{x\cos\theta} \cos^k\theta \, \dif \theta \qquad \text{for } \ k\ge 2.\label{I0-k}
\end{align}
where $\mI_1(x)$ is the modified Bessel's function of first kind with eigenvalue $1$. Let
\begin{align}\label{K-flat}
	K_R(z,r)\vcentcolon=&\dfrac{1}{\sqrt{2\pi}}\mF^{-1}\Big\{ \dfrac{\mI_0(r \cdot )}{\mI_0(R\cdot)} \Big\}(z) =  \!\int_{\R}\! \dfrac{\mI_0(r\xi)}{\mI_0(R\xi)} \dfrac{e^{iz\xi}}{2\pi} \, \dif \xi = \!\int_{0}^{\infty}\!\! \dfrac{\mI_0(r\xi)}{\mI_0(R\xi)}\dfrac{\cos(z \xi)}{\pi} \, \dif\xi,
\end{align}
where we used the fact that $\mI_0(-x)=\mI_0(x)$. Then, taking inverse Fourier transform with respect to $\xi\leftrightarrow z$ on (\ref{hatPsi}) we obtain that
\begin{equation}\label{fPsi}
	\Upsilon_R(z,r) = \int_{\R}\!\! \psi(x) K_R(z-x,r)\, \dif x.
\end{equation}
It follows that the first derivative in $r$ direction is given by:
\begin{subequations}\label{fPsi-l}
	\begin{gather}
		\d_r\Upsilon_R(z,r) = \int_{\R}\!\! \psi(x) \d_r K_R(z-x,r)\, \dif x, \label{fPsi-l1}\\ 
		\text{where } \ \d_r K_R(z,r)\vcentcolon = \frac{1}{\pi}\!\int_{0}^{\infty}\!\! \dfrac{\mI_1(r\xi)}{\mI_0(R\xi)}\xi\cos(z \xi) \, \dif\xi, \quad \text{for } \ r\in[0,R). \label{fPsi-l2} 
	\end{gather}
\end{subequations}
\begin{remark}\label{rem:flatKernel}
	At $r=R$, $\d_r K_R(z,R)$ is not well-defined since $\lim_{\xi\to\infty}\frac{\mI_1(R\xi)}{\mI_0(R\xi)}=1$, which implies that the integral (\ref{fPsi-l2}) diverges. But $\d_r K_R(z,r)$ is well-defined for $0\le r<R$. Hence one can construct the Dirichlet-Neumann operator $G[R]$ as the limit of $r\to R^{-}$.
\end{remark}
\begin{remark}\label{rem:scaledPhi}
    Denote $\mS \vcentcolon= \Omega_{1} = \R\times[0,1]$ as the strip domain. We define:
    \begin{equation}\label{Phi}
        \Upsilon(z,y) \vcentcolon= \Upsilon_R(z,yR) = \int_{\R}\psi(x) \cdot \tfrac{1}{R}K_{1}(\tfrac{z-x}{R}, y)\, \dif x \qquad \text{for } \  (z,y)\in\mS.
    \end{equation}
    Since $\Upsilon_R$ solves (\ref{flat}), it follows that $\Upsilon$ is solution to the following problem:
    \begin{subequations}\label{divPhi}
    \begin{align}
        &-\div_{(z,y)} \big( A_R \cdot \nabla_{(z,y)} \Upsilon  \big)=0 && \text{in } \ (z,y)\in\mS,\label{divPhi1}\\
		&\Upsilon(z,1)= \psi(z), \quad \d_y\Upsilon(z,0)=0 && \text{for } \ z\in\R,\label{divPhi2}\\
		\text{where } \ \ &  \nabla_{(z,y)}\vcentcolon= (\d_z,\d_y)^{\top}, \quad \ \ \div_{(z,y)}\vcentcolon= \d_z+ \d_y,  && A_R \vcentcolon= \begin{pmatrix}
            y R^2 & 0 \\ 0 & y  
		\end{pmatrix}.\label{AR}
    \end{align}
    \end{subequations}
\end{remark} 
\subsubsection{Estimates on modified Bessel's functions}

%-------------
%   Aymptotic of Bessels function
%-------------
%\begin{lemma}Let $\mI_0(z)$ be the modified Bessel function as above, then \[\frac1{\mI_0^2(\xi)}\int_1^\xi \mI_0^2(z)dz\sim \frac12, \quad \xi\to +\infty.\]\end{lemma}\begin{proof}We have the well known asymptotic formula \[\mI_0(z)=\frac{e^z}{\sqrt{2\pi z}}(1+O(\frac1z)).\] Then \[\frac1{\mI_0^2(\xi)}\int_1^\xi \mI_0^2(z)dz=\frac{\int_1^\xi\frac{e^{2z}}{{2\pi z}}(1+O(\frac1{z^2}))dz}{\frac{e^{2\xi}}{{2\pi \xi}}(1+O(\frac1{\xi^2}))}\sim \frac{\int_1^\xi\frac{e^{2z}}{{z}}dz}{\frac{e^{2\xi}}{{\xi}}}\sim\frac{\frac{e^{2\xi}}{{\xi}}}{\frac{2e^{2\xi}}{{\xi}}-\frac{e^{2\xi}}{{\xi^2}}}\sim \frac12,\] where the last line follows from L'Hospital's rule.\end{proof}

\begin{proposition}\label{prop:I1I0}
    Let $\mI_0(x)$ be the modified Bessel's function of first kind with eigenvalue $0$. Then for any $k\in\mathbb{N}\cup\{ 0 \}$, $y\in(0,1]$, and $x>0$,
    \begin{equation}\label{I1I0}
        \Big| \dfrac{\mI_0^{(k)}( y x)}{\mI_0(x)} \Big|^2 \le \pi^{2-2 y} \Big( \int_{0}^{\pi} e^{x\cos\theta}\,\dif \theta \Big)^{-2(1-y)} = |\mI_0(x)|^{-2(1-y)}.
    \end{equation}
\end{proposition}
\begin{proof}
    we observe that for any $k\in\mathbb{N}\cup\{0\}$ and $0< y \le 1$, 
    \begin{gather}
		\dfrac{\mI_0^{(k)}( y x)}{\mI_0(x)} = \dfrac{\smallint_{0}^{\pi} e^{ y x \cos\theta} \cos^k\! \theta \, \dif \theta}{\smallint_{0}^{\pi} e^{x\cos\theta}\,\dif \theta} = \Big( \int_{0}^{\pi} e^{x\cos\theta}\,\dif \theta \Big)^{-1+ y} \int_{0}^{\pi} |q(\theta,x)|^{y} \cos^k \!\theta\, \dif \theta,\label{I1I0'}\\
		\text{where } \ q(\theta,x) \vcentcolon= \Big( \int_{0}^{\pi} e^{x\cos\theta}\,\dif \theta \Big)^{-1} e^{x\cos\theta} \ \text{ which implies } \ \int_{0}^{\pi} q(x,\theta) \, \dif \theta = 1.\nonumber
    \end{gather}
    Exponentiating by $\frac{1}{y}$ on both sides of (\ref{I1I0'}), then applying Jensen's inequality,
    \begin{subequations}
    \begin{gather*}
        \Big| \dfrac{\mI_0^{(k)}( y x)}{\mI_0(x)} \Big|^{\frac{1}{y}} %\le \Big( \int_{0}^{\pi} e^{x\cos\theta}\,\dif \theta \Big)^{\frac{y-1}{y}} \Big(\int_{0}^{\pi} |q(\theta,x)|^{y} \cos^k\!\theta\, \dif \theta\Big)^{\frac{1}{y}}
		\le \Big( \dfrac{1}{\pi}\int_{0}^{\pi} e^{x\cos\theta}\,\dif \theta \Big)^{\frac{y-1}{y}}\! \int_{0}^{\pi} q(x,\theta) |\cos\theta|^{\frac{k}{y}}\, \dif \theta \le \pi^{\frac{1}{y}-1} \Big( \int_{0}^{\pi} e^{x\cos\theta}\,\dif \theta \Big)^{\frac{y-1}{y}}.
		%\Big| \dfrac{\mI_0(y x)}{\mI_0(x)} \Big|^{\frac{1}{y}} \le \Big( \int_{0}^{\pi} e^{x\cos\theta}\,\dif \theta \Big)^{\frac{y-1}{y}} \Big(\int_{0}^{\pi} |q(\theta,x)|^{y} \, \dif \theta\Big)^{\frac{1}{y}} \le \Big( \int_{0}^{\pi} e^{x\cos\theta}\,\dif \theta \Big)^{\frac{y-1}{y}}  \pi^{\frac{1}{y}-1} \int_{0}^{\pi} q(x,\theta) \, \dif \theta = \pi^{\frac{1}{y}-1} \Big( \int_{0}^{\pi} e^{x\cos\theta}\,\dif \theta \Big)^{\frac{y-1}{y}}.
    \end{gather*}
    \end{subequations}
    We obtain (\ref{I1I0}) by multiplying both sides of the above inequalities with $2y$.
\end{proof}

\begin{proposition}\label{prop:i1i0}
	Let $\mI_0(x)$ be the modified Bessel's function of first kind with eigenvalue $0$, and denote $\mI_0^{(k)}(x) \vcentcolon= \frac{\dif^k \mI_0}{\dif x^k} (x)$. Then for all $x\in\R$ and $k\in\mathbb{N}\cup\{0\}$, 
	\begin{equation}\label{i1i0}
		\int_{0}^{1}\!\! \Big| \dfrac{\mI_0^{(k)}(y x)}{\mI_0(x)} \Big|^2 \, \dif y  \le 1  \qquad \text{ and } \qquad |x|\int_{0}^{1} \Big| \dfrac{\mI_0^{(k)}(y x)}{\mI_0(x)} \Big|^2 \, \dif y \le 3.
	\end{equation}
\end{proposition}
\begin{proof}
First, it can be verified that $\mI_0^{(k)}(-x)=\mI_0^{(k)}(x)$ for even $k\in\mathbb{N}\cup\{0\}$ and $\mI_0^{(k)}(-x)=-\mI_0^{(k)}(x)$ for odd  $k\in\mathbb{N}$. This implies $x\mapsto |\mI_0^{(k)}(x)|^2$ is an even function for all $k\in\mathbb{N}\cup\{0\}$. Thus it suffices to prove the case only for $x\ge 0$. By the integral representation (\ref{I0-k}), we have for all $k\in\mathbb{N}\cup\{0\}$, $y\in[0,1]$, and $x\ge 0$,
\begin{equation*}
\big| \mI_0^{(k)} (yx)\big| = \bigg|\dfrac{1}{\pi}\int_{0}^{\pi}e^{yx\cos\theta}\cos^k\theta\, \dif \theta\bigg| \le \bigg|\dfrac{1}{\pi}\int_{0}^{\pi}e^{yx\cos\theta}\, \dif \theta\bigg| = \big| \mI_0(yx) \big| \le \big| \mI_0(x) \big|,
\end{equation*}
where in the last inequality, we used the fact that $x\mapsto \mI_0(x)$ is monotone increasing in the domain $x\ge 0$. This proves the first inequality in (\ref{i1i0}).

Next, we prove the second inequality in (\ref{i1i0}). By integral formula (\ref{I0}),
\begin{align}\label{Ipi3}
	|\mI_0(x)| = \dfrac{1}{\pi} \int_{0}^{\pi}\!\! e^{x\cos\theta}\, \dif \theta  \ge  \dfrac{1}{\pi} \int_{0}^{\pi/3}\!\! e^{x\cos\theta}\, \dif \theta \ge \dfrac{1}{3} e^{x/2}\quad \text{for } \ x\ge 0.
\end{align}
Combining the above with Proposition \ref{prop:I1I0}, one has for all $k\in\mathbb{N}\cup \{0\}$ and $x\ge 0$,
\begin{align*}
	x\!\!\int_{0}^{1}\! \Big| \dfrac{\mI_0^{(k)}(yx)}{\mI_0(x)} \Big|^2 \dif y \le x\!\! \int_{0}^{1}\!\! \big|\mI_0(x)\big|^{-2(1-y)} \dif y %\le x \int_{0}^{1} \big(\frac{1}{3} e^{x/2}\big)^{-2(1-y)}\, \dif y 
	\le 9 x e^{-x}\!\! \int_0^1\!\!\! \big(9^{-1} e^x \big)^y \dif y = \dfrac{x(1-9e^{-x})}{x- \ln 9} \le 3. 
\end{align*}
This proves the second inequality of (\ref{i1i0}), which concludes the proof.
\end{proof}
\subsubsection{Elliptic Dirichlet Boundary Estimate}
In what follows, we will denote $\mathscr{H}^k(\Omega_R)$ as the Sobolev space that for each $k\in\R$:
\begin{align*}
	&\mathscr{H}^k(\Omega_R) \vcentcolon= \big\{ f\in L_{\text{loc}}^1(\Omega_R) \ \big\vert\ \|f\|_{\mathscr{H}^k(\Omega_R)} <\infty  \big\},\\
	\text{where } \ &\| f \|_{\mathscr{H}^k(\Omega_R)}^2 \vcentcolon= \sum_{ 0 \le m\le k} \int_{0}^{R}\!\! \|\d_r^{m} f(\cdot,r)\|_{H^{k-m}(\R)}^2 \, r \dif r \quad (m\in\mathbb{N}\cup \{0\}).
\end{align*} 
Moreover, we will also use the notation: $\absm{\xi}\vcentcolon= \sqrt{1+|\xi|^2}$ for $\xi\in\R$.
\begin{proposition}\label{prop:PhiR}
    Suppose $\psi\in H^{s}(\R)$, and let $\Upsilon_R$ be the solution to (\ref{flat}) taking the form (\ref{fPsi}). Then $\Upsilon_R\in \mathscr{H}^{s+1/2}(\Omega_R)$, and there exists a constant $C(s)>0$ depending only on $s>0$ such that for $\nabla \vcentcolon= (\d_z,\d_r)^{\top}$,
    \begin{equation*}
		\| \Upsilon_R \|_{\mathscr{H}^{s+\frac{1}{2}}(\Omega_R)} \!\!\le\! R(1+R) C(s) \| \psi \|_{H^{s}(\R)}, \ \ \|\nabla \Upsilon_R \|_{\mathscr{H}^{s-\frac{1}{2}}(\Omega_R)} \!\!\le\! R(1+R) C(s) \| \d_z \psi \|_{H^{s-1}(\R)}.
    \end{equation*}
\end{proposition}
\begin{proof}
    Using (\ref{hatPsi}), Parseval's theorem, and $\xi^{2k}\le \absm{\xi}^{2k}$ for $\xi\in\R$, we obtain:
    \begin{align}\label{temp:PHI1}
		&\| \Upsilon_R \|_{\mathscr{H}^{s+\frac{1}{2}}(\Omega_R)}^2 = \sum_{0 \le k\le s+\frac{1}{2}} \int_{0}^{R}\! \int_{\R}\!\! \big\{ \absm{\xi}^{2(s+\frac{1}{2}-k)} |\d_r^k \wh{\Upsilon}_R|^2 \big\}\, r \dif \xi \dif r\\
		%=& \sum_{0\le k\le s+\frac{1}{2}}\int_{0}^{R}\!\!\int_{\R}\! \absm{\xi}^{2s+1-2k} |\wh{\psi}(\xi)|^2 \xi^{2k} \Big|\dfrac{\mI_0^{(k)}(r\xi)}{\mI_0(R\xi)}\Big|^2 \, r \dif \xi \dif r\nonumber\\
		=& \int_{\R}\! \absm{\xi}^{2s} |\wh{\psi}(\xi)|^2\!\!\!\! \sum_{0 \le k\le s+\frac{1}{2}}\int_{0}^{R}\!\! \absm{\xi}^{1-2k} \xi^{2k} \Big|\dfrac{\mI_0^{(k)}(r\xi)}{\mI_0(R\xi)}\Big|^2  \, r \dif r \dif \xi\nonumber\\
		\le& \int_{\R}\! \absm{\xi}^{2s}|\wh{\psi}(\xi)|^2 \Big\{\!\!\sum_{0 \le k\le s+\frac{1}{2}}\int_{0}^{R}\!\! \absm{\xi} \Big|\dfrac{\mI_0^{(k)}(r\xi)}{\mI_0(R\xi)}\Big|^2  \, r \dif r\Big\}\dif \xi =\vcentcolon \int_{\R}\! \absm{\xi}^{2s} |\wh{\psi}(\xi)|^2 \mathfrak{k}_{s}(\xi)\, \dif \xi.\nonumber
    \end{align}
    Applying the change of variable $y=\tfrac{r}{R}\in[0,1]$, and using the fact that $\xi^{2k}\le \absm{\xi}^{2k}$, it follows from Proposition \ref{prop:i1i0} that for all $\xi\in\R$ and $s > 0$,
    \begin{align}\label{temp:Jxi}
		\mathfrak{k}_{s}(\xi) \vcentcolon=& \!\!\!\sum_{0 \le k\le s+\frac{1}{2}}\int_{0}^{R}\!\! \absm{\xi} \Big|\dfrac{\mI_0^{(k)}(r\xi)}{\mI_0(R\xi)}\Big|^2  \, r \dif r =\!\!\! \sum_{0 \le k\le s+\frac{1}{2}}\int_{0}^{1}\!\! R^2 \sqrt{1+\xi^2} \Big|\dfrac{\mI_0^{(k)}(y R \xi)}{\mI_0(R\xi)}\Big|^2  \,  y \dif y\\
		\le & \sqrt{2}\!\!\!\sum_{0 \le k\le s+\frac{1}{2}}\!\!\!\! \big\{ R^2 + 3 R  \big\} \le R(1+R)C(s).\nonumber
    \end{align}
    Substituting the above estimate into (\ref{temp:PHI1}), we obtain that
    \begin{align*}
		\| \Upsilon_R \|_{\mathscr{H}^{s+\frac{1}{2}}(\Omega_R)}^2 \le \int_{\R}\!\!\absm{\xi}^{2s}|\wh{\psi}(\xi)|^2 \mathfrak{k}_{s}(\xi)\, \dif \xi 
		%\le R(1+R)C(s) \int_{\R}\! \absm{\xi}^{2s} |\wh{\psi}(\xi)|^2 \dif \xi 
		\le R(1+R)C(s) \| \psi \|_{H^{s}(\R)}^2.
    \end{align*}
    By Parseval's theorem and identity (\ref{hatPsi}), we have 
    \begin{align*}
		&\| \d_z \Upsilon_R \|_{\mathscr{H}^{s-\frac{1}{2}}(\Omega_R)}^2 = \sum_{0 \le k\le s+\frac{1}{2}} \int_{0}^{R}\!\!\!\! \int_{\R}\! \absm{\xi}^{2(s-\frac{1}{2}-k)} |\d_r^k \wh{\d_z\Upsilon_R}|^2 \, r \dif \xi \dif r\nonumber\\ 
		=& \sum_{0 \le k\le s-\frac{1}{2}}\int_{0}^{R}\!\!\!\!\int_{\R}\! \absm{\xi}^{2s-1-2k} |\wh{\psi}(\xi)|^2 \xi^{2k+2} \Big|\dfrac{\mI_0^{(k)}(r\xi)}{\mI_0(R\xi)}\Big|^2 \, r \dif \xi \dif r.
    \end{align*}
    Since $\xi^{2k}\le \absm{\xi}^{2k}$ for all $\xi\in\R$, it follows from (\ref{temp:Jxi}) that
    \begin{align*}
		&\| \d_z \Upsilon_R \|_{\mathscr{H}^{s-\frac{1}{2}}(\Omega_R)}^2  = \sum_{0 \le k\le s-\frac{1}{2}}\int_{0}^{R}\!\!\!\!\int_{\R}\! \absm{\xi}^{2s-1-2k} |\wh{\psi}(\xi)|^2 \xi^{2k+2} \Big|\dfrac{\mI_0^{(k)}(r\xi)}{\mI_0(R\xi)}\Big|^2 \, r \dif \xi \dif r\nonumber\\
		\le & \sum_{0 \le k\le s-\frac{1}{2}}\int_{0}^{R}\!\!\!\!\int_{\R}\! \absm{\xi}^{2s-2} \xi^2 |\wh{\psi}(\xi)|^2 \absm{\xi} \Big|\dfrac{\mI_0^{(k)}(r\xi)}{\mI_0(R\xi)}\Big|^2 \, r \dif \xi \dif r %\nonumber\\ =& \int_{\R}\! \absm{\xi}^{2s-2}\xi^2 |\wh{\psi}(\xi)|^2 \Big\{\!\! \sum_{0 \le k\le s-\frac{1}{2}} \int_{0}^{R}\!\! \absm{\xi} \Big|\dfrac{\mI_0^{(k)}(r\xi)}{\mI_0(R\xi)}\Big|\, r\dif r \Big\} \dif \xi \nonumber\\
		=\!\! \int_{\R}\! \absm{\xi}^{2s-2} \xi^2 |\wh{\psi}(\xi)|^2 \mathfrak{k}_{s}(\xi) \, \dif \xi\nonumber\\
		\le & R(1+R)C(s)\!\! \int_{\R}\! \absm{\xi}^{2s-2} \xi^2 |\wh{\psi}(\xi)|^2 \dif \xi %= R(1+R)C(s)\!\! \int_{\R}\! \absm{\xi}^{2s-2} |\wh{\d_z \psi}(\xi)|^2 \dif \xi 
		= R(1+R)C(s) \| \d_z \psi \|_{H^{s-1}(\R)}^2.
    \end{align*}
    Similarly, we also have
    \begin{align*}
        &\| \d_r \Upsilon_R \|_{\mathscr{H}^{s-\frac{1}{2}}(\Omega_R)}^2 = \sum_{0 \le k\le s-\frac{1}{2}} \int_{0}^{R}\!\! \int_{\R}\! \absm{\xi}^{2(s-\frac{1}{2}-k)} |\d_r^{k+1}\wh{\Upsilon}_R|^2 r \dif \xi \dif r\\
		= & \sum_{0 \le k\le s-\frac{1}{2}} \int_{0}^{R}\!\! \int_{\R}\! \absm{\xi}^{2s-1-2k} |\wh{\psi}(\xi)|^2 \xi^{2k+2} \Big|\dfrac{\mI_0^{(k+1)}(r\xi)}{\mI_0(R\xi)}\Big|^2 \, r \dif \xi \dif r\\
		\le & \sum_{0 \le k\le s-\frac{1}{2}} \int_{0}^{R}\!\!\! \int_{\R}\! \absm{\xi}^{2s-2} \xi^2 |\wh{\psi}(\xi)|^2 \absm{\xi} \Big|\dfrac{\mI_0^{(k+1)}(r\xi)}{\mI_0(R\xi)}\Big|^2 \, r \dif \xi \dif r %\\ =&  \int_{\R}\! \absm{\xi}^{2s-2} \xi^2 |\wh{\psi}(\xi)|^2 \Big\{\!\! \sum_{0 \le k\le s-\frac{1}{2}} \int_{0}^{R}\! \absm{\xi} \Big| \dfrac{\mI_{0}^{(k+1)}(r\xi)}{\mI_0(R\xi)} \Big|^2 \, r\dif r \Big\} \dif \xi\\
		= \!\! \int_{\R}\! \absm{\xi}^{2s-2} \xi^2 |\wh{\psi}(\xi)|^2 \mathfrak{k}_{s+1}(\xi) \dif \xi\\
		\le & R(1+R)C(s) \int_{\R}\! \absm{\xi}^{2s-2}\xi^2 |\wh{\psi}(\xi)|^2 \dif \xi %=C(s) \int_{\R}\! \absm{\xi}^{2s-2} |\wh{\d_z \psi}(\xi)|^2 \dif \xi 
		= R(1+R)C(s) \| \d_z \psi \|_{H^{s-1}(\R)}^2.
    \end{align*}
    Combining the estimates for $\d_z \Upsilon_R$ and $\d_r\Upsilon_R$, we obtain that
    \begin{equation}\label{trace}
        \|\nabla \Upsilon_R\|_{\mathscr{H}^{s-\frac{1}{2}}(\Omega_R)} \le  R(1+R)C(s) \|\d_z \psi\|_{H^{s-1}(\R)}.
    \end{equation}
    This proves the proposition.
\end{proof}

\begin{corollary}\label{corol:Phi}
	Suppose $\psi\in H^{s}(\R)$. Let $\Upsilon$ be the solution to (\ref{divPhi}) in Remark \ref{rem:scaledPhi}. Then $\Upsilon\in \mathscr{H}^{s+\frac{1}{2}}(\mS)$, where $\mS\vcentcolon = \Omega_1=\R\times[0,1]$. Moreover there exists a constant $C=C(R,s)>0$ independent of $\psi$ such that for $\nabla \vcentcolon= (\d_z,\d_y)^{\top}$,
	\begin{equation*}
		\| \Upsilon \|_{\mathscr{H}^{s+\frac{1}{2}}(\mS)} \le C \| \psi \|_{H^{s}(\R)}, \ \text{ and } \ \|\nabla \Upsilon \|_{\mathscr{H}^{s-\frac{1}{2}}(\mS)} \le C \| \d_z \psi \|_{H^{s-1}}.
	\end{equation*}
\end{corollary}
\begin{proof}
	Denote $k=s+\frac{1}{2}$, and let $\absm{D}\vcentcolon= \sqrt{1+|\d_z|^2}$ be the Fourier multiplier. It can be verified from (\ref{Phi}) that $\d_y^m \Upsilon (z,y) = R^m \d_r^m \Upsilon_R (z,yR)$. Thus integrating this in $(z,y)\in\mS$ with the weight $y$, and by the substitution $r=yR\in[0,R]$, we have
	\begin{align*}
		\int_{0}^1\!\! \int_{\R}\! |\absm{D}^{k-m}\d_y^m \Upsilon(z,y)|^2 \, y\dif z \dif y %= R^{2m}\int_{0}^1\!\! \int_{\R}\! |\absm{D}^{k-m}\d_r^m \Upsilon_R (z,yR)|^2 \, y\dif z \dif y
		= R^{2(m-1)}  \int_{0}^{R}\!\!\! \int_{\R} \! | \absm{D}^{k-m} \d_r^{m} \Upsilon_R (z,r) |^2\, r\dif z \dif r 
	\end{align*}
	Summing the above over $0 \le m\le k$, and using Proposition \ref{prop:PhiR}, we get 
	\begin{align*}
		\|\Upsilon\|_{\mathscr{H}^{k}(\mS)}^2 =& \!\!\sum_{0\le m \le k} \int_{0}^1\!\! \int_{\R}\! |\absm{D}^{k-m}\d_y^m \Upsilon|^2 \, y\dif z \dif y %\\ =& \sum_{0\le m \le k}\!\!\!\! R^{2m-2}\!\! \int_{0}^{R}\!\!\! \int_{\R} \! | \absm{D}^{k-m} \d_r^{m} \Upsilon_R |^2\, r\dif z \dif r\\
		= \sum_{0\le m \le k}\!\!\!\! R^{2(m-1)}\!\! \int_{0}^{R}\!\!\!\! \| \d_r^{m} \Upsilon_R (\cdot, r) \|_{H^{k-m}(\R)}^2\, r \dif r\\
		\le& \sup\limits_{0\le m \le k} R^{2(m-1)}\|\Upsilon_R\|_{\mathscr{H}^k(\Omega_R)}^2 \le  (1+\tfrac{1}{R}) \sup\limits_{0\le m \le s+\frac{1}{2}} R^{2m} C(s) \|\psi\|_{H^s(\R)}^2.
	\end{align*}
	The estimate for $\nabla \Upsilon=(\d_z \Upsilon, \d_y \Upsilon)^{\top}$ can be obtained in the exact same way.
	%Similarly for $\d_z \Upsilon$, we also have  \begin{align*} &\|\d_z\Upsilon\|_{\mathscr{H}^{k-1}(\mS)}^2\\ =& \sum_{0\le m \le k-1} \int_{0}^1\!\! \int_{\R}\! |\absm{D}^{k-m}\d_y^m \d_z\Upsilon|^2 \, y\dif z \dif y = \sum_{0\le m \le k-1}\!\!\!\! R^{2(m-1)}\!\! \int_{0}^{R}\!\!\!\! \| \d_r^{m} \d_z\Upsilon_R(\cdot,r)\|_{H^{k-m}(\R)}^2\, r \dif r\\ \le& \sup_{0\le m\le k} R^{2(m-1)} \|\d_z \Upsilon_R\|_{\mathscr{H}^{k-1}(\Omega_R)}^2 \le (1+\tfrac{1}{R}) \sup_{0\le m \le s+\frac{1}{2}} R^{2m} C(s)\|\d_z\psi\|_{H^{s-1}(\R)}^2. \end{align*}
\end{proof}

\subsection{The general surface of revolution}

\subsubsection{Existence of weak solution in non-flat strip}
Suppose that the surface of revolution is described by the equation $r-\eta(z)=0$. For a given boundary data $\psi$, we set $\Upsilon(z,y)=\Upsilon_R(z,yR)$ to be the solution of (\ref{divPhi}) as mentioned in Remark \ref{rem:scaledPhi}. Using this, we define 
\begin{equation}\label{barPsi}
	\bar{\Upsilon}(z,r) \vcentcolon= \Upsilon\big(z,\frac{r}{\eta(z)}\big) \qquad \text{for } \ (z,r)\in \Omega_{\eta} \equiv \{ (z^{\prime},r^{\prime}) \,|\, 0 \le r^{\prime} \le \eta(z^{\prime}) \}  
\end{equation}
Then $\bar{\Upsilon}\big(z,\eta(z)\big)=\psi(z)$ for $z\in\R$. Taking the derivatives, one has
\begin{equation*}
	\d_z \bar{\Upsilon}(z,r) = \d_z \Upsilon \big(z,\frac{r}{\eta(z)}\big) - \dfrac{r\eta^{\prime}(z)}{|\eta(z)|^2} \d_y \Upsilon \big(z,\frac{r}{\eta(z)}\big), \quad \d_r \bar{\Upsilon}(z,r) = \dfrac{1}{\eta(z)}  \d_y \Upsilon \big(z,\dfrac{r}{\eta(z)} \big).
\end{equation*}
Taking $L^2$ norm in the domain $(z,r)\in \Omega_{\eta}$, then using the substitution $y=\frac{r}{\eta}$ and the fact that $\frac{r}{\eta}\le 1$, we obtain that
\begin{align*}
	\iint_{\Omega_{\eta}}\!\! |\bar{\Upsilon}|^2\, r \dif r \dif z =& \iint_{\Omega_{\eta}}\!\! |\bar{\Upsilon}(z,\dfrac{r}{\eta(z)})|^2 \, r \dif r \dif z =\int_{\R}\!\int_{0}^{1}\!\! |\eta(z)|^2 |\Upsilon(z,y)|^2 \, y \dif y \dif z  \\ 
	\le & \|\eta\|_{L^{\infty}(\R)}^2 \|\Upsilon\|_{\mathscr{H}^0(\mS)}^2.\\
	\iint_{\Omega_{\eta}}\!\! |\d_z \bar{\Upsilon}|^2 \, r \dif r \dif z \le& 2 \iint_{\Omega_{\eta}} \big|\d_z \Upsilon (z,\dfrac{r}{\eta})\big|^2 \, r \dif r \dif z + 2\iint_{\Omega_{\eta}} \dfrac{|r \d_z\eta|^2}{|\eta|^4} \big|\d_y \Upsilon (z,\dfrac{r}{\eta})\big|^2 \, r \dif r \dif z\\
	=& 2\int_{\R}\! \int_{0}^{1}\!\! \eta^2 |\d_z \Upsilon(z,y)|^2\, y \dif y \dif z + 2\int_{\R}\! \int_{0}^{1} |\d_z\eta|^2 |\d_y\Upsilon(z,y) |^2 \, y^3 \dif y \dif z\\
	\le & 2 \big\{ \|\eta\|_{L^\infty(\R)}^2  + \| \d_z\eta \|_{L^{\infty}(\R)}^2 \big\} \| \nabla \Upsilon \|_{\mathscr{H}^0(\mS)}^2,\\
	\iint_{\Omega_{\eta}}\!\! |\d_r \bar{\Upsilon}|^2 \, r \dif r \dif z = & \iint_{\Omega_{\eta}}\!\! \dfrac{1}{|\eta(z)|^2} |\d_y \Upsilon \big(z,\dfrac{r}{\eta(z)}\big) |^2 \, r \dif r \dif z\\
	=& \int_{\R}\! \int_{0}^{1} |\d_y \Upsilon(z,y)|^2 \, y \dif y \dif z \le  \|\d_y \Upsilon\|_{H^0(\mS)}^2
\end{align*}
Using Corollary \ref{corol:Phi}, we obtain the estimate
\begin{equation}\label{barPsiTrace}
	\|\bar{\Upsilon}\|_{\mathscr{H}^1(\Omega_{\eta})} %\le 2\big\{ \|\eta\|_{L^{\infty}(\R)}^2 + \|\d_z\eta\|_{L^{\infty}(\R)}^2 \big\} \| \Upsilon \|_{\mathscr{H}^1(\mS)} 
	\le C \big\{ \|\eta\|_{L^{\infty}(\R)}^2 + \|\d_z\eta\|_{L^{\infty}(\R)}^2 \big\} \| \psi \|_{H^{\frac{1}{2}}(\R)}.
\end{equation}
where the norm $\|\cdot\|_{\mathscr{H}(\Omega_\eta)}$ is defined as
\begin{equation}
	\| f \|_{\mathscr{H}^s(\Omega_\eta)}^2 \vcentcolon= \sum_{|\alpha|=0}^{s}\iint_{\Omega_\eta}\!\! \big| \nabla_{(z,r)}^{\alpha} f \big| \, r \dif r \dif z.  \qquad \text{for } \ f(z,r)\in H^s(\Omega_{\eta}),
\end{equation}
where $\nabla_{(z,r)}\equiv (\d_z, \d_r)$ and $\alpha=(\alpha_1,\alpha_2)\in\mathbb{N}^2$ is the multi-index with $|\alpha|=\alpha_1+\alpha_2$ and $\nabla_{(z,r)}^{\alpha}\equiv \d_z^{\alpha_1} \d_r^{\alpha_2}$. Let $\mathring{\Psi}(z,r)$ be the weak solution to the problem:
\begin{subequations}\label{nonflat-ext}
	\begin{align}
		&-L\mathring{\Psi} \equiv -\big\{ \d_z (r \d_z \mathring{\Psi}) + \d_r (r \d_r \mathring{\Psi}) \big\} = L\bar{\Upsilon} && \text{for } \ (z,r)\in \Omega_{\eta},\label{nonflat-ext-1}\\
		&\mathring{\Psi}(z,\eta(z)) = 0, \quad \d_r \mathring{\Psi}(z,0) = 0 && \text{for } \ z\in\R.\label{nonflat-ext-2}
	\end{align}
\end{subequations}
By Lax-Milgram theorem, a unique weak solution $\mathring{\Psi}$ exists in the space:
\begin{equation*}
	\mathfrak{X}(\Omega_\eta)\vcentcolon=\{ \vp \in \mathscr{H}^2(\Omega_{\eta}) \,\,\vert\,\, \vp\vert_{r=\eta(z)}=0, \ \d_r \vp \vert_{r=0}=0 \}.
\end{equation*}
Now setting $\Psi= \mathring{\Psi} + \bar{\Upsilon}$. Then one can verify that $\Psi$ is a unique weak solution to (\ref{ellip}). Using (\ref{barPsiTrace}), one can obtain the estimate
\begin{equation}\label{PhiTrace}
	\|\Psi\|_{\mathscr{H}^1(\Omega_{\eta})} \le  C \big\{ \|\eta\|_{L^{\infty}(\R)}^2 + \|\d_z\eta\|_{L^{\infty}}^2 \big\} \| \psi \|_{H^{\frac{1}{2}}(\R)} .
\end{equation}
\subsubsection{Coordinate transformation into the flat strip}
In order to obtain an estimate for the Dirichlet-Neumann operator acting on certain Sobolev spaces, we first obtain estimates for the solution to the elliptic problem (\ref{ellip}). For this, we transform (\ref{ellip}) into the flat strip domain $\mS$ by introducing the following coordinate transformation: given a curve $r =\eta(z)\in \mC^2(\R)$, we set the transformation $\mT\vcentcolon \Omega_{\eta} \to \mS$ as follows, for $(z,r)\in \Omega_{\eta}$,
\begin{equation}\label{coordT}
	(z,y)=\mT(z,r) = \big(z,\zeta(z,r)\big), \qquad \text{where } \ \zeta(z,r)\vcentcolon=\dfrac{r}{\eta(z)}.
\end{equation}
The corresponding inverse function $\mT^{-1}\vcentcolon \mS\to \Omega_{\eta}$ is given as follows, for $(z,y)\in \mS$,
\begin{equation}\label{coordInv}
	(z,r) = \mT^{-1}(z,y) = (z,\rho(z,y)), \qquad \text{where } \ \rho(z,y)\vcentcolon= y \eta(z).
\end{equation}
Note that the Jacobian for $\mT$ and $\mT^{-1}$ are given by
\begin{equation}\label{Jacobian}
	\Big| \dfrac{\d(z,y)}{\d(z,r)} \Big| = \d_r \zeta, \qquad \Big| \dfrac{\d(z,r)}{\d(z,y)} \Big| = \d_y \rho.
\end{equation}
For function $f(z,r)\in \mC^{1}(\Omega_{\eta})$, we set $\tilde{f}(z,y)\vcentcolon= f(z,\rho(z,y))=f\circ \mT^{-1}(z,y)$. If $\eta\in \mC^1(\R)$, then $\tilde{f}(z,y)\in \mC^1(\mS)$. In light of this, we define the operators $\mT_{\ast}$, $\mT_{\ast}^{-1}$ as the adjoint operators for $\mT$, $\mT^{-1}$ respectively:
\begin{align*}
	\mT_{\ast}\vcentcolon& \ \mC^1(\mS)\to \mC^1(\Omega_{\eta}), && \text{defined as } \quad \mT_{\ast}\tilde{f}\vcentcolon= \tilde{f}\circ \mT = f,\\
	\mT_{\ast}^{-1}\vcentcolon& \ \mC^1(\Omega_{\eta})\to\mC^1(\mS), && \text{defined as } \quad \mT_{\ast}^{-1}f\vcentcolon= f\circ \mT^{-1} = \tilde{f}.
\end{align*}
By construction, $\zeta(z,\rho(z,y))=y$. It follows from chain rule that
\begin{equation*}
	\d_y \rho(z,y) = \dfrac{1}{\d_r\zeta}\Big\vert_{r=\rho(z,y)}, \qquad \d_z \rho (z,y) = - \dfrac{\d_z \zeta}{\d_r \zeta} \Big\vert_{r=\rho(z,y)}.
\end{equation*}
Furthermore, differentiating $f(z,r)= (\mT_{\ast} \circ \mT_{\ast}^{-1} f) (z,r) = (\mT_{\ast} \tilde{f})(z,r) = \tilde{f}(z,\zeta(z,r)) $, one can verify the following relations holds:
\begin{subequations}\label{chain}
	\begin{gather}
		\d_z f(z,r) %= \mT_{\ast}(\d_z \tilde{f})(z,r) + \d_z \zeta(z,r) \cdot \mT_{\ast}(\d_y \tilde{f}) (z,r) = \mT_{\ast}(\d_z \tilde{f})(z,r) - \big\{\mT_{\ast}\big(\dfrac{\d_z \rho}{\d_r \rho}\big) \cdot \mT_{\ast}(\d_y \tilde{f}) \big\} (z,r) = \big\{ \mT_{\ast} \circ \big( \d_z - \dfrac{\d_z \rho}{\d_r \rho} \d_y \big) \tilde{f} \big\} (z,r) =\big\{ \mT_{\ast} \circ \d_1 \tilde{f} \big\}(z,r)
		=\big\{ \mT_{\ast} \circ \d_1 \circ \mT_{\ast}^{-1} f \big\}(z,r), \quad  \d_r f(z,r) %= \d_r \zeta(z,r) \cdot \mT_{\ast}(\d_y \tilde{f})(z,r) = \big\{ \mT_{\ast}\big( \dfrac{1}{\d_y \rho} \big) \cdot \mT_{\ast}(\d_y \tilde{f}) \big\}(z,r) = \big\{ \mT_{\ast} \circ \big(\dfrac{1}{\d_y \rho} \d_y\big) \tilde{f} \big\}(z,r) = \big\{ \mT_{\ast} \circ \d_2 \tilde{f}  \big\}(z,r)
		= \big\{ \mT_{\ast} \circ \d_2 \circ \mT_{\ast}^{-1} f  \big\}(z,r),\label{chain-a}\\
		\text{where} \qquad  \d_1 \vcentcolon= \d_z - \dfrac{\d_z \rho}{\d_y \rho} \d_y, \qquad \d_2 \vcentcolon= \dfrac{1}{\d_y \rho} \d_y. \label{chain-b}
	\end{gather}
\end{subequations}
We denote $v(z,y)\vcentcolon= \Psi(z,\rho(z,y))$. Then applying the operator $\mT_{\ast}^{-1}$ on the equation (\ref{ellip}), and using (\ref{chain}), it can be verified that $-\d_1( \rho \d_1 v ) - \d_2 ( \rho \d_2 v ) =0$. Multiplying this equation with $\d_y \rho$, and integrating by parts, the equation can be rewritten in the divergence form as follows
\begin{align}\label{reformDiv}
	&-\div_{(z,y)}\big(A\cdot\nabla_{(z,y)} v\big)=0, \qquad \text{in } \ (z,y)\in \mS,\\
	\text{where } \ & A=A(\rho) \vcentcolon= \begin{pmatrix}
		\rho\d_y\rho & -\rho \d_z \rho\\
		-\rho\d_z \rho & \frac{\rho(1+|\d_z\rho|^2)}{\d_y\rho}
	\end{pmatrix}=\begin{pmatrix}
		y \eta^2 & -y^2 \eta \d_z \eta\\
		-y^2 \eta \d_z \eta & y(1+y^2|\d_z \eta|^2)
	\end{pmatrix}.\nonumber
\end{align}
Thus we obtain the system of equations in the flat strip $(z,y)\in\mS$:
\begin{subequations}\label{reform}
	\begin{align}
		&-\div_{(z,y)}\big( A\cdot \nabla_{(z,y)} v \big) = 0 && \text{for } \ (z,y) \in \mS,\label{reform-1}\\
		& v(z,1) = \psi(z), \quad \d_y v (z,0) = 0  && \text{for } \ z\in \R. \label{reform-2}
	\end{align}
\end{subequations}
Under this coordinate, the Dirichlet-Neumann operator (\ref{GDN}) is translated into
\begin{equation}\label{DN-flat}
	G[\eta](\psi) = \Big\{ \dfrac{1+y|\d_z \eta|^2}{\eta} \d_y v - \d_z \eta \d_z v \Big\}\Big\vert_{y=1}.
\end{equation}
Let $\mathring{\Psi}$ be solution constructed in (\ref{nonflat-ext}). We set $w(z,y)\vcentcolon=\mathring{\Psi}(z,\rho(z,y))$. Moreover, for a given data $\psi(z)$, recall that $\Upsilon(z,y)$ is the function constructed in (\ref{Phi}) as 
\begin{equation}\label{Phi1}
	\Upsilon(z,y) = \!\! \int_{\R}\!\! \psi(x)\cdot \tfrac{1}{R} K_1(\tfrac{z-x}{R},y) \, \dif x \quad \text{ where } \ K_1(z,y)= \int_{0}^{\infty}\! \dfrac{\mI_0(y\xi)}{\mI_0(\xi)} \dfrac{\cos(z\xi)}{\pi}\, \dif \xi,
\end{equation}
and we also recall $\bar{\Upsilon}(z,r)=\Upsilon(z,\tfrac{r}{\eta})$ in (\ref{barPsi}). Then $\mT_{\ast}^{-1}\bar{\Upsilon}(z,y) = \bar{\Upsilon}\big(z,\rho(z,y)\big) = \Upsilon(z,y)$. Thus applying $\mT_{\ast}^{-1}$ on both side of (\ref{nonflat-ext}), it holds that $w$ solves the problem:
\begin{subequations}\label{wdiv}
	\begin{align}
		&-\div_{(z,y)}\big(A\cdot\nabla_{(z,y)} w\big)= \div_{(z,y)}\big(A\cdot\nabla_{(z,y)} \Upsilon\big) && \text{in } \ (z,y)\in \mS,\label{wdiveq}\\
		& w(z,1)=0, \qquad \d_y w(z,0)=0 && \text{for } z\in\R.\label{wdivbdry}
	\end{align}
\end{subequations}

\begin{definition}\label{def:wWeak}
	Denote the space $\fD_{\mS}$ as:
	\begin{align}\label{DmS}
		%\fD_{\mS}^{\text{c}} \vcentcolon= \big\{ \vp \in \mC^{1}(\mS) \ \big\vert \ & \vp(\cdot,1) = 0 \ \text{in $L^{\frac{1}{2}}(\R)$ in the sense of trace,}\nonumber\\ &\exists N\in\mathbb{N} \text{ s.t. } \forall y\in [0,1]: \ \supp\big(\vp(\cdot,y)\big)\subseteq [-N,N] \big\}.\nonumber\\
		\fD_{\mS} \vcentcolon= \Big\{ \vp \in \mathscr{H}^{1}(\mS) \ \big\vert \ & \lim\limits_{y\to 1^{-}}\vp(\cdot,y) = 0 \ \text{in $L^{\frac{1}{2}}(\R)$ in the sense of trace.} \ \Big\}.
	\end{align}
	Then $w\in \mathscr{H}^1(\mS)$ is said to be a weak solution to the boundary value problem (\ref{wdiv}) if for all $\vp\in \fD_\mS$,
	\begin{equation*}
		\iint_{\mS} \nabla_{(z,y)} \vp \cdot \big( A\nabla_{(z,y)} w \big)\, \dif z \dif y = -\iint_{\mS} \nabla_{(z,y)} \vp \cdot \big( A\nabla_{(z,y)} \Upsilon \big)\, \dif z \dif y. 
	\end{equation*}
	Moreover, $v\in \mathscr{H}^1(\mS)$ is said to be a weak solution to the boundary value problem (\ref{reform}) if $v=w+\Upsilon$ with $w$ being the weak solution to (\ref{wdiv}).
\end{definition}
\begin{definition}\label{def:solOp}
	For given functions $\eta(z)\vcentcolon\R\to (0,\infty)$ and $\vp(z)\vcentcolon \R\to\R$, we denote $E_{\eta}[\vp](z,y)$ as the weak solution of (\ref{reform}) with the boundary data $E_{\eta}[\vp]\vert_{y=1}=\vp$.
\end{definition}
\begin{remark}
	If $\eta\equiv R>0$ is a constant function, then the problem (\ref{reform}) coincides with the scaled flat problem (\ref{divPhi}) as stated in Remark \ref{rem:scaledPhi}. Therefore one has
	\begin{equation*}
		E_{R}[\vp](z,y) = \int_{\R}\!\! \vp(x)  K_R(z-x,yR) \, \dif x =\int_{\R}\!\! \vp(x) \cdot \tfrac{1}{R} K_1(\tfrac{z-x}{R},y) \, \dif x.
	\end{equation*}
	In particular, $\Upsilon$ constructed in Remark \ref{rem:scaledPhi}, (\ref{Phi}) is given by $\Upsilon = E_R[\psi]$.
\end{remark}
\subsubsection{Elliptic estimates for potential function \texorpdfstring{$v$}{v}}
Throughout this section, we will use the following notation for simplicity:
\begin{equation*}
	\nabla \equiv \nabla_{(z,y)} = (\d_z, \d_y)^{\top}, \quad f_z\equiv \d_z f, \quad  f_y\equiv \d_y f \quad \text{for any function } \ f(z,y).
\end{equation*}
Moreover, we will denote $\absm{D}\vcentcolon= (1+|\d_z|^2)^{\frac{1}{2}}$ as the Fourier multiplier acting with respect to the variable $z\in\R$, in other words:
\begin{equation*}
	\absm{D} f (z) \vcentcolon= \dfrac{1}{\sqrt{2\pi}}\int_{\R}\! \wh{f}(\xi) e^{iz\xi} \dif \xi, \quad \text{ for functions } f(z).
\end{equation*}
\begin{remark}\label{rem:eigen}
	Some properties of the matrix $A$ is summarised in this remark:
	\begin{enumerate}[label=\textnormal{(\roman*)},ref=\textnormal{(\roman*)}]
		\item\label{item:eigen1} The characteristic polynomial $\det( A - \lambda I_{2\times2} )=0$ is given by
		\begin{equation*}
			\lambda^2 - y(1+\eta^2+y^2\eta_z^2)\lambda + y^2 \eta^2 =0.
		\end{equation*}
		Solving the quadratic equation, the eigenvalue is given by
		\begin{align*}
			\lambda_1,\, \lambda_2 %= \dfrac{y}{2} \Big\{ (1+ \eta^2 + y^2\eta_z^2 ) \pm \sqrt{\big(1+\eta^2+y^2\eta_z^2\big)^2-4\eta^2} \Big\}
			= \dfrac{y}{2} \Big\{ (1+ \eta^2 + y^2\eta_z^2 ) \pm \sqrt{\big(|1-\eta|^2+y^2\eta_z^2\big)\big(|1+\eta|^2+y^2\eta_z^2\big)} \Big\}.
		\end{align*}
		It can be shown that $\lambda_1>0$ and $\lambda_2>0$ for $y>0$.
		\item\label{item:eigen2} The matrix $A$ can be decomposed as follows:
		\begin{equation*}
			A=\textrm{Q}^{\top} \cdot \textrm{Q}, \qquad \text{where } \ \textrm{Q}=\sqrt{y}\begin{pmatrix}
				\eta& -y\eta_z\\
				0 & 1
			\end{pmatrix}.
		\end{equation*}
		Thus for any vector $w=(w^{z},w^{y})^{\top}$, we have the coercivity:
		\begin{align*}
			w^{\top} \cdot A \cdot w =& w^{\top} \cdot \mathrm{Q}^{\top} \cdot \mathrm{Q} \cdot w = (\mathrm{Q}\cdot w)^{\top} \cdot (\mathrm{Q}\cdot w)\\
			&= | \mathrm{Q} w |^2 = y|\eta w^z - y \eta_z w^y|^2 + y|w^y|^2 \ge 0. 
		\end{align*}
	\end{enumerate}
\end{remark}

\begin{lemma}\label{lemma:cacci}
Denote $\teta\equiv \eta-R$, and suppose $(\teta, \psi) \in H^{s+\frac{1}{2}}(\R)\times H^{s}(\R)$ for some $s>2+\frac{1}{2}$. Moreover, assume that
\begin{align}\label{assum:cacci}
    \|\teta\|_{L^{\infty}(\R)}< R \quad \text{ and } \quad \|\teta_z\|_{L^{\infty}(\R)}<\infty.
\end{align}
Then there exists a constant $C=C(s,R)>0$ depending only on $s$ and $R$ such that for all $0\le m \le s$ the solution $w(z,y)$ to (\ref{wdiv}) satisfies:
\begin{align*}
	\int_{0}^1 \|\nabla w(\cdot,y)\|_{H^{m-\frac{1}{2}}(\R)}^2\, y \dif y %\le C |\fl(\teta)|^{-2m-1} \big| \fU_s(\teta) \big|^{2m+1} \|\psi_z\|_{H^{m-1}(\R)}^2
    \le C \Big| \dfrac{\fU_{s}(\teta)}{\fl(\teta)} \Big|^{2m+1} \|\psi_z\|_{H^{m-1}(\R)}^2,
\end{align*} 
where the functional $\fl(\teta)$ and $\fU_s(\teta)$ are defined as
	\begin{subequations}\label{assum:Mell}
		\begin{align}
			\fl(\teta) \vcentcolon=& \min\Big\{\frac{1}{2}, \frac{(R-\|\teta\|_{L^{\infty}(\R)})^2}{1+2\|\teta_z\|_{L^{\infty}(\R)}^2} \Big\}, \label{assum:M} \\
			\fU_s(\teta) \vcentcolon=& \max\Big\{ R\|\teta\|_{H^{s-\frac{1}{2}}}, R\|\teta_z\|_{H^{s-\frac{1}{2}}} , \|\teta^2\|_{H^{s-\frac{1}{2}}}, \|\teta_z\teta\|_{H^{s-\frac{1}{2}}}, \|\teta_z^2\|_{H^{s-\frac{1}{2}}}\Big\}.\label{assum:ell}
		\end{align}
	\end{subequations}
\end{lemma}
\begin{proof}
	The proof is divided into $3$ steps.
	\paragraph{Step 1: Coercivity and Upper Bound of $A$.} Solely in terms of $\teta\equiv \eta-R$, the matrix $A$ can be rewritten as
	\begin{equation}\label{temp:Ateta}
		A= y\begin{pmatrix}
			\teta^2-2R \teta + R^2 & - y \teta \teta_z + y R \teta_z\\
			- y \teta \teta_z + y R \teta_z & 1+y^2 \teta_z^2
		\end{pmatrix}.
	\end{equation}
	For any $a,\, b \in \R$, one has $|a|\le |a-b| + |b|$. Thus by Cauchy-Schwartz's inequality, $|a|^2 \le (1+\frac{1}{\delta})|a-b|^2 + (1+\delta)|b|^2$ with $\delta\in(0,1)$. This yields that
	\begin{equation*}
		|a-b|^2 \ge \dfrac{\delta}{1+\delta}|a|^2 - \delta |b|^2 \qquad \text{for all } \ a,\,b \in\R \ \text{ and } \ \delta\in(0,1).
	\end{equation*} 
	Setting $\delta_1\vcentcolon= \frac{1}{2}\|\teta_z\|_{L^{\infty}(\R)}^{-2}$ and $\delta_2\vcentcolon= \big(R-\|\teta\|_{L^{\infty}(\R)}\big)^2$. Then by Remark \ref{rem:eigen}\ref{item:eigen2} and assumption (\ref{assum:cacci}), one has that for any $v=(v^z,v^y)\in\R^2$ and $(z,y)\in\mS=\R\times[0,1]$,
	\begin{align*}
		v\cdot(Av) =& y | R v^z + \teta v^z - y \teta_z v^y |^2 + y |v^y|^2\\
		\ge& \tfrac{y\delta_1}{1+\delta_1} |R+\teta|^2 |v^z|^2 - y^3 \delta_1|\teta_z|^2 |v^y|^2 + y |v^y|^2\\
		\ge& \tfrac{y\delta_1}{1+\delta_1} \big(R - \|\teta\|_{L^{\infty}(\R)}\big)^2 |v^z|^2 + y \big(1-\delta_1\|\teta_z\|_{L^{\infty}(\R)}^2\big) |v^y|^2 \ge \tfrac{ \delta_2 \delta_1 }{(1+\delta_1)} y|v^z|^2 + \tfrac{1}{2} y |v^y|^2.    
	\end{align*}
	Until to the end of this proof, we shall denote
	\begin{equation*}
		\fl\vcentcolon= \min\big\{ \dfrac{1}{2}, \dfrac{\delta_1\delta_2}{1+\delta_1} \big\}= \min\Big\{\dfrac{1}{2},\dfrac{(R-\|\teta\|_{L^{\infty}(\R)})^2}{1+2\|\teta_z\|_{L^{\infty}(\R)}^2}\Big\}.
	\end{equation*}
	Then we have the coercivity condition:
	\begin{equation}\label{temp:coerc}
		v\cdot (Av) \ge \fl  y  |v|^2 \qquad \text{for all} \ v\in R^2 \ \text{ and } \ (z,y)\in\mS.
	\end{equation}
	Moreover, recalling $A_R$ in (\ref{AR}), Remark \ref{rem:scaledPhi}, we define the difference matrix $\tA$ as
	\begin{equation}\label{tildeA}
		\tA \vcentcolon= A(\rho)-A_R = y\begin{pmatrix}
			\teta^2 - 2R\teta & -y\teta\teta_z + yR\teta_z \\
			-y\teta\teta_z + yR\teta_z & y^2 \teta_z^2
		\end{pmatrix}.
	\end{equation} 
	Since $\teta\in H^{s+\frac{1}{2}}(\R)$ with $s>\frac{5}{2}$, it follows from the definition (\ref{assum:M}), Propositions \ref{prop:Sobcomp}--\ref{prop:clprod}, and Sobolev embedding theorem $L^{\infty}(\R)\xhookrightarrow[]{} H^{s-2}(\R)$ that for all $y\in[0,1]$,
	\begin{equation}\label{tAEST}
		\|\tA(\cdot,y)\|_{H^{s-\frac{1}{2}}(\R)} \le y \fU_s(\teta), \qquad \|\tA(\cdot,y)\|_{L^{\infty}(\R)} %\le C_{\text{Sob}} \|\tA(\cdot,y)\|_{H^{s-\frac{1}{2}}(\R)} 
		\le y C_{\text{Sob}}\fU_s(\teta). 
	\end{equation}
	\paragraph{Step 2: $H^1$-Estimate.} Since $-\div \big( A_R \nabla  \Upsilon \big)=0 $ by construction (\ref{divPhi}). It follows from (\ref{tildeA}) that equation (\ref{wdiveq}) can be rewritten as
	\begin{equation}\label{wdiveq'}
		-\div\big(A\cdot\nabla w\big)= \div\big(\tA\cdot\nabla \Upsilon\big)
	\end{equation}
	Since $w\in \mathscr{H}^1(\mS)$, we can get a sequence of smooth function $\vp_k\in \mC_c^{\infty}(\mS)$ such that $\nabla \vp_k \to \nabla w$ in $\mathscr{H}^0(\mS)$. Taking $\vp_k$ as the test function in the weak form of (\ref{wdiveq'}), then letting $k\to\infty$, we obtain
	\begin{align*}
		\int_0^1\!\!\int_{\R}\! \nabla w \cdot A \nabla w \, \dif z \dif y = -\int_{0}^{1}\!\! \int_{\R}\! \nabla w \cdot \tA \nabla \Upsilon \, \dif z \dif y.
	\end{align*}
	Thus applying Cauchy-Schwartz's inequality, the coercivity condition (\ref{temp:coerc}), the upper estimate of $\tA$ in (\ref{tAEST}), and Corollary \ref{corol:Phi}, it follows that
	\begin{align}\label{temp:H1}
		\dfrac{\fl}{2}\! \int_0^1\!\! \|\nabla w\|_{L^2}^2\, y \dif y \le& \dfrac{C|\fU_{s}(\teta)|^2}{2\fl}\! \int_{0}^1\!\! \|\nabla \Upsilon\|_{L^2}^2\, y\dif y \le \dfrac{C|\fU_{s}(\teta)|^2}{2\fl} \|\d_z \psi\|_{H^{-\frac{1}{2}}}^2.
	\end{align}
	\paragraph{Step 3: $H^{s+\frac{1}{2}}$-Estimate.} Assume by induction that for $0\le k \le s-\frac{1}{2}$, we have
	\begin{equation}\label{temp:ind}
		\int_{0}^1\! \|\absm{D}^{k-1}\nabla w\|_{L^2(\R)}^2\,y\dif y \le \frac{C}{\fl^{2k}} \big| \fU_s(\teta) \big|^{2k}\|\d_z \psi\|_{H^{k-\frac{3}{2}}(\R)}^2.
	\end{equation}
	Our aim is then to show this inequality with $k-1$ replaced by $k$. Fix $0\le k \le s-\frac{1}{2}$. Let $\chi(\xi)\in C_{c}^{\infty}(\R)$ be such that $\chi=1$ if $|\xi|\le \frac{1}{2}$ and $\chi=0$ if $|\xi|\ge 1$. Moreover, there exists some constant $C>0$ such that $|\chi^{(i)}(\xi)|\le C$ for $i=1,\dotsc,4$. Then for $h\in(0,1)$, we define $\chi_{h}(\xi)\vcentcolon= \chi(h\xi)$. It follows that for $i=1,\dotsc,4$,
	\begin{equation*}
		\supp(\chi_h)\subseteq[-\tfrac{1}{h},\tfrac{1}{h}], \qquad \supp(\chi_h^{(i)})\subseteq[-\tfrac{1}{h},-\tfrac{1}{2h}]\cup[\tfrac{1}{2h},\tfrac{1}{h}].
	\end{equation*}
	Moreover, for each $\xi\in\R$, $\chi_h(\xi)\to 1$ and $\chi^{(i)}(h\xi)\to 0$ as $h\to 0^+$ with $i=1,\dotsc,4$. Using this we define the operators:
	\begin{equation}\label{temp:symb}
		\absm{D}_{h}^{k} f (z,y) \vcentcolon= \dfrac{1}{\sqrt{2\pi}} \int_{\R} \absm{\xi}^{k} \chi_h(\xi) \wh{f}(\xi,y) e^{iz\xi} \, \dif z \quad \text{ and } \quad \Lambda_{h}^{2k}\vcentcolon=(\absm{D}_{h}^{k})^2.
	\end{equation}
	Since this is a linear operator, it follows from the condition (\ref{wdivbdry}) that $\Lambda_h^{2k} w(z,1) = 0$ and $\d_y(\Lambda_h^{2k}w)(z,0)=0$ for all $z\in\R$. Moreover, since $\chi_h$ is a function with compact support, $z\mapsto \big(\Lambda_h^{2k} w, (\Lambda_h^{2k} w)_y \big)(z,y)\in H^{\infty}(\R)$ for each $y\in[0,1]$. Using $\Lambda_h^{2k} w$ as the test function for the weak form of (\ref{wdiveq'}), it follows that
	\begin{align}\label{temp:cacci0}
		\iint_{\mS} (A \nabla w) \cdot (\nabla \Lambda_h^{2k}w)\, \dif z \dif y =  -\iint_{\mS} (\tA \nabla \Upsilon) \cdot (\nabla \Lambda_h^{2k} w)\, \dif z \dif y.
	\end{align}
	For $f=w$ or $-\Upsilon$, and $\mA=A$ or $\tA$, we denote the following integral:
	\begin{equation}\label{temp:I}
		\mathscr{I}(\mA,f) \vcentcolon= \iint_{\mS} (\mA \nabla f) \cdot (\nabla \Lambda_h^{2k}w)\, \dif z \dif y.
	\end{equation}
	Computing the first spatial derivative of $\Lambda_h^{2k} w$, we get
	\begin{align*}
		\nabla \Lambda_h^{2k} w (z,y) %=\dfrac{1}{\sqrt{2\pi}} \int_{\R}\! \absm{\xi}^{2k} |\chi_h(\xi)|^2 \big( i\xi \wh{w}, \wh{w_{y}} \big)^{\top} e^{i\xi z}\, \dif \xi 
		= \dfrac{1}{\sqrt{2\pi}} \int_{\R}\! \absm{\xi}^{2k} |\chi_h(\xi)|^2 \big( \wh{w_{z}}, \wh{w_{y}} \big)^{\top} e^{iz\xi}\, \dif \xi = \absm{D}_h^k \big( \absm{D}_h^k \nabla w \big).
	\end{align*}
	Since $\absm{D}_{h}^k$ is a Fourier multiplier in the variable $z\in\R$, and $A_R$ defined in (\ref{AR}) depends only on $(y,R)$, one has $[\absm{D}_h^{k},A_R]=0$ hence $[\absm{D}_h^{k},\mA]=[\absm{D}_h^{k},\tA]$ for both $\mA=A$ and $\tA$. Using this and Parseval's theorem, $\mathscr{I}(\mA,f)$ can be rewritten as
	\begin{align}\label{temp:I12}
		\mathscr{I}(\mA,f) %=& \iint_{\mS}\! \absm{D}_h^k \nabla w \cdot \big\{ \mA \absm{D}_h^k \nabla f + [\absm{D}_h^k,\mA] \nabla f \big\} \, \dif z \dif y\nonumber\\
		=& \iint_{\mS}\! \absm{D}_h^k \nabla w \cdot \big\{ \mA \absm{D}_h^k \nabla f + [\absm{D}_h^k,\tA] \nabla f \big\} \, \dif z \dif y.
	\end{align}
	By the construction (\ref{temp:symb}), the symbol $\absm{D}_h^k$ belong to the class $\Sym^k$ as defined in Appendix \ref{append:para}. We set $t_0 \vcentcolon= s-\frac{3}{2}$. Since $0\le k\le s-\frac{1}{2}$ and $s>2+\frac{1}{2}$, it follows that $t_0>\frac{1}{2}$ and $0 \le k\le t_0+1$. Thus by the commutator estimate, Proposition \ref{prop:commu}, we get
	\begin{gather}
		\big\|[\absm{D}_h^k,\tA] \nabla f \big\|_{L^2(\R)}  \le \mathcal{N}^{k}_h \| \d_z \tA(\cdot,y)\|_{H^{t_0}(\R)} \|\nabla f\|_{H^{k-1}(\R)}, \label{temp:commu0}\\
		\text{where } \quad \mathcal{N}^{k}_h \vcentcolon= \sup\limits_{\beta\le 4} \sup\limits_{\xi\in\R} \absm{\xi}^{\beta-k} \big|\d_\xi^{\beta}\big( \absm{\xi}^k \chi_h(\xi) \big)\big|.\nonumber
	\end{gather}
	To estimate $\mathcal{N}_h^k$, we use product rule to obtain that
	\begin{align*}
		\d_\xi^{\beta}\big( \absm{\xi}^k \chi_h(\xi) \big) = \sum_{\alpha=0}^{\beta} \binom{\beta}{\alpha} \d_{\xi}^{\beta-\alpha}\absm{\xi}^k \d_{\xi}^{\alpha}\chi_h(\xi) = \sum_{\alpha=0}^{\beta}  \binom{\beta}{\alpha} h^{\alpha} \chi^{(\alpha)}(h\xi) \cdot \d_{\xi}^{\beta-\alpha}\absm{\xi}^k.
	\end{align*}
	Since $\absm{\xi}=\sqrt{1+\xi^2}$, one has $|\d_{\xi}^{\beta-\alpha}\absm{\xi}^k| \le C(s) \absm{\xi}^{k-\beta+\alpha}$ for all $\xi\in\R$, $0\le k\le s-\frac{1}{2}$, and $\alpha,\beta\in \{1,2,3,4\}$. Moreover, one also has $\absm{\xi}\le \frac{\sqrt{2}}{h}$ for $\xi\in \supp\big(\d_\xi^{\alpha} \chi_h\big)$ with $\alpha=1,\dotsc,4$. Thus,
	\begin{align*}
		&\absm{\xi}^{\beta-k}\big|\d_\xi^{\beta}\big( \absm{\xi}^k \chi_h(\xi) \big)\big| \le C(s)\absm{\xi}^{\beta-k}\sum_{\alpha=0}^{\beta}  \binom{\beta}{\alpha} h^{\alpha} \big|\chi^{(\alpha)}(h\xi)\big| \absm{\xi}^{k-\beta+\alpha}\\
		\le& C(s)\sum_{\alpha=0}^{\beta}  \binom{\beta}{\alpha} h^{\alpha} \big|\chi^{(\alpha)}(h\xi)\big| \sqrt{2^{\alpha}}  h^{-\alpha} \le C(s),
	\end{align*}
	with $C(s)>0$ depending only on $s$. Substituting this in (\ref{temp:commu0}), we obtain
	\begin{equation}\label{temp:commu}
		\|[\absm{D}_h^k,\tA] \nabla f \|_{L^2(\R)} 
		\le C(s) \| \tA(\cdot,y)\|_{H^{s-\frac{1}{2}}(\R)} \|\nabla f\|_{H^{k-1}(\R)},
	\end{equation}
	Taking $f=w$ and $\mA=A$ in $\mathscr{I}(\mA,f)$ as defined in (\ref{temp:I12}), it follows from the coercivity condition (\ref{temp:coerc}) and the estimate (\ref{temp:commu}) that 
	\begin{align*}
		&\mathscr{I}(A,w) =  \iint_{\mS} \absm{D}_h^k \nabla w \cdot \big\{ A \absm{D}_h^k \nabla w + [\absm{D}_h^k,\tA] \nabla w \big\} \dif z \dif y \\
		\ge& \dfrac{\fl}{2} \iint_{\mS} |\absm{D}_h^k \nabla w|^2\, y \dif z \dif y - \dfrac{1}{2\fl}\iint_{\mS}\!  \big\|[\absm{D}_h^k,\tA]\nabla w(\cdot,y)\big\|_{L^2(\R)}^2\, \dfrac{\dif y}{y}\\
		\ge& \dfrac{\fl}{2} \iint_{\mS} |\absm{D}_h^k \nabla w|^2\, y \dif z \dif y - \frac{C}{2\fl}\int_0^1\! \|\tA(\cdot,y)\|_{H^{s-\frac{1}{2}}(\R)}^2\|\nabla w(\cdot,y)\|_{H^{k-1}(\R)}^2\,\frac{\dif y}{y}.
	\end{align*}
	On the other hand, setting $f=-\Upsilon$ and $\mA=\tA$ in $\mathscr{I}(\mA,f)$ in (\ref{temp:I12}), we also have
	\begin{align*}
		&\big|\mathscr{I}(\tA,-\Upsilon)\big| = \bigg|\iint_{\mS} \absm{D}_h^k \nabla w \cdot \big\{ \tA \absm{D}_h^k \nabla \Upsilon + [\absm{D}_h^k,\tA] \nabla \Upsilon \big\} \dif z \dif y \bigg| \\
		\le & \dfrac{\fl}{4} \iint_{\mS} |\absm{D}_h^k \nabla w|^2\, y \dif z \dif y + \frac{C}{\fl} \int_{0}^1\! \|\tA(\cdot,y)\|_{L^{\infty}(\R)}^2 \|\absm{D}_h^k \nabla \Upsilon(\cdot,y)\|_{L^2(\R)}^2\, \frac{\dif y}{y}\\ 
		&+\frac{C}{\fl} \int_0^1 \|\tA(\cdot,y)\|_{H^{s-\frac{1}{2}}(\R)}^2 \|\nabla \Upsilon(\cdot,y)\|_{H^{k-1}(\R)}^2\, \frac{\dif y}{y}.
	\end{align*}
	Combining the estimates for $\mathscr{I}(A,w)$ and $\mathscr{I}(\tA,-\Upsilon)$ into (\ref{temp:cacci0}), we obtain that
	\begin{align*}
		\dfrac{\fl}{4}\! \iint_{\mS}\!\! |\absm{D}^k_h \nabla w|^2\, y\dif z \dif y%\nonumber\\ 
		\le& \frac{C}{\fl}\int_0^1 \dfrac{1}{y} \|\tA(\cdot,y)\|_{H^{s-\frac{1}{2}}}^2\|\nabla w(\cdot,y)\|_{H^{k-1}}^2\,\dif y\nonumber\\
		&+ \frac{C}{\fl}\!\! \int_0^1 \dfrac{1}{y} \big\{ \|\tA(\cdot,y)\|_{L^{\infty}}^2 + \|\tA(\cdot,y)\|_{H^{s-\frac{1}{2}}}^2 \big\} \|\nabla \Upsilon(\cdot,y)\|_{H^k}^2\, \dif y.
	\end{align*}
	Applying (\ref{tAEST}) and Corollary \ref{corol:Phi}, we obtain
	\begin{align*}
		\dfrac{\fl}{4} \iint_{\mS}\! |\absm{D}^k_h \nabla w|^2\, y\dif z \dif y \le& \frac{C |\fU_{s}(\teta)|^2}{\fl} \int_0^1 \big\{ \|\nabla \Upsilon(\cdot,y)\|_{H^{k}(\R)}^2 + \|\nabla w(\cdot,y)\|_{H^{k-1}(\R)}^2 \big\}\, y\dif y\\
		\le & \frac{C |\fU_{s}(\teta)|^2}{\fl} \Big\{ \| \psi_z \|_{H^{k-\frac{1}{2}}(\R)}^2 + \int_0^1 \!\|\nabla w(\cdot,y)\|_{H^{k-1}(\R)}^2 \, y\dif y \Big\}.
	\end{align*}
	Taking limit $h\to 0^{+}$, we obtain the estimate:
	\begin{align*}
		\int_{0}^{1}\!\! \|\absm{D}^k\nabla w(\cdot,y)\|_{L^2(\R)}^2\, y \dif y \le \frac{C |\fU_{s}(\teta)|^2}{\fl^2} \Big\{ \| \psi_z \|_{H^{k-\frac{1}{2}}(\R)}^2 \!\! + \int_0^1 \!\|\nabla w(\cdot,y)\|_{H^{k-1}(\R)}^2 \, y\dif y \Big\}.
	\end{align*}
	By the induction hypothesis (\ref{temp:ind}), it follows that
	\begin{align*}
		\int_{0}^1\!\! \| \absm{D}^k \nabla w (\cdot,y)\|_{L^2(\R)}^2\, y \dif y%\\
		\le& \dfrac{C|\fU_{s}(\teta)|^2}{\fl^2} \Big\{ \| \psi_z \|_{H^{k-\frac{1}{2}}(\R)}^2\!\! + \frac{C}{\fl^{2k}} \big| \fU_s(\teta) \big|^{2k}\| \psi_z \|_{H^{k-\frac{3}{2}}(\R)}^2 \Big\}\\
		\le& \dfrac{C}{\fl^{2(k+1)}}\big|\fU_s(\teta)\big|^{2(k+1)} \| \psi_z \|_{H^{k-\frac{1}{2}}(\R)}^2, 
	\end{align*}
	which holds for $0\le k \le s-\frac{1}{2}$. Since $\|f\|_{H^k}\vcentcolon=\|\absm{D}^k f\|_{L^2}$, this proves the lemma.
\end{proof}

\begin{corollary}\label{corol:v0}
	Suppose the same assumption of Lemma \ref{lemma:cacci} holds. Then $v\vcentcolon= w + \Upsilon$ is a weak solution to the problem (\ref{reform}) and for all $0\le m\le s$,
	\begin{equation*}
		\int_{0}^{1}\! \big\{ \|v\|_{H^{m+\frac{1}{2}}(\R)}^2 + \|v_y\|_{H^{m-\frac{1}{2}}(\R)}^2 \big\}\, y\dif y \le C \Big| \dfrac{\fU_s(\teta)}{\fl(\teta)} \Big|^{2m+1} \|\psi_z\|_{H^{m-1}(\R)}^2.
	\end{equation*}
\end{corollary}
\begin{proof}
	Let $v(z,y)\vcentcolon= w(z,y) + \Upsilon(z,y)$. Then it can be verified that $v$ solves the boundary value problem (\ref{reform}) in the weak sense as defined in Definition \ref{def:wWeak}. To prove the estimate, we note that due to Corollary \ref{corol:Phi} and Lemma \ref{lemma:cacci},
	\begin{align*}
		\int_{0}^{1}\!\! \|(v_z,v_y)\|_{H^{m-\frac{1}{2}}(\R)}^2 \, y \dif y \le & 2 \int_{0}^{1}\!\! \big\{ \|(w_z,w_y)\|_{H^{m-\frac{1}{2}}(\R)}^2 + \|(\Upsilon_z,\Upsilon_y)\|_{H^{m-\frac{1}{2}}(\R)}^2 \big\}\, y\dif y\\
		\le & C(s) \Big\{ \Big| \dfrac{\fU_{s}(\teta)}{\fl(\teta)} \Big|^{2m+1} + 1 \Big\} \|\psi_z\|_{H^{m-1}(\R)}^2.
	\end{align*}
	Thus it only remains to estimate the term: $\int_{0}^{1}\!\|v\|_{H^{k}(\R)}^2\, y \dif y$ for $k\in[0,1)$. It follows from the fundamental theorem of calculus and the boundary condition (\ref{reform-2}) that for $k\in[0,1)$ and $ (z,y)\in \R\times [0,1]$.
	\begin{align*}
		|\absm{D}^k w(z,y)| = \Big|\int_{1}^{y} \absm{D}^k w_y(z,\tilde{y}) \, \dif \tilde{y} \Big| \le \int_{y}^1 |\absm{D}^k w_y(z,\tilde{y})| \, \dif \tilde{y}.
	\end{align*}
	Thus it follows from Minkowski's inequality that
	\begin{align*}
		&\Big\{\int_{0}^1\!\int_{\R}\! |\absm{D}^k w(z,y)|^2 \, y\dif z \dif y \Big\}^{\frac{1}{2}} \le  \bigg\{\int_{0}^1\!\int_{\R}\! \Big|\int_{y}^1 |\absm{D}^k w_y(z,\tilde{y})| \, \dif \tilde{y}\Big|^2\, y \dif z \dif y \bigg\}^{\frac{1}{2}}\\
		\le & \int_{0}^1\!\! \bigg\{ \int_{0}^1\!\int_{\R} \! |\absm{D}^k w_y(z,\tilde{y})|^2 \chi_{\{ y \le \tilde{y} \le 1 \}}\, y \dif z \dif y \bigg\}^{\frac{1}{2}}\, \dif \tilde{y}\\
		%=& \int_{0}^1\!\!\Big( \int_{\R}\, |\absm{D}^k w_y(z,\tilde{y})|^2\,\dif z \Big)^{\frac{1}{2}}  \Big(\int_{0}^{\tilde{y}} y \dif y\Big)^{\frac{1}{2}}\dif \tilde{y} = \dfrac{1}{\sqrt{2}} \int_{0}^1\!\! \Big( \int_{\R}\, |\absm{D}^k w_y(z,\tilde{y})|^2\,\dif z \Big)^{\frac{1}{2}}  \tilde{y} \dif \tilde{y}\\
		%\le& \dfrac{1}{\sqrt{2}} \bigg\{ \int_{0}^1\!\!\int_{\R}\! |\absm{D}^k w_y(z,\tilde{y})|^2\, \tilde{y} \dif \tilde{y}  \bigg\}^{\frac{1}{2}} \Big(\int_{0}^{1} \tilde{y} \dif \tilde{y}\Big)^{\frac{1}{2}}\\
		\le& \dfrac{1}{2} \bigg\{ \int_{0}^1\!\!\int_{\R}\! |\absm{D}^k w_y(z,\tilde{y})|^2\, \tilde{y} \dif \tilde{y}  \bigg\}^{\frac{1}{2}}  \le C(s) \Big|\dfrac{\fU_s(\teta)}{\fl(\teta)}\Big|^{2k+2} \|\psi_z \|_{H^{k-\frac{1}{2}}(\R)}.
	\end{align*}
	The above inequality, together with Corollary \ref{corol:Phi} concludes the proof.
\end{proof}

\begin{proposition}\label{prop:vyyInt}
	Suppose the same assumptions of Lemma \ref{lemma:cacci} hold. Then the weak derivative $\d_y^2 v$ exists in $L^2_{\text{loc}}(\mS)$ and for each $\delta\in(0,\frac{1}{2})$, $\d_y^2 v\in L^2(\R\times[\delta,1-\delta])$.
\end{proposition}
\begin{proof}
	Fix $\delta\in(0,\frac{1}{2})$, and let $\zeta(y)\in \mC^{\infty}_{c}(0,1)$ be a positive function such that $\zeta(y)=1$ if $y\in (\delta,1-\delta)$ and $\zeta(y)=0$ if $y\in (0,\frac{\delta}{2})\cup (1-\frac{\delta}{2},1)$. Let $v$ be the solution to (\ref{reform}). For $h\in(0,\frac{\delta}{2})$, we set $\vp_h\vcentcolon=-\Delta_y^{-h}\big( \zeta^2 \Delta_y^{h} v \big)$ where $\Delta_y^{\tau} f \vcentcolon= \frac{f(y+\tau)-f(y)}{\tau}$ for $\tau\in \R\backslash\{0\}$. For functions $f(y)$ and $g(y)$ and $h\in\R\backslash\{0\}$, one has the identities:
	\begin{subequations}\label{temp:fdi}
		\begin{gather}
			\int_{\R} f(y)\Delta_y^{-h} g(y) \,\dif y = - \int_{\R} g(y) \Delta_y^h f(y)\, \dif y,\\   \Delta_y^h (y^i f) = (y+h)^i\Delta_y^h f + p_{i} f, \qquad \text{for } \ i=1,2,3,
		\end{gather} 
	\end{subequations}
	where $p_1\equiv1$, $p_2\vcentcolon= 2y+h$, and $p_3\vcentcolon= 3y^2+3h y+h^2$. Using these, we have
	\begin{align*}
		\iint_{\mS}\!\! A_{22} v_y \d_y \vp_h \,\dif z \dif y
		%=& - \iint_{\mS}\!\! y(1+y^2\eta_z^2) v_y \Delta_y^{-h}\d_y(\zeta^2\Delta_y^h v)  \,\dif z \dif y\\
		%=& \iint_{\mS}\!\! \Delta_y^h(y v_y) \d_y(\zeta^2 \Delta_y^h v)\, \dif z \dif y + \iint_{\mS}\!\! \eta_z^2 \Delta_y^h(y^3v_y) \d_y(\zeta^2 \Delta_y^h v) \, \dif z \dif y\\
		%=& \iint_{\mS}\!\! \Delta_y^h(y v_y)  \big\{ 2\zeta^{\prime}\zeta \Delta_y^h v + \zeta^2 \Delta_y^h v_y \big\}\,\dif z \dif y + \iint_{\mS}\!\! \eta_z^2 \Delta_y^h(y^3v_y) \d_y(\zeta^2 \Delta_y^h v) \, \dif z \dif y\\
		%=& \iint_{\mS}\! (y+h) \zeta^2 |\Delta_y^h v_y|^2 \, \dif z \dif y +  \iint_{\mS}\!\! \zeta^2 v_y \Delta_y^h v_y\, \dif z \dif y + 2 \iint_{\mS} \zeta^{\prime}\zeta (y+h) \Delta_y^h v \Delta_y^h v_y\, \dif z \dif y\\ &+ 2\iint_{\mS}\!\! \zeta^{\prime}\zeta v_y \Delta_y^h v \, \dif z \dif y +  \iint_{\mS}\!\! \eta_z^2 \Delta_y^h(y^3v_y) \d_y(\zeta^2 \Delta_y^h v) \, \dif z \dif y\\
		=&\!\! \iint_{\mS}\! (y+h) \zeta^2\big\{ 1 + (y+h)^2 \eta_z^2 \big\}|\Delta_y^h v_y|^2\, \dif z \dif y+  \!\!\iint_{\mS}\!\! \zeta^2 v_y \Delta_y^h v_y\, \dif z \dif y\\ &+ 2\!\!\iint_{\mS}\! (y+h)\big\{1+\eta_z^2(y+h)^2\big\}\zeta^{\prime}\zeta \Delta_y^h v \Delta_y^h v_y \dif z \dif y\\ &+ 2\!\!\iint_{\mS}\!\! \zeta^{\prime}\zeta v_y \Delta_y^h v \, \dif z \dif y +\!\! \iint_{\mS}\!\! p_3 \eta_z^2 v_y \big(2\zeta^{\prime}\zeta \Delta_y^h v + \zeta^2 \Delta_y^h v_y \big)\dif z \dif y.
	\end{align*}
	Moreover, computing the term corresponding to $A_{11}$, we also have
	\begin{align*}
		\iint_{\mS}\!\! A_{11} v_z \d_z \vp_h \, \dif z \dif y
		%=& - \iint_{\mS}\! y \eta^2 v_z \Delta_y^{-h}(\zeta^2 \Delta_y^h v_z) \, \dif z \dif y = \iint_{\mS}\!\! \eta^2 \zeta^2 \Delta_y^h(y v_z) \Delta_y^h v_z \, \dif z \dif y \\ 
		=& \iint_{\mS}\!\! (y+h) \eta^2 \zeta^2 |\Delta_y^h v_z|^2 \,\dif z \dif y + \iint_{\mS}\!\! \eta^2 \zeta^2 v_z \Delta_y^h v_z \, \dif z \dif y.      
	\end{align*}
	It can be verified that $\vp_h \in \fD_{\mS}$. Thus taking $\vp_h$ as the test function for (\ref{reform}), and using the above calculation, we obtain that
	\begin{align}\label{temp:vyyL2}
		&\iint_{\mS}\! (y+h) \zeta^2\big\{ 1 + (y+h)^2 \eta_z^2 \big\}|\Delta_y^h v_y|^2\, \dif z \dif y + \iint_{\mS}\!(y+h)\zeta^2 \eta^2 |\Delta_y^h v_z|^2\, \dif z \dif y \nonumber\\
		=& -\iint_{\mS}\!\! \zeta^2 v_y \Delta_y^h v_y\, \dif z \dif y - 2 \iint_{\mS}\! (y+h)\big\{1+\eta_z^2(y+h)^2\big\}\zeta^{\prime}\zeta \Delta_y^h v \Delta_y^h v_y\, \dif z \dif y\nonumber\\ &- 2\iint_{\mS}\!\! \zeta^{\prime}\zeta v_y \Delta_y^h v \, \dif z \dif y - \iint_{\mS} p_3 \eta_z^2 v_y \big(2\zeta^{\prime}\zeta \Delta_y^h v + \zeta^2 \Delta_y^h v_y \big)\, \dif z \dif y \nonumber\\ &-\iint_{\mS}\!\! \eta^2 \zeta^2 v_z \Delta_y^h v_z\, \dif z \dif y - \iint_{\mS}\!\!\big\{ A_{12} v_y \d_y \vp_h + A_{21} v_z \d_y \vp_h \big\}\, \dif z \dif y =\vcentcolon \sum_{i=1}^6 I_i.
	\end{align}
	Using Corollary \ref{corol:v0}, we estimate the terms $I_1$-$I_5$ as
	\begin{align*}
		I_1 \le& \dfrac{C(s)}{\delta^2} \Big| \dfrac{\fU_s(\teta)}{\fl(\teta)} \Big|^{2} \|\psi_z\|_{H^{-1/2}(\R)}^2 + \dfrac{1}{6} \iint_{\mS} y \zeta^2 |\Delta_y^h v_y|^2 \, \dif z \dif y,\\
		I_2 \le & \dfrac{C(s)}{\delta^3} \Big| \dfrac{\fU_s(\teta)}{\fl(\teta)} \Big|^{2} \|\eta_z\|_{L^{\infty}(\R)}^4 \|\psi_z\|_{H^{-1/2}(\R)}^2  + \dfrac{1}{6}\iint_{\mS}\! (y+h) \zeta^2 |\Delta_y^h v_y|^2 \, \dif z \dif y, \\
		I_3 \le& \dfrac{C(s)}{\delta^2} \Big| \dfrac{\fU_s(\teta)}{\fl(\teta)} \Big|^{2} \|\psi_z\|_{H^{-1/2}(\R)}^2, \qquad I_4 \le C(s)\Big| \dfrac{\fU_s(\teta)}{\fl(\teta)} \Big|^{2}  \|\eta_z\|_{L^{\infty}(\R)}^4 \|\psi_z\|_{H^{-1/2}(\R)}^2,\\
		I_5 \le& \dfrac{C(s)}{\delta} \Big| \dfrac{\fU_s(\teta)}{\fl(\teta)} \Big|^{4} \|\eta\|_{L^{\infty}(\R)}^4 \|\psi_z\|_{H^{1/2}(\R)}^2 . 
	\end{align*}
	For the remaining $I_6$, we use identity (\ref{temp:fdi}) once again to get
	\begin{align*}
		\iint_{\mS}\!\! A_{12} v_z \d_y \vp_h \, \dif z \dif y
		%= & \iint_{\mS}\!\! \eta_z \eta y^2  v_z \Delta_y^{-h} \d_y(\zeta^2 \Delta_y^h v)\, \dif z \dif y\\
		%=& -\iint_{\mS}\!\! \eta_z \eta \Delta_y^h(y^2 v_z )  \big(2 \zeta^{\prime}\zeta \Delta_y^h v + \zeta^2 \Delta_y^h v_y \big)\, \dif z \dif y,\\
		=& -\iint_{\mS}\!\! (y+h)^2 \eta_z \eta\big\{ \zeta^2 \Delta_y^h v_z \Delta_y^h v_y + 2 \zeta^{\prime}\zeta  \Delta_y^h v \Delta_y^h v_z \big\}\,\dif z \dif y\\
		&-\iint_{\mS}\!\! p_2(y,h)\eta_z \eta \big\{ \zeta^2 v_z \Delta_y^h v_y + 2\zeta^{\prime}\zeta v_z \Delta_y^h v \big\}\,\dif z \dif y \\
		\iint_{\mS}\!\! A_{21} v_y \d_z \vp_h \, \dif z \dif y
		%=& \iint_{\mS}\!\! \eta_z \eta y^2 v_y \Delta_y^{-h}(\zeta^2 \Delta_y^h v_z )\, \dif z \dif y\\ 
		=& -\iint_{\mS}\!\! \zeta^2 \eta_z \eta \big\{ (y+h)^2 \Delta_y^h v_y + p_2(y,h) v_y \big\} \Delta_y^{h} v_z \, \dif z \dif y . 
	\end{align*}
	Thus, by Corollary \ref{corol:v0}, we estimate to get
	\begin{align*}
		I_6 =& - \iint_{\mS}\!\!\big\{ A_{12} v_y \d_y \vp_h + A_{21} v_z \d_y \vp_h \big\}\, \dif z \dif y\\
		\le & \dfrac{1}{6}\iint_{\mS}\! (y+h) \zeta^2 |\Delta_y^h v_y|^2 \, \dif z \dif y + C(s) \Big| \dfrac{\fU_s(\teta)}{\fl(\teta)} \Big|^{4} \|\eta_z\eta\|_{L^{\infty}(\R)}^2 \|\psi_z\|_{H^{1/2}(\R)}^2. 
	\end{align*}
	Substituting the estimates for $I_1$--$I_7$ to (\ref{temp:vyyL2}), and using the fact that $\zeta=1$ for $y\in [\delta,1-\delta]$ we have
	\begin{align*}
		\int_{\delta}^{1-\delta}\!\!\!\int_{\R}\!\!|\Delta_y^h v_y|^2\, \dif z \dif y  \le \dfrac{C(s)}{\delta^4} \Big| \dfrac{\fU_s(\teta)}{\fl(\teta)} \Big|^{6} \|\psi_z\|_{H^{1/2}(\R)}^2 < \infty.
	\end{align*}
	Since the above inequality holds true for arbitrarily small $h\in (0,\frac{\delta}{2})$, one can take the limit $h\to 0^{+}$ to show that $\d_y^2 v$ exists in $L^2\big(\R\times[\delta,1-\delta]\big)$.
\end{proof}
\begin{corollary}\label{corol:ae}
	Suppose the same assumptions of Lemma \ref{lemma:cacci} hold. Then for almost every $(z,y)\in\mS$,
	\begin{equation}\label{ae}
		y \eta^2 v_{zz} - 2 y^2 \eta \eta_z  v_{zy} + \big(1+2y^2\eta_z^2-y^2\eta \eta_{zz}\big) v_y +  y\big(1+y^2 \eta_z^2\big) v_{yy}=0.
	\end{equation}
\end{corollary}
\begin{proof}
	Take any $\vp\in \mC_{c}^{\infty}(\mS)$ as the test function for the weak form defined in Definition \ref{def:wWeak}. Then the corollary can be shown by applying the Fundamental lemma of calculus of variation and Proposition \ref{prop:vyyInt}.  
\end{proof}
\begin{lemma}\label{lemma:vyy}
	Suppose the same assumptions of Lemma \ref{lemma:cacci} hold. Then there exists a positive monotone increasing function $x\mapsto C(x)>0$ such that:
	\begin{align*}
		&\int_{0}^1\!\! \big\{ y^3\| \d_y^2 v(\cdot,y) \|_{H^{s-\frac{3}{2}}(\R)}^2 + y^5\|\d_y^3 v(\cdot,y) \|_{H^{s-\frac{5}{2}}(\R)}^2 + y^7\|\d_y^4 v(\cdot,y)\|_{H^{s-\frac{7}{2}}(\R)}^2 \big\} \, \dif y \\
		&\le C\big(\|\teta\|_{H^{s+\frac{1}{2}}(\R)}\big)\|\psi_z\|_{H^{s-1}(\R)}^2.
	\end{align*}
\end{lemma}
\begin{proof}
	Divide both sides of (\ref{ae}) by $(1+y^2\eta_z^2)$, it follows that for a.e. $(z,y)\in\mS$,
	\begin{gather}
		y v_{yy} = a_1 v_{zz} + a_2 v_{zy} + a_3 v_y,\label{temp:vyy-eq}\\
		\text{where } \  a_1\! \vcentcolon=\! -\dfrac{y(\teta + R)^2 }{1+y^2 \teta_z^2}, \ \ a_2\!\vcentcolon=\! \dfrac{2y^2 \teta_z (\teta + R) }{1+ y^2 \teta_z^2}, \ \ a_3\!\vcentcolon=\! \dfrac{y^2 (\teta+R)\teta_{zz}-1-2y^2 \teta_z^2}{1+ y^2 \teta_z^2}.\nonumber   
	\end{gather}
	Set $F(x,y)\vcentcolon=%= (1+y^2 x^2 )^{-1}-1
	-\frac{y^2x^2}{1+y^2x^2} $. Then $F_x = -\frac{2y^2 x}{(1+y^2 x^2)^2}$, $F_{xx} = \frac{2y^2(3y^2x^2-1)}{(1+y^2x^2)^3}$, and $F(0,y)=0$ for each $y\in[0,1]$. It follows that for each fixed $y\in(0,1)$,
	\begin{align*}
		&\min\limits_{x\in\R}F_x(x,y) = - \dfrac{3\sqrt{3}}{8} y &&\text{at the point } \ \underset{x\in\R}{\arg\min} F_x(x,y) = \dfrac{\sqrt{3}}{3y},\\
		&\max\limits_{x\in\R} F_x(x,y) = \dfrac{3\sqrt{3}}{8} y &&\text{at the point } \ \underset{x\in\R}{\arg\max} F_x(x,y) = -\dfrac{\sqrt{3}}{3y}.
	\end{align*}
	Thus applying Proposition \ref{prop:Sobcomp}, one has for $0\le m\le s-\frac{1}{2}$, and $y\in(0,1)$,
	\begin{equation}\label{temp:denSob}
		\big\| F\big(\teta_z(\cdot),y\big) \big\|_{H^{m}(\R)}%=\Big\|\dfrac{1}{1+y^2 \teta_z^2}-1\Big\|_{H^{m}(\R)} 
		\le \sup\limits_{x\in\R} |F_x(x,y)| \|\teta_z\|_{H^{m}(\R)} = \dfrac{3\sqrt{3}}{8} y\|\teta_z\|_{H^{m}(\R)}.
	\end{equation}
	Then $a_1$, $a_2$, and $a_3$ can be written in terms of $F(\teta_z,y)$ as
	\begin{align*}
		a_1 =& -y\big\{F\big(\teta_z,y\big)\{ \teta^2 + 2 R \teta +R^2 \} + \teta^2 + 2R \teta +R^2\big\},\\
		a_2 =& 2y^2 \big\{  \teta_z \teta F(\teta_z,y) + R \teta_z F(\teta_z,y) + \teta_z \teta + R \teta_z \big\},\\
		a_3 =& -\big\{ 1 + F\big(\teta_z,y\big) \big\} \big\{ 1 + 2 y^2 \teta_z^2 -y^2 (\teta+R)\teta_{zz} \big\}
	\end{align*}
	Since $\teta\in H^{s+\frac{1}{2}}(\R)$ for $s>2+\frac{1}{2}$, one can apply the product rule Proposition \ref{prop:clprod} and (\ref{temp:denSob}) to obtain that
	\begin{subequations}\label{temp:aEsts}
		\begin{gather}
			\|a_1(\cdot,y) - y R^2 \|_{H^{s-\frac{1}{2}}(\R)} \le C y \big\{1+\|\teta\|_{H^{s+\frac{1}{2}}(\R)}^2\big\}\big\{1+\| \teta_z \|_{H^{s-\frac{1}{2}}(\R)} \big\},\\
			\|a_2(\cdot,y)\|_{H^{s-\frac{1}{2}}(\R)} \le C y^2 \big\{ 1 + \|\teta\|_{H^{s+\frac{1}{2}}(\R)}^2 \big\} \big\{ 1+ \|\teta_{z}\|_{H^{s-\frac{1}{2}}(\R)}  \big\},\\
			\|a_3(\cdot,y)-1\|_{H^{s-\frac{3}{2}}(\R)} \le C \big\{ 1+ \|\teta\|_{H^{s+\frac{1}{2}}(\R)}^4 \big\} \big\{ 1+ \|\teta_{zz}\|_{H^{s-\frac{3}{2}}} \big\}.
		\end{gather}
	\end{subequations}
	Fix $\delta\in(0,\frac{1}{2})$. We recall $\absm{D}_h^k$ and $\Lambda_h^{2k}$ to be the operators defined in (\ref{temp:symb}). Let $k=s-\frac{3}{2}$. Multiplying (\ref{temp:vyy-eq}) with $y^2\Lambda_h^{2k}v_{yy}$, then integrating in $y\in I_{\delta} \vcentcolon= (\delta,1-\delta)$, we have by Parseval's theorem in $z\in\R$ that
	\begin{align*}
		&\int_{I_\delta}\!\int_{\R}\!\! \big| \absm{D}_h^{k} v_{yy}\big|^2 \, y^3\dif z \dif y = \int_{I_\delta}\!\int_{\R}\!\! \big\{ y^{\frac{1}{2}} a_1 v_{zz}+ y^{\frac{1}{2}} a_2 v_{zy}+ y^{\frac{1}{2}} a_3 v_{y}\big\}  (\Lambda_h^{2k} v_{yy})\, y^{\frac{3}{2}}\dif z \dif y\\
		=& \int_{I_\delta}\!\int_{\R}\! y^{\frac{3}{2}}\big(\absm{D}_h^{k} v_{yy}\big) \big\{y^{\frac{1}{2}} a_1 \absm{D}_h^{k} v_{zz} + y^{\frac{1}{2}}a_2 \absm{D}_h^{k} v_{zy} + y^{\frac{1}{2}} a_3 \absm{D}_h^{k} v_y\big\}\, \dif z \dif y \\
		&+\int_{I_\delta}\!\int_{\R}\! y^{\frac{3}{2}}\big(\absm{D}_h^{k} v_{yy}\big) \big\{y^{\frac{1}{2}} \big[\absm{D}_h^{k},a_1\big]  v_{zz} + y^{\frac{1}{2}} \big[\absm{D}_h^{k},a_2\big] v_{zy} + y^{\frac{1}{2}} \big[\absm{D}_h^{k},a_3\big] v_y \big\}\, \dif z \dif y.
	\end{align*}
	By Cauchy-Schwartz's inequality, we obtain that
	\begin{align}\label{temp:vyycom}
		&\dfrac{1}{2}\int_{I_\delta}\!\int_{\R}\!\! \big| \absm{D}_h^{k} v_{yy}\big|^2 \, y^3\dif z \dif y\nonumber\\
		\le& C\!\!\int_{I_\delta}\! \big\{ R^4 y^3  \| v_{zz} \|_{H^{k} }^2 + y \| v_{y} \|_{H^{k} }^2 + y \|(a_1-y R^2)(\cdot,y)\|_{L^{\infty} }^2 \| v_{zz} \|_{H^{k} }^2 \big\} \dif z \dif y\nonumber\\
		&+ C\!\!\int_{I_\delta}\! \big\{ y \|a_2(\cdot,y)\|_{L^{\infty} }^2 \| v_{zy} \|_{H^{k} }^2 + y \|(a_3-1)(\cdot,y)\|_{L^{\infty} }^2 \| v_{y} \|_{H^{k} }^2  \big\} \dif z \dif y\nonumber\\
		&+ C\!\!\int_{I_\delta}\!\!\! \big\{\big\|\big[ \absm{D}_h^{k} , a_1 \big] v_{zz}\big\|_{L^2 }^2 + \big\|\big[ \absm{D}_h^{k} , a_2 \big] v_{zy}\big\|_{L^2 }^2 + \big\|\big[ \absm{D}_h^{k} , a_3 \big] v_{y}\big\|_{L^2 }^2 \big\} \, y \dif z \dif y.
	\end{align}
	By Corollary \ref{corol:v0}, Proposition \ref{prop:commuLam}, and (\ref{temp:aEsts}), we have 
	\begin{align*}
		\int_{0}^{1}\!\!\big\|\big[ \absm{D}_h^{k} , a_1 \big] v_{zz}\big\|_{L^2(\R)}^2\, y\dif y   \le& C\big|\mathcal{N}_h^{k}\big|^2\|\d_z a_1\|_{H^{s-\frac{3}{2}}}^2 \int_{0}^{1}\!\!  \|v_{zz}\|_{H^{s-\frac{5}{2}}(\R)}^2\, y \dif y\\
		\le & C\big\{1+\|\teta\|_{H^{s+\frac{1}{2}}(\R)}^6\big\} \Big| \dfrac{\fU_s(\teta)}{\fl(\teta)} \Big|^{2s-1} \|\psi_z\|_{H^{s-2}(\R)}^2,\\
		\int_{0}^{1}\!\!\big\|\big[ \absm{D}_h^{k} , a_2 \big] v_{zy}\big\|_{L^2(\R)}^2\, y\dif y  \le& C\big|\mathcal{N}_h^{k}\big|^2\|\d_z a_2\|_{H^{s-\frac{3}{2}}}^2 \int_{0}^{1}\!\! \|v_{zy}\|_{H^{s-\frac{5}{2}}(\R)}^2\, y \dif y \\
		\le & C\big\{1+\|\teta\|_{H^{s+\frac{1}{2}}(\R)}^6\big\} \Big| \dfrac{\fU_s(\teta)}{\fl(\teta)} \Big|^{2s-1} \|\psi_z\|_{H^{s-2}(\R)}^2,\\
		\int_{0}^{1}\!\!\big\|\big[ \absm{D}_h^{k} , a_3 \big] v_{y}\big\|_{L^2(\R)}^2\, y \dif y \le & C\big|\mathcal{N}_h^{k}\big|^2\|\d_z a_3\|_{H^{s-\frac{3}{2}}}^2 \int_{0}^{1}\!\!\|v_{y}\|_{H^{s-\frac{5}{2}}(\R)}^2 \, y\dif y \\
		\le& C(s) C\big\{1+\|\teta\|_{H^{s+\frac{1}{2}}(\R)}^{10}\big\} \Big| \dfrac{\fU_s(\teta)}{\fl(\teta)} \Big|^{2s-3} \|\psi_z\|_{H^{s-3}(\R)}^2,
	\end{align*}
	where $\mathcal{N}_h^k$ is defined in (\ref{temp:commu0}), and it can be shown with the same argument that $|\mathcal{N}_h^k|\le C(s)$ for $k=s-\frac{3}{2}$. Putting the above estimates in (\ref{temp:vyycom}), then by Corollary \ref{corol:v0} and the Sobolev embedding theorem: $L^{\infty}(\R)\xhookrightarrow[]{}H^{s-\frac{3}{2}}(\R)$ for $s>2+\frac{1}{2}$, one has
	\begin{align*}
		\dfrac{1}{2}\int_{I_\delta}\!\int_{\R}\!\! \big| \absm{D}_h^{s-\frac{3}{2}} v_{yy}\big|^2 \, y^3\dif z \dif y \le C\big\{1+\|\teta\|_{H^{s+\frac{1}{2}}(\R)}^{10}\big\} \Big| \dfrac{\fU_s(\teta)}{\fl(\teta)} \Big|^{2s+1} \|\psi_z\|_{H^{s-1}(\R)}^2.
	\end{align*}
	Taking the limits $h\to 0^{+}$ then $\delta\to 0^{+}$, it follows that
	\begin{align}\label{temp:vyyb}
		\int_{0}^{1}\!\! \big\| v_{yy}(\cdot, y)\big\|_{H^{s-\frac{3}{2}}(\R)}^2 \, y^3 \dif y \le C\big\{1+\|\teta\|_{H^{s+\frac{1}{2}}(\R)}^{10}\big\} \Big| \dfrac{\fU_s(\teta)}{\fl(\teta)} \Big|^{2s+1} \|\psi_z\|_{H^{s-1}(\R)}^2.
	\end{align}
	Next, we show the estimates for $\d_y^3 v$. Taking $\d_y$ on $\{a_i\}_{i=1}^3$, we obtain that
	\begin{align*}
		(a_1)_y =& -(\teta+R)^2\big\{ 1 +  F(\teta_z,y) \big\} - y (\teta+R)^2 F_y(\teta_z,y),\\
		(a_2)_y =& 4y \big\{ \teta_z (\teta+R) F(\teta_z,y) + \teta_z \teta + R \teta_z \big\} + 2 y^2 \teta_z (\teta+R) F_y(\teta_z,y),\\
		(a_3)_y =& F_y(\teta_z,y)\big\{ y^2 (\teta+R)\teta_{zz}-2y^2 \teta_z^2 -1 \big\}+ 2y \big\{ (\teta+R)\teta_{zz} - 2 \teta_z^2 \big\}\big\{1+F(\teta_z,y)\big\}, 
	\end{align*}
	where $F_{y}(x,y)=-\frac{2x^2 y}{(1+x^2y^2)^2}$ and $F_{xy}(x,y) = \frac{4xy(x^2y^2-1)}{(1+x^2y^2)^3}$. Thus by Propositions \ref{prop:Sobcomp}--\ref{prop:clprod},
	\begin{gather*}
		\|(a_1)_y(\cdot,y) - R^2 \|_{H^{s-\frac{1}{2}}(\R)} \le C \big\{1+\|\teta\|_{H^{s+\frac{1}{2}}(\R)}^2\big\} \big\{ 1 + \|\teta_z\|_{H^{s-\frac{1}{2}}(\R)} \big\},\\
		\|(a_2)_y(\cdot,y) \|_{H^{s-\frac{1}{2}}(\R)} \le C y \big\{ 1+\|\teta\|_{H^{s+\frac{1}{2}}(\R)}^2 \big\} \big\{ 1+ \|\teta_z\|_{H^{s-\frac{1}{2}}(\R)} \big\},\\
		\|(a_3)_y(\cdot,y)\|_{H^{s-\frac{3}{2}}(\R)} \le C y \big\{ 1 + \|\teta\|_{H^{s+\frac{1}{2}}(\R)}^3 \big\} \big\{ 1+ \|\teta_{zz}\|_{H^{s-\frac{3}{2}}(\R)} \big\}.
	\end{gather*}
	Recall the difference quotient operator $\Delta_y^h$ defined in the proof of Proposition \ref{prop:vyyInt}. For $h\in(0,\frac{\delta}{2})$, applying $\Delta_y^h$ on (\ref{temp:vyy-eq}), it follows that for a.e. $(z,y)\in\R\times[\delta,1-\delta]$,
	\begin{align*}
		(y+h) \Delta_y^h v_{yy} = -v_{yy} + a_1^h \Delta_y^h v_{zz} + a_2^h \Delta_y^h v_{yz} + a_3^h \Delta_y^h v_{y} + v_{zz} \Delta_y^h a_1 + v_{zy} \Delta_y^h a_2 + v_{y} \Delta_{y}^h a_3,
	\end{align*}
	where we denote $a_i^h(z,y) \vcentcolon= a_i(z,y+h)$. We multiply both sides of the above equation with $y^4 \Delta_y^h \absm{D}^{2s- 5} v_{yy} $, and integrate the resultant equality in $(z,y)\in\R\times[\delta,1-\delta]$. Then repeating the same previous estimates, it follows that there exists some positive monotone increasing function $x\mapsto \tilde{C}(x)$ such that
	\begin{align*}
		\int_{\delta}^{1-\delta}\!\!\!\! \int_{\R}\!\! \big| \Delta_y^h \absm{D}^{s-\frac{5}{2}} v_{yy} \big|^2 \, y^5 \dif z \dif y \le \tilde{C}\big( \|\teta\|_{H^{s+\frac{1}{2}}} \big) \Big\{\|\psi_z\|_{H^{s-1}}^2 + \int_{0}^{1}\!\! \|v_{yy}\|_{H^{s-\frac{5}{2}}}^2 \, y^3 \dif z \dif y \Big\}.
	\end{align*}
	Taking the limits $h\to 0^{+}$, then $\delta\to 0^{+}$, it follows that $y^\frac{5}{2}\d_{y}^3 v$ exists in $L_{y}^2\big(0,1;H_z^{s-\frac{5}{2}}(\R)\big)$ and for a.e. $(z,y)\in \mS$
	\begin{equation}\label{temp:vyyy}
		y\d_y^3 v =  (a_3-1) v_{yy}  + a_1 \d_y\d_z^2 v + a_2 \d_y^2 \d_z v + v_{zz} (a_1)_y + v_{zy} (a_2)_y + v_{y} (a_3)_y
	\end{equation}
	Finally, the estimate for $\d_y^4 v$ is obtained as follows: \textbf{(1)} compute $\d_y^2 a_i$ for $i=1,2,3$, and obtain the bounds of their Sobolev norms using Propositions \ref{prop:Sobcomp}--\ref{prop:clprod}; \textbf{(2)} apply $\Delta_y^h$ on both sides of (\ref{temp:vyyy}), multiply the resulting equation with $y^{6}\Delta_y^h \absm{D}^{2s-7} \d_y^3 v$, then integrate in $(z,y)\in\R\times[\delta,1-\delta]$; (3) estimate the right hand side using Cauchy-Schwartz's inequality, Proposition \ref{prop:commu}, Sobolev norms of $\{a_i\}_{i=1}^3$, and Corollary \ref{corol:v0}; \textbf{(4)} take limit $h\to 0^+$ then $\delta\to 0^+$ to show that $y^{\frac{7}{2}}\d_y^4 v$ exists in $L_{y}^2\big(0,1;H_z^{s-\frac{7}{2}}(\R)\big)$.
\end{proof}

\begin{corollary}\label{corol:yC}
	Suppose the same assumptions of Lemma \ref{lemma:cacci} hold. 
	%%%%%%%%%%%%%%%%%%%%%%%%%%%%%%%%%%%%%%%%%%%
	%%%%%%%%%%%%% OLD STATEMENT %%%%%%%%%%%%%%%
	%%%%%%%%%%%%%%%%%%%%%%%%%%%%%%%%%%%%%%%%%%%
	\iffalse
	Then for each fixed $y_0\in(0,1)$, the following regularity holds:
	\begin{align*}
		v_y \in \mC_{y}^{0}\big([y_0,1];H_z^{s-1}(\R)\big), \quad v_{yy} \in \mC_{y}^{0}\big([y_0,1];H_z^{s-2}(\R)\big), \quad \d_y^3v \in \mC_{y}^{0}\big([y_0,1];H_z^{s-3}(\R)\big).
	\end{align*}
	\fi
	%%%%%%%%%%%%%%%%%%%%%%%%%%%%%%%%%%%%%%%%%%%
	%%%%%%%%%%%%% OLD STATEMENT %%%%%%%%%%%%%%%
	%%%%%%%%%%%%%%%%%%%%%%%%%%%%%%%%%%%%%%%%%%%
	Then
	\begin{align*}
		v_y \in \mC_{y}^{0}\big((0,1];H_z^{s-1}(\R)\big), \quad v_{yy} \in \mC_{y}^{0}\big((0,1];H_z^{s-2}(\R)\big), \quad \d_y^3v \in \mC_{y}^{0}\big((0,1];H_z^{s-3}(\R)\big).
	\end{align*}
\end{corollary}
\begin{proof}
	Let $y_0\in(0,1)$ be an arbitrary point. Then Lemma \ref{lemma:vyy} implies that
	\begin{align*}
		&\int_{y_0}^1\!\! \big\{\| \d_y^2 v \|_{H^{s-\frac{3}{2}}}^2 + \|\d_y^3 v \|_{H^{s-\frac{5}{2}}}^2 + \|\d_y^4 v\|_{H^{s-\frac{7}{2}}}^2 \big\} \, \dif y  \le y_0^{-7}C\big(\|\teta\|_{H^{s+\frac{1}{2}}}\big)\|\psi_z\|_{H^{s-1}}^2.
	\end{align*}
	By the interpolation lemma for Bochner spaces, Proposition \ref{prop:bochner}, we have
	\begin{align*}
		v_y \in \mC_{y}^{0}\big([y_0,1];H_z^{s-1}(\R)\big), \quad v_{yy} \in \mC_{y}^{0}\big([y_0,1];H_z^{s-2}(\R)\big), \quad \d_y^3v \in \mC_{y}^{0}\big([y_0,1];H_z^{s-3}(\R)\big),
	\end{align*}
	which is true for arbitrary $y_0\in(0,1)$, hence the result follows. 
\end{proof}
\subsubsection{Construction of Dirichlet-Neumann operator}\label{sssec:consDN}
For $k\in\R$, we denote $H_0^{k}(\R)$ as the closure of $\mC_c^{\infty}(\R)$ with the $H^k(\R)$ norm. Recall that for a given function $\vp(z)\vcentcolon \R\to \R$, we denoted $E_{\eta}[\vp]$ as the solution to (\ref{reform}) with boundary data $E_\eta(\vp)\vert_{y=1}=\vp$. Then for a given $\eta-R\in H^{s+\frac{1}{2}}(\R)$, and $m\in[\frac{1}{2},s]$, we set the bilinear functional $\mathfrak{G}[\eta]:H^{m}(\R)\times H^{\frac{1}{2}}_0(\R) \to \R$ as:
\begin{equation}\label{bilinear}
	\mathfrak{G}[\eta](\psi,\vp) \vcentcolon= \iint_{\mS} \nabla E_{\eta}[\vp]\cdot \big( A \nabla E_{\eta}[\psi] \big)\, \dif z \dif y.
\end{equation}
Note that this is well-defined since $E_{\eta}[\vp]\in \mathscr{H}^1(\mS)$ if $\vp\in H^{\frac{1}{2}}(\R)$ by Lemma \ref{lemma:cacci}.
\begin{remark}
	To see how the Dirichlet-Neumann operator classically defined in (\ref{DN-flat}) follows from the bilinear form in (\ref{bilinear}), we assume that all terms appearing in $(\ref{bilinear})$ are smooth and compactly supported. Denote $v\equiv E_{\eta}[\psi]$. Then integrating by parts in (\ref{bilinear}) and using the fact that $v$ is a classical solution to (\ref{reform}), one has
	\begin{align*}
		\mathfrak{G}[\eta](\psi,\vp) = \int_{\R} \vp \cdot \big\{ A_{21} v_{z} + A_{22} v_{y} \big\}\big\vert_{y=1} \, \dif z \qquad \text{for any } \ \vp\in \mC_{c}^{\infty}(\R).
	\end{align*}
	By the expression of $A$ in (\ref{reformDiv}) and classical definition of $G[\eta](\psi)$ given in (\ref{DN-flat}), 
	\begin{align*}
		&\{A_{21} v_z + A_{22} v_y\}\vert_{y=1} = \{ - y^2 \eta_z \eta v_z + y(1+y^2|\eta_z|^2) v_y \}\vert_{y=1}\\
		=& \Big\{ y^2 \eta \Big( \dfrac{1+y|\eta_z|^2}{\eta}v_y - \eta_z v_z \Big) + (1-y)y v_y \Big\}\Big\vert_{y=1} \\
		=& \eta \Big\{ \dfrac{1+y|\eta_z|^2}{\eta}v_y - \eta_z v_z \Big\}\Big\vert_{y=1}= \eta\, G[\eta](\psi). 
	\end{align*}
	Therefore we see that heuristically, the bilinear form (\ref{bilinear}) corresponds to
	\begin{equation*}
		\mathfrak{G}[\eta](\psi,\vp) = \int_{\R} \vp \cdot \eta\, G[\eta](\psi)\, \dif z \qquad \text{for any } \ \vp\in \mC_{c}^{\infty}(\R).
	\end{equation*}   
\end{remark}
\paragraph{Case 1: \texorpdfstring{$m=\frac{1}{2}$}{m=1/2}.} Applying the estimate of Lemma \ref{lemma:cacci}, one has
\begin{align*}
	|\mathfrak{G}[\eta](\psi,\vp)| \le& \|A\|_{L^{\infty}} \bigg( \int_{0}^{1}\! \big\|\nabla E_{\eta}[\psi]\big\|_{L^2(\R)}^2 \, y \dif y\bigg)^{\frac{1}{2}} \bigg( \int_{0}^{1}\! \big\|\nabla E_{\eta}[\vp]\big\|_{L^2(\R)}^2 \, y \dif y\bigg)^{\frac{1}{2}}\\
	\le & C\big\{ 1 + \fU_s(\teta) \big\}  \Big|\dfrac{\fU_{s}(\teta)}{\fl(\teta)}\Big|^2 \|\psi\|_{H^{\frac{1}{2}}(\R)} \|\vp\|_{H^{\frac{1}{2}}(\R)}.  
\end{align*}
If we set the functional $\tilde{G}[\eta](\psi) \vcentcolon H_{0}^{\frac{1}{2}}(\R) \to \R $ as $\big\{\tilde{G}[\eta](\psi)\big\}(\vp) \vcentcolon= \mathfrak{G}[\eta](\psi,\vp)$, then $\tilde{G}[\eta](\psi) \in H^{-\frac{1}{2}}(\R)$ where $H^{-\frac{1}{2}}(\R)$ denotes the functional dual space of $H_{0}^{\frac{1}{2}}(\R)$. %with $L^2(\R)$ being the pivot space. 
Moreover, from the previous estimate, one has that
\begin{equation*}
	\|\tilde{G}[\eta](\psi)\|_{H^{-\frac{1}{2}}(\R)} \le C \{ 1 + \fU_s(\teta) \}  \dfrac{|\fU_s(\teta)|^2}{|\fl(\teta)|^2} \|\psi\|_{H^{\frac{1}{2}}(\R)}.
\end{equation*}

\paragraph{Case 2: \texorpdfstring{$\frac{1}{2}<m\le s$}{1/2<m<=s}.} Let $\psi\in H^{m}(\R)$ for $\frac{1}{2}<m\le s$. Then by the same argument as above, the functional $ \tilde{G}[\eta](\psi)$ belongs to the space $H^{-\frac{1}{2}}(\R)$. Thus we aim to show that $\absm{D}^{m-1}\tilde{G}[\eta](\psi)$ defined in the distributional sense can be extended into a functional on $L^2(\R)$. To do this, we first construct the following extension: let $\chi\in \mC_{c}^{\infty}(\R)$ be a positive even function such that
\begin{equation}\label{chi}
	\int_{\R}|\chi(\xi)|^2\, \dif \xi =1, \qquad  \chi(\xi)=\left\{\begin{aligned}
		&1 && \text{if $|\xi|\le \tfrac{1}{2}$,}\\
		&0 && \text{if $|\xi|\ge 1$.}
	\end{aligned}\right.
\end{equation}
For $f(z)\in L^1(\R)$, we define
\begin{equation}\label{dagger}
	f^{\dagger}(z,y) \vcentcolon= \chi\big( (1-y) D \big) f = \dfrac{1}{\sqrt{2\pi}}\int_{\R}\!\! \chi\big( (1-y) \xi \big) \wh{f}(\xi) e^{iz\xi}\dif \xi \quad \text{for } \ (z,y)\in \mS. 
\end{equation}
It follows that for fixed $0\le y<1$, $f^{\dagger}(\cdot,y)\in \mC^{\infty}(\R)$ if $f\in L^1(\R)$, and $f^\dagger(z,y)\to f(z)$ as $y\to 1^{-}$ for a.e. $z\in\R$. In addition, one can verify that
\begin{equation}\label{daggerC}
	\absm{D}^{\alpha} f^{\dagger} = \dfrac{1}{\sqrt{2\pi}}\int_{\R}\!\! \absm{\xi}^{\alpha} \chi\big((1-y)\xi\big) \wh{f} e^{iz\xi} \dif \xi = \big(\absm{D}^{\alpha} f\big)^{\dagger}, \quad \text{for } \ \alpha\in\R.
\end{equation}
The following estimate for $f^{\dagger}$ holds
\begin{proposition}\label{prop:daggerL2}
	For $f\in L^1(\R)$, let $f^{\dagger}(z,y)$ be the function constructed in (\ref{dagger}). Then the following estimate holds: 
	\begin{equation*}
		\int_{0}^{1} \int_{\R} |\absm{D}^{-\frac{1}{2}} \nabla f^{\dagger}|^2 \dif z \dif y \le 2 \| f \|_{L^2(\R)}.
	\end{equation*}
\end{proposition}
\begin{proof}
	Define $\mX(\xi)\vcentcolon= \int_{-\infty}^{\xi} |\chi(\zeta)|^2\, \dif \zeta$ to be the primitive function. Then $\mX(\xi)-\mX(0)=-\frac{1}{2}$ for $\xi\le -1$ and $\mX(\xi)-\mX(0)=\frac{1}{2}$ for $\xi\ge 1$. Thus we have
	\begin{equation*}
		\big|\frac{\mX(\xi)-\mX(0)}{\xi}\big|\le \frac{1}{2|\xi|}\le \frac{1}{1+|\xi|}  \qquad \text{if }  \ |\xi|\ge 1. 
	\end{equation*}
	On the other hand if $|\xi|\le 1$ then by Mean Value theorem there exists $\theta\in(0,\xi) \subset (-1,1)$ such that
	\begin{equation*}
		\big|\frac{\mX(\xi)-\mX(0)}{\xi}\big| = |\chi(\theta)|^2 \le 1 \le \frac{2}{1+|\xi|},
	\end{equation*}
	where the last inequality holds since $|\xi|\le 1$. In summary we have
	\begin{equation}\label{temp:mX}
		\big|\frac{\mX(\xi)-\mX(0)}{\xi}\big| \le \dfrac{2}{1+|\xi|} \qquad \text{for all } \ \xi\in\R.
	\end{equation}
	By Plancherel's theorem and change of variable $\tilde{y}=1-y$, one gets that
	\begin{align*}
		&\int_{0}^1\!\! \int_{\R} \big| \absm{D}^{-\frac{1}{2}} \nabla f^{\dagger}  \big|^2 \dif z \dif y  %= \int_{0}^{1}\!\! \int_{\R} \absm{\xi}^{-1} \xi^2 \big|\chi\big((1-y)\xi\big)\big|^2 |\wh{f}(\xi)|^2\, \dif \xi \dif y 
		\le \int_{0}^{1}\!\! \int_{\R} \absm{\xi} \big|\chi\big((1-y)\xi\big)\big|^2 |\wh{f}(\xi)|^2\, \dif \xi \dif y \\
		=&  \int_{\R}  \Big\{\int_{0}^1\!\! |\chi(\tilde{y}\xi)|^2 \dif \tilde{y}\Big\} \absm{\xi} |\wh{f}(\xi)|^2 \dif \xi = \int_{\R}\! \big|\dfrac{\mX(\xi)-\mX(0)}{\xi}\big| \absm{\xi} |\wh{f}(\xi)|^2 \dif \xi\\
		\le & \int_{\R}\! \dfrac{2}{1+|\xi|} \absm{\xi} |\wh{f}(\xi)|^2 \dif \xi \le 2 \int_{\R} |\wh{f}(\xi)|^2 \dif \xi = 2\|f\|_{L^2(\R)}^2.
	\end{align*}
	This completes the proof.
\end{proof}
\begin{proposition}\label{prop:DNConstruct}
	Let $s>\frac{5}{2}$ and $m\in (\frac{1}{2},s]$. For $\psi \in H^{m}(\R)$, define the distribution
	\begin{equation*}
		\big\{\absm{D}^{m-1}\tilde{G}[\eta](\psi)\big\}(\vp)\vcentcolon=\mathfrak{G}[\eta](\psi,\absm{D}^{m-1}\vp) \quad \text{for } \ \vp\in\mC_{c}^{\infty}(\R).
	\end{equation*}
	Then $\absm{D}^{m-1} \tilde{G}[\eta](\psi)$ can be extended as a bounded linear map on $L^2(\R)$, and there exists a unique $g_{\eta}[\psi]\in L^2(\R)$ such that for all $\vp\in L^2(\R)$,
	\begin{gather*}
		\big\{ \absm{D}^{m-1} \tilde{G}[\eta](\psi) \big\}(\vp) = \int_{\R} \vp g_{\eta}[\psi] \dif z,\\ \|g_{\eta}[\psi]\|_{L^2(\R)}\le C \big(1+\fU_s(\teta)\big) \dfrac{|\fU_s(\teta)|^{m+\frac{1}{2}}}{|\fl(\teta)|^{m+\frac{1}{2}}}\|\psi_z\|_{H^{m-1}(\R)}. \quad	
	\end{gather*}
\end{proposition}
\begin{proof}
	First we denote $v\equiv E_{\eta}[\psi]$. Fix $\vp(z)\in \mC_{c}^{\infty}(\R)$. Then from the construction (\ref{dagger}) and Definition \ref{def:solOp}, it follows that for all $\alpha\in\R$,
	\begin{equation}\label{temp:vphiExt}
		\vp^{\dagger}\vert_{y=1}=\vp=E_{\eta}[\vp]\vert_{y=1}\, \ \text{ and } \ \big(\absm{D}^{\alpha}\vp\big)^{\dagger}\big\vert_{y=1}= \absm{D}^{\alpha}\vp = E_{\eta}\big[ \absm{D}^{\alpha} \vp \big]\big\vert_{y=1}.
	\end{equation}
	Thus by Lemmas \ref{lemma:cacci}-\ref{lemma:vyy} and (\ref{dagger}), $E_{\eta}\big[\absm{D}^{\alpha}\vp\big]-\big(\absm{D}^{\alpha}\vp\big)^{\dagger}\in \fD_{\mS}$ for $\alpha\in\R$ where $\fD_{\mS}$ is the functional space defined in Definition \ref{def:wWeak}. Using this as a test function for the weak form of $v$ defined in Definition \ref{def:wWeak}, one has that
	\begin{align*}
		\iint_{\mS}\nabla E_{\eta}\big[\absm{D}^{\alpha}\vp\big] \big(A\nabla v\big)\, \dif z\dif y =  \iint_{\mS}\nabla \big(\absm{D}^{\alpha}\vp\big)^{\dagger} \big(A\nabla v\big)\, \dif z\dif y.
	\end{align*}
	Therefore we have
	\begin{align*}
		\mathfrak{G}[\eta]\big(\psi,\absm{D}^{\alpha}\vp\big) = \iint_{\mS}\nabla E_{\eta}\big[\absm{D}^{\alpha}\vp\big] (A\nabla v) \dif z \dif y = \iint_{\mS}\nabla \big(\absm{D}^{\alpha}\vp\big)^{\dagger} (A\nabla v) \dif z \dif y.
	\end{align*}
	Taking $\alpha=m-1$, then by the relation (\ref{daggerC}) and Parseval's theorem, it follows that
	\begin{align}\label{temp:bG}
		&\big\{\absm{D}^{m-1}\tilde{G}[\eta](\psi)\big\}(\vp)\vcentcolon=\mathfrak{G}[\eta]\big(\psi,\absm{D}^{m-1}\vp\big) = \iint_{\mS}\!\! \nabla \absm{D}^{m-1} \vp^{\dagger} \big(A\nabla v\big) \, \dif z \dif y\\
		=& \iint_{\mS}\!\! \nabla \absm{D}^{-\frac{1}{2}} \vp^{\dagger} \big( \absm{D}^{m-\frac{1}{2}} A\nabla v \big) \, \dif z \dif y = \iint_{\mS}\!\! \nabla \absm{D}^{-\frac{1}{2}} \vp^{\dagger} \big( \absm{D}^{m-\frac{1}{2}} A\nabla v \big) \, \dif z \dif y\nonumber\\
		=& \iint_{\mS}\!\! \nabla \absm{D}^{-\frac{1}{2}} \vp^{\dagger} \big\{ A \absm{D}^{m-\frac{1}{2}} \nabla v + \big[\absm{D}^{m-\frac{1}{2}},\tA\big] \nabla  v \big\} \, \dif z \dif y\nonumber\\
		\le & \bigg(\iint_{\mS}  \big| \nabla \absm{D}^{-\frac{1}{2}} \vp^{\dagger} \big|^2 \bigg)^{\frac{1}{2}}\bigg( \iint_{\mS} \big\{ A\absm{D}^{m-\frac{1}{2}}\nabla v + \big[\absm{D}^{m-\frac{1}{2}},\tA\big]\nabla v \big\}^2 \bigg)^{\frac{1}{2}}.\nonumber
	\end{align}
	By Proposition \ref{prop:daggerL2}, we have the estimate:
	\begin{align}\label{temp:varphiL2}
		\iint_{\mS}\!\! | \nabla \absm{D}^{-\frac{1}{2}} \vp^{\dagger} |^2\, \dif z \dif y \le 2\|\vp\|_{L^2(\R)}^2.
	\end{align}
	On the other hand by Lemma \ref{corol:v0} and (\ref{tAEST}), it follows that
	\begin{align}\label{temp:bG1}
		&\iint_{\mS} |A\absm{D}^{m-\frac{1}{2}}\nabla v|^2\, \dif z \dif y \le \iint_{\mS} \dfrac{1}{y^2} \|(\tA+A_R)(\cdot,y)\|_{L^{\infty}}^2 \|\nabla v(\cdot,y)\|_{H^{m-\frac{1}{2}}}^2 \, y^2 \dif y \\
		\le & \big(1+\fU_{s}(\teta)\big)^2 \int_{0}^1 \! \|\nabla v\|_{H^{m-\frac{1}{2}}}^2 \, y \dif y \le C\big(1+\fU_{s}(\teta)\big)^2\dfrac{|\fU_s(\teta)|^{2m+1}}{|\fl(\teta)|^{2m+1}}\| \psi_z\|_{H^{m-1}}^2.\nonumber
	\end{align}
	In addition, setting $\delta=1$, $k=m-\frac{1}{2}$, and $t_0=s-\frac{3}{2}$ in Proposition \ref{prop:commuLam}, we also have
	\begin{align}\label{temp:bG2}
		&\iint_{\mS}  \big| \big[\absm{D}^{m-\frac{1}{2}},\tA\big]\nabla v \big|^2\, \dif z \dif y \le \int_{0}^1 \big\|\big[\absm{D}^{m-\frac{1}{2}},\tA\big]\nabla v(\cdot, y)\big\|_{L^2(\R)}^2 \dif y\\
		\le& C\!\!\int_{0}^{1}\!\!\| \tA(\cdot, y) \|_{H^{s-\frac{1}{2}}(\R)}^2 \|\nabla v(\cdot,y)\|_{H^{m-\frac{3}{2}}(\R)}^2\, \dif y \nonumber\\
		\le& C |\fU_s(\teta)|^2 \int_{0}^1\!\! \|\nabla v(\cdot,y)\|_{H^{m-\frac{3}{2}}(\R)}^2\, y\dif y \le C \dfrac{|\fU_{s}(\teta)|^{2m+1}}{|\fl(\teta)|^{2m-1}} \| \psi_z \|_{H^{m-2}(\R)}^2,\nonumber      
	\end{align}
	Thus substituting (\ref{temp:varphiL2})--(\ref{temp:bG2}) into (\ref{temp:bG}), one has that for all $\vp\in\mC_{c}^{\infty}(\R)$,
	\begin{align*}
		\big\{\absm{D}^{m-1}\tilde{G}[\eta](\psi)\big\}(\vp)%=\mathfrak{G}[\eta]\big(\psi,\absm{D}^{m-1}\vp\big) 
		\le C\big(1+\fU_s(\teta)\big) \dfrac{|\fU_s(\teta)|^{m+\frac{1}{2}}}{|\fl(\teta)|^{m+\frac{1}{2}}}\|\psi_z\|_{H^{m-1}(\R)} \|\vp\|_{L^2(\R)}.
	\end{align*}
	By taking a sequence $\vp_k\to \vp$ in $L^2(\R)$ with $\{\vp_k\}_{k\in\mathbb{N}}\in \mC_{c}^{\infty}(\R)$, the domain of functional $\absm{D}^{m-1}\tilde{G}[\eta](\psi)$ can be extended to $L^2(\R)$. By Riesz representation theorem, there is a unique $g_{\eta}[\psi]\in L^{2}(\R)$ such that $\big\{\!\absm{D}^{m-1}\tilde{G}[\eta](\psi)\big\}(\vp) \!=\! \int_{\R} \vp g_{\eta}[\psi] \, \dif z$ for $\vp\in L^2(\R)$, and
	\begin{equation*}
		\|g_{\eta}[\psi]\|_{L^2(\R)} \le C\big(1+\fU_s(\teta)\big) \dfrac{|\fU_s(\teta)|^{m+\frac{1}{2}}}{|\fl(\teta)|^{m+\frac{1}{2}}}\|\psi_z\|_{H^{m-1}(\R)}.
	\end{equation*}  
	This concludes the proof.
\end{proof}
For all $\vp\in \mC_{c}^{\infty}$, it follows from Proposition \ref{prop:DNConstruct} and Parseval's theorem that
\begin{align*}
	\big\{\tilde{G}[\eta](\psi)\big\}(\vp) %= \big\{\tilde{G}[\eta](\psi)\big\}\big(\absm{D}^{m-1}\absm{D}^{1-m}\vp\big) 
	=&\big\{\absm{D}^{m-1}\tilde{G}[\eta](\psi)\big\}\big(\absm{D}^{1-m}\vp\big)\\ =& \int_{\R}\! g_{\eta}[\psi]\absm{D}^{1-m}\vp\, \dif z =  \int_{\R}\! \vp \absm{D}^{1-m}g_{\eta}[\psi]\, \dif z,
\end{align*}
where $\absm{D}^{1-m}g_{\eta}[\psi]\in H^{m-1}(\R)$. In light of this, we redefine $\tilde{G}[\eta](\psi)\vcentcolon=\absm{D}^{1-m}g_{\eta}[\psi]$.
Then by the estimate for $g_{\eta}[\psi]$ in Proposition \ref{prop:DNConstruct}, we obtain that for $\frac{1}{2}<m\le s$,
\begin{equation}\label{tildeG}
	\| \tilde{G}[\eta](\psi) \|_{H^{m-1}(\R)} = \| g_{\eta}[\psi] \|_{L^{2}(\R)} \le C\big(1+\fU_s(\teta)\big) \dfrac{|\fU_s(\teta)|^{m+\frac{3}{2}}}{|\fl(\teta)|^{m+\frac{1}{2}}}\|\psi_z\|_{H^{m-1}(\R)}. 
\end{equation}
Finally, we set the Dirichlet-Neumann operator $G[\eta](\psi)$ as 
\begin{equation*}
	G[\eta](\psi) \vcentcolon= \dfrac{1}{\eta} \tilde{G}[\eta](\psi).
\end{equation*}
By the Sobolev embedding theorem for composite function, Proposition \ref{prop:Sobcomp}, we have
\begin{equation*}
	\big\|\tfrac{1}{\eta}-\tfrac{1}{R}\big\|_{H^{s+\frac{1}{2}}(\R)} \!\!= \big\|\tfrac{1}{\teta+R}-\tfrac{1}{R}\big\|_{H^{s+\frac{1}{2}}(\R)}\! \le \sup\limits_{x\in\R}|x+R|^{-2} \|\teta\|_{H^{s+\frac{1}{2}}(\R)} \! \le R^{-2}\|\teta\|_{H^{s+\frac{1}{2}}(\R)}.
\end{equation*}
Since $s>2+\frac{1}{2}$ and $0\le m\le s$, it holds that $(s+\frac{1}{2})+ (m-1) >0$. Then one can apply the product rule for Sobolev spaces, Proposition \ref{prop:clprod}, and (\ref{tildeG}) to get
\begin{align*}
	&\|G[\eta](\psi)\|_{H^{m-1}(\R)} \le \big\|\big(\tfrac{1}{\eta}-\tfrac{1}{R}\big)\tilde{G}[\eta](\psi)\big\|_{H^{m-1}(\R)} + \dfrac{1}{R} \|\tilde{G}[\eta](\psi)\|_{H^{m-1}(\R)} \\
	\le& \big\|\tfrac{1}{\eta}-\tfrac{1}{R}\big\|_{H^{s+\frac{1}{2}}(\R)} \big\|\tilde{G}[\eta](\psi)\big\|_{H^{m-1}(\R)} + \dfrac{1}{R} \|\tilde{G}[\eta](\psi)\|_{H^{m-1}(\R)}\\
	\le&\big\{ R^{-1} +   R^{-2}\|\teta\|_{H^{s+\frac{1}{2}}(\R)} \big\} \big\|\tilde{G}[\eta](\psi)\big\|_{H^{m-1}(\R)}\\
	\le & C \dfrac{R+ \|\teta\|_{H^{s+\frac{1}{2}}}}{R^2} \big(1+\fU_s(\teta)\big)\dfrac{|\fU_s(\teta)|^{m+\frac{1}{2}}}{|\fl(\teta)|^{m+\frac{1}{2}}} \|\psi_z\|_{H^{m-1}(\R)}.
\end{align*}
Therefore we proved the following lemma:
\begin{lemma}\label{lemma:GSob}
Denote $\teta\vcentcolon=\eta-R$. Let $\teta \in H^{s+\frac{1}{2}}(\R)$ with $s>\frac{5}{2}$. Then $\psi\mapsto G[\eta](\psi)$ is a bounded linear map from $H^{m}(\R)$ to $H^{m-1}(\R)$ if $\frac{1}{2}< m \le s$, and from $H_0^{\frac{1}{2}}(\R)$ to $H^{-\frac{1}{2}}(\R)$ if $m=\frac{1}{2}$. Moreover, there exists a positive constant $C=C(s,R)>0$ such that for all $\frac{1}{2}\le m\le s$ and $\psi\in H^{m}(\R)$,
\begin{equation*}
    \|G[\eta](\psi)\|_{H^{m-1}(\R)} \le C\dfrac{R+ \|\teta\|_{H^{s+\frac{1}{2}}}}{R^2} \big(1+\fU_s(\teta)\big) \dfrac{|\fU_s(\teta)|^{m+\frac{1}{2}}}{|\fl(\teta)|^{m+\frac{1}{2}}} \|\psi_z\|_{H^{m-1}(\R)},
\end{equation*} 
where the functional $\fl(\teta)$ and $\fU_s(\teta)$ are defined as
\begin{align*}
    \fl(\teta) \vcentcolon=& \min\Big\{\frac{1}{2}, \frac{(R-\|\teta\|_{L^{\infty}(\R)})^2}{1+2\|\d_z\eta\|_{L^{\infty}(\R)}^2} \Big\},\\
    \fU_s(\teta) \vcentcolon=& \max\Big\{ R\|\teta\|_{H^{s-\frac{1}{2}}}, R\|\d_z\eta\|_{H^{s-\frac{1}{2}}} , \|\teta^2\|_{H^{s-\frac{1}{2}}}, \|\teta\d_z\eta\|_{H^{s-\frac{1}{2}}}, \|(\d_z\eta)^2\|_{H^{s-\frac{1}{2}}}\Big\}.
\end{align*}
\end{lemma}
\begin{remark}
	Here, we present another way to construct the operator $G[\eta](\psi)$. Fix $y_0\in(0,1)$ small (say $y_0=\frac{1}{4}$). We denote $I_0\equiv [y_0,1]$ and $\mS^0\equiv\R\times[y_0,1]$. Then Lemma \ref{lemma:cacci} and Corollary \ref{corol:Phi} implies that $v(z,y)$ satisfies 
	\begin{align*}
		\int_{I_0}\! \|\nabla v(\cdot,y)\|_{H^{s-\frac{1}{2}}(\R)}^2 \, \dif y \le \dfrac{C(s)}{y_0} \Big|\dfrac{\fU_s(\teta)}{\fl(\teta)}\Big|^{2s+1} \| \psi_z \|_{H^{s-1}(\R)}^2.
	\end{align*}
	Define $\mathcal{W}\vcentcolon=A_{21} v_z + A_{22} v_y = \rho ( \d_2 v - \d_z \rho \d_1 v )$, where recall that $\d_1 \vcentcolon= \d_z - \tfrac{\d_z \rho}{\d_y \rho} \d_y$, and $\d_2 \vcentcolon= \tfrac{1}{\d_y \rho} \d_y$. First, we see from Lemma \ref{lemma:cacci} that $\mathcal{W}\in L^{2}(\mS^0)$. Since $v$ is the solution to (\ref{reform}), it follows that that,
	\begin{align}\label{dydz}
		\d_ y \mathcal{W} = \d_y (A_{21} v_z + A_{22} v_y) = -\d_z( A_{11} v_z + A_{21} v_y ) = -\d_z\big( \rho \d_y   \rho \d_1 v \big).
	\end{align}
	For any $\vp(z)\in H_0^1(\R)$ and $\chi(y)\in L^2(I_0)$, we obtain by (\ref{dydz}) that
	\begin{align}\label{H-1temp1}
		&\Big|\iint_{\mS^0}\!\! \d_y \mathcal{W} \chi \vp \, \dif y \dif z\Big| = \Big|\int_{I_0} \int_{\R}  \!\! \d_z \big( \rho \d_y \rho \d_1 v \big) \chi \vp \, \dif z \dif y  \Big|\\
		=& \Big| \int_{I_0} \int_{\R} \!\!  \rho \d_y \rho \d_1 v  \chi \d_z \vp \, \dif z \dif y \Big|
		%\le& \Big( \iint_{\mS^0}\!\! |\rho \d_y \rho \d_1 v |^2 \dif z \dif y  \Big)^{\frac{1}{2}} \Big( \int_{I_0}\!\! |\chi|^2\, \dif y \Big)^{\frac{1}{2}} \Big( \int_{\R}\!\! |\d_z \vp|^2 \, \dif z \Big)^{\frac{1}{2}}. \nonumber\\
        \le \Big( \iint_{\mS^0}\!\! |\rho \d_y \rho \d_1 v |^2 \dif z \dif y  \Big)^{\frac{1}{2}} \|\chi\|_{L^2(I_0)} \|\d_z\varphi\|_{L^2(\R)}. \nonumber
	\end{align}
	Using (\ref{Jacobian}), (\ref{chain}), and Lemma \ref{lemma:cacci} with $m=\frac{1}{2}$, we have
	\begin{align}
		&\iint_{\mS^0}\!\! |\rho \d_y \rho \d_1 v |^2  \, \dif y \dif z \le \dfrac{2}{y_0}\int_{0}^1\!\! \int_{\R}\!\big\{ y \eta^2 |v_z|^2 + y^3 \eta^2 |\eta_z|^2 |v_y|^2 \big\} \, y \dif y\nonumber\\
		\le & \dfrac{4}{y_0}(\|\teta\|_{L^{\infty}}^2 + R^2)(1+\|\teta_z\|_{L^{\infty}}^2) \int_{0}^1\!\!\|\nabla v(\cdot,y)\|_{L^2}^2\, y\dif y \le C(s) \dfrac{|\fU_s(\teta)|^{4}}{|\fl(\teta)|^2} \|\psi\|_{H^{\frac{1}{2}}}.  \label{H-1temp2}
	\end{align}
	Combining this with (\ref{H-1temp1}), we obtain that $\d_y \mathcal{W} \in L^2\big( I_0 ; H^{-1}(\R) \big)$, where $H^{-1}$ is the functional dual space of $H_0^1$. Since $\mathcal{W}\in L^2\big( I_0 ; L^2(\R) \big)$, we can apply the Interpolation theorem for Bochner spaces, Proposition \ref{prop:bochner} to obtain that $\mathcal{W} \in \mC^0\big( I_0 ; H^{-\frac{1}{2}}(\R) \big)$. Therefore the trace $\lim_{y\to 1^{-}} \mathcal{W}(\cdot,y)$ is well defined in $H^{-\frac{1}{2}}(\R)$. By coordinate transformation law (\ref{coordT})--(\ref{coordInv}) and (\ref{chain}), one can construct operator of the form:
	\begin{equation*}
		\tilde{G}[\eta](\psi) = \eta ( \d_r \Psi - \d_z \eta \d_z \Psi ) \vert_{r=\eta(z)} =\rho (\d_2v - \d_z\rho \d_1 v) \vert_{y=1} = \lim\limits_{y\to 1^{-}} \mathcal{W}(\cdot, y)\in H^{-\frac{1}{2}}(\R). 
	\end{equation*}
	By (\ref{H-1temp2}) and Proposition \ref{prop:bochner}, there exists a positive monotone increasing function $x\mapsto C(x)$ such that
	\begin{equation*}
		\| \tilde{G}[\eta](\psi) \|_{H^{-1/2}} \le \| \mathcal{W} \|_{L^2(I_0;H^{1/2})}^{1/2} \| \d_y\mathcal{W} \|_{L^2(I_0;H^{-1/2})}^{1/2}   \le  C\big(\|\teta\|_{H^{s+1/2}}\big) \| \psi \|_{H^{1/2}}.
	\end{equation*}
	The Dirichlet-Neumann operator $G[\eta](\psi)$ is then defined as $G[\eta](\psi) \vcentcolon= \tfrac{1}{\eta} G[\eta](\psi)$. By Propositions \ref{prop:Sobcomp}--\ref{prop:clprod}, one can verify that 
	\begin{align*}
		\|G[\eta](\psi)\|_{H^{-1/2}(\R)} 
		\le C\big(\|\teta\|_{H^{s+1/2}(\R)}\big) \| \psi \|_{H^{1/2}(\R)}.
	\end{align*}
\end{remark}

%-----------------------------
%        subsection           
%-----------------------------

\subsection{Shape Derivative of Dirichlet-Neumann Operator}
This section is devoted to proving the following theorem:
\begin{theorem}\label{thm:shape}
Let $s>\frac{5}{2}$ and $1\le \sigma \le s$. Suppose $\psi\in H^{\sigma}(\R)$ and $\eta-R\in H^{s+\frac{1}{2}}(\R)$ satisfy $\eta\ge c$ for some $c>0$. Then there exists a neighbourhood $\mathcal{U}_{\eta} \subset H^{s+\frac{1}{2}}(\R)$ such that $\eta\in \mathcal{U}_{\eta}$ and the mapping
\begin{equation}
    \varrho \mapsto G[\varrho] (\psi) \in H^{\sigma-1}(\R) \ \text{ is differentiable in } \ \varrho\in \mathcal{U}_{\eta} \subset H^{s+\frac{1}{2}}(\R) .
\end{equation}
Moreover, for $h\in H^{s+\frac{1}{2}}(\R)$, the shape derivative $\dif_{\eta} G[\eta] (\psi) \cdot h$ is given by
\begin{gather}
\dif_{\eta} G[\eta] (\psi) \cdot h \vcentcolon= \lim\limits_{\ep\to 0} \dfrac{G[\eta+\ep h](\psi) - G[\eta](\psi)}{\ep} = - G[\eta](h\mathcal{B}) - \d_z (h V) - h \dfrac{\mathcal{B}}{\eta},\label{dG}\\
\text{where } \quad  \mathcal{B}\vcentcolon= \dfrac{\d_z \eta \d_z \psi + G[\eta](\psi)}{1+|\d_z \eta|^2}, \qquad V \vcentcolon= \d_z \psi- \mathcal{B} \d_z \eta .\label{BV}
\end{gather}
\end{theorem}

\begin{remark}\label{rem:BV}
By the definition of DN operator (\ref{DN-flat}), and using the fact that $v\vert_{y=1}=\psi$, one sees that $\mB$ can be written as 
\begin{align*}
\mathcal{B}%=& \dfrac{ \d_z \eta \d_z \psi + G[\eta](\psi) }{1+|\d_z \eta|^2}\nonumber\\ 
=& \dfrac{ \eta \d_z \eta \d_z v + (1+y|\d_z \eta|^2)\d_y v - \eta \d_z \eta \d_z v }{\eta (1+y|\d_z \eta|^2)}\Big\vert_{y=1} = \dfrac{\d_y v}{\eta}\big\vert_{y=1}.
\end{align*}
Translating this back into the original cylindrical coordinate system $(z,r)$, with the transformation law (\ref{coordT}) and (\ref{chain}), one sees that $( \eta^{-1}\d_y v)(z,\tfrac{r}{\eta}) = \d_r \Psi(z,r)$, where $\Psi$ is the velocity potential function solving (\ref{ellip}). Thus $\mB= \d_r \Psi \vert_{r=\eta(z)}$, which indicates that $\mB$ is the radial velocity of the fluid at free boundary surface. In addition, since $\psi(z)= v(z,y)\vert_{y=1}$, it follows from the transformation law (\ref{coordT}) and (\ref{chain}) that
\begin{align*}
V = \big\{ \d_z v - \dfrac{\d_y v}{\eta} \d_z \eta \big\}\big\vert_{y=1} = \big\{ \d_z v - \dfrac{y \d_z \eta}{\eta}  \d_y v \big\}\big\vert_{y=1} = \d_z \Psi \vert_{r=\eta(z)}.
\end{align*}
This implies that $V$ is the axial velocity of the fluid at free boundary surface.       
\end{remark}
\subsubsection{Surface perturbation}  
To show Theorem \ref{thm:shape}, we recall the Dirichlet boundary problem (\ref{reform}) in the flat strip $(z,y)\in \mS \equiv \R\times[0,1]$ stated as:
\begin{subequations}\label{vflat}
\begin{align}
&-\div\big(A(\rho)\cdot\nabla v\big) = 0 && \text{for } \ (z,y) \in \mS,\label{vflat-1}\\
& v(z,1)  = \psi(z), \quad \d_y v(z,0) = 0  && \text{for } \ z\in \R,\label{vflat-2}
\end{align}
\vspace*{-0.7cm}
\begin{equation}\label{vdiv}
\text{with } \ \ A(\rho) \vcentcolon= \begin{pmatrix}
\rho\d_y\rho & -\rho \d_z \rho\\
-\rho\d_z \rho & \frac{\rho(1+|\d_z\rho|^2)}{\d_y\rho}
\end{pmatrix}=\begin{pmatrix}
y \eta^2 & -y^2 \eta \d_z \eta\\
-y^2 \eta \d_z \eta & y(1+y^2|\d_z \eta|^2)
\end{pmatrix},
\end{equation}
\end{subequations}
where $\rho(z,y)\vcentcolon=y\eta(z)$, and we used the notations $\nabla\equiv(\d_z,\d_y)^{\top}$, $\div\equiv\d_z + \d_y$. Expanding (\ref{vflat-1}), and dividing the resultant equation by $y\eta^2$, we also have
\begin{gather}
\mL_{\eta}v \vcentcolon= \alpha \d_y^2 v + \d_z^2 v + \beta \d_z \d_y v - \gamma \d_y v =0,\label{Leta} \\
	\text{where } \ \alpha\vcentcolon= \dfrac{1+y^2|\d_z\eta|^2}{\eta^2}, \quad \beta\vcentcolon= - \dfrac{2y\d_z\eta}{\eta}, \quad \gamma= \dfrac{y^2\eta \d_z^2 \eta - 2y^2 |\d_z\eta|^2 -1}{y\eta^2}.\nonumber
\end{gather}
For given $\ep>0$ and $h\in\mC_c^{\infty}(\R)$, we define the surface perturbations to be
\begin{subequations}\label{Rep}
\begin{gather}
\eta_{\ep}(z)\vcentcolon= \eta(z)+ \ep h(z), \qquad A_{\ep}\vcentcolon= A(\rho_{\ep})=A(\rho+\ep y h),\\
\text{where } \ \rho_\ep(z,y)\vcentcolon= y ( \eta + \ep h ) = \rho(z,y) + \ep y h(z).
\end{gather}
\end{subequations}
\begin{subequations}\label{vpert}
Moreover, we let $v_{\ep}(z,y)$ be the solution to the perturbed Dirichlet boundary problem:
\begin{align}
&-\div\big( A_{\ep}\cdot \nabla v_{\ep} \big) = 0 && \text{for } \ (z,y) \in \mS,\label{vpert-1}\\
& v_{\ep}(z,1) = \psi(z), \qquad  \d_y v_{\ep}(z,0) = 0  && \text{for } \ z\in \R,\label{vpert-2}
\end{align}
\end{subequations}
In addition, we define the limits
\begin{equation}\label{dvdAdxi}
\dif v\cdot h \vcentcolon= \lim\limits_{\ep\to 0 } \dfrac{v_{\ep}-v}{\ep}, \quad \dif \rho\cdot h\vcentcolon= \lim\limits_{\ep\to 0} \dfrac{\rho_{\ep}-\rho}{\ep}, \quad \dif A \cdot h \vcentcolon= \lim\limits_{\ep\to 0} \dfrac{A(\rho_{\ep})-A(\rho)}{\ep}.
\end{equation}
\begin{subequations}\label{dvh}
It immediately follows that $\dif \rho \cdot h = y h(z)$. Furthermore, subtracting (\ref{vflat}) with (\ref{vpert}), dividing the resultant equation with $\ep$, and taking the limit $\ep\to 0$, it can be verified that $\dif v\cdot h$ solves the problem:
\begin{align}
&-\div\big( A \cdot \nabla( \dif v \cdot h) \big) = \div\big( (\dif A \cdot h) \cdot \nabla v \big) && \text{for } \ (z,y)\in\mS,\\
& (\dif v\cdot h) (z,1) = 0 \qquad \d_y(\dif v \cdot h)(z,0)=0 && \text{for } \ z\in\R.
\end{align}
\end{subequations}
\begin{proposition}\label{prop:dG}
Suppose $(\eta,\psi)\in H^{s+\frac{1}{2}}(\R)\times H^s(\R)$ for some $s>0$ large enough, and let $G[\eta](\psi)$ be the operator defined by (\ref{DN-flat}). For $h\in\mC_{c}^{\infty}(\R)$, set $\dif_\eta G[\eta](\psi)\cdot h$ to be the shape derivative given in (\ref{dG}), and $\dif v\cdot h$ be the solution to (\ref{dvh}). Then 
\begin{align*}
\dif_{\eta} G[\eta](\psi)\cdot h =& \Big\{ \dfrac{1+y|\d_z \eta|^2}{\eta}\d_y (\dif v \cdot h) - \d_z \eta \d_z (\dif v \cdot h)\Big\}\Big\vert_{y=1}\nonumber\\
&+ \Big\{ \d_z h \dfrac{2y\d_z \eta}{\eta} \d_y v - \d_z h \d_z v - h\dfrac{1+y|\d_z \eta|^2}{\eta^2} \d_y v \Big\}\Big\vert_{y=1} 
\end{align*}
\end{proposition}
\begin{proof}
Computing $\frac{1}{\ep}\{ G[\eta_{\ep}](\psi) - G[\eta](\psi) \}$ using (\ref{DN-flat}) and (\ref{Rep}), we have
\begin{align*}
& \dfrac{1}{\ep} \big\{ G[\eta_{\ep}](\psi) - G[\eta](\psi) \big\} \\
=& \dfrac{1}{\ep}\Big\{ \dfrac{1+y|\d_z\eta+\ep \d_z h|^2}{\eta+\ep h} \d_y v_{\ep} - (\d_z \eta + \ep \d_z h) \d_z v_{\ep} -\dfrac{1+y|\d_z \eta|^2}{\eta}\d_y v + \d_z \eta \d_z v \Big\}\Big\vert_{y=1}\\
%=& \dfrac{1}{\ep}\Big\{ \dfrac{1+y(|\d_z \eta|^2 + 2 \ep \d_z\eta \d_z h + \ep^2 |\d_z h|^2)}{\eta+\ep h}\d_y v_{\ep} - \dfrac{1+ y|\d_z\eta|^2}{\eta}\d_y v \Big\}\Big\vert_{y=1}\\ &- \d_z \eta \d_z \big(\dfrac{v_{\ep}-v}{\ep}\big)\big\vert_{y=1}- \d_zh \d_z v_{\ep}\vert_{y=1}\\
=&  \dfrac{1+ y |\d_z \eta|^2}{\ep} \Big\{ \dfrac{\d_y v_{\ep}}{\eta+\ep h} - \dfrac{\d_y v}{\eta} \Big\} + \dfrac{2 y\d_z \eta \d_z h}{\eta+\ep h} \d_y v_{\ep} + \ep \d_y v_{\ep} \dfrac{y|\d_z h|^2}{\eta+\ep h}\\ &- \d_z \eta  \d_z \big(\dfrac{v_{\ep}-v}{\ep}\big)\big\vert_{y=1}- \d_zh \d_z v_{\ep}\vert_{y=1}\\
=& (1+ y |\d_z \eta|^2) \Big\{ \dfrac{1}{\eta+\ep h} \d_y\big( \dfrac{v_{\ep}-v}{\ep} \big) - \dfrac{h \d_y v}{\eta(\eta+\ep h)} \Big\} + \dfrac{2 y\d_z \eta \d_z h}{\eta+\ep h} \d_y v_{\ep} \\ &+ \ep \d_y v_{\ep} \dfrac{y|\d_z h|^2}{\eta+\ep h} - \d_z \eta \d_z \big(\dfrac{v_{\ep}-v}{\ep}\big)\big\vert_{y=1}- \d_zh \d_z v_{\ep}\vert_{y=1}\\
=& \Big\{ \dfrac{1+y|\d_z\eta|^2}{\eta+\ep h} \d_y \big(\dfrac{v_{\ep}-v}{\ep}\big) - \d_z \eta \d_z \big(\dfrac{v_{\ep}-v}{\ep}\big) \Big\}\Big\vert_{y=1} \\ &+ \Big\{ \dfrac{2y\d_z \eta \d_z h}{\eta+\ep h}\d_y v_{\ep} - \d_z h \d_z v_{\ep} - h\dfrac{1+ y|\d_z\eta|^2}{\eta(\eta+\ep h)}\d_y v  \Big\}\Big\vert_{y=1} + \ep \d_y v_{\ep} \dfrac{y |\d_z h|^2}{\eta+\ep h}\Big\vert_{y=1}
\end{align*}
Taking the limit on the above, and using (\ref{dvdAdxi}), we have
\begin{align*}
\dif_{\eta} G[\eta](\psi)\cdot h =& \lim\limits_{\ep\to 0} \dfrac{G[\eta_\ep](\psi)-G[\eta](\psi)}{\ep} \\
=&  \Big\{ \dfrac{1+y|\d_z \eta|^2}{\eta}\d_y (\dif v \cdot h) - \d_z \eta \d_z (\dif v \cdot h)\Big\}\Big\vert_{y=1} \\
&+ \Big\{ \d_z h \dfrac{2y\d_z \eta}{\eta} \d_y v - \d_z h \d_z v - h\dfrac{1+y|\d_z \eta|^2}{\eta^2} \d_y v \Big\}\Big\vert_{y=1}.
\end{align*}
This concludes the proof.
\end{proof}
With few lines of calculation, one can also verify the following proposition:
\begin{proposition}\label{prop:dAh}
Let $A(\rho)$ be the matrix in (\ref{vdiv}), and $\dif A\cdot h$ in (\ref{dvdAdxi}) then
\begin{align*}
\dif A \cdot h =&\begin{pmatrix}
\rho h + y h \d_y \rho & - y (\rho \d_z h + h \d_z \rho)\\
- y (\rho \d_z h + h \d_z \rho) & \dif A_{22} \cdot h
\end{pmatrix} = \begin{pmatrix}
2 y \eta h & -y^2 \d_z(\eta h) \\
-y^2 \d_z(\eta h) & 2 y^3 \d_z \eta \d_z h
\end{pmatrix},
\end{align*}
where $\dif A_{22}\cdot h\vcentcolon= |\d_y\rho|^{-2}\{2 y \rho \d_z \rho \d_y \rho \d_z h - h\rho + y h \d_y \rho + h |\d_y\rho|^2(y\d_y\rho-\rho)\}=2 y^3 \d_z \eta \d_z h$.
\end{proposition}

\subsubsection{Cylindrical harmonic extension of \texorpdfstring{$h\mB$}{hB}}
The next lemma is a special case of the result obtained by Lannes in \cite{lannes} (Lemma 3.22).
\begin{lemma}\label{lemma:BigTheta}
Define $\varTheta_h\vcentcolon=\frac{\dif \rho \cdot h}{\d_y \rho} \d_y v = \frac{yh}{\eta} \d_y v$. Then $\varTheta_h$ solves the following problem
\begin{align*}
&-\textnormal{\div}\big( A(\rho) \cdot \nabla \varTheta_h \big) = \textnormal{\div}\big( (\dif A\cdot h) \cdot \nabla v \big) && \text{for } \ (z,y)\in \mS,\\
&\varTheta_h(z,1) = h\mB, \qquad \d_y \varTheta_h(z,0) = 0 && \text{for } \ z\in\R.
\end{align*}
\end{lemma}  
\begin{proof}
First, we compute $\d_y \varTheta_h = \frac{h}{\eta}\d_y v + \frac{yh}{\eta} \d_y^2 v$. Therefore it follows from (\ref{vflat-2}) that $\d_y \varTheta_h(z,0) = (\tfrac{h}{\eta}\d_y v + \tfrac{yh}{\eta} \d_y^2 v)\vert_{y=0} = 0$. Moreover, from Remark \ref{rem:BV}, we have the identity $h\mB=h\tfrac{1}{\eta}\d_y v\vert_{y=1} = \tfrac{yh}{\eta}\d_y v\vert_{y=1}=\varTheta_h\vert_{y=1}$. Using the expression $\varTheta_h\vcentcolon= \frac{\dif \rho \cdot h}{\d_y \rho} \d_y v$ we obtain
\begin{align*}
\div ( A \nabla \varTheta_h ) =& \div\big( \dfrac{\dif \rho \cdot h}{\d_y \rho} A\cdot \nabla \d_y v \big) +  \div\Big( \d_y v A\cdot \nabla \big( \dfrac{\dif \rho \cdot h}{\d_y \rho} \big)  \Big)\\
=& \nabla^{\top}\big(\dfrac{\dif \rho \cdot h}{\d_y \rho}\big) \cdot A \cdot \nabla \d_y v + \dfrac{\dif \rho \cdot h}{\d_y \rho} \div\big(  A\cdot \nabla \d_y v \big) + \div\Big( \d_y v A\cdot \nabla \big( \dfrac{\dif \rho \cdot h}{\d_y \rho} \big)  \Big).
\end{align*}
Next, using the fact that $v$ satisfies $\div (A\nabla v)=0$, one has
\begin{align*}
\div ( A \nabla \varTheta_h )
\!\!= \nabla^{\top}\big(\dfrac{\dif \rho \cdot h}{\d_y \rho}\big) \cdot A \cdot \nabla \d_y v - \dfrac{\dif \rho \cdot h}{\d_y \rho} \div\big(   \d_y A \cdot \nabla v \big) + \div\Big( \d_y v A\cdot \nabla \big( \dfrac{\dif \rho \cdot h}{\d_y \rho} \big)  \Big).
\end{align*}
By the product rule, $A\nabla \d_y v + \d_y A \nabla v=\d_y(A\nabla v)$, hence we have
\begin{align*}
\div( A \nabla \varTheta_h )
=& \nabla^{\top}\big(\dfrac{\dif \rho \cdot h}{\d_y \rho}\big) \cdot A \cdot \nabla \d_y v - \div\big( \dfrac{\dif \rho \cdot h}{\d_y \rho} \d_y A \cdot \nabla v \big) \\
&+ \nabla^{\top}\big(\dfrac{\dif \rho \cdot h}{\d_y \rho}\big)\cdot \d_y A \cdot \nabla v + \div\Big( \d_y v A\cdot \nabla \big( \dfrac{\dif \rho \cdot h}{\d_y \rho} \big)  \Big)\\
=& \nabla^{\top}\!\big(\dfrac{\dif \rho \cdot h}{\d_y \rho}\big)\cdot \d_y \big(A\!\cdot\!\nabla v\big) - \div\big( \dfrac{\dif \rho \cdot h}{\d_y \rho} \d_y A \!\cdot\!\nabla v \big) + \div\Big( \d_y v A\!\cdot\! \nabla \big( \dfrac{\dif \rho \cdot h}{\d_y \rho} \big)  \Big).
\end{align*}
Using the product rule and $\div(A\nabla v)=0$ once again, we have
\begin{align*}
\div ( A \nabla \varTheta_h ) =& \d_y \Big( \nabla^{\top}\big( \dfrac{\dif \rho \cdot h}{\d_y \rho}\big) \cdot A \cdot \nabla v\Big) - \d_y \nabla^{\top} \big(\dfrac{\dif \rho\cdot h}{\d_y \rho}\big) A\cdot \nabla v \\ & - \div\big( \dfrac{\dif \rho \cdot h}{\d_y \rho} \d_y A \cdot \nabla v \big) + \div\Big( \d_y v A\cdot \nabla \big( \dfrac{\dif \rho \cdot h}{\d_y \rho} \big)  \Big)\\
=& \d_y \Big( \Big\{ A^{\top}\cdot \nabla\big( \dfrac{\dif \rho \cdot h}{\d_y \rho}\big) \Big\}^{\top} \cdot \nabla v\Big) - \div\Big(\d_y\big(\dfrac{\dif \rho\cdot h}{\d_y \rho}\big) A\cdot \nabla v\Big) \\ & - \div\big( \dfrac{\dif \rho \cdot h}{\d_y \rho} \d_y A \cdot \nabla v \big) + \div\Big( \d_y v A\cdot \nabla \big( \dfrac{\dif \rho \cdot h}{\d_y \rho} \big)  \Big).
\end{align*}
Combining the second and third terms in the right hand side of above equation, we get 
\begin{align*}
&\div ( A \nabla \varTheta_h )\\
=& \d_y \Big( \Big\{ A^{\top}\cdot \nabla\big( \dfrac{\dif \rho \cdot h}{\d_y \rho}\big) \Big\}^{\top} \cdot \nabla v\Big) - \div\Big( \d_y\big(\dfrac{\dif \rho \cdot h}{\d_y \rho}  A\big) \cdot \nabla v \Big) + \div\Big( \d_y v A\cdot \nabla \big( \dfrac{\dif \rho \cdot h}{\d_y \rho} \big)  \Big).
\end{align*}
Using the fact that $A$ is symmetric, it follows that
\begin{gather*}
\div( A\cdot \nabla \varTheta_h ) = \div( Q\cdot \nabla v ), \quad \text{ where }	\\
Q \vcentcolon=\begin{pmatrix}
0_{2\times 1} & A\cdot \nabla\big(\dfrac{\dif \rho\cdot h}{\d_y \rho}\big)
\end{pmatrix} + \begin{pmatrix}
0_{1\times 2}\\
\Big\{ A\cdot \nabla\big(\dfrac{\dif \rho\cdot h}{\d_y \rho}\big)\Big\}^{\top}
\end{pmatrix} - \d_y \big( \dfrac{\dif \rho \cdot h}{\d_y \rho} A \big),
\end{gather*}
where $0_{n\times m}$ denotes the $0$ matrix with size $n$ rows and $m$ columns. Therefore it is left to show that $Q=-\dif A \cdot h$. By $\rho=y\eta$ and $\dif \rho \cdot h = yh$,
\begin{align*}
A\cdot \nabla\big( \dfrac{\dif \rho \cdot h}{\d_y \rho} \big) =& \begin{pmatrix}
	y \eta^2 & -y^2 \eta \d_z \eta\\
	-y^2 \eta \d_z \eta & y(1+y^2|\d_z \eta|^2)
\end{pmatrix} \cdot \begin{pmatrix}
y \eta^{-2} (\eta \d_z h - h \d_z \eta)\\
\eta^{-1} h
\end{pmatrix}\\
=& \begin{pmatrix}
y^2(\eta \d_z h - 2 h \d_z \eta)\\
yh \eta^{-1} + 2 y^3 h \eta^{-1}|\d_z \eta|^2 - y^3 \d_z \eta \d_z h
\end{pmatrix}.
\end{align*}
In addition, we also have that
\begin{align*}
-\d_y \big( \dfrac{\dif \rho \cdot h}{\d_y \rho} A \big) =& \d_y \begin{pmatrix}
-y^2 h \eta & y^3 h \d_z \eta\\
y^3 h \d_z \eta & -y^2 h \eta^{-1} +- y^4 h \eta^{-1} |\d_z \eta|^2 
\end{pmatrix}\\
 =& \begin{pmatrix}
-2 y h \eta & 3 y^2 h \d_z \eta\\
 3 y^2 h \d_z \eta & -2 y h \eta^{-1} - 4 y^3 h \eta^{-1} |\d_z \eta|^2 
\end{pmatrix}.
\end{align*}
Using these expressions and Proposition \ref{prop:dAh}, we can compute $Q$ to obtain
\begin{align*}
Q \vcentcolon=&\begin{pmatrix}
	0_{2\times 1} & A\cdot \nabla\big(\dfrac{\dif \rho\cdot h}{\d_y \rho}\big)
\end{pmatrix} + \begin{pmatrix}
	0_{1\times 2}\\
	\Big\{ A\cdot \nabla\big(\dfrac{\dif \rho\cdot h}{\d_y \rho}\big)\Big\}^{\top}
\end{pmatrix} - \d_y \big( \dfrac{\dif \rho \cdot h}{\d_y \rho} A \big)\\
=& \begin{pmatrix}
	0_{2\times 1} & A\cdot \nabla\big(\dfrac{\dif \rho\cdot h}{\d_y \rho}\big)
\end{pmatrix} + \begin{pmatrix}
-2yh\eta & 3 y^2 h \d_z \eta\\
y^2(h \d_z \eta + \eta \d_z h) & - \tfrac{y h}{\eta} - y^3 \d_z \eta \d_z h - \tfrac{2 y^3 h |\d_z \eta|^2}{\eta}
\end{pmatrix}\\
=& \begin{pmatrix}
-2yh\eta & y^2 (h \d_z \eta + \eta \d_z h)\\
y^2 (h\d_z \eta + \eta \d_z h) & -2y^3 \d_z \eta \d_z h
\end{pmatrix} = - \dif A \cdot h.
\end{align*}
This concludes the proof of the lemma.
\end{proof}
Define $\vartheta_h\vcentcolon= \varTheta_h-\dif v\cdot h$. Then combining Lemma \ref{lemma:BigTheta}, and the fact that $\dif v\cdot h$ solves the problem (\ref{dvh}), we obtain the following:
\begin{corollary}\label{corol:hBext} $\vartheta_h\vcentcolon= \varTheta_h-\dif v\cdot h$ solves the problem
\begin{subequations}\label{varthetah}
\begin{align}
&-\div\big(A(\rho)\cdot\nabla \vartheta_h\big) = 0 && \text{for } \ (z,y)\in \mS,\\
&\vartheta_h(z,1) = h\mB, \qquad \d_y \vartheta_h(z,0)=0 && \text{for } \ z\in\R. 
\end{align}
\end{subequations}
\end{corollary}
\subsubsection{Shape derivative of \texorpdfstring{$G[\eta]$}{G[eta]}, proof of Theorem \ref{thm:shape}.}
In this subsection, we use Proposition \ref{prop:dG} and Corollary \ref{corol:hBext} to conclude the proof for Theorem \ref{thm:shape}. 
\begin{proof}[Proof of Theorem \ref{thm:shape}]
Since $\vartheta_h= \varTheta_h-\dif v\cdot h$ solves (\ref{varthetah}), it follows from the definition of DN operator (\ref{DN-flat}) that
\begin{align*}
G[\eta](h\mathcal{B})
=& \Big\{ \dfrac{1+y|\d_z \eta|^2}{\eta}\d_y \varTheta_h - \d_z \eta \d_z \varTheta_h \Big\}\Big\vert_{y=1}\\
&- \Big\{ \dfrac{1+y|\d_z\eta|^2}{\eta}\d_y (\dif v\cdot h) - \d_z \eta \d_z (\dif v\cdot h) \Big\}\Big\vert_{y=1}
\end{align*}
Combining the above with Proposition \ref{prop:dG}, we obtain that
\begin{align}\label{dGhtemp}
\dif_{\eta} G[\eta](\psi)\cdot h =& - G[\eta](h \mathcal{B}) + \Big\{ \dfrac{1+y|\d_z\eta|^2}{\eta}\d_y \varTheta_h - \d_z \eta \d_z \varTheta_h \Big\}\Big\vert_{y=1} \nonumber\\
&+ \Big\{ \d_z h \dfrac{2y\d_z \eta}{\eta} \d_y v - \d_z h \d_z v - h\dfrac{1+y|\d_z \eta|^2}{\eta^2} \d_y v \Big\}\Big\vert_{y=1} \nonumber\\
=\vcentcolon& - G[\eta](h \mathcal{B}) + I\vert_{y=1}.
\end{align}
Using $\varTheta_h=\frac{yh}{\eta}\d_y v$ and equation (\ref{Leta}), we have
\begin{align*}
I = & \dfrac{1+y|\d_z\eta|^2}{\eta}\d_y \varTheta_h - \d_z \eta \d_z \varTheta_h + \d_z h \dfrac{2y\d_z \eta}{\eta} \d_y v - \d_z h \d_z v - h\dfrac{1+y|\d_z \eta|^2}{\eta^2} \d_y v \\
=& h \dfrac{1+ y |\d_z \eta|^2}{\eta^2} \d_y v + h\dfrac{y+y^2|\d_z\eta|^2}{\eta^2}\d_y^2 v - h \dfrac{y\d_z\eta}{\eta} \d_z \d_y v - \d_z h \dfrac{y\d_z\eta}{\eta} \d_y v \\
&+ h \dfrac{y |\d_z\eta|^2}{\eta^2} \d_y v + \d_z h \dfrac{2y\d_z\eta}{\eta}\d_y v - \d_z h \d_z v - h \dfrac{1+y|\d_z\eta|^2}{\eta^2}\d_y v\\
%=& h \Big\{ \dfrac{y+y^2|\d_z\eta|^2}{\eta^2}\d_y^2 v + \dfrac{y|\d_z\eta|^2}{\eta^2}\d_y v -\dfrac{y\d_z\eta}{\eta}\d_z\d_y v \Big\} + \d_z h \Big\{ \dfrac{y\d_z\eta}{\eta}\d_y v - \d_z v \Big\}\\
%=& h \Big\{ \dfrac{y+y^2|\d_z\eta|^2}{\eta^2}\d_y^2 v + \dfrac{y|\d_z\eta|^2}{\eta^2}\d_y v -\dfrac{y\d_z\eta}{\eta}\d_z\d_y v \Big\} + \d_z \Big\{ h \big(\d_z\eta\dfrac{y\d_y v}{\eta} - \d_z v\big) \Big\}\\ &- h \Big\{ \dfrac{y\d_z^2 \eta}{\eta} \d_y v - \dfrac{y|\d_z\eta|^2}{\eta^2}\d_y v + \dfrac{y \d_z \eta}{\eta} \d_z\d_y v - \d_z^2 v \Big\}\\
%=& h \Big\{ \d_z^2 v - \dfrac{2y\d_z\eta}{\eta}\d_z\d_y v + \dfrac{2y^2|\d_z\eta|^2-y^2\eta\d_z^2\eta}{y \eta^2} \d_y v + \dfrac{y+y^2|\d_z\eta|^2}{\eta^2}\d_y^2 v \Big\}\\ &+ \d_z \Big\{ h \big(\d_z\eta\dfrac{y\d_y v}{\eta} - \d_z v\big) \Big\}\\ 
=& h \Big\{ \d_z^2 v - \dfrac{2y\d_z\eta}{\eta}\d_z\d_y v + \dfrac{1+2y^2|\d_z\eta|^2-y^2\eta\d_z^2\eta}{y \eta^2} \d_y v + \dfrac{1+y^2|\d_z\eta|^2}{\eta^2}\d_y^2 v \Big\}\\ &+ \d_z \Big\{ h \big(\d_z\eta\dfrac{y\d_y v}{\eta} - \d_z v\big) \Big\} + h\Big\{ \dfrac{y-1}{\eta^2}\d_y^2 v - \dfrac{\d_y v}{y \eta^2} \Big\}\\
=& - \d_z \Big\{ h \big(\d_z v - \d_z\eta\dfrac{y\d_y v}{\eta}\big) \Big\} + h\Big\{ \dfrac{y-1}{\eta^2}\d_y^2 v - \dfrac{\d_y v}{y \eta^2} \Big\}
\end{align*}
By the boundary conditions (\ref{vflat-2}), $(\d_z v , \d_z^2 v)\vert_{y=1}=(\d_z \psi,\d_z^2 \psi)$. Moreover, by Remark \ref{rem:BV}, $\d_z( \frac{\d_ y v}{\eta})\vert_{y=1} = \d_z \mathcal{B}$. Using these, it follows that
\begin{align*}
I\vert_{y=1} =& - \d_z \Big\{ h \big(\d_z v - \d_z\eta\dfrac{y\d_y v}{\eta}\big) \Big\}\Big\vert_{y=1} - h \dfrac{\d_y v}{y \eta^2} \Big\vert_{y=1}\\
=& - \d_z \Big\{ h \big( \d_z \psi - \mathcal{B} \d_z \eta  \big) \Big\} - h \dfrac{\mathcal{B}}{\eta} = -\d_z (h V) - h \dfrac{\mathcal{B}}{\eta}.
\end{align*}
Substituting the above into (\ref{dGhtemp}), it follows that
\begin{equation*}
\dif_{\eta} G[\eta](\psi)\cdot h = - G[\eta](h\mathcal{B}) + I \vert_{y=1} = - G[\eta](h\mathcal{B}) - \d_z (h V) - h \dfrac{\mathcal{B}}{\eta}.
\end{equation*}
This conclude the proof of Theorem \ref{thm:shape}.
\end{proof}
\begin{corollary}\label{corol:RdGh}
Define $\tilde{G}[\eta](\psi) \vcentcolon= \eta\cdot G[\eta](\psi)$. If $\dif_{\eta} G[\eta](\psi)\cdot h$ exists then
\begin{equation*}
\dif_\eta \tilde{G}[\eta](\psi)\cdot h = - \tilde{G}[\eta](h\mB) -\d_z (h \eta V) \qquad \text{for } \ h\in\mC_{c}^{\infty}(\R).
\end{equation*}
\end{corollary}
\begin{proof}
By Theorem \ref{thm:shape}, we have
\begin{align*}
&\dif_\eta \tilde{G}[\eta](\psi)\cdot h = \lim\limits_{\ep\to 0} \dfrac{1}{\ep} \Big\{ (\eta+\ep h) G[\eta+\ep h](\psi) - \eta G[\eta](\psi) \Big\}\\
=& \eta \lim\limits_{\ep\to 0} \dfrac{1}{\ep} \Big\{ G[\eta+\ep h](\psi) - G[\eta](\psi) \Big\} + h \lim\limits_{\ep \to 0} G[\eta+\ep h](\psi)\\
=& -\eta G[\eta](h \mB) - \eta\d_z (h V) - h\mB  + h G[\eta](\psi)  \\
=& -\tilde{G}[\eta](h\mB) -\d_z(h\eta V) + h V \d_z \eta - h\mB + h G[\eta](\psi).
\end{align*}
Thus we claim that $ V \d_z \eta - \mB +  G[\eta](\psi)=0$. By the definition (\ref{BV}) for $\mB$ and $V$,
\begin{align*}
&V \d_z \eta - \mB + G[\eta](\psi) = \d_z\eta\d_z\psi - |\d_z\eta|^2 \mB -\mB + G[\eta](\psi)\\
 =& \d_z\eta\d_z\psi + G[\eta](\psi) - (1+|\d_z\eta|^2) \mB = \d_z\eta\d_z\psi + G[\eta](\psi) - \{ \d_z \eta \d_z \psi + G[\eta](\psi) \} =0.
\end{align*}
This completes the proof.
\end{proof}

\begin{corollary}[Perturbation with respect to cylinder]
Let $R>0$, and denote $G[R](\psi)$ as the Dirichlet-Neumann operator associated with the cylinder $\eta\equiv R$. Then 
\begin{align*}
\dif_{\eta} G[\eta](\psi)\cdot h \vert_{\eta=R} = -G[R]\big(hG[R](\psi)\big)-\d_z(h\d_z\psi) - \dfrac{hG[R](\psi)}{R}, \quad \text{for } \ h\in\mC_{c}^{\infty}(\R).
\end{align*}
\end{corollary}
\begin{proof}
For $\eta(z)=R>0$, we have $\d_z \eta = 0$, hence $\mathcal{B}=G[R](\psi)$ and $V=\d_z\psi$. Substituting these expression in Theorem \ref{thm:shape}, we obtain the result.
\end{proof}

\subsection{Cancellation in \texorpdfstring{$G[\eta](\mB)+\d_z V$}{G[eta](B)+ V\_z}}
\begin{proposition}\label{prop:Letab} Let $\fb\vcentcolon=\tfrac{1}{\eta}\d_y v$, and recall $\mL_{\eta}$ in (\ref{Leta}). Then $\mL_{\eta} \fb = \tfrac{1}{y^2\eta^3}\d_y v$.
\end{proposition}
\begin{proof}
Recalling the definition of $\alpha$, $\beta$, $\gamma$ in (\ref{Leta}), we compute the derivative of $\fb$:
\begin{gather*}
\d_y \fb = \dfrac{\d_y^2 v}{\eta}, \quad \d_y^2 \fb = \dfrac{\d_y^3 v}{\eta}, \quad \d_z \fb %= \dfrac{\d_y\d_z v}{\eta} - \dfrac{\d_z \eta \d_y v}{\eta^2}
=\dfrac{\d_y\d_z v}{\eta}+\dfrac{\beta\d_y v}{2y\eta}, \\
\d_y\d_z \fb %= \dfrac{\d_y^2 \d_z v}{\eta} - \dfrac{\d_z\eta \d_y^2 v}{\eta^2}
=\dfrac{\d_y^2\d_z v}{\eta} + \dfrac{\beta\d_y^2 v}{2y\eta}, \quad \d_z^2 \fb %= \dfrac{\d_y\d_z^2 v}{\eta} - 2 \dfrac{\d_z\eta \d_y \d_z v}{\eta^2}  + \dfrac{2|\d_z\eta|^2-\eta \d_z^2 \eta}{\eta^3}\d_y v
= \dfrac{\d_y\d_z^2 v}{\eta} + \dfrac{\beta\d_y\d_z v}{y\eta} - \big( \gamma + \dfrac{1}{y \eta^2} \big) \dfrac{\d_y v}{y\eta}.
\end{gather*}
Applying the operator $\mL_{\eta}$ defined in (\ref{Leta}) on $\fb$, and using the above, we obtain that
\begin{align}\label{temp:Letab}
\mL_{\eta} \fb =& \alpha\d_y^2 \fb + \d_z^2 \fb + \beta \d_y \d_z \fb -\gamma \d_y \fb\nonumber\\
=& \dfrac{\alpha \d_y^3 v}{\eta}  + \dfrac{\d_y\d_z^2 v}{\eta} + \dfrac{\beta \d_y\d_z v}{y\eta} - \big(\gamma+\dfrac{1}{y\eta^2}\big)\dfrac{\d_y v}{y\eta} + \beta\big( \dfrac{\d_y^2 \d_z v}{\eta} + \dfrac{\beta \d_y^2 v}{2y\eta} \big) - \dfrac{\gamma \d_y^2 v}{\eta}\nonumber\\
=& \dfrac{1}{\eta}\d_y \big(\alpha \d_y^2 v\big) - \dfrac{\d_y^2 v}{\eta}\d_y \alpha + \dfrac{1}{\eta}\d_y \d_z^2 v + \dfrac{\beta\d_y\d_z v}{y\eta} -\dfrac{\gamma \d_y v}{y\eta} - \dfrac{\d_y v}{y^2 \eta^3} \nonumber\\ &+ \dfrac{1}{\eta}\d_y\big( \beta \d_y \d_z v \big) -\dfrac{\d_y\d_z v}{\eta} \d_y \beta + \dfrac{\beta^2\d_y^2 v}{2y\eta} - \dfrac{1}{\eta}\d_y\big(\gamma \d_y v\big) + \dfrac{\d_y v}{\eta} \d_y \gamma \nonumber\\
=& \dfrac{1}{\eta} \d_y \big\{ \alpha \d_y^2 v + \d_z^2 v + \beta \d_y \d_z v - \gamma \d_y v \big\} - \dfrac{\d_y^2 v}{\eta} \d_y \alpha - \dfrac{\d_y \d_z v}{\eta}\d_y \beta + \dfrac{\d_y v}{\eta} \d_y \gamma \nonumber\\ &+ \dfrac{\beta\d_y\d_z v}{y\eta} - \dfrac{\gamma \d_y v}{y\eta} + \dfrac{\beta^2\d_y^2 v}{2y\eta} - \dfrac{\d_y v}{y^2 \eta^3}.  
\end{align}
Evaluating the derivatives of $\alpha$, $\beta$, and $\gamma$, we obtain that 
\begin{gather*}
\d_y \alpha = \d_y \dfrac{1+y^2|\d_z\eta|^2}{\eta^2} = \dfrac{2y|\d_z\eta|^2}{\eta^2} = \dfrac{\beta^2}{2y}, \quad \d_y \beta = - \dfrac{2\d_z\eta}{\eta} = \dfrac{\beta}{y},\\
\d_y \gamma = \d_y \dfrac{y^2 \eta \d_z^2 \eta - 2 y^2 |\d_z\eta|^2 - 1}{y\eta^2} = \dfrac{\eta\d_z^2\eta-2|\d_z\eta|^2}{\eta^2} + \dfrac{1}{y^2\eta^2} = \dfrac{\gamma}{y} + \dfrac{2}{y^2\eta^2}.
\end{gather*}
Using these, it follows that
\begin{align*}
-\dfrac{\d_y^2 v}{\eta} \d_y\alpha - \dfrac{\d_y\d_z v}{\eta} \d_y \beta + \dfrac{\d_y v}{\eta} \d_y \gamma = - \dfrac{\beta^2 \d_y^2 v}{2y\eta} - \dfrac{\beta \d_y \d_z v}{y\eta} + \dfrac{\gamma \d_y v}{y\eta} + \dfrac{2\d_y v}{y^2 \eta^3}. 
\end{align*}
Substituting the above into (\ref{temp:Letab}), we obtain:
\begin{align*}
\mL_{\eta} \fb = \dfrac{1}{\eta} \d_y \big\{ \alpha \d_y^2 v + \d_z^2 v + \beta \d_y \d_z v - \gamma \d_y v \big\} + \dfrac{\d_y v}{y^2 \eta^3} = \dfrac{\d_y v}{y^2 \eta^3},
\end{align*}
where we used the fact that $\mL_{\eta}v\equiv (\alpha \d_y^2  + \d_z^2  + \beta \d_y \d_z  - \gamma \d_y) v  =0$.
\end{proof}

Before proceeding onto the next lemma, we construct the following cut-off function. Let $\chi(x)\in \mC_{c}^{\infty}(\R)$ be a positive even function such that $\chi(x)>0$ for $|x|\le \tfrac{1}{2}$, $\chi(x)=0$ for $|x|\ge 1$, and $\int_{\R}\chi(x)\dif x = 1$. For each $y_0\in(0,1)$, define $\mX_0\in \mC^{\infty}[0,1]$ as:
\begin{equation*}
	\mX_0(y) \vcentcolon= \int_{-\infty}^{y} \chi\big(\tfrac{8}{y_0}(x-\tfrac{3y_0}{8})\big)\, \dif x \quad \text{for } \ y\in [0,1].
\end{equation*}
It can be verified that there exists a generic constant $c>0$ such that for $y\in[0,1]$,
\begin{equation}\label{mX0}
	0\le \mX_0(y) \le 1, \quad |\mX_0^{\prime}(y)| \le c, \quad |\mX_0^{\prime\prime}(y)| \le \tfrac{c}{y_0}, \quad \mX_0(y) =\left\{\begin{aligned}
		&1 && \text{if } \ y\in[\tfrac{y_0}{2},1],\\
		&0 && \text{if } \ y\in[0,\tfrac{y_0}{4}].
	\end{aligned}\right.
\end{equation}

\begin{lemma}\label{lemma:cancel}
Define $\vartheta\vcentcolon= E_{\eta}[\mB]$ where $E_{\eta}$ is the solution operator in Definition \ref{def:solOp}. For fixed $y_0\in(0,1]$, set $\varpi \vcentcolon= \vartheta - \mX_0 \fb$. Then
\begin{subequations}\label{RegUpsilon}
\begin{gather}
\varpi \in \mC_{y}^0\big([y_0,1];H_{z}^{s}(\R)\big), \quad \varpi_y \in \mC_{y}^0\big([y_0,1];H_{z}^{s-1}(\R)\big),\\
\sup\limits_{y_0 \le y\le 1}\Big\{ \|\varpi_z(\cdot,y)\|_{H^{s-1}(\R)}^2 + \|\varpi_y(\cdot,y)\|_{H^{s-1}(\R)}^2 \Big\} \le \dfrac{1}{y_0^5} C\big(\|\teta\|_{H^{s+\frac{1}{2}}(\R)}\big)\|\psi\|_{H^{s}(\R)}^2.
\end{gather}
\end{subequations}  
\end{lemma}
\begin{proof}
Throughout this proof, we denote the constant $C_{\ast}\equiv C\big(\|\teta\|_{H^{s+\frac{1}{2}}}\big)\|\psi\|_{H^{s}}^2$. Using Proposition \ref{prop:Letab} and definition (\ref{Leta}) for $\mL_{\eta}$, one has
\begin{align*}
\mL_{\eta}(\mX_0\fb) =& \alpha (\mX_0\fb)_{yy} + \mX_0 \fb_{zz} + \beta (\mX_0 \fb_z)_y - \gamma (\mX_0 \fb)_y\\
%=& \alpha \big( \mX_0 \d_y^2 \fb + 2 \mX_0^{\prime} \d_y \fb + \mX_0^{\prime\prime}  \fb \big) + \mX_0 \d_z^2 \fb\\ &+ \mX_0 \beta \d_y\d_z \fb + \mX_0^{\prime} \beta \d_z \fb - \mX_0 \gamma \d_y\fb - \mX_0^{\prime} \gamma \fb\\
%=& \mX_0 \big( \alpha \d_y^2 + \d_z^2 + \beta\d_y \d_z - \gamma \d_y \big) \fb + 2\mX_0^{\prime} \alpha \d_y \fb + \mX_0^{\prime\prime} \alpha \fb + \mX_0^{\prime} \beta \d_z \fb - \mX_0^{\prime} \gamma \fb\\
=& \mX_0 \mL_{\eta}\fb + 2\mX_0^{\prime} \alpha \fb_y + \mX_0^{\prime\prime} \alpha \fb + \mX_0^{\prime} \beta \fb_z - \mX_0^{\prime} \gamma \fb\\
%=& \mX_0 \dfrac{\d_y v}{y^2 \eta^3} + 2\mX_0^{\prime}  \dfrac{\alpha \d_y^2 v}{\eta} + \mX_0^{\prime\prime}  \dfrac{\alpha\d_y v}{\eta} + \mX_0^{\prime} \dfrac{\beta\d_y\d_z v}{\eta} + \mX_0^{\prime}\dfrac{\beta^2\d_y v}{2y\eta} - \mX_0^{\prime}  \dfrac{\gamma\d_y v}{\eta}\\
=& \dfrac{2\mX_0^{\prime}\alpha}{\eta} v_{yy} +  \dfrac{\mX_0^{\prime}\beta}{\eta} v_{zy} + \big( \dfrac{\mX_0}{y^2\eta^2} + \mX_0^{\prime\prime} \alpha + \dfrac{\mX_0^{\prime}\beta^2}{2y} - \mX_0^{\prime} \gamma \big) \dfrac{v_y}{\eta}.
\end{align*}
Using the above equation, one calculates to obtain that
\begin{align}\label{temp:mX0b}
&\div\big(A\nabla(\mX_0 \fb)\big) %= y\eta^2 \d_z^2 (\mX_0 \fb) - 2y^2 \eta \eta_z \d_z \d_y (\mX_0 \fb) - (y^2 \eta \d_z^2 \eta - 2y^2 |\d_z\eta|^2 -1)\d_y (\mX_0 \fb) + y(1+y^2|\d_z\eta|^2) \d_y^2(\mX_0 \fb)
= y \eta^2 \mL_{\eta} (\mX_0 \fb)\\ 
=& 2y\mX_0^{\prime} \eta\alpha  v_{yy} + y\mX_0^{\prime} \eta \beta v_{zy} + \big( \dfrac{\mX_0}{y\eta} + y\mX_0^{\prime\prime} \eta \alpha + \dfrac{\mX_0^{\prime} \eta \beta^2}{2} - y \mX_{0}^{\prime} \eta \gamma  \big) v_y \nonumber\\
=&\vcentcolon \fg_1 v_{yy} + \fg_2 v_{zy} + \fg_3 v_y.\nonumber
\end{align}
Since $\teta\in H^{s+\frac{1}{2}}(\R)$, it can be verified using (\ref{Leta}) and Propositions \ref{prop:Sobcomp}--\ref{prop:clprod} that
\begin{subequations}\label{temp:fg}
\begin{gather}
\fg_1(z,y)-\dfrac{2y\mX_0^{\prime}(y)}{R} \in \mC_{y}^{\infty}\big( [0,1]; H_z^{s-\frac{1}{2}}(\R) \big), \quad \fg_2(z,y) \in \mC_{y}^{\infty}\big( [0,1]; H_z^{s-\frac{1}{2}}(\R) \big),\\
\fg_3 - \dfrac{\mX_0(y)}{yR} - \dfrac{y\mX_0^{\prime\prime}(y)}{R} -  \dfrac{\mX_0^{\prime}(y)}{R} \in \mC_{y}^{\infty}\big( [0,1]; H_z^{s-\frac{3}{2}}(\R) \big).
\end{gather}
\end{subequations}
By construction (\ref{mX0}) and Remark \ref{rem:BV}, $\mX_0 \fb\vert_{y=1}= \tfrac{1}{\eta} v_y = \mB$, and $ \d_y (\mX_0 \fb) \vert_{y=0}=0$. If we set $\vartheta \vcentcolon= E_{\eta}[\mB]$ and $\varpi\vcentcolon= \vartheta - \mX_0 \fb$, then (\ref{temp:mX0b}) implies that $\varpi$ solves
\begin{subequations}\label{temp:Upsilon}
\begin{align}
&-\div(A\nabla \varpi) = \fg_1 v_{yy} + \fg_2 v_{zy} + \fg_3 v_{y} && \text{for } \ (z,y)\in\mS\label{temp:Upsilon-1}\\
& \varpi(z,1)=0, \qquad \d_y \varpi(z,0)=0 && \text{for } \ z\in\R.\label{temp:Upsilon-2}
\end{align}
\end{subequations}
By the definition of $\mB$ given in (\ref{BV}), Lemma \ref{lemma:GSob}, and Propositions \ref{prop:Sobcomp}--\ref{prop:clprod}, one has $\|\mB\|_{H^{s-1}(\R)}^2\le C\big(\|\teta\|_{H^{s+\frac{1}{2}}(\R)}\big)\|\psi_z\|_{H^{s-1}(\R)}^2$. Replacing the boundary data $\psi$ by $\mB$ in the problem (\ref{reform}), it follows from Corollary \ref{corol:v0} and Lemma \ref{lemma:vyy} that
\begin{equation}\label{temp:upsilonEst}
	\int_{0}^1\!\!\! \big\{ \|\varpi(\cdot,y)\|_{H^{s-\frac{1}{2}}}^2 \!+\! \|\nabla\varpi(\cdot,y)\|_{H^{s-\frac{3}{2}}}^2 \big\} y \dif y \le \dfrac{1}{y_0^2}C\big(\|\teta\|_{H^{s+\frac{1}{2}}}\big)\|\psi_z\|_{H^{s-1}}^2 \le \dfrac{C_{\ast}}{y_0^2}.
\end{equation}
Set $m=s-\tfrac{1}{2}$ and $\ep\in(0,1)$. Recall the operator $\Lambda_{\ep}^{2m}$ in (\ref{temp:symb}). Multiplying (\ref{temp:Upsilon-1}) with $\Lambda_{\ep}^{2m} \varpi$, and integrating in $(z,y)\in\mS$ using the boundary condition (\ref{temp:Upsilon-2}), we get
\begin{align*}
\iint_{\mS}\!\! (A \nabla \varpi) \cdot \nabla \Lambda_\ep^{2m} \varpi \dif z \dif y = \iint_{\mS}\!\! \big\{ \fg_1 v_{yy} + \fg_2 v_{zy} + \fg_3 v_y \big\} \Lambda_{\ep}^{2m} \varpi \,\dif z \dif y = \vcentcolon \sum_{j=1}^3 I_j. 
\end{align*}
Repeating the same ellipticity argument in the proof of Lemma \ref{lemma:cacci}, we obtain that
\begin{align}\label{temp:upsilonEst-m}
&\dfrac{\fl(\teta)}{2} \iint_{\mS}\!\! |\absm{D}_{\ep}^m \nabla \varpi|^2 y\dif z \dif y \\ \le& \dfrac{C(s)|\fU_{s}(\teta)|^2}{\fl(\teta)}\!\! \int_{0}^1\!\!\! \|\nabla \varpi(\cdot,y)\|_{H^{m-1}}^2\, y\dif y + \sum_{j=1}^{3}|I_j| \le \dfrac{C_{\ast}}{y_0^2} + \sum_{j=1}^{3}|I_j|,\nonumber 
\end{align}
where we also used (\ref{temp:upsilonEst}) in the last inequality. First, we estimate $I_1$. By Parseval's theorem and (\ref{temp:symb}), one has 
\begin{align*}
I_1=& \int_{0}^1\!\!\! \int_{\R} \fg_1  v_{yy} \Lambda_\ep^{2m} \varpi \, \dif z \dif y = \int_{0}^1\!\!\! \int_{\R} \absm{D}_{\ep}^{m-1}\big(\fg_1  v_{yy}\big) \absm{D}_{\ep}^{m+1} \varpi \, \dif z \dif y \\
=&\int_{0}^1\!\!\!\int_{\R} \absm{D}_{\ep}^{m+1} \varpi \big\{\fg_1 \absm{D}_{\ep}^{m-1} v_{yy} + \big[\absm{D}_{\ep}^{m-1},\fg_1\big]v_{yy} \big\}\, \dif z \dif y.
\end{align*}
By construction (\ref{mX0}), we have $\supp\big(\fg_1(z,\cdot)\big)\in[\tfrac{y_0}{4},\tfrac{y_0}{2}]$ for each $z\in\R$. Thus by Cauchy-Schwartz's inequality, Lemma \ref{lemma:vyy}, Proposition \ref{prop:commu}, and (\ref{temp:fg}), we get
\begin{align*}
|I_1|\le& \dfrac{\fl(\teta)}{12} \iint_{\mS}\!\! |\absm{D}_{\ep}^{m+1}\varpi|^2\, y \dif z \dif y + \dfrac{C}{y_0^4}\!\int_{0}^{1}\!\! \|\fg_1(\cdot,y)\|_{L^{\infty}}^2 \| v_{yy}(\cdot,y)\|_{H^{m-1}}^2\, y^3\dif y \\ &+ \dfrac{C}{y_0^4} \int_{0}^{1} \big\| \big[ \absm{D}_{\ep}^{m-1}, \fg_1 \big] v_{yy}(\cdot,y) \big\|_{L^2}^2 \, y^3\dif y\\
%\le& \dfrac{\fl(\teta)}{12} \int_{0}^1 \|\absm{D}_{\ep}^{m}\varpi\|_{H^1}^2\, y \dif y + \dfrac{C}{y_0^4}\!\int_{0}^{1} \| v_{yy}(\cdot,y)\|_{H^{m-1}}^2\, y^3\dif y \\ &+ \dfrac{C}{y_0^4} \int_{0}^{1}  \|\d_z\fg_1(\cdot,y)\|_{H^{m-2}}^2\|v_{yy}\|_{H^{m-2}}^2 \, y^3\dif y\\ 
\le& \dfrac{\fl(\teta)}{12} \int_{0}^1 \|\absm{D}_{\ep}^{m}\varpi\|_{H^1}^2\, y \dif y + \dfrac{C}{y_0^4}\!\int_{0}^{1} \| v_{yy}(\cdot,y)\|_{H^{s-\frac{3}{2}}}^2\, y^3\dif y \\ &+ \dfrac{C}{y_0^4} \int_{0}^{1}  \|\d_z\fg_1(\cdot,y)\|_{H^{s-\frac{5}{2}}}^2\|v_{yy}\|_{H^{s-\frac{5}{2}}}^2 \, y^3\dif y\\ 
\le& \dfrac{\fl(\teta)}{12} \int_{0}^1 \|\absm{D}_{\ep}^{m}\varpi\|_{H^1}^2\, y \dif y + \dfrac{1}{y_0^4}C_{\ast}.
\end{align*}
Repeating the same argument as $I_1$, one can also estimate $I_2$ and $I_3$ to get
\begin{align*}
&|I_2| %= \Big| \int_0^1\!\!\!\int_{\R} \fg_2 v_{zy} \Lambda_\ep^{2m} \varpi\, \dif z \dif y \Big| = \Big| \int_{0}^{1}\!\!\!\int_{\R} \absm{D}_{\ep}^{m-1}\big( \fg_2 v_{zy} \big) \absm{D}_{\ep}^{m+1}\varpi \, \dif z \dif y \Big|\\ =& \Big| \iint_{\mS}\!\! \absm{D}_{\ep}^{m+1}\varpi \big\{ \fg_2 \absm{D}_{\ep}^{m-1} v_{zy} + \big[\absm{D}_{\ep}^{m-1},\fg_2\big] v_{zy} \big\} \Big|\\ \le& \dfrac{\fl(\teta)}{12}\iint_{\mS}\!\! |\absm{D}_{\ep}^{m+1}\varpi|^2\, y \dif z \dif y + C \int_{0}^{1} \tfrac{1}{y} \|\fg_2(\cdot,y)\|_{L^{\infty}}^2 \|v_{zy}(\cdot,y)\|_{H^{m-1}}^2\, \dif y \\&+ C \int_{0}^{1} \tfrac{1}{y}\|\fg_2(\cdot,y)\|_{H^{m-2}}^2 \|v_{zy}(\cdot,y)\|_{H^{m-2}}^2 \, \dif y\\ \le& \dfrac{\fl(\teta)}{12}\int_{0}^{1}\!\! \|\absm{D}_{\ep}^{m}\varpi(\cdot,y)\|_{H^1}^2\, y \dif y + C \int_{0}^{1}\!\! \|v_{zy}(\cdot,y)\|_{H^{s-\frac{3}{2}}}^2\, y\dif y \\ &+ C \int_{0}^{1} \tfrac{1}{y^2} \|\fg_2(\cdot,y)\|_{H^{s-\frac{5}{2}}}^2 \|v_{zy}(\cdot,y)\|_{H^{s-\frac{5}{2}}}^2 \, y \dif y\\ &
\le \dfrac{\fl(\teta)}{12}\int_{0}^{1}\!\! \|\absm{D}_{\ep}^{m}\varpi(\cdot,y)\|_{H^1(\R)}^2\, y \dif y + C_{\ast},\\
&|I_3| %=\Big| \int_{0}^{1}\!\!\!\int_{\R} \fg_{3} v_y \Lambda_{\ep}^{2m} \varpi \dif z \dif y \Big| = \Big| \int_{0}^{1}\!\!\!\int_{\R} \absm{D}_{\ep}^{m-1}\big(\fg_{3} v_y\big) \absm{D}_{\ep}^{m+1} \varpi \dif z \dif y \Big|\\ =&\Big|  \int_{0}^{1}\!\!\!\int_{\R}\!\! \absm{D}_{\ep}^{m+1} \varpi \big\{ \fg_{3}\absm{D}^{m-1}_{\ep} v_y + \big[ \absm{D}^{m-1}, \fg_3 \big] v_y \big\} \dif z \dif y \Big|\\ \le & \dfrac{\fl(\teta)}{12}\int_{0}^{1}\!\!\!\int_{\R}\!\! |\absm{D}_{\ep}^{m+1} \varpi|^2\, y \dif z \dif y + \dfrac{C}{y_0^2} \int_{0}^{1}\!\! \|\fg_3(\cdot,y)\|_{L^{\infty}}^2 \|v_y(\cdot,y)\|_{H^{m-1}}^2 \, y \dif y \\ &+ \dfrac{C}{y_0^2}\int_{0}^{1}\!\! \|\d_z \fg_3(\cdot,y)\|_{H^{m-2}}^2 \|v_y(\cdot,y)\|_{H^{m-2}}^2\, y \dif y\\ \le& \dfrac{\fl(\teta)}{12}\int_{0}^{1}\!\!\!\int_{\R}\!\! |\absm{D}_{\ep}^{m+1} \varpi|^2\, y \dif z \dif y + \dfrac{C}{y_0^2} \int_{0}^{1}\!\! \|\fg_3(\cdot,y)\|_{L^{\infty}}^2 \|v_y(\cdot,y)\|_{H^{s-\frac{3}{2}}}^2 \, y \dif y \\ &+ \dfrac{C}{y_0^2}\int_{0}^{1}\!\! \|\d_z \fg_3(\cdot,y)\|_{H^{s-\frac{5}{2}}}^2 \|v_y(\cdot,y)\|_{H^{s-\frac{5}{2}}}^2\, y \dif y\\ &
\le \dfrac{\fl(\teta)}{12}\int_{0}^{1}\!\!\|\absm{D}_{\ep}^{m}\varpi(\cdot,y)\|_{H^1(\R)}^2\, y \dif y + \dfrac{C_{\ast}}{y_0^4}.
\end{align*}
Substituting the estimates for $I_1$--$I_3$ into (\ref{temp:upsilonEst-m}), one gets
\begin{align*}
\dfrac{\fl(\teta)}{2} \iint_{\mS}\!\! |\absm{D}_{\ep}^m \nabla \varpi|^2 y\dif z \dif y \le \dfrac{\fl(\teta)}{4}\!\!\int_{0}^{1}\!\!\|\absm{D}_{\ep}^{m}\varpi(\cdot,y)\|_{H^1(\R)}^2\, y \dif y + \dfrac{1}{y_0^4}C_{\ast}.
\end{align*}
Adding $\tfrac{\fl(\teta)}{2}\iint_{\mS}|\absm{D}_{\ep}^m\varpi|^2\, y \dif z \dif y$ on both sides of the above inequality, then applying the estimate (\ref{temp:upsilonEst}) with $m=s-\frac{1}{2}$, we have
\begin{align*}
&\dfrac{\fl(\teta)}{2} \int_{0}^{1}\!\! \big\{ \|\absm{D}_{\ep}^m \varpi(\cdot,y)\|_{H^1(\R)}^2 + \|\absm{D}_{\ep}^m \varpi_y(\cdot,y)\|_{L^2(\R)}^2 \big\} \, y\dif y\\
=&\dfrac{\fl(\teta)}{2} \iint_{\mS} \big\{ |\absm{D}_{\ep}^m \varpi|^2 + |\absm{D}_{\ep}^m \nabla\varpi|^2 \big\}\, y\dif z \dif y \\
\le& \dfrac{\fl(\teta)}{2} \int_{0}^{1}\!\! \| \varpi(\cdot,y)\|_{H^m(\R)}^2\, y \dif y + \dfrac{\fl(\teta)}{4}\!\!\int_{0}^{1}\|\absm{D}_{\ep}^{m}\varpi(\cdot,y)\|_{H^1(\R)}^2\, y \dif y + \dfrac{1}{y_0^4}C_{\ast}\\
\le& \dfrac{\fl(\teta)}{4}\!\!\int_{0}^{1}\|\absm{D}_{\ep}^{m}\varpi(\cdot,y)\|_{H^1(\R)}^2\, y \dif y +\dfrac{1}{y_0^4} C_{\ast}.
\end{align*}
Since $m=s-\frac{1}{2}$, taking $\ep\to 0^+$ in the above estimates, we obtain:
\begin{equation}\label{temp:UpsilonFinal}
\dfrac{\fl(\teta)}{4} \int_{0}^{1}\!\! \big\{ \| \varpi(\cdot,y) \|_{H^{s+\frac{1}{2}}(\R)}^2 + \|\varpi_y(\cdot,y)\|_{H^{s-\frac{1}{2}}(\R)}^2 \big\}\, y\dif y \le \dfrac{1}{y_0^4}C_{\ast}.
\end{equation}
Since $\varpi$ solves (\ref{temp:Upsilon}), one can repeat the same difference quotient scheme appeared in the proof of Proposition \ref{prop:vyyInt} and Lemma \ref{lemma:vyy} to obtain the following estimate for $\varpi_{yy}$:
\begin{align}\label{temp:Upsilonyy}
&\int_{0}^{1}\!\!\|\varpi_{yy} (\cdot,y)\|_{H^{s-\frac{3}{2}}(\R)}^2\, y^3\dif y \\
%\le& C_{\ast} + C\int_{0}^{1}\!\!\Big\{ \|\fg_1\|_{L^{\infty}}^2\|v_{yy}\|_{H^{s-\frac{3}{2}}}^2 + \big\|\big[\absm{D}_{\ep}^{s-\frac{3}{2}},\fg_1\big]v_{yy}\big\|_{L^2}^2  \Big\}\, y^3 \dif y  \\ & +C\int_{0}^{1}\!\!\Big\{ \|\fg_2\|_{L^{\infty}}^2\|v_{zy}\|_{H^{s-\frac{3}{2}}}^2 + \big\|\big[\absm{D}_{\ep}^{s-\frac{3}{2}},\fg_2\big]v_{zy}\big\|_{L^2}^2  \Big\}\, y^3 \dif y \\ &+C\int_{0}^{1}\!\!\Big\{ \|\fg_3\|_{L^{\infty}}^2 \|v_y\|_{H^{s-\frac{3}{2}}}^2 + \big\|\big[ \absm{D}_{\ep}^{s-\frac{3}{2}}, \fg_3 \big]v_y\big\|_{L^2(\R)}^2 \Big\}\, y^3\dif y\\
\le& C_{\ast} \!+\! C \!\! \int_{0}^{1}\!\!\! \|\nabla v_{y}\|_{H^{s-\frac{3}{2}}}^2 \, y^3 \dif y + \dfrac{C}{y_0^2}\! \int_{0}^{1}\!\!\! \big\|\big[\absm{D}^{s-\frac{3}{2}},\fg_3\big]v_y\big\|_{L^2}^2 \, y^3 \dif y \le C_{\ast}\big(1+\dfrac{1}{y_0^2}\big),\nonumber
\end{align}
where in the last line we used Lemma \ref{lemma:vyy}, Proposition \ref{prop:commu}, and (\ref{temp:fg}). Using (\ref{temp:UpsilonFinal}) and (\ref{temp:Upsilonyy}) in the interpolation theorem, Proposition \ref{prop:bochner}, it is shown that (\ref{RegUpsilon}) holds, which concludes the proof of this lemma.
\end{proof}
\begin{corollary}\label{corol:B-dV}
Let $\mB$ and $V$ be given in (\ref{BV}). Then $G[\eta](\mB)+ \d_z V \in H^{s-1}(\R)$.
\end{corollary}
\begin{proof}
Since $\vartheta=E_{\eta}[\mB]$ and $\fb=\tfrac{1}{\eta}v_y$, by definition (\ref{DN-flat}) of $G[\eta]$, and (\ref{Leta}) we have
\begin{align}\label{temp:cancel1}
G[\eta](\mB)%=\Big\{ \dfrac{1+y|\eta_z|^2}{\eta} \vartheta_y - \eta_z \vartheta_z \Big\}\Big\vert_{y=1}
=& (\eta \alpha \vartheta_y - y\eta_z \vartheta_z)\vert_{y=1}\\
=& (\eta \alpha \varpi_y - y\eta_z \varpi_z + \eta \alpha (\mX_0 \fb)_y - y\eta_z (\mX_0 \fb)_z)\vert_{y=1}\nonumber\\
%=& \big(\eta \alpha \varpi_y - y\eta_z \varpi_z + \mX_0^{\prime} \alpha v_y\big)\vert_{y=1} + \mX_0  \big( \alpha v_{yy} - \dfrac{y\eta_z}{\eta}v_{yz} + \dfrac{y\eta_z^2}{\eta^2} v_y \big)\vert_{y=1}\nonumber\\
=& \big(\eta \alpha \varpi_y - y\eta_z \varpi_z \big)\vert_{y=1} + \mX_0  \big( \alpha v_{yy} - \dfrac{y\eta_z}{\eta}v_{yz} + \dfrac{y\eta_z^2}{\eta^2} v_y \big)\vert_{y=1},\nonumber
\end{align}
where we used the fact that $\mX_0^{\prime}(y)\vert_{y=1}=0$ from the construction (\ref{mX0}). Since $\mL_{\eta} v =0$, using the definition of $\alpha$, $\beta$, and $\gamma$ in (\ref{Leta}), we get: 
\begin{align*}
&\alpha v_{yy} - \dfrac{y\eta_z}{\eta} v_{yz} + \dfrac{y\eta_z^2}{\eta^2} v_y = -v_{zz} -\beta v_{yz} + \gamma v_{y} - \dfrac{y\eta_z}{\eta} v_{yz} + \dfrac{y\eta_z^2}{\eta^2} v_y\\
%=& -v_{zz} + \dfrac{2 y \eta_z}{\eta} v_{yz} + \dfrac{y^2\eta\d_z^2\eta - 1 -2y^2 \eta_z^2}{y\eta^2}v_y - \dfrac{y\eta_z}{\eta} v_{yz} + \dfrac{y\eta_z^2}{\eta^2} v_y\\
=& -v_{zz} + \dfrac{2y \eta_z}{\eta} v_{yz} + \dfrac{y}{\eta} v_y \d_z^2 \eta - \dfrac{v_y}{y\eta^2} - \dfrac{2y\eta_z^2}{\eta^2} v_y - \dfrac{y\eta_z}{\eta} v_{yz} + \dfrac{y\eta_z^2}{\eta^2} v_y\\
=& -v_{zz} + \dfrac{y}{\eta} (\eta_z v_{yz} + v_y \d_z^2\eta ) - \dfrac{1+y^2\eta_z^2}{y\eta^2} v_y = -v_{zz} + \dfrac{y}{\eta} \d_z (\eta_z v_y) - \dfrac{1+y^2\eta_z^2}{y\eta^2} v_y\\
=& -v_{zz} + \d_z(y \eta_z \fb ) + \dfrac{y\eta_z^2}{\eta^2} v_y - \dfrac{1+y^2\eta_z^2}{y\eta^2} v_y = -\d_z ( \d_z v - y \fb \d_z\eta ) - \dfrac{v_y}{y\eta^2}.
\end{align*}
Substituting the above into (\ref{temp:cancel1}), we obtain that
\begin{equation}\label{temp:cancel2}
G[\eta](\mB) = (\eta \alpha \varpi_y - y\eta_z \varpi_z - \dfrac{\mX_0}{y\eta^2}v_y )\vert_{y=1} - \mX_0 \d_z( \d_z v - y \fb \d_z\eta )\vert_{y=1}.
\end{equation}
By Remark \ref{rem:BV} and (\ref{mX0}), $\fb\vert_{y=1} = \tfrac{v_y}{\eta} \vert_{y=1} = \mB$ and $\mX_0(y)\vert_{y=1}=1$, hence
\begin{equation*}
\mX_0 \d_z(\d_z v - y\fb\d_z\eta )\vert_{y=1} = \d_z (\d_z\psi - \mB \d_z \eta) = \d_z V.
\end{equation*}
Putting this limit in (\ref{temp:cancel2}), it follows that
\begin{equation*}
G[\eta](\mB) + \d_z V = \big(\eta \alpha \varpi_y - y\eta_z \varpi_z - \dfrac{\mX_0}{y\eta^2}v_y \big)\vert_{y=1} =\vcentcolon \fr \vert_{y=1}.
\end{equation*}
Using Corollary \ref{corol:yC}, Lemma \ref{lemma:cancel} and Propositions \ref{prop:Sobcomp}--\ref{prop:clprod}, one has the regularity $\fr\in \mC_{y}^{0}\big((0,1];H_z^{s-1}(\R)\big)$. Thus we conclude that $G[\eta](\mB) + \d_z V = \fr\vert_{y=1}\in H^{s-1}(\R)$.  
\end{proof}
%%%%%%%%%%%%%%%%%%%%%%%%%%%%%%%%%%%%%%%%%%%%%%%%%%%%%%
%%%%%%%%%%%%%%%%%%%%%%%%%%%%%%%%%%%%%%%%%%%%%%%%%%%%%%

%---------------------
%     Section
%---------------------

\section{Paralinearization and principal symbol of \texorpdfstring{$G[\eta](\psi)$}{G[eta](psi)}}\label{sec:DN-paralinearization}

\subsection{Notations for paradifferential operators.}
The main aim of this subsection is to paralinearise and determine the principal symbol of Dirichlet-Neumann operator $G[\eta](\psi)$. The notations and main results for paradifferential calculus used throughout this section can be found in Appendix \ref{append:para}. Specifically, for symbol $a\in\Gamma_{\theta}^m(\R)$, we denote $T_a$ as the paradifferential operator defined in Definition \ref{def:parad}. Moreover, we will use the following notations throughout the entirety of this section:
\begin{subequations}\label{brackets}
\begin{enumerate}[label=\textnormal{(\roman*)}]
\item Let $f(z)$, $g(z)$ be two functions in some Sobolev space with sufficient regularity. We set the function:
\begin{equation}\label{bonyProd}
\fp(f,g) \equiv f\cdot g - T_{f} g - T_{g} f.
\end{equation}
\item Let $a(x,\xi)$, $b(x,\xi)$ be two symbols in $\Gamma_{\theta}^m(\R)$ for some $m,\, \theta\ge 0$. For any function $f(z)$ in some Sobolev space with sufficient regularity, we set the operator: 
\begin{equation}\label{bonyCom}
\fq[a,b](f) \equiv T_a T_b f - T_{a\sharp b} f.
\end{equation}
\item Let $F(x)\in \mC^{\infty}(\R)$. Then for any function $f(x)$ in some Sobolev space with sufficient regularity, we set the operator:
\begin{equation}\label{bonyChain}
\fC[F](f)\equiv F(f) - F(0) - T_{F^{\prime}(f)} f. 
\end{equation} 
\end{enumerate}
\end{subequations}
Using Remark \ref{rem:paraC}, we have the following statement for constant paraproducts.
\begin{proposition}\label{prop:constR}
Let $c\in\R\backslash\{0\}$ be a non-zero constant. Then the following holds
\begin{enumerate}[label=\textnormal{(\roman*)},ref=\textnormal{(\roman*)}]
\item\label{item:constR1} If $f\in H^{k}(\R)$ for some $k\in\R$, then $\fp(f,c) \in H^{\infty}(\R)$. In addition, for all $l\in\R$, $\|\fp(f,c)\|_{H^{l}(\R)} \le C(l,k) \|f\|_{H^{k}(\R)}$.
\item\label{item:constR2} If $a\in \Gamma_{0}^{m}(\R)$ and $f\in H^{k}(\R)$ for $k\in\R$ and $m\ge 0$, then one has the regularity: $\fq[c,a](f)$, $\fq[a,c](f) \in H^{\infty}(\R)$. Moreover, for all $l\in\R$,
\begin{equation*}
 \|\fq[c,a](f)\|_{H^{l}(\R)} + \|\fq[a,c](f)\|_{H^{l}(\R)} \le C(l,k) \mM_0^{m}(a) \|f\|_{H^k(\R)},
\end{equation*} 
where $\mM_\theta^m(\cdot)$ is the semi-norm defined in Definition \ref{def:symbols}.
\end{enumerate}
\end{proposition} 

\subsection{Alinhac's good unknown \texorpdfstring{$U=\psi-T_{\mB}\eta$}{U=psi-TBeta} and its extension.}
To fix notations, we write the elliptic equation (\ref{Leta}) as:
\begin{gather}
\mL_{\eta} v \vcentcolon=\alpha \d_y^2 v + \d_z^2 v  + \beta \d_z\d_y v - \gamma \d_y v  = 0, \label{Leta2} \\
\text{where } \quad \alpha\vcentcolon= \dfrac{1+y^2|\d_z\eta|^2}{\eta^2}, \quad \beta\vcentcolon= -\dfrac{2y \d_z \eta}{\eta}, \quad \gamma \vcentcolon= \dfrac{y^2 \eta \d_z^2 \eta - 1 - 2 y^2 |\d_z \eta|^2}{y \eta^2}.\nonumber
\end{gather}
In addition, we also define:
\begin{equation}\label{talga}
\talpha \vcentcolon= \alpha - \tfrac{1}{R^2}, \ \text{ and } \ \tgamma \vcentcolon= \gamma + \tfrac{1}{yR^2}. 
\end{equation}
Then by Propositions \ref{prop:Sobcomp}--\ref{prop:clprod}, one has that for all $y\in(0,1]$
\begin{subequations}\label{albega}
\begin{gather}
\big\|\talpha(\cdot, y)\big\|_{H^{s-\frac{1}{2}}(\R)} \le C \|\teta\|_{H^{s+\frac{1}{2}}(\R)}^3, \quad \|\beta(\cdot,y)\|_{H^{s-\frac{1}{2}}(\R)} \le C y \|\teta\|_{H^{s+\frac{1}{2}}(\R)}^2,\label{albega-1}\\
\big\| \tgamma (\cdot, y) \big\|_{H^{s-\frac{3}{2}}(\R)} \le \dfrac{C}{y} \|\teta\|_{H^{s+\frac{1}{2}}(\R)} + C y \|\teta\|_{H^{s+\frac{1}{2}}(\R)}^3. \label{albega-2}
\end{gather}
\end{subequations}
We set the Alinhac's good unknown $U$, and the tangential velocity component $V$ as
\begin{equation}\label{AGU}
U\vcentcolon=\psi- T_{\mB} \eta \ \text{ and } \ V\vcentcolon= \d_z\psi - \mB \d_z \eta \quad \text{where } \  \mB\vcentcolon=\frac{\d_z \eta \d_z \psi + G[\eta](\psi)}{1+|\d_z\eta|^2}.
\end{equation}
Motivated by the cancellation Lemma \ref{lemma:cancel}, we also define:
\begin{equation*}
u \vcentcolon= v - T_{\fb}\rho, \ \text{ and } \ \omega \vcentcolon= \d_z v - \fb \d_z \rho \qquad \text{where } \ \rho=y\eta \ \text{ and } \  \fb\vcentcolon= \dfrac{\d_y v}{\eta}.
\end{equation*} 
We remark that $u$ and $\omega$ are respectively the extension functions of $U$ and $V$ in the sense that $u\vert_{y=1} = U$, and $\omega\vert_{y=1} = V$, since $\fb\vert_{y=1}=\mB$.

\begin{proposition}\label{prop:Db}
Let $s>\frac{5}{2}$. Assume $(\teta,\psi)\in H^{s+\frac{1}{2}}(\R)\times H^{s}(\R)$. Then
\begin{gather*}
\fb \in \mC_{y}^0\big((0,1]; H_z^{s-1}(\R)\big), \qquad \d_z \fb, \ \d_y \fb  \in \mC_{y}^0\big((0,1]; H_z^{s-2}(\R)\big),\\
\d_z^2 \fb, \ \ \d_y^2 \fb, \ \ \d_z\d_y\fb \ \in \mC_{y}^0\big((0,1]; H_z^{s-3}(\R)\big).
\end{gather*}
\end{proposition}
\begin{proof}
Taking derivative on $\fb=\frac{1}{\eta} \d_y v$, we have
\begin{gather*}
\d_y \fb = \frac{\d_y^2 v}{\eta}, \quad \d_y^2 \fb = \dfrac{\d_y^3 v}{\eta}, \quad \d_z \fb = \dfrac{\d_y\d_z v }{\eta} - \dfrac{\d_z \eta}{\eta^2}  \d_y v, \quad \d_z\d_y \fb= \dfrac{\d_z\d_y^2 v}{\eta} - \dfrac{\d_z \eta}{\eta^2} \d_y^2 v,\\
\d_z^2 \fb = \dfrac{ \d_z^2 \d_y v}{\eta} - 2 \dfrac{\d_z \eta}{\eta^2} \d_z\d_y v + 2 \dfrac{|\d_z\eta|^2}{\eta^3} \d_y v - \dfrac{\d_z^2 \eta}{\eta^2} \d_y v.  
\end{gather*}
The result follows from Corollary \ref{corol:yC} and Propositions \ref{prop:Sobcomp}--\ref{prop:clprod}.
\end{proof}

\begin{proposition}\label{prop:b}
If $v$ solves (\ref{Leta2}), then $\fb=\frac{1}{\eta}\d_y v$, $\rho=y\eta$, and $\omega=\d_z v - \fb \d_z\rho$ satisfies the equation:
\begin{equation*}
\eta \alpha \d_y \fb + \d_z \omega - \d_z\rho \d_z \fb  + \dfrac{\fb}{\rho}  = 0.
\end{equation*}
\end{proposition}
\begin{proof}
Reordering the terms in (\ref{Leta2}), we have
\begin{align*}
0=\mL_{\eta} v =& \dfrac{1+y^2|\d_z\eta|^2}{\eta} \dfrac{\d_y^2 v}{\eta} + \d_z^2 v - \dfrac{2y\d_z \eta}{\eta} \d_z\big( \eta \dfrac{\d_y v}{\eta} \big)\\ &+ \dfrac{1+ y^2 |\d_z\eta|^2}{y\eta}\dfrac{\d_y v}{\eta} + \dfrac{y^2|\d_z \eta|^2 - y^2 \eta \d_z^2 \eta}{y\eta} \dfrac{\d_y v}{\eta} \\
=& \eta \alpha \d_y \fb + \d_z^2 v - 2y\dfrac{|\d_z\eta|^2}{\eta}\fb - 2y\d_z\eta \d_z\fb + \dfrac{\fb}{y\eta}\\
&+ y\dfrac{|\d_z\eta|^2}{\eta} \fb  + y\dfrac{|\d_z\eta|^2}{\eta} \fb - y \fb \d_z^2\eta\\
=& \eta \alpha \d_y \fb + \d_z^2 v - 2\d_z\rho \d_z\fb + \dfrac{\fb}{\rho} - \fb \d_z^2 \rho\\
=& \eta \alpha \d_y \fb + \d_z (\d_z v - \fb \d_z \rho) - \d_z\rho \d_z\fb + \frac{\fb}{\rho}.  
\end{align*}
Since $\omega=\d_z v - \fb \d_z \rho$, this proves the proposition.
\end{proof}

\subsection{Paralinearization of cylindrical non-flat Laplacian operators.}
\begin{lemma}[Paralinearization of cylindrical non-flat Laplacian]\label{lemma:paraEllip}
Suppose $s>\frac{5}{2}$ and $(\teta,\psi)\in H^{s+\frac{1}{2}}(\R)\times H^{s}(\R)$. Set $\delta\vcentcolon= \min\{\tfrac{1}{2},s-\tfrac{5}{2}\}$. Define paradifferential operator:
\begin{equation*}
P_{\eta} \vcentcolon= T_{\alpha} \d_y^2 + \d_z^2 + T_{\beta} \d_z\d_y - T_{\gamma}\d_y.
\end{equation*}
Then $u=v-T_{\fb}\rho$ solves the paradifferential equation:
\begin{equation}\label{Petaf}
P_{\eta} u = f \quad \text{for some } \ f(y,z)\in \mC_y^0\big( (0,1] ; H_z^{s-\frac{1}{2}+\delta}(\R) \big).
\end{equation}
Moreover, there exists a positive monotone increasing function $x\mapsto C(x)$ such that
\begin{equation*}
 \|f(\cdot,y)\|_{H^{s-\frac{1}{2}+\delta}(\R)} \le \tfrac{1}{y} C\big( \|\teta\|_{H^{s+\frac{1}{2}}(\R)} \big) \|\psi\|_{H^{s}(\R)}. 
\end{equation*} 
\end{lemma}  
\begin{proof}
We split the proof into 3 steps:
\paragraph{Step 1: Paralinearization of $\mL_{\eta} v=0$.} Denote $\talpha\vcentcolon= \alpha - \frac{1}{R^2}$ and $\tgamma \vcentcolon= \gamma + \frac{1}{yR^2}$. By the linearization rule (\ref{bonyProd}), we have
\begin{subequations}\label{Para-albega}
\begin{align}
	\alpha \d_y^2 v =& T_{\alpha} \d_y^2 v + T_{\d_y^2 v} \talpha + \fp(\talpha,\d_y^2 v) + \fp(\tfrac{1}{R^2}, \d_y^2 v)\\
	\beta \d_{z}\d_{y} v =& T_{\beta} \d_{z}\d_{y} v + T_{\d_{z}\d_{y} v} \beta + \fp(\beta, \d_{z}\d_{y} v ),\\
	-\gamma \d_y v =& - T_{\gamma} \d_y v - T_{\d_y v} \tgamma - \fp(\tgamma, \d_y v) + \tfrac{1}{y}\fp( \tfrac{1}{R^2}, \d_y v ).
\end{align}
\end{subequations}
Therefore the fact that $v$ is a solution to (\ref{Leta2}) implies that 
\begin{align*}
 0=\mL_{\eta} v =& P_{\eta} v + T_{\d_y^2 v} \talpha + T_{ \d_{zy} v} \beta - T_{\d_y v} \tgamma \\
 & +\fp( \talpha, \d_y^2 v ) + \fp(\tfrac{1}{R^2},\d_y^2 v) + \fp( \beta, \d_{zy} v ) - \fp( \tgamma, \d_y v ) + \tfrac{1}{y} \fp( \tfrac{1}{R^2}, \d_y v ).
\end{align*}
Rewriting in terms of $v=u+T_{\fb} \rho$, and reorganising, one has
\begin{align}\label{temp:Pu}
P_{\eta} u =&- P_{\eta} T_{\fb}\rho - T_{\d_y^2 v} \talpha - T_{ \d_{z}\d_{y} v} \beta + T_{\d_y v} \tgamma\\
& -\fp( \talpha, \d_y^2 v ) - \fp( \beta, \d_{z}\d_{y} v ) + \fp( \tgamma, \d_y v ) -\fp(\tfrac{1}{R^2},\d_y^2 v) - \tfrac{1}{y} \fp( \tfrac{1}{R^2}, \d_y v ).\nonumber
\end{align}
Our main aim is to show that the right hand side of (\ref{temp:Pu}) belongs to the functional space $\mC_y^0\big((0,1];H_z^{s-\frac{1}{2}+\delta}(\R)\big)$. By Bony's Theorem \ref{thm:bony}, Corollary \ref{corol:v0}, and (\ref{albega}) it follows that for each $y\in (0,1]$.
\begin{gather*}
\|\fp( \talpha, \d_y^2 v )(\cdot,y)\|_{H^{2s-\frac{5}{2}}} %\le C \|\alpha-\tfrac{1}{R^2}\|_{H^{s-\frac{1}{2}}}^3 \|\d_y^2 v\|_{H^{s-\frac{3}{2}}} 
\le \tilde{C}\big( \|\teta\|_{H^{s+\frac{1}{2}}} \big) \|\d_z \psi\|_{H^{s-1}},\\
\|\fp(\beta,\d_{zy}v)(\cdot,y)\|_{H^{2s-\frac{5}{2}}} %\le C\|\beta\|_{H^{s-\frac{1}{2}}}^2\|\d_{zy}v\|_{H^{s-\frac{3}{2}}} 
\le y \tilde{C}\big( \|\teta\|_{H^{s+\frac{1}{2}}} \big) \|\d_z \psi\|_{H^{s-1}},\\
\|\fp ( \tgamma, \d_y v )(\cdot,y)\|_{H^{2s-\frac{5}{2}}} %\le C\big\| (\gamma+\tfrac{1}{yR^2}) \big\|_{H^{s-\frac{3}{2}}} \|\d_y v\|_{H^{s-\frac{1}{2}}} 
\le \tfrac{1}{y} \tilde{C}\big( \|\teta\|_{H^{s+\frac{1}{2}}} \big) \|\d_z \psi\|_{H^{s-1}}.
\end{gather*}
Moreover, by Proposition \ref{prop:constR}\ref{item:constR1}, and Corollary \ref{corol:v0}, we also have
\begin{equation*}
\|\fp(\tfrac{1}{R^2},\d_y^2 v)\|_{H^{2s-\frac{5}{2}}} + \|\tfrac{1}{y}\fp( \tfrac{1}{R^2}, \d_y v )\|_{H^{2s-\frac{5}{2}}} \le \tfrac{1}{y} \tilde{C}\big(\|\teta\|_{H^{s+\frac{1}{2}}}\big) \|\d_z \psi\|_{H^{s-1}}.
\end{equation*} 
Thus by Corollary \ref{corol:yC}, we get the regularities
\begin{equation*}
\fp( \talpha, \d_y^2 v ), \ \fp( \beta, \d_{zy} v ), \ \fp( \tgamma, \d_y v ), \ \fp(\tfrac{1}{R^2},\d_y^2 v), \ \tfrac{1}{y} \fp( \tfrac{1}{R^2}, \d_y v ) \in \mC_y^0\big((0,1];H_{z}^{2s-\frac{5}{2}}(\R)\big).
\end{equation*}
Since $s>\frac{5}{2}$ and $\delta=\min\{\frac{1}{2},s-\frac{5}{2}\}$, we have $s-\frac{1}{2}+\delta\le 2s - \frac{5}{2}$. Putting the above regularities into (\ref{temp:Pu}), it holds that there is an explicit functional $f_1(\cdot, \cdot)$ such that
\begin{equation}\label{temp:f1}
	P_{\eta} u + P_{\eta} T_{\fb}\rho + T_{\d_y^2 v} \talpha + T_{\d_z\d_y v} \beta - T_{\d_y v} \tgamma = f_1(\teta,v) \in \mC_y^0\big((0,1];H_{z}^{s-\frac{1}{2}+\delta}(\R)\big)
\end{equation}
\paragraph{Step 2: Reduction of $P_{\eta}T_{\fb}\rho$.} Computing the term $P_{\eta} T_{\fb} \teta$ using the Leibniz rule, Proposition \ref{prop:paraLeib}, one gets
\begin{align}\label{temp:PTeta}
P_{\eta} T_{\fb} \rho =& T_{\alpha} T_{\d_y^2 \fb} \rho + 2 T_{\alpha} T_{\d_y \fb} \eta + T_{\d_z^2 \fb} \rho + 2 T_{\d_z\fb} \d_z\rho + T_{\fb} \d_z^2 \rho\nonumber\\
&+T_{\beta} T_{\d_z\d_y \fb} \rho + T_{\beta}T_{\d_y \fb} \d_z \rho + T_{\beta}T_{\d_z\fb}\eta + T_{\beta}T_{\fb} \d_z\eta - T_{\gamma} T_{\d_y \fb} \rho - T_{\gamma} T_{\fb} \eta\nonumber\\
%=& T_{\fb}\d_z^2\rho + 2 T_{\d_z\fb}\d_z\rho + T_{\beta}T_{\d_y \fb} \d_z\rho + T_{\beta} T_{\fb} \d_z\eta + T_{\alpha}T_{\d_y^2\fb} \rho + 2T_{\alpha}T_{\d_y \fb} \eta\nonumber\\ &+ T_{\d_z^2 \fb} \rho + T_{\beta} T_{\d_z\d_y \fb} \rho + T_{\beta} T_{\d_z\fb} \eta - T_{\gamma}T_{\d_y\fb}\rho -T_{\gamma}T_{\fb} \eta\nonumber\\
%=& T_{\fb} \d_z^2\rho 2 T_{\d_z\fb} \d_z\rho + T_{\beta\d_y\fb} \d_z\rho + \fq[\beta,\d_z\fb](\d_z\rho) + T_{\beta\fb}\d_z\eta + \fq[\beta,\fb](\d_z\eta)\nonumber\\ & + T_{\alpha \d_y^2\fb} \rho + \fq[\alpha,\d_y^2\fb](\rho) + 2 T_{\alpha \d_y \fb} \eta + 2 \fq[\alpha,\d_y\fb](\eta) + T_{\d_z^2 \fb} \rho \nonumber\\ & + T_{\beta\d_z\d_y\fb}\rho + \fq[\beta,\d_z\d_y\fb](\rho) + T_{\beta \d_z\fb} \eta + \fq[\beta,\d_z\fb](\eta) \nonumber\\ & - T_{\gamma \d_y\fb}\rho - \fq[\gamma,\d_y\fb](\rho) - T_{\gamma\fb} \eta - \fq[\gamma,\fb](\eta)\nonumber\\
=& T_{\fb} \d_z^2\rho + 2 T_{\d_z\fb} \d_z\rho + T_{\beta\d_y\fb} \d_z\rho  + T_{\beta\fb}\d_z\eta \nonumber\\ & + T_{\alpha \d_y^2\fb} \rho + 2 T_{\alpha \d_y \fb} \eta + T_{\d_z^2 \fb} \rho + T_{\beta\d_z\d_y\fb}\rho + T_{\beta \d_z\fb} \eta - T_{\gamma \d_y\fb}\rho  - T_{\gamma\fb} \eta \nonumber\\
&+ \fq[\beta,\d_z\fb](\d_z\rho)+ \fq[\beta,\fb](\d_z\eta)+ \fq[\alpha,\d_y^2\fb](\rho)+ 2 \fq[\alpha,\d_y\fb](\eta) \nonumber\\
&+ \fq[\beta,\d_z\d_y\fb](\rho) + \fq[\beta,\d_z\fb](\eta) - \fq[\gamma,\d_y\fb](\rho) - \fq[\gamma,\fb](\eta).
\end{align}
Using the estimates (\ref{albega}) and Propositions \ref{prop:Db} and \ref{prop:clprod}, we obtain that
\begin{align*}
\alpha \d_y^2 \fb, \ \alpha \d_y \fb, \ \d_z^2\fb, \ \beta \d_z \d_y \fb, \ \beta \d_z\fb, \ \gamma \d_y \fb, \ \gamma\fb \in \mC_{y}^0\big( (0,1]; H_z^{s-3}(\R) \big).
\end{align*}
Since $\delta\le s-\frac{5}{2}$, it follows that $\frac{1}{2}-(1-\delta) \le s-3$. Hence,
\begin{equation*}
\alpha \d_y^2 \fb, \ \alpha \d_y \fb, \ \d_z^2\fb, \ \beta \d_z \d_y \fb, \ \beta \d_z\fb, \ \gamma \d_y \fb, \ \gamma\fb\in \mC_{y}^0\big( (0,1]; H_z^{\frac{1}{2}-(1-\delta)}(\R) \big).
\end{equation*}
Since $\teta\in H^{s+\frac{1}{2}}(\R)$, applying Theorem \ref{thm:paraPEst}, one has
\begin{subequations}\label{temp:TetaReg}
\begin{align}
\big\{ y T_{\alpha \d_y^2 \fb} + 2T_{\alpha\d_y\fb} + yT_{\d_z^2\fb} \big\} \teta \in& \ \mC_{y}^0\big( (0,1]; H_z^{s-\frac{1}{2}+\delta}(\R) \big),\\
\big\{ yT_{\beta\d_z\d_y \fb} + T_{\beta\d_z\fb} - yT_{\gamma\d_y \fb} - T_{\gamma\fb} \big\} \teta \in& \ \mC_{y}^0\big( (0,1]; H_z^{s-\frac{1}{2}+\delta}(\R) \big).
\end{align}
\end{subequations}
Moreover, by (\ref{albega}), Proposition \ref{prop:Db}, Theorem \ref{thm:adjprod}\ref{item:prod}, and Remark \ref{rem:paraC}\ref{item:paraC1},
\begin{subequations}\label{temp:PTetaR}
\begin{gather}
\fq[\alpha,\d_{y}^2\fb](\rho), \ \fq[\alpha,\d_y\fb](\eta), \ \fq[\gamma,\d_y\fb](\rho), \ \fq[\gamma,\fb](\eta) \in \mC_{y}^0\big( (0,1]; H_z^{s-\frac{1}{2}+\delta}(\R) \big),\\
\fq[\beta,\fb_z](\d_z\rho), \ \fq[\beta,\fb](\d_z\eta), \ \fq[\beta,\fb_{zy}](\rho), \ \fq[\beta,\fb_z](\eta) \in \mC_{y}^0\big( (0,1]; H_z^{s-\frac{1}{2}+\delta}(\R)\big).
\end{gather}
\end{subequations}
Using the regularities (\ref{temp:TetaReg}) and (\ref{temp:PTetaR}) in the equation (\ref{temp:PTeta}), it holds that there exists an explicit functional $f_2(\cdot, \cdot)$ such that
\begin{align}\label{temp:f2}
P_{\eta}T_{\fb} \rho - T_{\fb}\d_z^2 \rho - 2 T_{\d_z\fb} \d_z\rho & \nonumber\\
- T_{\beta\d_y \fb} \d_z\rho - T_{\beta \fb} \d_z\eta &= f_2(\teta,v) \in \mC_{y}^0\big((0,1];H_{z}^{s-\frac{1}{2}+\delta}(\R)\big).
\end{align}
Thus Substituting (\ref{temp:f2}) into (\ref{temp:f1}), we obtain that
\begin{align}
P_{\eta} u =&f_1 - f_2 - T_{\d_y^2 v} \talpha - T_{\d_z\d_y v} \beta + T_{\d_y v} \tgamma\nonumber\\
&-T_{\fb}\d_z^2\rho - 2 T_{\d_z\fb} \d_z\rho - T_{\beta\d_y \fb} \d_z\rho - T_{\beta \fb} \d_z\eta =\vcentcolon f_1 - f_2 + f_3.
\end{align}
Therefore, it is left to show $f_3=f_3(\teta,v)$ belongs to the space $\mC_{y}^0\big((0,1];H_{z}^{s-\frac{1}{2}+\delta}(\R)\big)$.

\paragraph{Step 3: Cancellations in $f_3(\teta,v)$.} Set $F_R(x)\vcentcolon= \tfrac{1}{x+R}-\tfrac{1}{R}$ and $\bar{F}_R(x)\vcentcolon= \tfrac{1}{(x+R)^2}-\tfrac{1}{R^2}$. Then $\talpha$, $\beta$, and $\tgamma$ can be expressed as
\begin{equation*}
\talpha=\dfrac{|\d_z\rho|^2}{\eta^2} + \bar{F}_{R}(\teta), \quad \beta = - 2 \dfrac{\d_z \rho}{\eta}, \quad \tgamma = \dfrac{\d_z^2 \rho}{\eta} - 2y\dfrac{|\d_z\eta|^2}{\eta^2} - \dfrac{\bar{F}_R(\teta)}{y}.   
\end{equation*}
Using the linearization rules (\ref{brackets}), we obtain
\begin{align*}
\talpha =& T_{1/\eta^2} |\d_z\rho|^2 + T_{|\d_z\rho|^2} \eta^{-2} + \bar{F}_{R}(\teta)\\ 
=& T_{1/\eta^2} \big\{ 2T_{\d_z\rho} \d_z\rho + \fp( \d_z\rho,\d_z\rho ) \big\}+ T_{|\d_z\rho|^2} \bar{F}_R(\teta) + \bar{F}_{R}(\teta)\\
=& 2 T_{\d_z\rho/\eta^2} \d_z\rho + 2 \fq[\eta^{-2},\d_z\rho](\d_z\rho) + T_{1/\eta^2}\fp(\d_z\rho,\d_z\rho) + T_{|\d_z\rho|^2} \bar{F}_R(\teta) + \bar{F}_R(\teta) \\
=&\vcentcolon 2 T_{\d_z\rho/\eta^2} \d_z\rho + \fR_{\talpha}.
\end{align*}
Since $\d_y^2 v = \eta \d_y \fb$ and $\beta=-2 \tfrac{\d_z\rho}{\eta}$, applying the above with the operator $T_{\d_y^2 v}$, we get
\begin{align}\label{temp:Ralpha}
- T_{\d_y^2 v} \talpha  =& - T_{\eta \d_y \fb } \big( 2 T_{\d_z\rho/\eta^2} \d_z\rho + \fR_{\talpha} \big)\nonumber\\ %=& -2 T_{\d_y \fb \d_z\rho/\eta}\d_z\rho - \fq[\eta\d_y \fb, \eta^{-2} \d_z\rho ](\d_z\rho) - T_{\eta \d_y \fb} \fR_{\talpha}\\
=& T_{ \beta \d_y \fb }\d_z\rho - \fq[\eta\d_y \fb, \eta^{-2} \d_z\rho ](\d_z\rho) - T_{\eta \d_y \fb} \fR_{\talpha} =\vcentcolon T_{ \beta \d_y \fb }\d_z\rho + \bar{\fR}_{\talpha}.
\end{align}
Moreover, by the linearization rules (\ref{brackets}) we also have 
\begin{align*}
\beta =& -2 T_{1/\eta} \d_z\rho -2 T_{\d_z\rho} \eta^{-1} -2 \fp( \eta^{-1}, \d_z\rho ) \\
=& -2 T_{1/\eta} \d_z\rho - 2T_{\d_z\rho} F_R(\teta) -2 \fp( F_R(\teta), \d_z\rho ) -2 \fp( R^{-1}, \d_z\rho )\\
=&\vcentcolon -2 T_{1/\eta} \d_z\rho + \fR_{\beta}, 
\end{align*}
Since $\d_z\d_y v = \eta \d_z \fb + \fb \d_z \eta$, and $\beta=-2 \tfrac{\d_z\rho}{\eta}$, applying the above with operator $T_{\d_z\d_y v}$, it follows that
\begin{align}\label{temp:Rbeta}
&-T_{\d_z\d_y v} \beta = -T_{\eta \d_z\fb} \beta - T_{\fb \d_z\eta} \beta\nonumber\\
=&  2 T_{\eta \d_z\fb} T_{1/\eta} \d_z\rho + 2 T_{\fb \d_z \eta} T_{1/\eta} \d_z\rho  - T_{\d_z\d_y v} \fR_{\beta}\nonumber\\
=& 2 T_{\d_z\fb} \d_z \rho + 2 \fq[\eta \d_z\fb, \eta^{-1}](\d_z\rho) + 2 T_{\fb \d_z\eta/\eta} \d_z\rho + 2 \fq[\fb\d_z\eta,\eta^{-1}](\d_z\rho) - T_{\d_z\d_y v}\fR_{\beta}\nonumber\\
=& 2 T_{\d_z\fb} \d_z \rho -  T_{ \beta \fb} \d_z\eta + 2 \fq[\eta \d_z\fb, \eta^{-1}](\d_z\rho)  + 2 \fq[\fb\d_z\eta,\eta^{-1}](\d_z\rho) - T_{\d_z\d_y v}\fR_{\beta}\nonumber\\
=&\vcentcolon 2 T_{\d_z\fb} \d_z \rho -  T_{ \beta \fb} \d_z\eta + \bar{\fR}_{\beta}.
\end{align}
Next, linearising $\tgamma$ by the rules (\ref{brackets}), we also have
\begin{align*}
\tgamma =& T_{1/\eta} \d_z^2 \rho + T_{\d_z^2 \rho } F_R(\teta) + \fp( F_R(\teta), \d_z^2 \rho ) + \fp( R^{-1} , \d_z^2 \rho )\\ &- 2y \big\{ T_{1/\eta^2} |\d_z\eta|^2 + T_{|\d_z\eta|^2} \bar{F}_R(\teta) + \fp( \bar{F}_R(\teta), |\d_z\eta|^2 ) + \fp( R^{-2}, |\d_z\eta|^2 ) \big\} - \tfrac{1}{y}\bar{F}_R(\teta)\\
=& T_{1/\eta} \d_z^2 \rho -2y T_{1/\eta^2} \big\{ 2T_{\d_z\eta} \d_z\eta + \fp( \d_z\eta, \d_z\eta ) \big\} + T_{\d_z^2 \rho } F_R(\teta) + \fp( F_R(\teta), \d_z^2 \rho )\\ & + \fp( \tfrac{1}{R} , \d_z^2 \rho ) - 2y \big\{ T_{|\d_z\eta|^2} \bar{F}_R(\teta) + \fp( \bar{F}_R(\teta), |\d_z\eta|^2 ) + \fp( \tfrac{1}{R^2}, |\d_z\eta|^2 ) \big\} - \tfrac{1}{y}\bar{F}_R(\teta)\\
=& T_{1/\eta} \d_z^2 \rho + 2 T_{\beta/\eta} \d_z\eta - 4y \fq[\eta^{-2},\d_z\eta](\d_z\eta) - 2y T_{1/\eta^2} \fp(\d_z\eta,\d_z\eta)  + T_{\d_z^2 \rho } F_R(\teta) \\ & + \fp( F_R(\teta), \d_z^2 \rho )  + \fp( R^{-1} , \d_z^2 \rho )  - 2y T_{|\d_z\eta|^2} \bar{F}_R(\teta)  - 2y \fp( \bar{F}_R(\teta), |\d_z\eta|^2 ) \\ & - 2y \fp( R^{-2}, |\d_z\eta|^2 ) - \tfrac{1}{y}\bar{F}_R(\teta)\\
=&\vcentcolon T_{1/\eta} \d_z^2 \rho + 2 T_{\beta/\eta} \d_z\eta + \fR_{\tgamma}.
\end{align*}
Applying $T_{\d_y v}$ on the above expression, we obtain that
\begin{align}\label{temp:Rgamma}
&T_{\d_y v} \tgamma = T_{\eta \fb} \tgamma = T_{\eta \fb} T_{1/\eta} \d_z^2 \rho + 2 T_{\eta \fb} T_{\beta/\eta} \d_z\eta + T_{\eta \fb} \fR_{\tgamma}\nonumber\\
=& T_{\fb} \d_z^2 \rho + 2 T_{\beta\fb}\d_z\eta  + \fq[\eta \fb, \eta^{-1}](\d_z^2\rho) + 2 \fq[\eta\fb,\eta^{-1}\beta](\d_z\eta) + T_{\eta \fb}\fR_{\tgamma}\nonumber\\
=&\vcentcolon T_{\fb} \d_z^2 \rho + 2 T_{\beta\fb}\d_z\eta + \bar{\fR}_{\tgamma}.
\end{align}
Adding (\ref{temp:Ralpha})--(\ref{temp:Rgamma}) together, we obtain that
\begin{align*}
&-T_{\d_y^2 v} \talpha - T_{\d_z\d_y v}\beta + T_{\d_y v } \tgamma\\ 
=& T_{ \beta \d_y \fb }\d_z\rho  + 2 T_{\d_z\fb} \d_z \rho -  T_{ \beta \fb} \d_z\eta  + T_{\fb} \d_z^2 \rho + 2 T_{\beta\fb}\d_z\eta + \bar{\fR}_{\talpha} + \bar{\fR}_{\beta} + \bar{\fR}_{\tgamma}\\
=& T_{ \beta \d_y \fb }\d_z\rho  + 2 T_{\d_z\fb} \d_z \rho  + T_{\fb} \d_z^2 \rho + T_{\beta\fb}\d_z\eta + \bar{\fR}_{\talpha} + \bar{\fR}_{\beta} + \bar{\fR}_{\tgamma}.
\end{align*}
Substituting this into the definition of $f_3(\teta,v)$ we obtain that
\begin{align*}
f_3(\teta,v) \vcentcolon=& -T_{\d_y^2 v} \talpha - T_{\d_z\d_y v }\beta + T_{\d_y v }\tgamma -T_{\fb}\d_z^2\rho - 2 T_{\d_z\fb} \d_z\rho - T_{\beta\d_y \fb} \d_z\rho - T_{\beta \fb} \d_z\eta\\
=& T_{ \beta \d_y \fb }\d_z\rho  + 2 T_{\d_z\fb} \d_z \rho  + T_{\fb} \d_z^2 \rho + T_{\beta\fb}\d_z\eta + \bar{\fR}_{\talpha} + \bar{\fR}_{\beta} + \bar{\fR}_{\tgamma}\\
&-T_{\fb}\d_z^2\rho - 2 T_{\d_z\fb} \d_z\rho - T_{\beta\d_y \fb} \d_z\rho - T_{\beta \fb} \d_z\eta\\
=& \bar{\fR}_{\talpha} + \bar{\fR}_{\beta} + \bar{\fR}_{\tgamma}.
\end{align*}
The remainder terms $\bar{\fR}_{\talpha}$, $\bar{\fR}_{\beta}$, and $\bar{\fR}_{\tgamma}$ are given by:
\begin{align*}
\bar{\fR}_{\talpha} %\vcentcolon=& -\fq[\eta\d_y\fb,\eta^{-2}\d_z\rho](\d_z\rho) - T_{\eta\d_y\fb} \fR_{\talpha}\\
%=& -T_{\d_y^2 v} \big\{ 2 \fq[\eta^{-2},\d_z\rho](\d_z\rho) + T_{1/\eta^2}\fp(\d_z\rho,\d_z\rho) + T_{|\d_z\rho|^2} \bar{F}_R(\teta) + \bar{F}_R(\teta) \big\}\\ & -\fq[\eta\d_y\fb,\eta^{-2}\d_z\rho](\d_z\rho)\\
\vcentcolon=& 2 T_{\d_y^2 v} \fq[\eta^{-2},\d_z\rho](\d_z\rho) - T_{\d_y^2 v} T_{1/\eta^2} \fp(\d_z\rho,\d_z\rho) - T_{\d_y^2 v} T_{|\d_z\rho|^2} \bar{F}_R(\teta)\\ &-T_{\d_y^2 v} \bar{F}_R(\teta) -\fq[\eta\d_y\fb,\eta^{-2}\d_z\rho](\d_z\rho),\\
\bar{\fR}_{\beta} %\vcentcolon=& 2 \fq[\eta \d_z\fb, \eta^{-1}](\d_z\rho)  + 2 \fq[\fb\d_z\eta,\eta^{-1}](\d_z\rho) - T_{\d_z\d_y v}\fR_{\beta}\\
%=& T_{\d_z\d_y v} \big\{ 2T_{\d_z\rho} F_R(\teta) + 2 \fp( F_R(\teta), \d_z\rho ) + 2 \fp( R^{-1}, \d_z\rho ) \big\}  \\ &+2 \fq[\eta \d_z\fb, \eta^{-1}](\d_z\rho)  + 2 \fq[\fb\d_z\eta,\eta^{-1}](\d_z\rho)\\
\vcentcolon=& 2 T_{\d_z\d_y v}T_{\d_z\rho} F_R(\teta) + 2 T_{\d_z\d_y v}\fp( F_R(\teta), \d_z\rho ) + 2 T_{\d_z\d_y v}\fp( R^{-1}, \d_z\rho ) \\ &+2 \fq[\eta \d_z\fb, \eta^{-1}](\d_z\rho)  + 2 \fq[\fb\d_z\eta,\eta^{-1}](\d_z\rho),\\
\bar{\fR}_{\tgamma} %\vcentcolon=& \fq[\eta \fb, \eta^{-1}](\d_z^2\rho) + 2 \fq[\eta\fb,\eta^{-1}\beta](\d_z\eta) + T_{\d_y v}\fR_{\tgamma}\\
\vcentcolon=& \fq[\eta \fb, \eta^{-1}](\d_z^2\rho) + 2 \fq[\eta\fb,\eta^{-1}\beta](\d_z\eta) - 4y T_{\d_y v}\fq[\eta^{-2},\d_z\eta](\d_z\eta)\\ &- 2y T_{\d_y v} T_{1/\eta^2} \fp(\d_z\eta,\d_z\eta)  + T_{\d_y v} T_{\d_z^2 \rho } F_R(\teta) + T_{\d_y v}\fp( F_R(\teta), \d_z^2 \rho ) \\ & + T_{\d_y v}\fp( R^{-1} , \d_z^2 \rho )  - 2y T_{\d_y v}T_{|\d_z\eta|^2} \bar{F}_R(\teta)  - 2y T_{\d_y v}\fp( \bar{F}_R(\teta), |\d_z\eta|^2 ) \\ & - 2y T_{\d_y v}\fp( R^{-2}, |\d_z\eta|^2 ) - \tfrac{1}{y}T_{\d_y v}\bar{F}_R(\teta).
\end{align*} 
By Proposition \ref{prop:Db}, $\fb\in \mC_y^0\big((0,1];H_z^{s-1}(\R)\big)$ and $\d_y \fb$, $\d_z \fb \in \mC_y^0\big((0,1];H_z^{s-2}(\R)\big)$. Using this with Propositions \ref{prop:Sobcomp}--\ref{prop:clprod}, and Theorem \ref{thm:adjprod}, terms with the least regularity in $\bar{\fR}_{\talpha}$, $\bar{\fR}_{\beta}$, and $\bar{\fR}_{\tgamma}$ are estimated as
\begin{gather*}
\fq[\eta\d_y\fb,\eta^{-2}\d_z\rho](\d_z\rho), \ \  2\fq[\eta\d_z\fb,\eta^{-1}](\d_z\rho) \ \in \mC_y^0\big((0,1];H_z^{2s-3}(\R)\big),\\
\fq[\eta \fb,\eta^{-1}](\d_z^2\rho) \in \mC_y^0\big((0,1];H_z^{2s-3}(\R)\big)\subset \mC_y^0\big((0,1];H_z^{s-\frac{1}{2}+\delta}(\R)\big),
\end{gather*}
where we used the fact that $2s-3 \ge s-\frac{1}{2}+\delta$ since $\delta \le s-\frac{5}{2}$. Other terms in $\bar{\fR}_{\talpha}$, $\bar{\fR}_{\beta}$, and $\bar{\fR}_{\tgamma}$ can be estimated similarly using Theorems \ref{thm:paraL2}, \ref{thm:adjprod}, \ref{thm:paraPEst}, \ref{thm:bony}, Propositions \ref{prop:Sobcomp}--\ref{prop:clprod}, and Corollary \ref{corol:yC}. Therefore we conclude that
\begin{align*}
f_3(\teta,v)= \bar{\fR}_{\talpha} + \bar{\fR}_{\beta} + \bar{\fR}_{\tgamma} \in \mC_{y}^0\big((0,1];H_{z}^{s-\frac{1}{2}+\delta}(\R)\big).
\end{align*}
This proves the lemma.
\end{proof}

\begin{lemma}[Decomposition of \texorpdfstring{$P_{\eta}$}{Peta}]\label{lemma:decomP}
Define $\delta\vcentcolon= \min\{\frac{1}{2},s-\frac{5}{2}\}$, and let $\teta\in H^{s+\frac{1}{2}}(\R)$. There are symbols $a=a^{(1)}+a^{(0)}$ and $A=A^{(1)}+A^{(0)}$ such that for each fixed $y\in(0,1]$,
\begin{align*}
(z,\xi)\mapsto\big(a^{(1)},A^{(1)}\big)(z,y,\xi) \in \mathring{\Gamma}_{s-1}^{1}(\R) \subseteq \mathring{\Gamma}_{3/2+\delta}^1(\R),\\
 (z,\xi)\mapsto \big(a^{(0)},A^{(0)}\big)(z,y,\xi) \in \mathring{\Gamma}_{s-2}^{0}(\R) \subseteq \mathring{\Gamma}_{1/2+\delta}^0(\R),
\end{align*}
and $P_{\eta}$ admits the decomposition:
\begin{equation*}
	P_{\eta} = T_{\alpha}(\d_y-T_a)(\d_y-T_{A}) + \mR_0 + \mR_1 \d_y,
\end{equation*}
where $\mR_0$, $\mR_1$ are some operators satisfying: 
\begin{equation*}
\|\mR_1\|_{H^{k}(\R)\to H^{k+\frac{1}{2}+\delta}(\R)}<\infty, \quad \|\mR_0\|_{H^{k}(\R)\to H^{k-\frac{1}{2}+\delta}(\R)}<\infty, \quad \text{for all } \  k\in\R.
\end{equation*} 
Moreover, $a=a^{(1)}+a^{(0)}$ and $A=A^{(1)}+A^{(0)}$ take the explicit form:
\begin{subequations}\label{aA}
\begin{gather}
	a^{(1)}(z,y,\xi) = \dfrac{ i y \xi \eta \d_z\eta  - |\xi| \eta}{1+y^2 |\d_z\eta|^2 }, \quad A^{(1)}(z,y,\xi)= \dfrac{ i y \xi \eta \d_z\eta + |\xi| \eta}{1+y^2|\d_z\eta|^2},\label{a0A0}\\
	a^{(0)}(z,y,\xi) = \dfrac{\gamma}{\alpha} + \dfrac{1}{2y} + \dfrac{\beta^2}{8y\alpha} \{1+ iy \d_z\eta \sgn(\xi)\},\label{a1}\\
	A^{(0)}(z,y,\xi) = -\dfrac{1}{2y} - \dfrac{\beta^2}{8y\alpha} \{ 1 + iy \d_z\eta \sgn(\xi) \}.\label{A1}
\end{gather}
\end{subequations} 
\end{lemma}
\begin{proof}
Using the notation given in (\ref{bonyCom}), we have
\begin{align*}
&T_{\alpha} (\d_y - T_a) (\d_y - T_A) = T_{\alpha}\{\d_y^2 - (T_a + T_A) \d_y + T_a T_A \}\\
=& T_{\alpha} \{\d_y^2 - (T_a + T_A) \d_y + T_{a\sharp A} \} + T_{\alpha} \fq[a,A],\\
\text{where } \ & a\sharp A = a^{(1)} A^{(1)} + a^{(1)} A^{(0)} + a^{(0)}A^{(1)} + \tfrac{1}{i} \d_\xi a^{(1)} \d_z A^{(1)} + \mathcal{O}(\xi^{0}).
\end{align*}
Comparing the above with $P_{\eta}=T_{\alpha}\d_y^2 + \d_z^2 + T_{\beta}\d_z\d_y - T_{\gamma}\d_y$, our aim is to seek symbols $a=a^{(1)}+a^{(0)}$ and $A=A^{(1)}+A^{(0)}$ such that
\begin{equation}\label{temp:aA}
(a\sharp A)(z,\xi) = -\dfrac{\xi^2}{\alpha(z)} \quad \text{ and } \quad a(z,\xi)+A(z,\xi) = \dfrac{\gamma(z)-i\xi \beta(z)}{\alpha(z)}.
\end{equation}
In light of this, we set the first order symbols, $a^{(1)}$ and $A^{(1)}$ to satisfy the algebraic system of equations: 
\begin{equation*}
a^{(1)} A^{(1)} = -\dfrac{\xi^2}{\alpha} \quad \text{ and } \quad a^{(1)}+A^{(1)} = -i\xi \dfrac{ \beta}{\alpha}.
\end{equation*}
This is equivalent to the statement that $a^{(1)}$ and $A^{(1)}$ are the two roots to the quadratic polynomial $\mathcal{P}(x) = x^2 + \tfrac{i\xi\beta}{\alpha} x - \tfrac{\xi^2}{\alpha}  = 0$. Thus the quadratic formula implies that
\begin{align*}
a^{(1)},\, A^{(1)} %= \dfrac{1}{2}\bigg\{ - i\xi \dfrac{\beta}{\alpha} \pm \sqrt{ - \xi^2 \dfrac{\beta^2}{\alpha^2} + \dfrac{4\xi^2}{\alpha} } \bigg\}
= \dfrac{1}{2\alpha} \big\{ - i\xi \beta \pm |\xi|\sqrt{4\alpha-\beta^2} \big\} = \dfrac{ i  y \d_z\eta\, \xi \pm |\xi|}{\eta\alpha} .
\end{align*}
We wish to choose $a^{(1)}$, $A^{(1)}$ such that $-\text{Re}\, a^{(1)}\ge c|\xi|$ for some constant $c>0$, hence
\begin{equation}\label{temp:a1A1}
 a^{(1)}\vcentcolon= \dfrac{ i  y\d_z\eta\, \xi - |\xi|}{\eta\alpha} \quad \text{ and } \quad A^{(1)}\vcentcolon=\dfrac{ i  y\d_z\eta\, \xi + |\xi|}{\eta\alpha}.
\end{equation}
Since $\d_z\eta$, $\alpha(\cdot,y)\in H^{s-\frac{1}{2}}(\R) \subseteq W^{s-1,\infty}(\R)$ where the function space $W^{k,\infty}(\R)$ is defined in Definition \ref{def:holder}. Proposition \ref{prop:clprod} and (\ref{temp:a1A1}) imply that
\begin{equation*}
a^{(1)}, \ A^{(1)}\in \mathring{\Gamma}_{s-1}^{1}(\R) \subseteq \mathring{\Gamma}_{3/2+\delta}^{1}(\R), \qquad \text{since } \ s-\tfrac{5}{2} \ge  \delta.  
\end{equation*}
With $a^{(1)}$, $A^{(1)}$ determined, equations (\ref{temp:aA}) requires $a^{(0)}$ and $A^{(0)}$ to satisfy the system of algebraic equations:
\begin{align*}
a^{(0)}A^{(1)} + a^{(1)} A^{(0)} + \frac{1}{i} \d_\xi a^{(1)} \d_z A^{(1)} =0 \quad \text{ and } \quad a^{(0)}+A^{(0)} = \frac{\gamma}{\alpha}.
\end{align*}
This can be explicitly solved to give
\begin{subequations}\label{temp:a0A0'}
\begin{align}
	a^{(0)} =& \dfrac{1}{A^{(1)}-a^{(1)}} \Big( i \d_\xi a^{(1)} \d_z A^{(1)} - \dfrac{\gamma}{\alpha} a^{(1)} \Big),\label{a0'}\\
	A^{(0)} =& \dfrac{1}{a^{(1)}-A^{(1)}} \Big( i \d_\xi a^{(1)} \d_z A^{(1)} - \dfrac{\gamma}{\alpha} A^{(1)} \Big).\label{A0'}
\end{align}
\end{subequations}
%%%%%%%%%%%%%%%%%%%%%%%%%%%%%%%%%%%%%%%%%%%%%%
%%%%%%%%% Intermediate Calculations %%%%%%%%%%
%%%%%%%%%%%%%%%%%%%%%%%%%%%%%%%%%%%%%%%%%%%%%%
\iffalse
\begin{align*}
(1+y^2|\d_z\eta|^2) i \d_\xi a^{(1)} =& -y \eta \d_z\eta - i \sgn(\xi) \eta,\\
(1+y^2 |\d_z\eta|^2 )^2 \d_z A^{(1)} =& -iy^3 \xi |\d_z\eta|^4 - i y^2 \xi \eta^2 |\d_z\eta|^2 \gamma + 2i y \xi |\d_z\eta|^2 + i\xi \eta^2 \gamma + iy^{-1} \xi\\
&- 3y^2 |\xi| (\d_z\eta)^3 - 2y |\xi| \eta^2 \d_z\eta \gamma -|\xi| \d_z\eta, \\
(1+y^2|\d_z\eta|^2)^3 i \d_\xi a^{(1)} \d_z A^{(1)} =& \gamma \eta^3 (1+y^2|\d_z\eta|^2) (iy\xi\d_z\eta +|\xi|) + \tfrac{1}{y}|\xi|\eta(1+y^2|\d_z\eta|^2)^2 \\
&+ y\eta |\d_z\eta|^2 (1+y^2|\d_z\eta|^2) (iy\xi\d_z\eta+|\xi|),\\
i \d_\xi a^{(1)} \d_z A^{(1)} =& \dfrac{|\xi|}{y \eta \alpha} + \dfrac{\gamma}{\alpha}A^{(1)} + \dfrac{y |\d_z\eta|^2 }{\eta^2 \alpha} A^{(1)}. 
\end{align*}
\fi
%%%%%%%%%%%%%%%%%%%%%%%%%%%%%%%%%%%%%%%%%%%%%%
%%%%%%%%% Intermediate Calculations %%%%%%%%%%
%%%%%%%%%%%%%%%%%%%%%%%%%%%%%%%%%%%%%%%%%%%%%%
Computing $a^{(0)}$ and $A^{(0)}$ while using (\ref{Leta2}) and (\ref{temp:a0A0'}), we get
\begin{subequations}\label{temp:a0A0}
\begin{align}
a^{(0)} %=& \dfrac{\gamma}{\alpha} + \dfrac{1}{2y} + \dfrac{y|\d_z\eta|^2}{2\eta^2 \alpha} \{1+ iy \d_z\eta \sgn(\xi)\}\\
=&\dfrac{\gamma}{\alpha} + \dfrac{1}{2y} + \dfrac{\beta^2}{8y\alpha} \{1+ iy \d_z\eta\, \sgn(\xi)\},\label{a0}\\
A^{(0)} %=& -\dfrac{1}{2y} - \dfrac{y|\d_z\eta|^2 }{2\eta^2 \alpha} \{ 1 + iy \d_z\eta\, \sgn(\xi) \}\\
=&-\dfrac{1}{2y} - \dfrac{\beta^2}{8y\alpha} \{ 1 + iy \d_z\eta\, \sgn(\xi) \}.\label{A0} 
\end{align}
\end{subequations}
Since $(\gamma-\tfrac{1}{yR^2})(\cdot,y)\in H^{s-\frac{3}{2}}(\R) \subseteq W^{s-2,\infty}(\R)$. Proposition \ref{prop:clprod} and (\ref{temp:a0A0}) imply that
\begin{equation*}
a^{(0)}, \ A^{(0)}\in \mathring{\Gamma}_{s-2}^{0}(\R) \subseteq \mathring{\Gamma}_{1/2+\delta}^{0}(\R), \qquad \text{since } \ s-\tfrac{5}{2} \ge  \delta.  
\end{equation*}
By construction (\ref{temp:a1A1})--(\ref{temp:a0A0}), we have
\begin{align*}
&T_{\alpha}(\d_y-T_{a})(\d_y -T_A)= T_{\alpha} \{\d_y^2 - (T_A + T_a) \d_y + T_{a\sharp A} \} + T_{\alpha} \fq[a,A]\\
=& T_{\alpha} \d_y^2 + T_{\alpha} T_{\beta/\alpha} \d_z\d_y  - T_{\alpha} T_{\gamma/\alpha}\d_y + T_{\alpha} T_{1/\alpha} \d_z^2 + T_{\alpha} \fq[a,A]\\
=& T_{\alpha} \d_y^2 + T_{\beta} \d_z \d_y - T_{\gamma} \d_y + \d_z^2 + \fq[\alpha,\sigma_1] \d_y - \fq[\alpha,\gamma/\alpha] \d_y + \fq[\alpha,\sigma_2] + T_{\alpha} \fq[a,A],  
\end{align*}
where $\sigma_1(z,\xi)\vcentcolon= i\xi \beta(z) $ and $\sigma_2(z,\xi)\vcentcolon= -\xi^2 / \alpha(z) $. Since $\alpha$, $\beta\in H^{s-\frac{1}{2}}(\R) \subseteq W^{s-1,\infty}(\R)$ where the function space $W^{k,\infty}(\R)$ is defined in Definition \ref{def:holder}. Thus $\sigma_1 \in \Gamma_{s-1}^{1}(\R)$ and $\sigma_2 \in \Gamma_{s-1}^{2}(\R)$. Define $\mR_1\vcentcolon=\fq[\alpha,\sigma_1] - \fq[\alpha,\gamma/\alpha]$ and $\mR_0\vcentcolon= \fq[\alpha,\sigma_2]+T_{\alpha}\fq[a,A]$. Then by Theorems \ref{thm:paraL2} and \ref{thm:adjprod}\ref{item:prod}, for any $k\in\R$,
\begin{gather*}
\|\mR_1\|_{H^{k}\to H^{k-(2-s)}} \le \mM_{0}^0(\alpha) \mM_{s-1}^{1}(\sigma_1) + \mM_0^0(\alpha) \mM_{0}^{0}(\gamma/\alpha),\\
\|\mR_0\|_{H^{k}\to H^{k-(3-s)}} \le \mM_{0}^{0}(\alpha)\mM_{s-1}^2(\sigma_2) + \|T_{\alpha}\|_{L^2\to L^2} \mM_{s-1}^1(a)\mM_{s-1}^1(A).
\end{gather*}
From $s-\frac{5}{2}\ge \delta$, one can verify that $k-(2-s)\ge k+\frac{1}{2}+\delta$ and $k-(3-s) \ge k-\frac{1}{2}+\delta$, for all $k\in\R$. Therefore we have
\begin{equation*}
\|\mR_1\|_{H^{k}\to H^{k+\frac{1}{2}+\delta}}<\infty, \quad \|\mR_0\|_{H^{k}\to H^{k-\frac{1}{2}+\delta}}<\infty, \quad \text{for all } \ k\in\R.
\end{equation*}
This concludes the proof this lemma.
\end{proof}

%\todo{we gain the epsilon regularity cancelation (the previous ) result and this one}
\begin{proposition}[Boundary Regularity]\label{prop:G-A}
Assume $\teta\in H^{s+\frac{1}{2}}(\R)$. Let $a=a^{(1)}+a^{(0)}$ be the symbol in (\ref{aA}). Given $r>0$, let $g\in \mC_y^0\big( (0,1]; H_z^{r}(\R) \big)$. Suppose $\mu\in \mC^0_y\big((0,1];H_z^{-\infty}(\R)\big)$ \textnormal{(}where $H^{-\infty}\vcentcolon= \cup_{-\infty<s}H^s$\textnormal{)} is such that it solves the equation: 
\begin{equation}\label{assum:G-A}
\d_z \mu - T_{a^{(1)}} \mu = T_{a^{(0)}} \mu + g \quad \text{ in } \ (z,y)\in \R\times(0,1]. 
\end{equation}
Then for all $\ep\in(0,1)$,
\begin{equation*}
\mu\vert_{y=1} \in H^{r+1-\ep}(\R).
\end{equation*}
\end{proposition}
\begin{proof}
We divide the proof into several steps:

\paragraph{Step 1: Determination of exponentially decaying symbol.} First, for a fixed $\tau\in (0,1]$, we wish to find a symbol $e_{\tau}(z,y,\xi)$ which solves the problem:
\begin{equation}\label{temp:e1DE}
\d_y e_{\tau} = -a^{(1)} e_{\tau} \ \ \text{ and } \ \ e_{\tau}(z,\tau,\xi) = 1, \quad \text{for } \ (z,y,\xi)\in\R\times(0,\tau]\times\R.
\end{equation}
In light of the partial fraction decomposition we seek $q_1$, $q_2$ such that
\begin{align*}
-a^{(1)}= \dfrac{|\xi|\eta - iy\xi \eta \d_z\eta }{1+y^2 |\d_z\eta|^2} = \dfrac{q_1}{1+ i y \d_z\eta} + \dfrac{q_2}{1-i y \d_z\eta} = \dfrac{q_1+q_2+i y \d_z\eta (q_2-q_1)}{1+y^2 |\d_z\eta|^2}.
\end{align*}
Thus we set $q_1+q_2\vcentcolon=|\xi|\eta$ and $q_2-q_1= -\xi\eta$. Solving this, we obtain that
\begin{equation*}
q_1 = \eta \dfrac{\xi+|\xi|}{2} = \eta \max\{0,\xi\} =\vcentcolon \eta \xi_{+}, \quad q_2 = -\eta \dfrac{\xi-|\xi|}{2} = -\eta \min\{0,\xi\} =\vcentcolon -\eta \xi_{-}.
\end{equation*}
Thus $-a^{(1)}$ admits the fraction decomposition
\begin{equation*}
-a^{(1)}(z,y,\xi) = \dfrac{\eta \xi_{+}}{1+ i y \d_z\eta} - \dfrac{\eta \xi_{-}}{1- iy\d_z\eta}.
\end{equation*} 
Integrating the above in $x$ from $x=\tau$ to $x=y$ for $y\in (0,\tau]$, 
\begin{align*}
-\int_{\tau}^{y}\!\! a^{(1)}(z,x,\xi)\, \dif x =& \eta \xi_{+}\int_{\tau}^{y}\!\! \dfrac{\dif x}{1+ ix \d_z\eta} - \eta \xi_{-} \int_{\tau}^{y}\!\! \dfrac{\dif x}{1-i x \d_z\eta}\\
=& \dfrac{\eta \xi_{+}}{i\d_z\eta} \int_{1+i\tau \d_z\eta}^{1+iy \d_z\eta}\! \dfrac{\dif x}{x} -  \dfrac{\eta \xi_{-}}{(-i\d_z\eta)} \int_{1-i\tau\d_z\eta}^{1-iy \d_z\eta}\! \dfrac{\dif x}{x}\\
=& -i \dfrac{\eta \xi_{+}}{\d_z\eta} \ln\big( \dfrac{1+iy\d_z\eta}{1+i\tau\d_z\eta} \big) - i \dfrac{\eta \xi_{-}}{\d_z\eta} \ln\big( \dfrac{1-iy\d_z\eta}{1-i\tau\d_z\eta} \big).
\end{align*}
Therefore the integral is given by
\begin{align}\label{temp:-a1Int}
-\int_{\tau}^{y}\!\! a^{(1)}(z,x,\xi)\,\dif x = -i \dfrac{\eta \xi_{+}}{\d_z\eta} \ln \zeta_{\tau} - i \dfrac{\eta \xi_{-}}{\d_z\eta} \ln\bar{\zeta_{\tau}} \ \text{ where } \ \zeta_{\tau}(z,y)\vcentcolon= \dfrac{1+iy\d_z\eta}{1+i\tau\d_z\eta}.
\end{align}
The polar form of a general complex number $\sigma\in\mathbb{C}$ is given by $\sigma=|\sigma|\exp\big(i\theta(\sigma)\big)$ where $|\sigma|=\sqrt{|\text{Re}\sigma|^2 + |\text{Im}\sigma|^2}$ and $\theta(\sigma)= \arctan(\frac{\text{Im}\sigma}{\text{Re}\sigma})$. Thus rewriting $\zeta_{\tau}$ in the standard complex form, we have
\begin{align*}
\zeta_{\tau} =& \dfrac{1+y\tau |\d_z\eta|^2}{1+\tau^2|\d_z\eta|^2} - i \dfrac{(\tau-y)\d_z\eta}{1+\tau^2|\d_z\eta|^2}, \ \text{ which implies } \ \theta(\zeta_{\tau}) = - \arctan\Big(\dfrac{(\tau-y)\d_z\eta}{1+y\tau|\d_z\eta|^2}\Big). 
\end{align*}
The principal value of complex Logarithmic function is give by $\ln \sigma = \ln |\sigma| + i \theta(\sigma)$. In addition $|\bar{\sigma}|=|\sigma|$ and $\theta(\bar{\sigma})=-\theta(\sigma)$. Thus, the equation (\ref{temp:-a1Int}) reduces to 
\begin{align*}
&-\int_{\tau}^{y}\!\! a^{(1)}(z,\tau,\xi)\, \dif \tau \doteq -i \dfrac{\eta\xi_{+}}{\d_z\eta} \big( \ln|\zeta_{\tau}| + i \theta(\zeta_{\tau}) \big) - i \dfrac{\eta\xi_{-}}{\d_z\eta} \big( \ln|\zeta_{\tau}| - i\theta(\zeta_{\tau}) \big)\\
=& |\xi| \dfrac{\eta}{\d_z\eta} \theta(\zeta_{\tau}) -i\xi \dfrac{\eta\ln|\zeta_{\tau}|}{\d_z\eta} = - |\xi| \dfrac{\eta}{\d_z\eta} \arctan\Big( \dfrac{(\tau-y)\d_z\eta}{1+y\tau|\d_z\eta|^2} \Big) -i\xi \dfrac{\eta\ln|\zeta_{\tau}|}{\d_z\eta},
\end{align*}
where the notation ``$\doteq$" is understood as equality modulo addition by $2\pi k |\xi| \frac{\eta}{\d_z\eta}$ with $k\in\mathbb{Z}$. Since $\arctan(-x)=-\arctan x$, the above integral can be further written as 
\begin{align*}
-\int_{\tau}^{y}\!\! a^{(1)}(z,x,\xi)\, \dif x \doteq - |\xi| \dfrac{\eta}{|\d_z\eta|} \arctan\Big( \dfrac{(\tau-y)|\d_z\eta|}{1+y\tau|\d_z\eta|^2} \Big) -i\xi \dfrac{\eta\ln|\zeta_{\tau}|}{\d_z\eta} 
\end{align*}
%\todo{main difference as opposed to ABZ paper where the kernel is given by the exponential}
If we define the symbol $e_{\tau}(z,y,\xi)$ as
\begin{gather}
e_{\tau}(z,y,\xi)\vcentcolon= \exp\Big(- |\xi| \dfrac{\eta}{|\d_z\eta|} \arctan\Big( \dfrac{(\tau-y)|\d_z\eta|}{1+y\tau|\d_z\eta|^2} \Big) -i\xi \dfrac{\eta\ln|\zeta_{\tau}|}{\d_z\eta} \Big),\label{temp:e}\\
\text{where } \ |\zeta_{\tau}|(z,y) = \sqrt{\dfrac{1+y^2|\d_z\eta|^2}{1+\tau^2|\d_z\eta|^2}}.\nonumber 
\end{gather}
Then by construction $e_{\tau}$ solves the problem (\ref{temp:e1DE}).
Next, we claim that for each $m>0$, there exists a constant $C>0$ depending only on $m>0$ and $\|\teta\|_{H^{s+\frac{1}{2}}(\R)}$ such that for all $(z,y,\xi)\in\R\times[0,\tau]\times\R$,
\begin{equation}\label{temp:eIneq}
	0\le |e_{\tau}(z,y,\xi)|\le 1, \qquad |(\tau-y)\xi|^{m} |e_{\tau}(z,y,\xi)| \le C.
\end{equation}
The first assertion of (\ref{temp:eIneq}) immediately follows from the construction (\ref{temp:e}). To prove the second assertion, we define the function $F_{\tau}(x)$
\begin{equation*}
F_{\tau}(x)\vcentcolon= \dfrac{\eta}{|\d_z\eta|} \arctan\Big( \dfrac{x|\d_z\eta|}{1+(\tau-x)\tau|\d_z\eta|^2} \Big) \ \Rightarrow \ e = \exp\Big(-|\xi|F_{\tau}(\tau-y) - i \xi \dfrac{\eta \ln|\zeta_{\tau}|}{\d_z\eta} \Big).
\end{equation*}
Thus we have $|\xi F_{\tau}(\tau-y)|^m |e_\tau| \le C(m)$ for some constant $C(m)$ depending only on $m>0$. Next by Mean value theorem, for $x>0$, $\arctan(x)= \frac{x}{1+x_{\ast}^2}$ for some $x_{\ast}\in (0,x)$. It follows that $\frac{x}{1+x^2}\le \arctan(x) \le x$ for $x>0$. Using this, one has for all $y\in[0,\tau]$:
\begin{align*}
 F_{\tau}(\tau-y) \ge (\tau-y) \dfrac{\eta(1+y\tau|\d_z\eta|^2)}{(1+\tau^2|\d_z\eta|^2)(1+y^2|\d_z\eta|^2)}\ge (\tau-y) \dfrac{R-\|\teta\|_{L^{\infty}}}{(1+\|\d_z\teta\|_{L^{\infty}}^2)^2}.
\end{align*}
Therefore we conclude that there exists a constant $C=C\big(m,\|\teta\|_{H^{s+\frac{1}{2}}}\big)>0$ such that
\begin{equation*}
|(\tau-y)\xi|^m |e_{\tau}| \le \dfrac{(1+\|\d_z\eta\|_{L^{\infty}}^2)^{2m}}{(R-\|\teta\|_{L^{\infty}})^m} |\xi F_{\tau}(\tau-y)|^m |e_{\tau}| \le C. 
\end{equation*}

\paragraph{Step 2: Induction on the regularity of $\mu$.} For simplicity, throughout this step, we use the notation $\|\cdot\|_{k}\equiv \|\cdot\|_{H^{k}(\R)}$ for $k\in\R$. By construction, it holds that $\d_y e_{\tau} = - e_{\tau} a^{(1)}  $, therefore from assumption (\ref{assum:G-A}), we have
\begin{equation}\label{temp:dyODE}
	\d_y ( T_{e_{\tau}} \mu ) = -T_{e_{\tau} a^{(1)}} \mu + T_{e_{\tau}}\d_y \mu = (T_{e_{\tau}}T_{a^{(1)}} - T_{e_{\tau} a^{(1)}}) \mu + T_{e_{\tau}}T_{a^{(0)}} \mu  + T_{e_{\tau}} g.
\end{equation} 
In addition $\mu\in \mC_y^0\big((0,1];H_z^{-\infty}(\R)\big)$ by the assumption of proposition. Thus there exists $k_0\in \R$ such that $\mu\in \mC_y^0\big((0,1];H_z^{k_0}(\R)\big)$. Fix $\ep\in(0,1)$. We set $k_n\vcentcolon= k_0+n(1-\ep)$ for $n\in\mathbb{N}$. Let $r>0$ be the index of regularity for $g$. If we define $N\vcentcolon= \lceil \tfrac{r-k_0}{1-\ep} \rceil $, then $k_N \ge r > k_{N-1}$. We choose $y_0\in(0,1)$ so small such that $2^N y_0 < 1$, and denote $y_n\vcentcolon= 2^n y_0$. For each $n\in \{0,1,2,\dotsc,N\}$, we set:
\begin{equation}\label{temp:muInd}
Q(n)\vcentcolon=\sup_{ y_n \le y \le 1} \|\mu(\cdot,y)\|_{H^{k_n}(\R)}\equiv \sup_{ y_n \le y \le 1} \|\mu(\cdot,y)\|_{k_n}. 
\end{equation} 
By our assumption, it follows that the base case is given: $Q(0)<\infty$. Therefore our aim is to prove via induction that $Q(n)<\infty$ for all $1\le n \le N$. Assume $Q(n)<\infty$. Take $\tau\in (y_n,1]$. By construction (\ref{temp:e}), we have $e_{\tau}(z,\tau,D)=1$. Hence, integrating (\ref{temp:dyODE}) in $y\in[y_{n},\tau]$, we obtain
\begin{align}\label{temp:TeODE}
\mu(z,\tau) =& (1-T_{1})\mu(z,\tau) + (T_{e_{\tau}} \mu)(z,y_n)+ \int_{y_n}^{\tau}\!\! T_{e_{\tau}} T_{a^{(0)}} \mu(z,y) \, \dif y  \\ 
&+ \int_{y_n}^{\tau}\!\! \big\{T_{e_{\tau}}T_{a^{(1)}} - T_{e_{\tau} a^{(1)}}\big\} \mu(z,y) \, \dif y  + \int_{y_n}^{\tau}\!\! T_{e_{\tau}} g (z,y) \, \dif y.\nonumber
\end{align}
Taking $H^{k_{n+1}}(\R)$-norm on both sides of (\ref{temp:TeODE}) with respect to the variable $z\in\R$,
\begin{align}\label{temp:mu}
\|\mu(\cdot,\tau)\|_{k_{n+1}} \le& \|(1-T_1)\mu(\cdot,\tau)\|_{k_{n+1}} + \|T_{e_{\tau}}\mu(\cdot,y_n)\|_{k_{n+1}}+ \int_{y_n}^{\tau}\!\! \| T_{e_{\tau}} T_{a^{(0)}} \mu(\cdot,y) \|_{k_{n+1}}\,\dif y \nonumber\\
& + \int_{y_n}^{\tau}\!\!  \|T_{e_{\tau}}g(\cdot,y)\|_{k_{n+1}}\,\dif y + \int_{y_n}^{\tau}\!\!  \|\{T_{e_{\tau}} T_{a^{(1)}}- T_{e_{\tau} a^{(1)}}\}\mu(\cdot,y)\|_{k_{n+1}}\,\dif y \nonumber\\
=&\vcentcolon \mathrm{(I)} + \mathrm{(II)} + \mathrm{(III)} + \mathrm{(IV)} + \mathrm{(V)}.
\end{align}
First, by Remark \ref{rem:paraC}\ref{item:paraC2}, we have that
\begin{equation*}
\mathrm{(I)}\vcentcolon=\| (1-T_1) \mu(\cdot, \tau) \|_{k_{n+1}} \le C(k_{n+1},k_0) \sup\limits_{y_0\le y\le 1}\|\mu(\cdot,y)\|_{k_0}=C(k_{n+1},k_0)Q(0).
\end{equation*}
To estimate $\mathrm{(II)}$, we take $m = 1-\ep$ in (\ref{temp:eIneq}). It holds that there exists a constant $C>0$ depending only on $\ep$ and $\|\teta\|_{H^{s+\frac{1}{2}}}$ such that
\begin{align*}
\|(\tau-y_n)^{1-\ep}\absm{D}^{1-\ep}T_{e_{\tau}}\mu(\cdot,y_n)\|_{k_n} \le C \|\mu(\cdot,y_n)\|_{k_n}< C Q(n),
\end{align*}
where the operator $\absm{D}=(1+|D|^2)^{\frac{1}{2}}$ denotes the Fourier multiplier with respect to the variable $z\in\R$. Using this and $k_{n+1}-k_{n}=1-\ep$, we obtain that
\begin{align}
\mathrm{(II)}\vcentcolon=&\|T_{e_{\tau}}\mu(\cdot,y_n)\|_{k_{n+1}}\nonumber\\
=& (\tau-y_n)^{k_n-k_{n+1}} \| (\tau-y_n)^{k_{n+1}-k_n}\absm{D}^{k_{n+1}-k_n} \absm{D}^{k_n-k_{n+1}} \absm{D}^{k_{n+1}} T_{e_{\tau}}\mu(\cdot,y_n)\|_{L^2(\R)}\nonumber\\
\le %&(\tau-y_n)^{\ep-1}\|(\tau-y_n)^{k_{n+1}-k_n}\absm{D}^{k_{n+1}-k_n}T_{e_{\tau}}\mu(\cdot,y_n)\|_{k_n}\nonumber\\
&(\tau-y_n)^{\ep-1}\|(\tau-y_n)^{1-\ep}\absm{D}^{1-\ep}T_{e_{\tau}}\mu(\cdot,y_n)\|_{k_n}\nonumber\\
%\le& C (\tau-y_n)^{\ep-1} \|\mu(\cdot,y_n)\|_{k_n}\nonumber\\
\le& C (\tau-y_n)^{-(1-\ep)} Q(n).\nonumber
\end{align}
For $\mathrm{(III)}$, the inequality (\ref{temp:eIneq}) implies that there exists a constant $C>0$ depending only on $\ep$ and $\|\teta\|_{H^{s+\frac{1}{2}}}$ such that for $y\in[y_n,\tau]$,
\begin{align}\label{temp:III'}
\big\| (\tau-y)^{1-\ep}\absm{D}^{1-\ep} T_{e_{\tau}} T_{a^{(0)}} \mu (\cdot, y) \big\|_{k_n} \le C \| T_{a^{(0)}} \mu (\cdot, y) \|_{k_n}. 
\end{align}
From Lemma \ref{lemma:decomP}, one has $a^{(0)}\in \Gamma_{1/2+\delta}^0(\R)\subset \Gamma_{0}^0(\R)$. Thus Theorem \ref{thm:paraL2} implies that the operator $T_{a^{(0)}}$ is of order $0$ and $\|T_{a^{(0)}}\|_{H^k\to H^k}\le C(k)\mM_{0}^0\big(a^{(0)}\big)$ for all $k\in\R$. Combining this with (\ref{temp:III'}), and using that $k_{n+1}=1-\ep+k_{n}$, we obtain:
\begin{align*}
\mathrm{(III)}\vcentcolon=&\int_{y_n}^{\tau}\!\! \|T_{e_\tau} T_{a^{(0)}} \mu\|_{k_{n+1}} (\cdot, y) \, \dif y \nonumber\\
=&\!\! \int_{y_n}^{\tau}\!\! (\tau-y)^{\ep-1} \big\| (\tau-y)^{1-\ep} T_{e_\tau} T_{a^{(0)}} \mu (\cdot, y)\big\|_{1-\ep+k_n} \, \dif y \nonumber\\
\le & \int_{y_n}^{\tau}\!\! (\tau-y)^{\ep-1} \|(\tau-y)^{1-\ep} \absm{D}^{1-\ep} T_{e_{\tau}} T_{a^{(0)}}\mu(\cdot,y)\|_{k_n} \, \dif y\nonumber\\
\le& C\!\!\int_{y_n}^{\tau}\!\!\! (\tau-y)^{\ep-1} \|T_{a^{(0)}}\mu(\cdot,y)\|_{k_n}\,\dif y  \le C \mM_0^{0}\big(a^{(0)}\big) \int_{y_n}^{\tau}\!\!\! (\tau-y)^{\ep-1} \|\mu(\cdot,y)\|_{k_n}\, \dif y  \nonumber\\
\le & C \mM_0^0\big(a^{(0)}\big) Q(n) \int_{y_n}^{\tau} (\tau-y)^{\ep-1}\, \dif y = C \mM_0^0\big(a^{(0)}\big) \dfrac{(\tau-y_n)^\ep}{\ep} Q(n).
\end{align*}
Next, we estimate the term $(\mathrm{IV})$. It follows from (\ref{temp:eIneq}) that there exists a constant $C>0$ depending only on $\ep$ and $\|\teta\|_{H^{s+\frac{1}{2}}}$ such that for $y\in[y_n,\tau]$,
\begin{align*}
\big\| (\tau-y)^{1-\ep}\absm{D}^{1-\ep} T_{e_\tau} g (\cdot, y) \big\|_{k_n} \le C \| g (\cdot, y) \|_{k_n} \le C \| g (\cdot, y) \|_{r}, 
\end{align*}
where we used the fact that $k_n\le r$ for $n\in\{0,1,2,\dotsc,N-1\}$. Using this and $k_{n+1}=1-\ep+k_n$, we obtain that
\begin{align*}
&\mathrm{(IV)}\vcentcolon=\int_{y_n}^{\tau}\!\! \|T_{e_\tau} g (\cdot, y)\|_{k_{n+1}} \, \dif y \!\! = \int_{y_n}^{\tau}\!\! (\tau-y)^{\ep-1}\| (\tau-y)^{1-\ep} T_{e_\tau} g (\cdot, y)\|_{1-\ep+k_n} \, \dif y \nonumber\\
\le &  \int_{y_n}^{\tau}\!\! (\tau-y)^{\ep-1} \|(\tau-y)^{1-\ep} \absm{D}^{1-\ep} T_{e_{\tau}} g(\cdot,y)\|_{k_n} \, \dif y\nonumber\\
\le & C \sup\limits_{y_0\le y\le 1}\|g(\cdot,y)\|_{r} \int_{y_n}^{\tau}\!\! (\tau-y)^{\ep-1}\, \dif y = \dfrac{C(\tau-y_n)^\ep}{\ep}\!\!\sup\limits_{y_0\le y\le 1}\|g(\cdot,y)\|_{r}.
\end{align*}
Next, we claim that for any $k\in\R$, there exists a constant $C>0$ depending only on $k$, $\ep$, and $\|\teta\|_{H^{s+\frac{1}{2}}}$ such that
\begin{align}\label{temp:TeTa}
\| (T_{e_\tau}T_{a^{(1)}} - T_{e_{\tau}a^{(1)}})(\cdot,y) \|_{H^k(\R)\to H^{k+1-\ep}(\R)} \le C (1-y)^{\ep-1}.
\end{align}
To show this, we set $\tilde{e}_{\tau}(z,y,\xi)\vcentcolon= (\tau-y)^{1-\ep} e_{\tau}(z,y,\xi)$, and our goal is to prove that the semi-norm $\mM_{1}^{\ep-1}\big(\tilde{e}_{\tau}(\cdot,y,\cdot)\big)$ is bounded uniformly in $\tau\in(0,1]$ and $y\in[0,\tau]$. Recall that $\mM_{\theta}^m(\sigma)$ is the semi-norm for symbols defined in Definition \ref{def:symbols}. For $y\in[0,\tau]$, $m\in\mathbb{N}$, and $|\xi|\ge \tfrac{1}{2}$, we have from (\ref{temp:e}) that 
\begin{gather*}
 |\xi|^{m-(\ep-1)}\d_\xi^m \tilde{e}_{\tau}(z,y,\xi) = - |(\tau-y)\xi|^{1-\ep} |\xi|^{m}  \Big\{ F_{\tau}(\tau-y) +i \dfrac{\eta \ln|\zeta_\tau|}{\d_z\eta} \Big\}^m e_{\tau}(z,y,\xi),\\
 \text{where } \ F_{\tau}(\tau-y)= \dfrac{\eta}{|\d_z\eta|}\arctan\Big( \dfrac{(\tau-y)|\d_z\eta|}{1+y\tau|\d_z\eta|^2} \Big), \quad |\zeta_{\tau}|(z,y) = \sqrt{\dfrac{1+y^2|\d_z\eta|^2}{1+\tau^2 |\d_z\eta|^2}}.  
\end{gather*}
Since $\arctan(x)\le x$, we have that for all $\tau\in(0,1]$ and $y\in[0,\tau]$,
\begin{align}\label{temp:Fup}
F_{\tau}(\tau-y) \le \dfrac{(\tau-y)\eta}{1+y\tau|\d_z\eta|^2} \le (\tau-y) \|\eta\|_{L^{\infty}(\R)}.
\end{align}
Furthermore, since $\ln|\zeta_{\tau}|(z,y)\to 0$ as $y\to \tau^{-}$, we can apply L'H\^opital's rule to obtain:
\begin{align*}
\lim\limits_{y\to \tau^{-}} \dfrac{\eta \ln|\zeta_{\tau}|}{(\tau-y)\d_z\eta}  \overset{H}{=} - \dfrac{\eta}{2\d_z\eta} \lim\limits_{y\to \tau^-} \dfrac{2y |\d_z\eta|^2}{1+y^2|\d_z\eta|^2} 
=-\dfrac{\eta \d_z\eta}{1+|\d_z\eta|^2}.   
\end{align*}
Therefore there exists a generic constant $C>0$ independent of $y$ and $\tau$ such that
\begin{equation}\label{temp:lnzetaUp}
\Big|\dfrac{\eta \ln|\zeta_{\tau}|}{\d_z\eta}\Big| \le C(1+\|\eta\d_z\eta\|_{L^{\infty}(\R)}) (\tau-y), \quad \text{for all } \ \tau\in(0,1] \ \text{ and } \ y\in[0,\tau]. 
\end{equation}
Substituting (\ref{temp:Fup})--(\ref{temp:lnzetaUp}) into the expression for $|\xi|^{m-(\ep-1)}\d_\xi^m \tilde{e}_{\tau}$, and using (\ref{temp:eIneq}), there exists $C=C(m,\ep,\|\teta\|_{H^{s+\frac{1}{2}}})>0$ independent of $\tau\in(0,1]$ and $y\in[0,\tau]$ such that
\begin{align*}
&|\xi|^{m-(\ep-1)}|\d_\xi^m \tilde{e}_{\tau}(z,y,\xi)| =  |(\tau-y)\xi|^{1-\ep} |\xi|^{m}  \Big\{ |F_{\tau}(\tau-y)| + \Big| \dfrac{\eta \ln|\zeta_\tau|}{\d_z\eta}\Big| \Big\}^m e_{\tau}(z,y,\xi),\\
\le &C  \big\{ 1 + \|\eta\|_{L^{\infty}(\R)} + \|\eta\d_z\teta\|_{L^{\infty}(\R)}\big\}^m |(\tau-y)\xi|^{m+1-\ep} e_{\tau}(z,y,\xi) \le C.
\end{align*}
Thus we have just proved that $\tilde{e}_{\tau}\in \Gamma_1^{\ep-1}(\R)$ and for all $\tau\in(0,1]$ and $y\in[0,\tau]$,
\begin{align*}
\mM_{1}^{\ep-1}\big(\tilde{e}_{\tau}(\cdot,y,\cdot)\big) = \sup\limits_{|m|\le 5/2} \sup\limits_{|\xi|\ge 1/2} \big\| \big(1+|\xi|\big)^{|m|-(\ep-1)} \d_\xi^m \tilde{e}_{\tau}(\cdot,y,\xi) \big\|_{W^{1,\infty}(\R)}<C,
\end{align*}
for some constant $C>0$ depending only on $\ep$ and $\|\teta\|_{H^{s+\frac{1}{2}}(\R)}$. Also, by Lemma \ref{lemma:decomP}, $a^{(1)}\in \Gamma_{3/2+\delta}^1(\R)\subset \Gamma_{1}^1(\R)$. Therefore applying Theorem \ref{thm:adjprod}\ref{item:prod}, it follows that the operator $T_{\tilde{e}_\tau}T_{a^{(1)}}-T_{\tilde{e}_{\tau}a^{(1)}}$ is of order $(\ep-1)+1-1=\ep-1$ and for all $k\in\R$,
\begin{align*}
\|T_{\tilde{e}}T_{a^{(1)}}-T_{\tilde{e}a^{(1)}}\|_{H^k(\R)\to H^{k+1-\ep}(\R)} \le C(k) \mM_1^{\ep-1}(\tilde{e}_{\tau}) \mM_1^1\big(a^{(1)}\big)<C.
\end{align*}
This proves (\ref{temp:TeTa}). To estimate $\mathrm{(V)}$, we use (\ref{temp:TeTa}) and $k_{n+1}=1-\ep+k_n$ to get
\begin{align*}
\mathrm{(V)} \vcentcolon= & \int_{y_n}^{\tau}\!\!  \|(T_{e_{\tau}} T_{a^{(1)}}- T_{e_{\tau} a^{(1)}})\mu(\cdot,y)\|_{k_{n+1}}\,\dif y \\
\le & C \int_{y_n}^{\tau}\!\! (\tau-y_0)^{\ep-1}  \|\mu(\cdot,y)\|_{k_n}\,\dif y = C \dfrac{(\tau-y_n)^{\ep}}{\ep} Q(n).
\end{align*}
Substituting the estimates for $\mathrm{(I)}$--$\mathrm{(V)}$ into (\ref{temp:mu}), we conclude that
\begin{align*}
\|\mu(\cdot,\tau)\|_{k_{n+1}} \le& C(k_{n+1},k_0) Q(0) + C(\tau-y_n)^{-(1-\ep)} Q(n)\\
&+ C \dfrac{(\tau-y_n)^\ep}{\ep} \big\{ Q(n) + \sup\limits_{y_0\le y\le 1}\|g(\cdot,y)\|_{r} \big\}.
\end{align*}
Since this inequality holds for arbitrary $\tau\in[2y_n,1]=[y_{n+1},1]$, taking supremum over $\tau\in [y_{n+1},1]$ on the above inequality, we obtain that for all $n\in\{1,2,\dotsc,N-1\}$,
\begin{align*}
Q(n+1)=&\sup\limits_{y_{n+1}\le \tau\le 1 }\|\mu(\cdot,\tau)\|_{k_{n+1}}\\
\le& C(k_{n+1},k_0) Q(0) + C y_n^{-(1-\ep)} Q(n) + \dfrac{C}{\ep} \big\{ Q(n) + \sup\limits_{y_0\le y\le 1}\|g(\cdot,y)\|_{r} \big\}.
\end{align*}
Finally, since $k_{N-1} < r\le k_N$, we have
\begin{equation}\label{temp:QN}
\sup_{y_N\le y\le1}\|\mu(\cdot,y)\|_r<Q(N)<\infty.
\end{equation}
Taking $\tau=1$ in (\ref{temp:dyODE}), then integrating (\ref{temp:dyODE}) in $y\in[y_N,1]$ and taking $H^{r+1-\ep}(\R)$ norm on both sides of the resultant equation, we have
\begin{align*}
\|\mu(\cdot,1)\|_{r+1-\ep} \le& \|(1-T_{1})\mu(\cdot,1)\|_{r+1-\ep} + \|T_{e_1} \mu (\cdot,y_N)\|_{r+1-\ep}\\ 
&+ \int_{y_N}^{1}\!\! \|T_{e_1} T_{a^{(0)}} \mu(\cdot,y)\|_{r+1-\ep}\,\dif y + \int_{y_N}^{1}\!\! \|T_{e_1} g(\cdot,y)\|_{r+1-\ep}\, \dif y\\
&+\int_{y_N}^{1}\!\!\|\{ T_{e_1}T_{a^{(1)}} - T_{e_1 a^{(1)}} \} \mu(\cdot,y)\|_{r+1-\ep}\, \dif y 
\end{align*}
Using (\ref{temp:QN}), and repeating the same estimates performed in $\mathrm{(I)}$--$\mathrm{(V)}$ but with the weight $(1-y)^{1-\ep}$ instead, we obtain that
\begin{align*}
\|\mu(\cdot,1)\|_{r+1-\ep}\le& C(k_N,k_0) Q(0) + C (1-y_N)^{-(1-\ep)} Q(N)\\
&+ C\dfrac{(1-y_N)^{\ep}}{\ep} \big\{ Q(N) + \sup\limits_{y_0\le y\le 1} \|g(\cdot,y)\|_{r} \big\}.
\end{align*}
This concludes the proof of the lemma.
\end{proof}

\begin{corollary}\label{corol:extrareg}
	Let $A(z,y,\xi)$ be the symbol given in Lemma \ref{lemma:decomP}. Then
	\begin{equation*}
		(\d_y u - T_A u)\vert_{y=1} \in H^{s+\frac{1}{2}}(\R).
	\end{equation*} 
\end{corollary}
\begin{proof}
Applying $T_{1/\alpha}$ on both sides of the paradifferential equation (\ref{Petaf}) in Lemma \ref{lemma:paraEllip}, and using the decomposition Lemma \ref{lemma:decomP}, we get 
\begin{equation*}
(\d_y - T_a)(\d_y -T_A)u = T_{1/\alpha} f - T_{1/\alpha}\mR_0 u - T_{1/\alpha} \mR_1 \d_y  - \fq[\alpha^{-1},\alpha](\d_y-T_a)(\d_y-T_A) u. 
\end{equation*}
If we set $\mu\vcentcolon= \d_y u - T_A u$, then $\mu$ solves the equation, 
\begin{gather*}
 \d_y \mu - T_{a^{(1)}} \mu = T_{a^{(0)}} \mu + g,\\
\text{where } \ g =  T_{1/\alpha} f - T_{1/\alpha}\mR_0 u - T_{1/\alpha} \mR_1 \d_y  - \fq[\alpha^{-1},\alpha](\d_y-T_a)(\d_y-T_A) u
\end{gather*} 
Using Lemma \ref{lemma:decomP}, Theorems \ref{thm:paraL2}, \ref{thm:adjprod}\ref{item:prod}, and Corollary \ref{corol:yC}, it can be verified that $g\in\mC_y^0\big((0,1];H_z^{s-\frac{1}{2}+\delta}(\R)\big)$, where $\delta\vcentcolon= \min\{\frac{1}{2},s-\frac{5}{2}\}$. Therefore one can apply Proposition \ref{prop:G-A} with $\ep=\delta$ to obtain that $\mu\vert_{y=1}= (\d_y u - T_{A} u)\vert_{y=1}\in \mC_y^0\big((0,1];H_z^{s+\frac{1}{2}}(\R)\big)$. This concludes the proof.
\end{proof}

\subsection{Paralinearization of Dirichlet-Neumann operator.}
The main theorem of this section is the paralinearization of cylindrical Dirichlet-Neumann operator $G[\eta](\psi)$: 
%\todo{this is the main result for DN operator which later implies the symmetrizer}
\begin{theorem}\label{thm:paraG}
Let $s>\tfrac{5}{2}$. Assume $(\teta,\psi)\in H^{s+\frac{1}{2}}(\R)\times H^{s}(\R)$ and $\|\teta\|_{L^{\infty}(\R)}<R$. Then the following paralinearization holds
\begin{gather*}
G[\eta](\psi) = T_{\lambda}U -T_{V}\d_z\eta + \fR_{G}(\teta,\psi), \\
\text{where } \quad U\vcentcolon= \psi - T_{\mB} \eta, \quad   V\vcentcolon=\d_z \psi - \mB \d_z \eta, \quad \mB\vcentcolon=\frac{\d_z\eta \d_z \psi - G[\eta](\psi)}{1+|\d_z\eta|^2},
\end{gather*}
where $\lambda(z,\xi)\vcentcolon=\lambda^{(1)}(\xi)+\lambda^{(0)}(z,\xi) \in \mathring{\Gamma}_{\infty}^{1}(\R) + \mathring{\Gamma}_{s-1}^{0}(\R)$ is given by:
\begin{equation}\label{lambda}
\lambda^{(1)}(\xi)= |\xi|, \ \ \text{ and } \ \  \lambda^{(0)}(z,\xi)= - \dfrac{1+2|\d_z\eta|^2 + i (\d_z\eta)^3 \sgn(\xi)}{2\eta}.
\end{equation}
In addition, there exists a positive monotone increasing function $x\mapsto C(x)$ such that
\begin{equation*}
\|\fR_{G}(\teta,\psi)\|_{H^{s+\frac{1}{2}}(\R)}\le C\big(\|\teta\|_{H^{s+\frac{1}{2}}(\R)}\big)\|\d_z\psi\|_{H^{s-1}(\R)}.
\end{equation*}
\end{theorem}
To paralinearise $G[\eta]$, we start by constructing its extended operator $\mG_{y}[\eta]$. Since the limit $\lim_{y\to 1^{-}}\d_y v $ exists for each $z\in\R$, we have that 
\begin{align*}
	G[\eta](\psi) =& \Big\{ \dfrac{1+y|\d_z\eta|^2}{\eta}\d_y v - \d_z \eta \d_z v \Big\}\Big\vert_{y=1}\\
	=& \Big\{ y(1-y)\dfrac{|\d_z\eta|^2}{\eta} \d_y v + \dfrac{1+ y^2 |\d_z \eta|^2}{\eta} \d_y v - (1-y)\d_z \eta \d_z v - y\d_z \eta \d_z v \Big\}\Big\vert_{y=1} \\
	=& \Big\{ \dfrac{1+ y^2 |\d_z \eta|^2}{\eta} \d_y v  - y \d_z \eta \d_z v \Big\}\Big\vert_{y=1}.
\end{align*}
In light of this, we define the operator $\mG_y[\eta](v)$ as 
\begin{equation}\label{extG}
	\mG_y[\eta](v) \vcentcolon= \dfrac{1+ y^2 |\d_z \eta|^2}{\eta} \d_y v - y \d_z \eta \d_z v = \eta\alpha \d_y v - \d_z \rho \d_z v.
\end{equation}
Then by the previous equation, one sees that $G[\eta](\psi)=\mG_y[\eta](v)\vert_{y=1}$.

\begin{lemma}\label{lemma:paraExtG}
Let $\mG_y[\eta]$ be defined in (\ref{extG}). Denote the quantities $\omega\vcentcolon=\d_z v - \fb\d_z \rho$, $\fb\vcentcolon= \frac{\d_y v}{\eta}$, $\rho\vcentcolon= y\eta$, $u\vcentcolon= v - T_{\fb}\rho$, and $\alpha=\frac{1+y^2|\d_z\eta|^2}{\eta^2}$. Then
\begin{align*}
\mG_y[\eta](v) =& T_{\eta\alpha} \d_y u - T_{\d_z \rho} \d_z u - T_{\d_z\omega} \rho -T_{\omega}\d_z \rho - T_{\fb/\eta} \teta + \fR_1(v,\teta),\\
\text{where } \ \fR_1(v,\eta)\vcentcolon=&\fp(\eta\alpha,\d_y v) - \fp(\d_z\rho,\d_z v) + y\fq[\eta\alpha,\d_y \fb](\teta) + y\fq[\eta\alpha,\fb](\teta)\nonumber\\ &- y\fq[\d_z\rho,\d_z \fb](\teta)-\fq[\d_z\rho,\fb](\d_z \rho) - \fq[\eta\fb,\eta^{-2}](\teta) \nonumber\\ & + T_{\eta\fb} \fC[F_R](\teta)+ 2 \fq[\fb,\d_z\rho](\d_z\rho) + T_{\fb}\fp(\d_z\rho,\d_z\rho) \nonumber\\ & - \fq\big[|\d_z\rho|^2\d_y v, \eta^{-2}\big](\teta) + T_{|\d_z\rho|^2\d_y v} \fC[F_R](\teta),
\end{align*}
with $F_R(x)\vcentcolon=\frac{1}{x+R}$, and the notations $\fp(\cdot,\cdot)$, $\fq[\cdot,\cdot](\cdot)$, $\fC[F_R](\cdot)$ are defined in (\ref{brackets}).  
\end{lemma}
\begin{proof}
By the linearization rule (\ref{bonyProd}), we get
\begin{align*}
\mG_y[\eta](v) = T_{\eta\alpha} \d_y v + T_{\d_y v} (\eta\alpha) + \fp(\eta\alpha,\d_y v) - T_{\d_z \rho} \d_z v - T_{\d_z v} \d_z \rho - \fp(\d_z\rho,\d_z v).
\end{align*}
Rewriting the above with $v=u+T_{\fb}\rho$, we get
\begin{align}\label{temp:1stG}
\mG_y[\eta](v) =&  T_{\eta\alpha} \d_y u + T_{\eta\alpha} \d_y T_{\fb}\rho + T_{\d_y v} (\eta\alpha) + \fp(\eta\alpha,\d_y v)\\
&- T_{\d_z \rho} \d_z u - T_{\d_z\rho} \d_z T_{\fb}\rho - T_{\d_z v} \d_z \rho - \fp(\d_z\rho,\d_z v)\nonumber\\
=&  T_{\eta\alpha} \d_y T_{\fb}\rho - T_{\d_z\rho} \d_z T_{\fb}\rho + T_{\d_y v} (\eta\alpha)  - T_{\d_z v} \d_z \rho  + T_{\eta\alpha} \d_y u - T_{\d_z \rho} \d_z u\nonumber\\
&+ \fp(\eta\alpha,\d_y v) - \fp(\d_z\rho,\d_z v)\nonumber\\
=\vcentcolon& \sum_{i=1}^4 I^{(i)} + T_{\eta\alpha} \d_y u - T_{\d_z \rho} \d_z u + \fp(\eta\alpha,\d_y v) - \fp(\d_z\rho,\d_z v).\nonumber
\end{align}
Using the Leibniz rule for paraproduct, Proposition \ref{prop:paraLeib} and linearization rule (\ref{bonyCom}), the term $I^{(1)}$ is rewritten as 
\begin{align*}
I^{(1)}\vcentcolon=& T_{\eta\alpha} \d_y T_{\fb} \rho = T_{\eta\alpha} \big\{ T_{\d_y \fb} \rho + T_{\fb} \d_y \rho \big\}\\ 
=& T_{\eta\alpha \d_y \fb} \rho + \fq[\eta\alpha,\d_y \fb](\rho) + T_{\eta\alpha \fb} \eta + \fq[\eta\alpha,\fb](\rho). 
\end{align*}
By Proposition \ref{prop:b}, we have $\eta \alpha \d_y \fb = \d_z\rho \d_z \fb - \d_z \omega  - \tfrac{\fb}{\rho}$. Moreover by the definition of $\fb\vcentcolon= \frac{\d_y v}{\eta}$, we have
\begin{align*}
I^{(1)} = -T_{\d_z\omega} \rho + T_{\d_z\rho \d_z \fb} \rho - T_{\fb/\rho} \rho + T_{\alpha \d_y v} \eta + \fq[\eta\alpha,\d_y \fb](\rho) + \fq[\eta\alpha,\fb](\rho).  
\end{align*} 
By the definition of $\alpha$ and $\fb$, we have
\begin{align*}
\alpha \d_y v = \dfrac{1+y^2|\d_z \eta|^2}{\eta^2} \d_y v = \dfrac{1}{\eta} \dfrac{ \d_y v}{\eta} + y^2 \frac{|\d_z\eta|^2}{\eta} \dfrac{\d_y v}{\eta} = \dfrac{y\fb}{\rho} + \dfrac{|\d_z \rho|^2}{\eta} \fb. 
\end{align*}
Substituting this into the expression for $I^{(1)}$, and denoting $\teta\vcentcolon= \eta-R$, we have
\begin{align*}
I^{(1)}%=& - T_{\d_z\omega} \rho + T_{\d_z\rho \d_z \fb }\rho - T_{\fb/\rho} \rho + T_{\fb/\rho} \rho + y^2 T_{|\d_z\eta|^2\fb /\eta} \teta + y\fq[\eta\alpha,\d_y \fb](\teta) + y\fq[\eta\alpha,\fb](\teta)\\
=- T_{\d_z\omega} \rho + T_{\d_z\rho \d_z \fb }\rho + T_{|\d_z\rho|^2\fb /\eta} \teta + y\fq[\eta\alpha,\d_y \fb](\teta) + y\fq[\eta\alpha,\fb](\teta),
\end{align*}
	where we also used Remark \ref{rem:paraC}\ref{item:paraC1}. Next, the term $I^{(2)}$ is rewritten using the linearization rule (\ref{bonyCom}) as
	\begin{align*}
		I^{(2)}\vcentcolon=& -T_{\d_z \rho} \d_z T_{\fb} \rho = - T_{\d_z \rho} \big\{ T_{\d_z \fb} \rho + T_{\fb} \d_z \rho \big\}\nonumber\\
		%=& - T_{\d_z\rho\d_z \fb} \rho - \fq[\d_z\rho,\d_z \fb](\rho) - T_{\fb\d_z\rho} \d_z \rho -\fq[\d_z\rho,\fb](\d_z \rho)\nonumber\\
		=& - T_{\d_z\rho\d_z \fb} \rho - T_{\fb\d_z\rho} \d_z \rho - \fq[\d_z\rho,\d_z \fb](\rho)-\fq[\d_z\rho,\fb](\d_z \rho).
	\end{align*}
	Adding $I^{(1)}$ and $I^{(2)}$, we have
	\begin{align}\label{temp:linear12}
		I^{(1)}+I^{(2)} 
		= & - T_{\d_z\omega} \rho- T_{\fb \d_z \rho} \d_z \rho + T_{|\d_z\rho|^2\fb/\eta} \teta + y\fq[\eta\alpha,\d_y \fb](\teta) \\
		& + y\fq[\eta\alpha,\fb](\teta) - \fq[\d_z\rho,\d_z \fb](\rho)-\fq[\d_z\rho,\fb](\d_z \rho). \nonumber
	\end{align}
	By the definition of $\alpha$, $\fb$, and linearization rules (\ref{bonyProd})--(\ref{bonyCom}), $I^{(3)}$ is rewritten as 
	\begin{align}\label{temp:lI3}
		I^{(3)} \vcentcolon=& T_{\d_y v} (\eta \alpha) = T_{\d_y v} \big( \dfrac{1}{\eta} + \dfrac{y^2|\d_z \eta|^2}{\eta} \big)\\
		=& T_{\d_y v}\eta^{-1} + T_{\d_y v} \big\{ T_{1/\eta}|\d_z\rho|^2 + T_{|\d_z\rho|^2}\eta^{-1} + \fp(\eta^{-1},|\d_z\rho|^2)  \big\}\nonumber\\
		=& T_{\d_y v} \eta^{-1} + T_{\d_y v/ \eta} |\d_z\rho|^2 + \fq[\d_y v,\eta^{-1}](|\d_z\rho|^2)\nonumber\\ &+ T_{|\d_z\rho|^2\d_y v}\eta^{-1} + \fq[\d_y v,|\d_z\rho|^2](\eta^{-1}) + T_{\d_y v} \fp(\eta^{-1},|\d_z\rho|^2)\nonumber\\
		=& T_{\d_y v}\eta^{-1} + T_{\fb} |\d_z\rho|^2 +  T_{|\d_z\rho|^2 \d_y v}\eta^{-1} + \fq[\d_y v,\eta^{-1}](|\d_z\rho|^2)\nonumber\\ & + \fq[\d_y v,|\d_z\rho|^2](\eta^{-1}) + T_{\d_y v} \fp(\eta^{-1},|\d_z\rho|^2)\nonumber\\
		=\vcentcolon& \sum_{i=1}^{3} I_i^{(3)} +\fq[\d_y v,\eta^{-1}](|\d_z\rho|^2) + \fq[\d_y v,|\d_z\rho|^2](\eta^{-1}) + T_{\d_y v} \fp(\eta^{-1},|\d_z\rho|^2).\nonumber 
	\end{align}
	Set $F_{R}(x)=\frac{1}{x+R}-\frac{1}{R}$, then $F_R(0)=0$, $F_{R}(\teta)=\eta^{-1}$ and $F_{R}^{\prime}(\teta)= -\eta^{-2}$ where $\teta\vcentcolon= \eta-R$. Thus it follows from linearization rules (\ref{bonyCom})--(\ref{bonyChain}) that
	\begin{align*}
		I^{(3)}_1 \vcentcolon=& T_{\d_y v} \eta^{-1}
		= T_{\eta \fb} \eta^{-1} = T_{\eta \fb} \big\{ - T_{1/\eta^2} \teta + \fC[F_R](\teta) \big\}\\
		=& - T_{\fb/\eta} \teta - \fq[\eta\fb,\eta^{-2}](\teta) + T_{\eta\fb} \fC[F_R](\teta).
	\end{align*}
	By the rule (\ref{bonyProd}), we also have
	\begin{align*}
		I^{(3)}_2\vcentcolon=& T_{\fb}|\d_z \rho|^2 = T_{\fb}\big\{ 2 T_{\d_z \rho} \d_z \rho + \fp(\d_z\rho,\d_z\rho) \big\}\\
		=& 2 T_{\fb \d_z \rho} \d_z \rho + 2 \fq[\fb,\d_z\rho](\d_z\rho) + T_{\fb}\fp(\d_z\rho,\d_z\rho).
	\end{align*}
	By (\ref{bonyProd})--(\ref{bonyChain}), we also have
	\begin{align*}
		I_3^{(3)} \vcentcolon=& T_{|\d_z\rho|^2 \d_y v}\eta^{-1} = T_{|\d_z\rho|^2\d_y v}\big\{ - T_{1/\eta^2} \teta + \fC[F_R](\teta) \big\}\\
		=& -T_{|\d_z\rho|^2 \fb/\eta} \teta - \fq\big[|\d_z\rho|^2\d_y v, \eta^{-2}\big](\teta) + T_{|\d_z\rho|^2\d_y v} \fC[F_R](\teta).  
	\end{align*}
	Putting $I_{1}^{(3)}$--$I_{3}^{(3)}$ in (\ref{temp:lI3}), we obtain that
	\begin{align*}
		I^{(3)} =& - T_{\fb/\eta} \teta + 2 T_{\fb \d_z \rho} \d_z \rho -T_{|\d_z\rho|^2 \fb/\eta} \teta \nonumber\\
		&- \fq[\eta\fb,\eta^{-2}](\teta) + T_{\eta\fb} \fC[F_R](\teta)+ 2 \fq[\fb,\d_z\rho](\d_z\rho) \\
		&+ T_{\fb}\fp(\d_z\rho,\d_z\rho) - \fq\big[|\d_z\rho|^2\d_y v, \eta^{-2}\big](\teta) + T_{|\d_z\rho|^2\d_y v} \fC[F_R](\teta).
	\end{align*}
	Adding $I^{(3)}$ and $I^{(4)}$, we obtain
	\begin{align}\label{temp:linear34}
		I^{(3)}+I^{(4)} =& -T_{\omega}\d_z \rho - T_{\fb/\eta} \teta +  T_{\fb \d_z \rho} \d_z \rho -T_{|\d_z\rho|^2 \fb/\eta} \teta \\
		&- \fq[\eta\fb,\eta^{-2}](\teta) + T_{\eta\fb} \fC[F_R](\teta)+ 2 \fq[\fb,\d_z\rho](\d_z\rho) \nonumber\\ &+ T_{\fb}\fp(\d_z\rho,\d_z\rho) - \fq\big[|\d_z\rho|^2\d_y v, \eta^{-2}\big](\teta) + T_{|\d_z\rho|^2\d_y v} \fC[F_R](\teta).\nonumber
	\end{align}
	Summing $I^{(1)}$--$I^{(4)}$ together, it follows that
	\begin{align*}
		\sum_{i=1}^4 I^{(i)} %=& - T_{\d_z\omega} \rho- T_{\fb \d_z \rho} \d_z \rho + T_{|\d_z\rho|^2\fb/\eta} \teta\nonumber\\ &-T_{\omega}\d_z \rho - T_{\fb/\eta} \teta +  T_{\fb \d_z \rho} \d_z \rho -T_{|\d_z\rho|^2 \fb/\eta} \teta \nonumber\\ &+ y\fq[\eta\alpha,\d_y \fb](\teta) + y\fq[\eta\alpha,\fb](\teta) - \fq[\d_z\rho,\d_z \fb](\rho)-\fq[\d_z\rho,\fb](\d_z \rho)\\ &- \fq[\eta\fb,\eta^{-2}](\teta) + T_{\eta\fb} \fC[F_R](\teta)+ 2 \fq[\fb,\d_z\rho](\d_z\rho) \nonumber\\ &+ T_{\fb}\fp(\d_z\rho,\d_z\rho) - \fq\big[|\d_z\rho|^2\d_y v, \eta^{-2}\big](\teta) + T_{|\d_z\rho|^2\d_y v} \fC[F_R](\teta)\nonumber\\
		=& - T_{\d_z\omega} \rho -T_{\omega}\d_z \rho - T_{\fb/\eta} \teta + y\fq[\eta\alpha,\d_y \fb](\teta) + y\fq[\eta\alpha,\fb](\teta)\nonumber\\ &- y\fq[\d_z\rho,\d_z \fb](\teta)-\fq[\d_z\rho,\fb](\d_z \rho) - \fq[\eta\fb,\eta^{-2}](\teta) \nonumber\\ & + T_{\eta\fb} \fC[F_R](\teta)+ 2 \fq[\fb,\d_z\rho](\d_z\rho) + T_{\fb}\fp(\d_z\rho,\d_z\rho) \nonumber\\ & - \fq\big[|\d_z\rho|^2\d_y v, \eta^{-2}\big](\teta) + T_{|\d_z\rho|^2\d_y v} \fC[F_R](\teta).
	\end{align*}
	Substituting this into (\ref{temp:1stG}), we get
	\begin{align*}
		\mG_{y}[\eta] =& T_{\eta\alpha} \d_y u - T_{\d_z \rho} \d_z u - T_{\d_z\omega} \rho -T_{\omega}\d_z \rho - T_{\fb/\eta} \teta\\
		&+\fp(\eta\alpha,\d_y v) - \fp(\d_z\rho,\d_z v) + y\fq[\eta\alpha,\d_y \fb](\teta) + y\fq[\eta\alpha,\fb](\teta)\nonumber\\ &- y\fq[\d_z\rho,\d_z \fb](\teta)-\fq[\d_z\rho,\fb](\d_z \rho) - \fq[\eta\fb,\eta^{-2}](\teta) \nonumber\\ & + T_{\eta\fb} \fC[F_R](\teta)+ 2 \fq[\fb,\d_z\rho](\d_z\rho) + T_{\fb}\fp(\d_z\rho,\d_z\rho) \nonumber\\ & - \fq\big[|\d_z\rho|^2\d_y v, \eta^{-2}\big](\teta) + T_{|\d_z\rho|^2\d_y v} \fC[F_R](\teta).
	\end{align*}
	This concludes the proof of this proposition.
\end{proof}
With Proposition \ref{prop:G-A} and Lemma \ref{lemma:paraExtG}, we are ready to paralinearise $G[\eta](\psi)$.  
\begin{proof}[Proof of Theorem \ref{thm:paraG}]
From Lemma \ref{lemma:paraExtG}, we have
\begin{align}\label{temp:mG}
\mG_y[\eta](v) =& T_{\eta\alpha} \d_y u - T_{\d_z \rho} \d_z u - T_{\d_z\omega} \rho -T_{\omega}\d_z \rho - T_{\fb/\eta} \teta + \fR_1(v,\teta)\\
=\vcentcolon & T_{\eta\alpha} \d_y u - T_{\d_z \rho} \d_z u - T_{\omega} \d_z\rho + \fR_1(v,\teta) + \fR_2(v,\teta).\nonumber
\end{align}
Using Theorems \ref{thm:adjprod}\ref{item:prod}, \ref{thm:bony}, it can be verified that
\begin{equation}\label{temp:fR1}
\fR_1(v,\teta) \in \mC_{y}^0\big((0,1];H_z^{2s-2}(\R)\big)\subset \mC_{y}^0\big((0,1];H_z^{s+\frac{1}{2}}(\R)\big), \quad \text{since } \ s>\frac{5}{2}.
\end{equation}
Moreover since $\d_z\omega = \d_z^2 v - \d_z\fb \d_z \rho + \fb \d_z^2 \rho \in \mC_{y}^0\big((0,1];H_{z}^{s-2}(\R)\big)$ with $s-2>\tfrac{1}{2}$, Theorem \ref{thm:paraL2} and Sobolev embedding theorem imply that $T_{\d_z\omega}$ is an operator of order $0$. Hence it follows that
\begin{equation}\label{temp:fR2}
\fR_2(v,\teta)\vcentcolon= - T_{\d_z\omega} \rho - T_{\fb/\eta} \teta \in \mC_{y}^0\big((0,1];H_z^{s+\frac{1}{2}}(\R)\big).
\end{equation}
The first two terms in (\ref{temp:mG}) can be rewritten as
\begin{align}\label{temp:first2}
&T_{\eta\alpha} \d_y u - T_{\d_z \rho} \d_z u = T_{\eta\alpha} \big( \d_y u - T_{A} u \big) + T_{\eta\alpha} T_{A} u - T_{\d_z \rho } \d_z u \\
%=& T_{(\eta\alpha) \sharp A} u - T_{\d_z\rho} \d_z u + T_{\eta\alpha} (\d_y u - T_A u) + \fq[\eta\alpha,A](u)\\
=& T_{\Theta} u + T_{\eta\alpha} (\d_y u - T_A u) + \fq[\eta\alpha,A](u) = \vcentcolon T_{\Theta} u + \fR_3(v,\teta),\nonumber
\end{align}
where $\Theta(z,y,\xi)$ is a symbol, given by:
\begin{align*}
\Theta \vcentcolon=& (\eta\alpha)\sharp A - i\xi \d_z\rho %= \sum_{|k|\le 1/2 + \delta} \dfrac{1}{i^{|k|}k!} \d_\xi^k (\eta \alpha) \d_z^kA - i\xi \d_z\rho 
= \eta \alpha A - i\xi \d_z\rho = \eta\alpha \big\{ A^{(1)} + A^{(0)} \big\} - i\xi \d_z\rho\\
=& (iy\xi\d_z\eta+|\xi|) - i\xi\d_z\rho + \eta\alpha A^{(0)} = |\xi| + \eta \alpha A^{(0)} =\vcentcolon \Theta^{(1)}+ \Theta^{(0)}.
\end{align*}
We define the symbol $\lambda(z,\xi)\vcentcolon=\lambda^{(1)}(z,\xi)+\lambda^{(0)}(z,\xi)$ as 
\begin{align*}
\lambda^{(1)}(z,\xi) \vcentcolon= \Theta(z,y,\xi)\vert_{y=1} = |\xi|, \quad \lambda^{(0)}(z,\xi) \vcentcolon= \Theta(z,y,\xi)\vert_{y=1} = \eta\alpha A^{(0)}(z,y,\xi)\vert_{y=1}.
\end{align*}
Then from Lemma \ref{lemma:decomP}, (\ref{A0}), we have that
\begin{align*}
\lambda^{(0)}(z,\xi) = \eta \alpha A^{(0)}(z,y,\xi) \vert_{y=1} %= - \dfrac{\eta\alpha}{2y} - \dfrac{\eta \beta^2}{8y}\big(1+iy\d_z\eta\,\sgn(\xi)\big)\big\vert_{y=1}
= - \dfrac{1+2|\d_z\eta|^2 + i (\d_z\eta)^3 \sgn(\xi)}{2\eta}. 
\end{align*}
From these expressions, it can be shown that $\lambda^{(1)}\in \mathring{\Gamma}_{\infty}^{1}(\R)$ and $\lambda^{(0)}\in\mathring{\Gamma}_{s-1}^{0}(\R)$. Moreover, by the definition (\ref{bonyCom}), we write
\begin{equation*}
q[\eta\alpha,A](u) = q\big[\eta\alpha, A^{(1)}\big](u) + q\big[\eta\alpha, A^{(0)}\big](u).
\end{equation*}
From Lemma \ref{lemma:decomP}, we have $A^{(1)}\in \mathring{\Gamma}_{s-1}^{1}(\R)$ and $A^{(0)}\in\mathring{\Gamma}_{s-1}^{0}(\R)$. Since $u = v+T_{\fb}\rho \in \mC_{y}^0\big((0,1];H_z^{s}(\R)\big)$ from Corollary \ref{corol:yC}, it follows by Theorem \ref{thm:adjprod}\ref{item:prod} that
\begin{align*}
q\big[\eta\alpha, A^{(1)}\big](u) & \in \mC_{y}^0\big((0,1];H_{z}^{2s-2}(\R)\big) \subset \mC_{y}^0\big((0,1];H_{z}^{s+\frac{1}{2}}(\R)\big),\\
q\big[\eta\alpha, A^{(0)}\big](u) & \in \mC_{y}^0\big((0,1];H_{z}^{2s-1}(\R)\big) \subset \mC_{y}^0\big((0,1];H_{z}^{s+\frac{1}{2}}(\R)\big).
\end{align*}
Furthermore, from Proposition \ref{prop:G-A} and Theorem \ref{thm:paraL2}, we also have
\begin{equation*}
T_{\eta\alpha} (\d_y u - T_A u) \vert_{y=1} \in H^{s+\frac{1}{2}}(\R).
\end{equation*}
Thus these regularities implies that
\begin{equation}\label{temp:fR3}
\fR_3(v,\teta)\vert_{y=1} = T_{\eta\alpha}(\d_y u-T_A u)\vert_{y=1} + q[\eta\alpha,A](u)\vert_{y=1} \in H^{s+\frac{1}{2}}(\R).
\end{equation}
Taking the limit $y\to 1^{-}$ in the equation (\ref{temp:first2}), and using the fact that $u\vert_{y=1}=(v-T_{\fb}\rho)\vert_{y=1} = (\psi-T_{\mB}\eta)=U$, we obtain
\begin{equation}\label{temp:first2limit}
\big(T_{\eta\alpha}\d_y u - T_{\d_z\rho} \d_z u \big)\big\vert_{y=1} = T_{\lambda} U + \fR_3(v,\teta)\vert_{y=1}.
\end{equation}
Taking limit $y\to 1^{-}$ in (\ref{temp:mG}), using the regularities (\ref{temp:fR1}), (\ref{temp:fR2}), (\ref{temp:fR3}), the limit (\ref{temp:first2limit}), and the fact that $\omega\vert_{y=1}=(\d_z v - \fb \d_z \rho)_{y=1}=(\d_z \psi - \mB \d_z \eta) = V$, we obtain 
\begin{align*}
G[\eta](\psi)=&\mG_y[\eta](v)\vert_{y=1} = (T_{\eta\alpha}\d_y u - T_{\d_z\rho} \d_z u)\vert_{y=1} - T_{\omega} \d_z\rho \vert_{y=1} + \{\fR_1 + \fR_2\}\vert_{y=1}\\
 =& T_{\lambda} U - T_{V} \d_z\eta + \fR_G(\teta,\psi),\\
\text{where } \quad \ &  \fR_G(\teta,\psi)\vcentcolon= \{\fR_1+\fR_2+\fR_3\}(v,\teta)\vert_{y=1} \in H^{s+\frac{1}{2}}(\R).
\end{align*}
This concludes the proof of Theorem \ref{thm:paraG}.
\end{proof}

\subsection{First Order Paralinearization}
In this section, we derive a simpler version of paralinearization for DN operator, which will be used for the proof of uniqueness.
\begin{lemma}[First Order Paralinearization]\label{lemma:1stOP}
Let $s>\tfrac{5}{2}$ and $\sigma\in[1,s-1]\backslash\{\tfrac{3}{2},\tfrac{5}{2}\}$. Assume $(\teta,\psi)\in H^{s+\frac{1}{2}}(\R)\times H^{\sigma}(\R)$ and $\|\teta\|_{L^{\infty}(\R)}<R$. Then
\begin{equation*}
	G[\eta](\psi) = T_{\lambda^{(1)}}\psi  + \fR_{\sigma}(\teta,\psi) \quad \text{ where } \ \lambda^{1}(\xi)=|\xi|,
\end{equation*}
and there exists a positive monotone increasing function $x\mapsto C(x)$ such that
\begin{equation*}
	\|\fR_{\sigma}(\teta,\psi)\|_{H^{\sigma}(\R)}\le C\big(\|\teta\|_{H^{s+\frac{1}{2}}}\big)\|\psi\|_{H^{\sigma}(\R)}.
\end{equation*}
\end{lemma}
In order to prove Lemma \ref{lemma:1stOP}, we first obtain the following simpler version of para-linearization for the equation $\mL_{\eta}v=0$.
\begin{proposition}\label{prop:1stOP}
Suppose assumptions of Lemma \ref{lemma:1stOP} holds, and $v$ solves \textnormal{(\ref{Leta2})} with the boundary data $\psi\in H^{\sigma}(\R)$. Then
\begin{equation*}
P_{\eta} v \in \mC_{y}^0\big((0,1];H_{z}^{\sigma-\frac{1}{2}}(\R)\big), \quad \text{ where } \ P_{\eta}\vcentcolon= T_{\alpha} \d_y^2 +\d_z^2 + T_{\beta} \d_y \d_z - T_{\gamma} \d_y.
\end{equation*}
\end{proposition}
\begin{proof}
Set $\psi\in H^{\sigma}(\R)$ as the boundary data for the problem (\ref{reform}). Then it follows from Corollary \ref{corol:v0} and \ref{corol:yC} that
\begin{subequations}\label{sigmaReg}
\begin{align}
&v \in \mC_{y}^0\big((0,1];H_{z}^{\sigma}(\R)\big), && \d_y v \in \mC_{y}^0\big((0,1];H_{z}^{\sigma-1}(\R)\big),\\
&\d_y^2 v \in \mC_{y}^0\big((0,1];H_{z}^{\sigma-2}(\R)\big), && \d_y^3 v \in \mC_{y}^0\big((0,1];H_{z}^{\sigma-3}(\R)\big).
\end{align}
\end{subequations} 
Repeating the same paralinearization in (\ref{Para-albega}), we obtain that
\begin{gather}
0=\mL_{\eta} v = P_{\eta} v + \big\{ T_{\d_y^2 v} \alpha + T_{\d_y\d_z v}\beta - T_{\d_y v} \gamma \big\} - \fr_1,\label{Peta1st}\\
\text{where } \ \fr_1 \vcentcolon= - \fp(\alpha,\d_y^2 v) - \fp(\beta,\d_y\d_z v) + \fp(\gamma,\d_y v).\nonumber
\end{gather}
Since $s>\tfrac{5}{2}$ and $\sigma\in [1,s-1]\backslash\{\tfrac{3}{2},\tfrac{5}{2}\}$, we apply Bony's Theorem \ref{thm:bony}\ref{item:err1} with the regularities (\ref{albega}) and (\ref{sigmaReg}) to obtain that
\begin{equation*}
	\fr_1 \in \mC_y^0\big((0,1]; H_{z}^{s+\sigma-3}(\R)\big) \subset \mC_y^0\big((0,1]; H_{z}^{\sigma-\frac{1}{2}}(\R)\big)
\end{equation*}
Next, for the rest of the proof, we will only present the case $1\le\sigma<\frac{3}{2}$. The proof for $\sigma>\tfrac{3}{2}$ can be obtained similarly except applying Theorem \ref{thm:paraL2} instead of Theorem \ref{thm:paraPEst}. By the regularities (\ref{sigmaReg}), it follows from Theorem \ref{thm:paraPEst} that for all $k\in\R$
\begin{equation*}
\|T_{\d_y^2 v}\|_{H^k\to H^{k-\frac{5}{2}+\sigma}}<\infty, \quad \|T_{\d_y\d_z v}\|_{H^k\to H^{k-\frac{5}{2}+\sigma}}<\infty, \quad \|T_{\d_y v}\|_{H^k\to H^{k-\frac{3}{2}+\sigma}} <\infty. 
\end{equation*}
Therefore it follows from (\ref{albega}) and $s>\frac{5}{2}$ that
\begin{equation*}
	T_{\d_y^2 v} \alpha, \quad T_{\d_z\d_y v} \beta, \quad T_{\d_y v}\gamma \in \mC_{y}^0\big((0,1];H_{z}^{s+\sigma-3}(\R)\big) \subset \mC_{y}^0\big((0,1];H_{z}^{\sigma-\frac{1}{2}}(\R)\big).
\end{equation*}
Substituting these regularities statement in (\ref{Peta1st}), we get
\begin{equation*}
P_{\eta} v = \fr_1 - \big\{ T_{\d_y^2 v} \alpha + T_{\d_y\d_z v}\beta - T_{\d_y v} \gamma \big\} \in \mC_{y}^0\big((0,1];H_{z}^{\sigma-\frac{1}{2}}(\R)\big).
\end{equation*}
This concludes the proof.
\end{proof}
\begin{corollary}\label{corol:1stOP}
Suppose assumptions of Lemma \ref{lemma:1stOP} holds, and $v$ solves \textnormal{(\ref{Leta2})} with the boundary data $\psi\in H^{\sigma}(\R)$. Let $A(z,y,\xi)$ be the symbol given in Lemma \ref{lemma:decomP}, then
\begin{equation*}
(\d_y v - T_{A} v )\vert_{y=1} \in H^{\sigma}(\R).
\end{equation*}
\end{corollary}
\begin{proof}
By the operator decomposition of $P_{\eta}$, Lemma \ref{lemma:decomP}, it follows that
\begin{gather*}
(\d_y-T_{a})(\d_y-T_{A})v = g\\
\text{where } \ g \vcentcolon= T_{\alpha^{-1}}\big\{P_\eta v  - \mR_0 v - \mR_1 \d_y v \big\} - \fq[\alpha^{-1},\alpha](\d_y-T_a)(\d_y-T_A)v  \\
\end{gather*}
By regularities (\ref{albega}) and (\ref{sigmaReg}), Theorem \ref{thm:paraL2}, and Proposition \ref{prop:1stOP},
\begin{equation*}
	T_{\alpha^{-1}} P_{\eta} v \in \mC_{y}^0\big((0,1];H_{z}^{\sigma-\frac{1}{2}}(\R)\big).
\end{equation*} 
Moreover, by Theorems \ref{thm:paraL2}, Lemma \ref{lemma:decomP}, and (\ref{sigmaReg}), it can also be verified that
\begin{equation*}
	T_{\alpha^{-1}}\big\{ \mR_0 v +  \mR_1 \d_y v \big\} \in \mC_{y}^0\big((0,1];H_{z}^{\sigma-\frac{1}{2}+\delta}(\R)\big),
\end{equation*} 
where $\delta\vcentcolon=\min\{ \frac{1}{2}, s-\frac{5}{2} \}$. Furthermore, by Theorem \ref{thm:adjprod}\ref{item:prod} and (\ref{sigmaReg}), we also have
\begin{equation*}
\fq[\alpha^{-1},\alpha](\d_y-T_a)(\d_y-T_A)v \in \mC_{y}^0\big((0,1];H_z^{s+\sigma-3}(\R)\big) \subseteq \mC_{y}^0\big((0,1];H_z^{\sigma-\frac{1}{2}}(\R)\big).
\end{equation*}
Therefore, we have just shown that $g\in \mC_{y}^0\big((0,1];H_z^{\sigma-\frac{1}{2}}(\R)\big)$. If we define $\mu\vcentcolon= \d_y v - T_{A} v$, then $\d_y \mu - T_{a^{(1)}} \mu = T_{a^{(0)}} \mu + g$. Therefore we can apply Proposition \ref{prop:G-A} with $\ep=\frac{1}{2}$ to conclude that $\mu\vert_{y=1}=(\d_y v - T_{A} v)\vert_{y=1}\in H^{\sigma}(\R)$. 
\end{proof}
\begin{proof}[Proof of Lemma \ref{lemma:1stOP}] For the entirety of this proof, we only present the case $1\le\sigma<\frac{3}{2}$. Proof for $\sigma>\tfrac{3}{2}$ can be obtained similarly except applying Theorem \ref{thm:paraL2} instead of Theorem \ref{thm:paraPEst}. Let $\mG_{y}[\eta](v)$ be the extended operator defined in (\ref{extG}). Then by Lemma \ref{lemma:paraExtG} and $u\vcentcolon= v-T_{\fb} \rho$, we have
\begin{align}\label{1stmG}
\mG_{y}[\eta](v) %=& T_{\eta\alpha} \d_y v -T_{\eta\alpha}\d_y T_{\fb}\rho - T_{\d_z\rho} \d_z v + T_{\d_z\rho} \d_z T_{\fb} \rho - T_{\d_z \omega} \rho - T_{\omega} \d_z \rho - T_{\fg/\eta} \teta + \fR_1(v,\teta)\nonumber\\
=&  T_{\eta\alpha} \d_y v - T_{\d_z\rho} \d_z v -T_{\eta\alpha}\d_y T_{\fb}\rho + T_{\d_z\rho} \d_z T_{\fb} \rho\\
&- T_{\d_z \omega} \rho - T_{\omega} \d_z \rho - T_{\fb/\eta} \teta + \fR_1(v,\teta),\nonumber
\end{align}
where $\rho=y\eta$, $\fb=\frac{1}{\eta}\d_y v$, $\omega=\d_z v - \fb \d_z\eta$. First, according to the definition of $\fR_1(v,\teta)$ stated in Lemma \ref{lemma:paraExtG}, one can verify using (\ref{albega}) and (\ref{sigmaReg}) that,
\begin{equation*}
	\fR_1(v,\teta) \in \mC_{y}^0\big((0,1];H_z^{\sigma}(\R)\big).
\end{equation*}
Next, from (\ref{sigmaReg}) and the fact that $\teta\in H^{s+\frac{1}{2}}(\R)$, we have
\begin{align*}
\fb, \ \omega \in \mC_{y}^0\big((0,1];H_z^{\sigma-1}(\R)\big), \quad \d_z\fb, \ \d_y\fb, \ \d_z\omega \in \mC_{y}^0\big((0,1];H_z^{\sigma-2}(\R)\big)
\end{align*}
Thus by Theorem \ref{thm:paraPEst}, we have that for all $k\in\R$,
\begin{gather*}
	\|T_{\fb}\|_{H^k\to H^{k-\frac{3}{2}+\sigma}}<\infty, \quad \|T_{\omega}\|_{H^k\to H^{k-\frac{3}{2}+\sigma}}<\infty,\\ 
	\|T_{\d_z\fb}\|_{H^k\to H^{k-\frac{5}{2}+\sigma}} <\infty, \quad \|T_{\d_y\fb}\|_{H^k\to H^{k-\frac{5}{2}+\sigma}} <\infty, \quad \|T_{\d_z\omega}\|_{H^k\to H^{k-\frac{5}{2}+\sigma}} <\infty.
\end{gather*}
Applying these and Theorem \ref{thm:paraL2}, we obtain using $s>\frac{5}{2}$ that
\begin{align*}
-T_{\eta\alpha}\d_y T_{\fb}\rho + T_{\d_z\rho} \d_z T_{\fb} \rho - T_{\d_z \omega} \rho - T_{\omega} \d_z \rho - T_{\fb/\eta} \teta \in %\mC_{y}^0\big((0,1];H_{z}^{s+\sigma-2}\big)\subset
\mC_{y}^0\big((0,1];H_{z}^{\sigma}(\R)\big).
\end{align*}
Substituting these into (\ref{1stmG}), it follows that
\begin{equation*}
\mG_{y}[\eta](v) = T_{\eta\alpha} \d_y v - T_{\d_z\rho} \d_z v + \fR_4(v,\teta), \quad \text{for some } \ \fR_4(v,\teta)\in \mC_{y}^0\big((0,1];H_{z}^{\sigma}(\R)\big).
\end{equation*} 
Finally, using the expression for symbol $A=A^{(1)}+A^{(0)}$ given in Lemma \ref{lemma:decomP}, we have
\begin{align}\label{1stPlambda}
	&T_{\eta\alpha} \d_y v - T_{\d_z \rho} \d_z v = T_{\eta\alpha} \big( \d_y v - T_{A} v \big) + T_{\eta\alpha} T_{A} v - T_{\d_z \rho } \d_z v \\
	=& T_{\lambda^{(1)}} v + T_{\eta\alpha A^{(0)}} v + T_{\eta\alpha} (\d_y v - T_A v) + \fq[\eta\alpha,A](v) \nonumber\\
	=& \vcentcolon T_{\lambda^{(1)}} v +T_{\eta\alpha} (\d_y v - T_A v)+ \fR_5(v,\teta),\nonumber
\end{align}
By Theorems \ref{thm:paraL2}, \ref{thm:adjprod}\ref{item:prod}, and regularity (\ref{sigmaReg}), one can verify that
\begin{equation*}
\fR_5(v,\teta) \vcentcolon= T_{\eta\alpha A^{(0)}} v + \fq[\eta\alpha,A](v) \in \mC_{y}^0\big((0,1];H_{z}^{\sigma}(\R)\big).
\end{equation*}
Substituting (\ref{1stPlambda}) into $\mG_{y}[\eta](v)$, we get
\begin{equation*}
\mG_y[\eta](v) = T_{\lambda^{(1)}} v + T_{\eta\alpha}(\d_y v - T_A v) + \fR_4(v,\teta) + \fR_5(v,\teta).
\end{equation*}
Taking the limit $y\to 1^{-}$ on the above, then applying Corollary \ref{corol:1stOP}, one has 
\begin{equation*}
 G[\eta](\psi) = \mG_y[\eta](v)  \vert_{y=1} = T_{\lambda^{(1)}} \psi + \fR_{\sigma}(\teta,\psi), \quad \text{for some } \ \fR_{\sigma}(\teta,\psi) \in \mC_{y}^0\big((0,1];H^{\sigma}_z(\R)\big).
\end{equation*}
This concludes the proof for Lemma \ref{lemma:1stOP}.
\end{proof}

%%%%%%%%%%%%%%%%%%%%%%%%%%%%%%%%%%%%%%%%%%%%%%%%%%%%%
%%%%%%%%%%%%%%%%%%%%%%%%%%%%%%%%%%%%%%%%%%%%%%%%%%%%%

%---------------------
%     Section
%---------------------

\section{Paralinearization of Zakharov's system}\label{sec:para-Zakh}
The main result of this section is Proposition \ref{prop:JohnWick-4}. It is a paralinearization of Zakharov's system in terms of the Alinhac's good unknowns $(\eta, U)$, where $U=\psi-T_{\mB}\eta$.

\subsection{Paralinearization of kinematic equation}
In this section we assume that 
\begin{equation}
(\teta,\psi) \equiv (\eta-R, \psi)\in C^0([0, T]; H^{s+\frac{1}{2}}(\R)\times H^s(\R)).
\end{equation}
The equation for free boundary is given in terms of $r=\eta(z, t)$, and we have the system
\begin{align}\label{zak-cyl1}
	\left\{\begin{aligned}
		&\d_t \eta -G[\eta](\psi)=0,\\
		&\d_t \psi +\dfrac{1}{2}|\d_z\psi|^2-\dfrac{1}{2}\dfrac{|\d_z \psi \d_z\eta+G[\eta](\psi)|^2}{1+|\d_z \eta|^2}-\mH(\eta)-\dfrac{1}{2R}=0, \quad 
	\end{aligned}\right.
\end{align}
for $(z, t)\in \R\times(0, \infty)$, where $G[\eta]\psi=(\d_r \Psi-\d_z \Psi \d_z \eta)\vert_{r=\eta(t,z)},$ and $\Psi(z,r)$ is the potential of velocity. Recall from the definition (\ref{BV}) and (\ref{AGU}), we have
\begin{equation}\label{BVU}
\mB \vcentcolon= \dfrac{\d_z\eta \d_z\psi + G[\eta](\psi)}{1+|\d_z\eta|^2}, \qquad V\vcentcolon= \d_z\psi - \mB \d_z\eta, \qquad U\vcentcolon= \psi - T_{\mB} \eta.
\end{equation}
By Theorem \ref{thm:paraG}, one has $G[\eta](\psi) = T_{\lambda}U -T_{V}\d_z\eta + \fR_{G}(\teta,\psi)$. It follows that 
\begin{align}\label{zak-cyl2}
	\left\{\begin{aligned}
		&\d_t \eta-T_{\lambda}U +T_{V}\d_z\eta - \fR_G(\teta,\psi)=0,\\
		&\d_t\psi +\dfrac{1}{2}|\d_z\psi|^2-\dfrac{1}{2}\dfrac{|\d_z \psi \d_z\eta+G[\eta](\psi)|^2}{1+|\d_z \eta|^2} -\mH(\eta)-\dfrac{1}{2R}=0, \quad 
	\end{aligned}\right.
\end{align}
for $(z, t)\in \R\times(0, \infty)$. Thus it is left to paralinearise the dynamic equation of (\ref{zak-cyl2}).

%%%%%%%%%%%%%%%%%%%%%%%%%%%%%%%%%%%%%%%%%%%%%%
%%%%%%%%%%%%%%%%%%%%%%%%%%%%%%%%%%%%%%%%%%%%%%

\subsection{Paralinearization of surface tension}
To paralinearise the dynamic equation of Zakharov's system, we start by paralinearising the surface tension, which is given by the mean curvature.
\begin{lemma}\label{lemma:mean}
	Let $\mH(\eta)$ be the mean curvature given in (\ref{H}). Define the symbol
	\begin{gather}
		\ell(z,\xi)\vcentcolon= \ell^{(2)}(z,\xi)+\ell^{(1)}(z,\xi), \quad \text{and} \quad \ell^{(0)}=\ell^{(0)}(z) \quad \text{ where }\label{ell}\\
		\ell^{(2)}\!\vcentcolon=\! \dfrac{|\xi|^2}{2(1+|\d_z\eta|^2)^{3/2}}, \quad \ell^{(1)}\! \vcentcolon=\! \dfrac{i\xi\d_z\eta (3\eta\d_z^2\eta-1-|\d_z\eta|^2)}{2\eta(1+|\d_z\eta|^2)^{5/2}}, \quad \ell^{(0)}\! \vcentcolon=\! \dfrac{1}{2\eta^2\sqrt{1+|\d_z\eta|^2}}.\nonumber
	\end{gather}
	Suppose $\teta\in H^{s+\frac{1}{2}}(\R)$ for $s>\frac{5}{2}$. Then there exists $\fR_{\mH}(\teta)\in H^{2s-\frac{5}{2}}(\R)$ such that
	\begin{gather*}
		\mH(\eta)+\dfrac{1}{2R} = - T_{\ell} \teta + T_{\ell^{(0)}}\teta + \fR_{\mH}(\teta) \quad \text{ with } \quad \|\fR_{\mH}(\teta)\|_{H^{2s-\frac{5}{2}}(\R)} \le C\big(\|\teta\|_{H^{s+\frac{1}{2}}(\R)}\big). 
	\end{gather*}
\end{lemma}
\begin{proof}
	The mean curvature $\mH(\eta)$ is given by (\ref{H}) as
	\begin{equation}\label{Htemp}
		\mH(\eta) = \d_z\Big( \dfrac{\d_z\eta}{2\sqrt{1+|\d_z\eta|^2}} \Big) - \dfrac{1}{2\eta\sqrt{1+|\d_z\eta|^2}} =\vcentcolon \mathrm{(I)}+\mathrm{(II)}.
	\end{equation}
	First, we set functions $F(x),\,\, F_R(x)\in\mC^{\infty}(\R)$ with $F(0)=F_{R}(0)=0$ to be
	\begin{equation}\label{FRF}
		F(x)\vcentcolon= \dfrac{1}{\sqrt{1+x^2}} - 1, \qquad F_R(x)\vcentcolon= \dfrac{1}{x+R} - \dfrac{1}{R}. 
	\end{equation}
	Then by the paralinearization rule (\ref{brackets}), one has
	\begin{gather*}
		\begin{aligned}
			\dfrac{\d_z\eta}{2\sqrt{1+|\d_z\eta|^2}} = \dfrac{1}{2} (1+F(\d_z\eta)) \d_z\eta
			%=& \dfrac{1}{2} T_{1+F(\d_z\eta)} \d_z\eta + \dfrac{1}{2} T_{\d_z\eta} F(\d_z\eta) + \fp\big(1+F(\d_z\eta),\d_z\eta\big)\\ =& \dfrac{1}{2} T_{1+F(\d_z\eta)}\d_z\eta + \dfrac{1}{2} T_{\d_z\eta} T_{F^{\prime}(\d_z\eta)} \d_z\eta + \dfrac{1}{2} T_{\d_z\eta} \fC[F](\d_z\eta)+ \fp\big(1+F(\d_z\eta),\d_z\eta\big)\\=& \dfrac{1}{2}T_{1+F(\d_z\eta)}\d_z\eta + \dfrac{1}{2}T_{\d_z\eta F^{\prime}(\d_z\eta)} \d_z\eta + \fR_{\mH}^1(\eta),\\
			=& T_{\mathfrak{M}}\d_z\eta + \fR_{\mH}^1(\eta), 
		\end{aligned}\\
		\begin{aligned}	
			\text{where } \ \ &\mathfrak{M}(z)\vcentcolon= \dfrac{1}{2}\big\{1+F(\d_z\eta)+\d_z\eta F^{\prime}(\d_z\eta)\big\} = \dfrac{1}{2(1+|\d_z\eta|^2)^{3/2}}, \\
			&\fR_{\mH}^1(\eta) \vcentcolon= \dfrac{1}{2}\fq\big[\d_z\eta,F^{\prime}(\d_z\eta)\big](\d_z\eta) + \dfrac{1}{2} T_{\d_z\eta} \fC[F](\d_z\eta)+ \fp\big(1+F(\d_z\eta),\d_z\eta\big).
		\end{aligned}
	\end{gather*}
	By Theorems \ref{thm:paraL2}--\ref{thm:bony}, it can be shown that $\|\fR_{\mH}^1(\eta)\|_{H^{2s-\frac{3}{2}}(\R)}\le C\big(\|\teta\|_{H^{s+\frac{1}{2}}(\R)}\big)$. Substituting this into $\mathrm{(I)}$, it follows that 
	\begin{equation*}
		\mathrm{(I)} = \d_z \Big( \dfrac{\d_z\eta}{2\sqrt{1+|\d_z\eta|^2}} \Big) = T_{\mathfrak{M}} \d_z^2 \eta + T_{\d_z \mathfrak{M}} \d_z\eta + \d_z \fR_{\mH}^1(\eta). 
	\end{equation*}
	Next, we consider $\mathrm{(II)}$. By the paralinearization rule (\ref{brackets}) and (\ref{FRF}), one has
	\begin{gather*}
		\begin{aligned}
			\mathrm{(II)}=& -\dfrac{1}{2\eta\sqrt{1+|\d_z\eta|^2}}= 
			%\\=& -\dfrac{1}{2\sqrt{1+|\d_z\eta|^2}}\Big(\dfrac{1}{\eta}-\dfrac{1}{R}\Big) - \dfrac{1}{2R\sqrt{1+|\d_z\eta|^2}}\\=& -\dfrac{1}{2}F_R(\teta)\big(1+F(\d_z\eta)\big) - \dfrac{1}{2R}\big(1+F(\d_z\eta)\big)\\ =&
			- \dfrac{1}{2} F_{R}(\teta) F(\d_z\eta) -\dfrac{1}{2} F_{R}(\teta) - \dfrac{1}{2R} F(\d_z\eta) - \dfrac{1}{2R}\\
			=& - T_{N_1} \d_z\eta + T_{N_0} \teta - \dfrac{1}{2R} - \dfrac{1}{2}\fR_{\mH}^2(\eta),
		\end{aligned}\\
		\begin{aligned}
			\text{where } \ \ N_1(z)\vcentcolon=& \dfrac{1}{2}\big\{ F_{R}(\teta) + \dfrac{1}{R} \big\} F^{\prime}(\d_z\eta) = - \dfrac{\d_z\eta}{2\eta(1+|\d_z\eta|^2)^{3/2}},\\
			N_0(z)\vcentcolon=& -\dfrac{1}{2}\big\{ F(\d_z\eta)  + 1 \big\}F_{R}^{\prime}(\teta) = \dfrac{1}{2\eta^2 \sqrt{1+|\d_z\eta|^2}},\\
			\fR^2_{H}(\eta) \vcentcolon= & \fq\big[F_{R}(\teta),F^{\prime}(\d_z\eta)\big](\d_z\eta) + T_{F_R(\teta)} \fC[F](\d_z\eta) + \fq\big[F(\d_z\eta),F_{R}^{\prime}(\teta)\big](\teta)\\& + T_{F(\d_z\eta)}\fC[F_R](\teta)+\fp\big(F_R(\teta),F(\d_z\eta)\big) + \fC[F_R](\teta) + \fC[F](\d_z\eta).
		\end{aligned}
	\end{gather*}
	By Theorems \ref{thm:paraL2}--\ref{thm:bony}, one can verify $\|\fR_{\mH}^2(\eta)\|_{H^{2s-\frac{3}{2}}(\R)} \le C\big(\|\teta\|_{H^{s+\frac{1}{2}}(\R)}\big)$. Setting:
	\begin{gather*}
		\ell^{(2)}(z,\xi)\vcentcolon= |\xi|^2 \mathfrak{M}(z) =  \dfrac{|\xi|^2}{2(1+|\d_z\eta|^2)^{3/2}}, \qquad \ell^{(0)}(z) \vcentcolon= N_0(z) = \dfrac{1}{2\eta^2 \sqrt{1+|\d_z\eta|^2}},\\
		\ell^{(1)}(z,\xi) = i\xi(N_1(z)-\d_z \mathfrak{M}) = i\xi \dfrac{\d_z\eta (3\eta\d_z^2\eta -1 - |\d_z\eta|^2)}{2\eta(1+|\d_z\eta|^2)^{5/2}}, \quad \fR_{\mH}(\eta) \vcentcolon= \d_z\fR_{\mH}^1-\dfrac{1}{2}\fR_{\mH}^2.
	\end{gather*}
	Then $\mH(\eta)=\mathrm{(I)}+\mathrm{(II)}= - T_{\ell} \teta + T_{\ell^{(0)}}\teta - \tfrac{1}{2R} + \fR_{\mH}(\eta)$, which concludes the proof.
\end{proof}

%%%%%%%%%%%%%%%%%%%%%%%%%%%%%%%%%%%%%%%%%%%%%%
%%%%%%%%%%%%%%%%%%%%%%%%%%%%%%%%%%%%%%%%%%%%%%

\subsection{Paralinearization of dynamic equation}
Before we proceed to paralinearise the dynamic equation of the Zakharov's system, we first state the following observation regarding $\mB$ and $V$:
\begin{remark}\label{rem:BVbernou}
According to the definition of $\mB$, $V$ in (\ref{BVU}), one has
\begin{align}\label{GBVeta}
G[\eta](\psi) = (1+|\d_z\eta|^2) \mB - \d_z\eta \d_z\psi = \mB - \d_z\eta (\d_z\psi - \mB \d_z\eta) = \mB -V\d_z\eta.
\end{align}
Using (\ref{BVU}) and (\ref{GBVeta}), it follows that
\begin{align}\label{Dynmain}
	&\dfrac{1}{2}|\d_z\psi|^2 - \dfrac{1}{2}\dfrac{|\d_z\eta\d_z\psi+G[\eta](\psi)|^2}{1+|\d_z\eta|^2} = \dfrac{1}{2}\big\{|\d_z\psi|^2 - \mB \big( \d_z\eta \d_z\psi + G[\eta](\psi) \big) \big\}\\
	=& \dfrac{1}{2} \big\{ \big(\d_z\psi - \mB \d_z \eta \big)\d_z\psi - \mB G[\eta](\psi) \big\} = \dfrac{1}{2}\big\{ V \d_z\psi - \mB (\mB- V\d_z\eta) \big\}\nonumber\\
	=& %\dfrac{1}{2}\big\{ V \d_z\psi + V \mB \d_z\eta - \mB^2 \big\} =
	\dfrac{1}{2}\big\{ V \d_z\psi + V \big( \d_z\psi - V \big) - \mB^2 \big\}= V \d_z\psi -\dfrac{1}{2}(V^2 + \mB^2).\nonumber
\end{align} 
Substituting this into the second equation of (\ref{zak-cyl2}), one has
\begin{equation}\label{bernou}
(\d_t \psi + V \d_z\psi) - \dfrac{1}{2}(V^2 + \mB^2) - \mH(\eta) - \dfrac{1}{2R} = 0.
\end{equation}
Recall from Remark \ref{rem:BV} that $\mB=\d_r\Psi\vert_{r=\eta(t,z)}$ and $V=\d_z\Psi\vert_{r=\eta(t,z)}$ are respectively the radial and axial velocity component at the surface $r=\eta(t,z)$. Therefore using (\ref{BVU}), (\ref{GBVeta}), and $\d_t\eta = G[\eta](\psi)$, one has
\begin{equation*}
\d_t\psi + V\d_z\psi = \d_t \Psi \big\vert_{r=\eta(t,z)} + (\mB^2 + V^2) , \quad \text{and} \quad \frac{1}{2}(V^2 + \mB^2) = \frac{1}{2}|\nabla \Psi|^2\big\vert_{r=\eta(t,z)}
\end{equation*}
Putting these into (\ref{bernou}), we recover the Bernoulli's law.
\end{remark}
Using Remark \ref{rem:BVbernou}, we obtain the following paralinearization: 
\begin{lemma}\label{lemma:ParaDyn}
Let $s>\frac{5}{2}$ and $\teta\in H^{s+\frac{1}{2}}(\R)$. Suppose $\psi\in H^{k}(\R)$ for $\frac{1}{2}\le k\le s$. Then
\begin{equation}\label{ParaDyn}
	\dfrac{1}{2}|\d_z\psi|^2-\dfrac{1}{2}\dfrac{|\d_z \psi \d_z\eta+G[\eta](\psi)|^2}{1+|\d_z \eta|^2}
	=
	T_V\d_z \psi-T_{V}T_{\mathcal{B}}\d_z \eta-T_{\mathcal B}G[\eta]\psi+\fR_{D}(\teta,\psi),
\end{equation}
for some $\fR_{D}(\teta,\psi)\in H^{2k-\frac{5}{2}}(\R)$ such that
\begin{equation*}
\|\fR_{D}(\teta,\psi)\|_{H^{2k-\frac{5}{2}}(\R)} \le C\big(\|\teta\|_{H^{s+\frac{1}{2}}(\R)}\big)\|\d_z\psi\|_{H^{k-1}(\R)}.
\end{equation*} 
\end{lemma}
\begin{proof}
Applying the paralinearization rule (\ref{brackets}) on $V\d_z\psi$, and using (\ref{BVU}), one has
\begin{align}\label{Vdzpsi}
V\d_z\psi =& T_{V} \d_z\psi + T_{\d_z\psi} V + \fp(V,\d_z\psi) = T_{V} \d_z\psi + T_{V} V + T_{\mB \d_z\eta} V + \fp(V,\d_z\psi)\\
=& T_{V} \d_z\psi + \dfrac{1}{2}V^2 - \dfrac{1}{2}\fp(V,V) + T_{\mB \d_z\eta} V + \fp(V,\d_z\psi).\nonumber
\end{align}
Similarly, using (\ref{GBVeta}), $\frac{1}{2}\mB^2$ can be paralinearised as
\begin{align}\label{mB2}
&\dfrac{1}{2}\mB^2= T_{\mB}\mB +\dfrac{1}{2}\fp(\mB,\mB) = T_{\mB} G[\eta](\psi) + T_{\mB}(V\d_z\eta) +\dfrac{1}{2}\fp(\mB,\mB)\\
%=& T_{\mB} G[\eta](\psi) + T_{\mB}T_{V}\d_z\eta  + T_{\mB} T_{\d_z\eta} V + T_{\mB} \fp(V,\d_z\eta) +\dfrac{1}{2}\fp(\mB,\mB)\nonumber\\
=& T_{\mB} G[\eta](\psi) + T_{V}T_{\mB} \d_z \eta - [T_{V},T_{\mB}]\d_z\eta + T_{\mB} T_{\d_z\eta} V  + T_{\mB} \fp(V,\d_z\eta) +\dfrac{1}{2}\fp(\mB,\mB).\nonumber
\end{align}
Substituting (\ref{Vdzpsi}) and (\ref{mB2}) into (\ref{Dynmain}), we obtain that
\begin{gather}
\dfrac{|\d_z\psi|^2}{2} - \dfrac{|\d_z\eta\d_z\psi+G[\eta](\psi)|^2}{2(1+|\d_z\eta|^2)} %= V \d_z\psi -\dfrac{1}{2}(V^2 + \mB^2)
= T_{V}\d_z\psi - T_{V}T_{\mB}\d_z\eta -T_{\mB}G[\eta](\psi) + \fR_{D}(\teta,\psi), \label{RD}\\
\begin{aligned}
\text{where } \ \fR_{D}(\teta,\psi)\vcentcolon=& \fp(V,\d_z\psi) - \tfrac{1}{2}\fp(V,V) + [T_{V},T_{\mB}]\d_z\eta\\
&- \fq[\mB,\d_z\eta](V)  - T_{\mB}\fp(V,\d_z\eta) - \tfrac{1}{2}\fp(\mB,\mB).  
\end{aligned}\nonumber
\end{gather}
By (\ref{BVU}), Lemma \ref{lemma:GSob}, and Propositions \ref{prop:Sobcomp}--\ref{prop:clprod}, if $(\teta,\psi)\in H^{s+\frac{1}{2}}(\R)\times H^{k}(\R)$ then
\begin{align}\label{BVs-1}
\|\mB\|_{H^{k-1}} + \|V\|_{H^{k-1}} \le C\big(\|\teta\|_{H^{s+\frac{1}{2}}}\big)\|\d_z\psi\|_{H^{k-1}}.
\end{align}
Using the above estimate, and Theorem \ref{thm:paraL2}, \ref{thm:bony}, we obtain that
\begin{align*}
\|\fR_{D}(\teta,\psi)\|_{H^{2k-\frac{5}{2}}} \le C\big(\|\teta\|_{H^{s+\frac{1}{2}}}\big)\|\d_z\psi\|_{H^{k-1}},
\end{align*}
which concludes the proof.
\end{proof}
Combining Lemmas \ref{lemma:mean} and \ref{lemma:ParaDyn} in (\ref{zak-cyl2}), we get that for $(z, t)\in \R\times(0, \infty)$,
\begin{align}\label{zak-cyl3}
	\left\{\begin{aligned}
		&\d_t\eta +T_{V}\d_z\eta -T_{\lambda}U  =\fR_G(\teta,\psi)\\
		&\d_t \psi -T_{\mB}G[\eta]\psi +T_V\d_z \psi-T_{V}T_{\mathcal{B}}\d_z \eta
		+
		T_{\ell}\eta=  T_{\ell^{(0)}}\teta + \fR_{\mH}(\teta) -\fR_{D}(\teta,\psi).
	\end{aligned}\right.
\end{align}

\subsection{Zakharov's system in Alinhac's good unknowns \texorpdfstring{$(\eta,U)$}{(eta,U)}}
Using $\d_t \eta=G[\eta](\psi) $ in the expression $U=\psi - T_{\mB} \eta$, we have that 
\begin{align*}
\d_t U=&\d_t \psi-T_{\mathcal B}\d_t\eta-T_{\partial_t\mathcal B}\eta = \d_t \psi  - T_{\mB} G[\eta](\psi) - T_{\d_t \mB} \eta,\\
T_V\d_z U=& T_{V}\d_z (\psi - T_{\mB} \eta) = T_V\d_z\psi- T_{V}T_{\mB}\d_z \eta - T_{V}T_{\d_z\mB} \eta.
\end{align*}
Substituting these into the second equation of \eqref{zak-cyl3}, we obtain 
\begin{align*}
\d_t U + T_{V} \d_z U + T_{\ell} \eta  %=& \d_t \psi - T_{\mB} G[\eta](\psi) - T_{\d_t \mB} \eta + T_{V} \d_z \psi -  T_{V} T_{\mB} \d_z \eta - T_{V} T_{\d_z \mB} \eta + T_{\ell} \eta\\ =& \d_t \psi - T_{\mB} G[\eta](\psi) + T_{V} \d_z \psi -  T_{V} T_{\mB} \d_z \eta + T_{\ell} \eta - T_{\d_t \mB} \eta - T_{V} T_{\d_z \mB} \eta\\
=& T_{\ell^{(0)}} \teta + \mR_{\mH}(\teta) -\mR_{D}(\teta,\psi) - T_{\d_t \mB} \eta - T_{V} T_{\d_z \mB} \eta.
\end{align*}
Thus the second equation in (\ref{zak-cyl3}) is rewritten as $\d_t U + T_{V} \d_z U + T_{\ell} \eta = f_{\ast}^2(\eta,\psi)$ where
\begin{equation*}
f_{\ast}^2(\eta,\psi) \vcentcolon= T_{\ell^{(0)}} \teta + \mR_{\mH}(\teta) -\mR_{D}(\teta,\psi) - T_{\d_t \mB} \eta - T_{V} T_{\d_z \mB} \eta
\end{equation*}
Next, we claim that $f_{\ast}^2(\eta,\psi)\in L^{\infty}\big(0,T;H^{s}(\R)\big)$ with
\begin{equation} \label{f2ast}
  \sup\limits_{ 0 \le t\le T} \| f_{\ast}^2(\eta,\psi) \|_{H^{s}(\R)} \le C\big(\|(\teta,\psi)\|_{H^{s+1/2}(\R)\times H^s(\R)}\big)
\end{equation}
First, $\d_z\mB\in H^{s-2}(\R)\xhookrightarrow[]{}L^{\infty}(\R)$, since $s>\frac{5}{2}$. This implies that $V,\,\d_z\mB\in \Gamma_{0}^{0}(\R)$, where $\Gamma_0^0$ is the symbol space in Definition \ref{def:symbols}. By Theorems \ref{thm:paraL2} and \ref{thm:adjprod}\ref{item:adj}, 
\begin{equation*}
    \| T_{V} T_{\d_z \mB} \eta \|_{H^{s+\frac{1}{2}}} \le C \| V \|_{L^{\infty}} \| \d_z \mB \|_{L^{\infty}} \|\teta\|_{H^{s+\frac{1}{2}}} \le C\big(\|(\teta,\psi)\|_{H^{s+1/2}\times H^s}\big).
\end{equation*}
Moreover, by Lemmas \ref{lemma:mean} and \ref{lemma:ParaDyn}, we also have
\begin{equation*}
    \| T_{\ell^{(0)}}\teta \|_{H^{s+\frac{1}{2}}} + \|\mR_{D}(\teta,\psi)\|_{H^{2s-\frac{5}{2}}} + \|\mR_{D}(\teta)\|_{H^{2s-\frac{5}{2}}} \le C\big(\|(\teta,\psi)\|_{H^{s+1/2}\times H^s}\big).
\end{equation*}
Thus to obtain (\ref{f2ast}), it remains to estimate $T_{\d_t\mB} \eta$, which is proved as follows:
\begin{lemma}\label{lem:JohnWick3}
There is a strictly positive non-decreasing function $x\mapsto C(x)$ such that 
\[
\|T_{\partial_t\mB}\eta\|_{H^{s}(\R)}\le C\big(\|(\teta,\psi)\|_{H^{s+1/2}(\R)\times H^s(\R)}\big).
\]
\end{lemma}
\begin{proof}
Using the first equation of (\ref{zak-cyl1}), one has
\begin{equation}\label{dtetas-1}
\|\d_t \eta\|_{H^{s-1}}=\|G[\eta](\psi)\|_{H^{s-1}}\le 
C(\|(\teta,\psi)\|_{H^{s+1/2}\times H^s}).
\end{equation}
From (\ref{BVs-1}), we also have the estimate
\begin{align*}
	\|\mB\|_{H^{s-1}} + \|V\|_{H^{s-1}} \le C\big(\|\teta\|_{H^{s+\frac{1}{2}}}\big)\|\d_z\psi\|_{H^{s-1}}.
\end{align*}
Finally, we claim that
\begin{equation}\label{dtpsis-1}
\|\d_t\psi(\cdot,t)\|_{H^{s-3/2}} \le 
C(\|(\teta,\psi)\|_{H^{s+1/2}\times H^s}).
\end{equation}
Recalling the definition of $\mB$ in (\ref{BV}), the second equation in (\ref{zak-cyl1}) can be written as
\begin{equation*}
\d_t \psi = - \dfrac{1}{2}|\d_z\psi|^2 - \dfrac{1}{2} (1+|\d_z\eta|^2) \mB^2 - \mH(\eta)-\frac{1}{2R}.
\end{equation*}
Since $\d_z\psi\in H^{s-1}(\R)$, $\mB\in H^{s-1}(\R)$, and $\mH(\eta)+\frac{1}{2R}\in H^{s-3/2}(\R)$, using the product rule for Sobolev functions, Proposition \ref{prop:clprod}, we get
\begin{align*}
\|\d_t\psi(\cdot,t)\|_{H^{s-3/2}} %\le C \|\d_z\psi\|_{H^{s-1}}^2 + C \big\{ 1+\|\d_z\eta\|_{H^{s-\frac{1}{2}}}^2 \big\} \|\mB\|_{H^{s-1}}^2 + \|\teta\|_{H^{s+1/2}(\R)}
\le C(\|(\teta,\psi)\|_{H^{s+1/2}\times H^s}).
\end{align*} 
By the product rule and the formula for the shape derivative, Theorem \ref{thm:shape}, we get
\begin{align*}
\d_t\big(G[\eta](\psi)\big) =&
\dif_{\eta} G[\eta](\psi)\cdot \d_t\eta +G[\eta](\d_t\psi)\\
=&%-G[\eta](\mB\d_t\eta)-\d_z(V\d_t\eta)-\frac{\mB}{\eta}\d_t\eta+G[\eta](\d_t\psi) =
G[\eta](\d_t\psi-\mB\d_t\eta)-\d_z(V\d_t\eta)-\frac{\mB}{\eta}\d_t\eta.
\end{align*}
Applying the above estimates (\ref{dtetas-1})--(\ref{dtpsis-1}), we get that 
\[
\big\|\d_t\big(G[\eta](\psi)\big)\big\|_{H^{s-5/2}}\le 
C(\|(\eta, \psi)\|_{H^{s+1/2}\times H^s}).
\]
If $s<3$ then we can apply Theorem \ref{thm:paraPEst} with $s-\frac{5}{2} = \frac{1}{2} - (3-s)$ to obtain that
\begin{align*}
   \| T_{\d_t \mB} \eta \|_{H^{s}} \le \| T_{\d_t \mB} \eta \|_{H^{2s-5/2}} \le&  C \| \d_t \mB \|_{H^{s-5/2}} \|\teta\|_{H^{s+1/2}} \le C(\|(\eta, \psi)\|_{H^{s+1/2}\times H^s}).
\end{align*}
If $s\ge 3$ then the same result follows by Sobolev embedding and Theorem \ref{thm:paraL2}. 
\end{proof}

Set $f_{\ast}^1(\eta,\psi)\vcentcolon= \mR_{G}(\teta,\psi)$. Then Theorem \ref{thm:paraG} and (\ref{f2ast}) yield that the system 
(\ref{zak-cyl3}) simplifies to 
\begin{gather}\label{zak-cyl4}
\left\{\begin{aligned}
&\d_t\eta -T_{\lambda}U +T_{V}\d_z\eta =f^1_{\ast}(\eta,\psi)\\
&\d_t U + T_{V}\d_z U+T_{\ell}\eta= f^2_{\ast}(\eta,\psi), \quad 
\end{aligned}\right.
\text{for } \quad (z, t)\in \R\times [0, T],
\end{gather}
where $f_{\ast}^1 \in L^{\infty}\big(0,T;H^{s+\frac{1}{2}}(\R)\big)$ and $f_{\ast}^2 \in L^{\infty}\big(0,T;H^{s}(\R)\big)$. We summarize the results of this section in the following 
\begin{proposition}\label{prop:JohnWick-4}
The Zakharov system has the following paralinearization
\begin{gather}
\label{zak-para} \left\{\begin{aligned}
&\d_t\eta+T_{V}\d_z\eta -T_{\lambda}U =f^1_{\ast},\\
&\d_t U+T_{V}\d_z U +T_{\ell}\eta=f^2_{\ast}, \quad 
\end{aligned}\right. \text{for } \quad (z, t)\in \R\times[0,T],
\end{gather}
where $f^1_{\ast}=\fR_G(\teta,\psi)$, $f^2_{\ast}=T_{\ell^{(0)}} \teta + \mR_{\mH}(\teta) -\mR_{D}(\teta,\psi) - T_{\d_t \mB} \eta - T_{V} T_{\d_z \mB} \eta$, and
\begin{align}\label{eq:Alinhacf}
\|f^1_{\ast}\|_{L^{\infty}(0,T; H^{s+1/2}(\R))} + \|f^2_{\ast}\|_{L^{\infty}(0,T; H^{s}(\R))} \le C\big(\|(\teta,\psi)\|_{H^{s+1/2}(\R)\times H^s(\R)}\big).
\end{align}
\end{proposition}

\begin{remark}\label{rem:JohnWick3}
The system \eqref{zak-para} can be rewritten as 
\begin{equation}\label{Mlaml}
(\partial_t+T_V\d_z )
\begin{pmatrix}
	\eta\\
	U
\end{pmatrix}
+
\mathscr{M}_{\lambda, \ell}
\begin{pmatrix}
	\eta\\
	U
\end{pmatrix}
=
\begin{pmatrix}
	f_{\ast}^1\\
	f_{\ast}^2
\end{pmatrix},
\quad \text{where} \quad
\mathscr{M}_{\lambda, \ell}\vcentcolon=
\begin{pmatrix}
	0 & -T_\lambda\\
	T_\ell & 0
\end{pmatrix}.
\end{equation}
In what follows we will also use the following notation. 
\[
\begin{pmatrix}
\eta\\
U
\end{pmatrix}
=
\mathscr{M}_{\mathcal B}
\begin{pmatrix}
\eta\\
\psi
\end{pmatrix}, \quad \text{where} \quad 
\mathscr{M}_{\mathcal B}
\vcentcolon=
\begin{pmatrix}
I & 0\\
-T_{\mathcal B}&I
\end{pmatrix}.
\]
Consequently, \eqref{zak-para} has the form 
%\begin{equation}
%(\partial_t+T_V\d_z )
%\mathscr{M}_{\mathcal B}
%\begin{pmatrix}
%	\eta\\
%	\psi
%\end{pmatrix}
%+
%\mathscr{M}_{\lambda, \ell}\mathscr{M}_{\mathcal B}
%\begin{pmatrix}
%	\eta\\
%	\psi
%\end{pmatrix}
%=
%\begin{pmatrix}
%	f^1_{\ast}\\
%	f^2_{\ast}
%\end{pmatrix}.
%\end{equation} Hence
\begin{equation}\label{eq:JohnWick33-f1f2}
\mathscr{M}_{\mathcal B}(\partial_t+T_V\d_z )
\begin{pmatrix}
	\eta\\
	\psi
\end{pmatrix}
+
\mathscr{M}_{\lambda, \ell}\mathscr{M}_{\mathcal B}
\begin{pmatrix}
	\eta\\
	\psi
\end{pmatrix}
=
\begin{pmatrix}
	f^1_{\ast}\\
	f^2_{\ast}
\end{pmatrix}-\big[\partial_t+T_V\d_z, \mathscr{M}_{\mathcal B}\big]\begin{pmatrix}
	\eta\\
	\psi
\end{pmatrix}.
\end{equation}
Multiplying both sides from left by $\mathscr{M}_{\mB}^{-1}$, (\ref{eq:JohnWick33-f1f2}) takes the form 
\begin{gather*}
(\partial_t+T_V\d_z +\mathscr L)\begin{pmatrix}
	\eta\\
	\psi
\end{pmatrix}=\begin{pmatrix}
    \tilde{f}^1\\ \tilde{f}^2
\end{pmatrix}, \quad \text{ where } \ \mathscr{L}\vcentcolon=\mathscr{M}_{\mathcal B}^{-1}\mathscr{M}_{\lambda, \ell}\mathscr{M}_{\mathcal B}, \\
\text{and } \  \begin{pmatrix}
    \tilde{f}^1\\ \tilde{f}^2
\end{pmatrix}\vcentcolon=\mathscr{M}_{\mathcal B}^{-1}\begin{pmatrix}
	f^1_{\ast}\\
	f^2_{\ast}
\end{pmatrix}- \mathscr{M}_{\mathcal B}^{-1} \big[\partial_t+T_V\d_z, \mathscr{M}_{\mathcal B}\big]\begin{pmatrix}
	\eta\\
	\psi
\end{pmatrix}. 
\end{gather*}
This is one of the main reformulations of Zakharov's system considered in Section \ref{sec:cauchy}.
\end{remark}

%---------------------
%     Section
%---------------------

\section{Symbolic Calculus and Symmetriser}\label{sec:symbolic-five}
The main aim of this section is to seek a suitable symmetrizer $\mathbf{S}=\big(\begin{smallmatrix} T_{p} & 0 \\ 0 & T_{q}  \end{smallmatrix}\big)$ for the paralinearised system (\ref{zak-para}). Here, $T_{p}$ and $T_q$ are two paradifferential operators chosen such that $\mathscr{M}_{\lambda, \ell}$ is transformed into a skew-symmetric matrix upon conjugation by $\mathbf{S}$. This then guarantees that Gr\"onwall's inequality can be applied without any loss of regularity. The key argument is to use the symbolic calculus proposed in \cite{ABZ}. The main results of this section is Proposition \ref{prop:qpgamma} and Theorem \ref{thm:symPHI}.  

\subsection{Symbols of \texorpdfstring{$\Sigma$}{Sigma}-class}\label{ssec:Sigma}
Following \cite{ABZ}, we introduce an important class of symbols to be used 
in this section. We will present a more general construction where the spatial dimension includes the range $d\ge 1$. In the context of current paper, we only require the special case where $d=1$.  
\begin{definition}\label{def:Sigma}
Let $\eta(t,x)\vcentcolon [0,T]\times \R^d \to (0,\infty)$ be $\teta \vcentcolon= \eta -R \in L^{\infty}\big(0,T ; H^{s+1/2}(\R)\big) $. For $m\in \R$, we say $a\in \Sigma^m$ if $a=a^{(m)}+a^{(m-1)}$ is such that 
\begin{equation*}
a^{(m)}(t, x, \xi)\!=\!\mathcal{Y}(\nabla \eta(t, x), \xi), \quad
a^{(m-1)}(t, x, \xi)\!=\!\!\sum_{|\alpha|=2}\mathcal{Z}_\alpha(\nabla \eta(t, x), \xi)\d_x^\alpha\eta(t, x), \ \text{where}
\end{equation*}
\begin{enumerate}[label=\textnormal{(\arabic*)},ref=\textnormal{(\arabic*)}]
\item\label{item:Sigma1} $T_a$ is an operator mapping from real-valued to real-valued; 
\item\label{item:Sigma2} $\mathcal{Y}(\zeta, \xi)\in \mC^\infty\big(\R^d\times(\R^d\setminus\{0\})\big)$ is homogeneous of degree $m$ in $\xi$ and there exists a continuous function $K(\zeta)>0$ such that 
\[
\mathcal{Y}(\zeta, \xi)\ge K(\zeta)|\xi|^m, \quad \forall (\zeta, \xi)\in \R^d\times(\R^d\setminus\{0\});
\]
\item\label{item:Sigma3} $\mathcal{Z}_\alpha\in \mC^\infty\big(\R^d\times(\R^d\setminus\{0\})\big)$
complex valued and homogeneous in $\xi$ of degree $m-1$.
\end{enumerate}
\end{definition}

See \cite{Metivier} Chapter 6. The proofs for Propositions \ref{prop:sharp-high-order} and \ref{prop:CZ}
are straightforward application of paradifferential calculus, and can be found in Section 4 of \cite{ABZ}.
%\todo{ say that the dimension d=1 but can be applied to any dimension}
\begin{proposition}\label{prop:sharp-high-order}
Let $m, m'\in \R$. The following statements hold:
\begin{enumerate}[label=\textnormal{(\arabic*)},ref=\textnormal{(\arabic*)}]
\item $a\in \Sigma^m, b\in \Sigma^{m'}$, then $T_aT_b\sim T_{a\natural b}$ with 
$ a\natural b\in \Sigma^{m+m'}$ given by 
\[
a\natural b \vcentcolon= a^{(m)}b^{(m')}+ a^{(m-1)}b^{(m')}+a^{(m)}b^{(m'-1)}+\frac1i\d_\xi a^{(m)}\cdot \d_x b^{(m')}, 
\]
\item Denote $\overline{a}$ as the complex conjugate of $a$. If $a\in \Sigma^m$, then $(T_a)^{\ast}\sim T_{a^{\star}}$ with 
\[
a^{\star}\in \Sigma^m \ \text{ given by } \ a^{\star}\vcentcolon=a^{(m)}+\overline{a^{(m-1)}}+\frac{1}{i}(\d_x\cdot\d_\xi)a^{(m)}.
\]
\end{enumerate}

\end{proposition}

\begin{proposition}\label{prop:CZ}
Let $m, \mu\in \R$ and $\teta\in H^{s-1}(\R)$ for $s>2+\frac{d}{2}$. Then there exists a positive non-decreasing $x\mapsto \mathcal{K}_1(x)$ such that for all $a\in \Sigma^m(\R^d)$ and $t\in [0,T],$ 
\[
\|T_{a(t)}u\|_{H^{\mu-m}}\le \mathcal{K}_1\big(\|\teta\|_{H^{s-1}}\big)\|u\|_{H^\mu}.
\]
\end{proposition}

\begin{proposition}\label{prop:coerc}
Let $m, \mu\in \R$ and $\teta\in H^{s-1}(\R)$ for $s>2+\frac{d}{2}$. Then there exists a strictly positive non-decreasing $x\mapsto \mathcal{K}_2(x)$ such that for all $a\in \Sigma^m(\R^d)$ and $t\in [0,T],$ 
\[
\|u\|_{H^{\mu+m}(\R^d)}\le \mathcal{K}_2(\|\nabla \eta\|_{H^{s-2}(\R^d)})\left\{\|T_{a(t)}u\|_{H^{\mu}(\R^d)}+\|u\|_{L^2(\R^d)}\right\}.
\]
\end{proposition}
\begin{proof}
For simplicity, we omit the dependence in time $t\ge 0$. Set $b(x,\xi)\vcentcolon=1/a^{(m)}(x,\xi)$, and choose $0<\ep<\min\{1,s-2-\frac{d}{2}\}$. Since $\teta\in H^{s-1}(\R^d)$ with $s>2+\frac{d}{2}$, the embedding $H^{s-2}(\R^d)\xhookrightarrow[]{} W^{s-2-d/2,\infty}(\R^d)$ implies that $a^{(m)}\in \Gamma_{\ep}^{m}(\R^d)$ and $b\in \Gamma_{\ep}^{-m}(\R^d)$. by Theorem \ref{thm:adjprod}\ref{item:prod}, we have
\begin{equation*}
T_{b} T_{a^{(m)}} = I + \tilde{\mR}, \quad \text{where } \ \tilde{\mR} \ \text{is an operator of order } \ (-\ep),
\end{equation*}
and $I$ is the identity operator. Since $a=a^{(m)}+a^{(m-1)}$, we have
\begin{equation*}
 u = T_{b} T_{a} u - \tilde{\mR} u - T_{b} T_{a^{(m-1)}} u.
\end{equation*}
Set $\mR\!\vcentcolon=\! -\tilde{\mR} - T_{b}T_{a^{(m-1)}}$. Then $(I- \mR) u = T_{b} T_{a} u$. Since $\xi \mapsto a^{(m-1)}(x,\xi)$ is homogeneous of order $m-1$ by Definition \ref{def:Sigma}\ref{item:Sigma3}. Thus we can apply Proposition \ref{prop:Hsym} with
%\todo{say that there is more in proposition Appendix, which we skip here}
\begin{equation*}
r\vcentcolon=m-1, \quad k\vcentcolon=s-3, \quad \delta \vcentcolon= s-2-\tfrac{d}{2}-\ep>0, \quad l\vcentcolon= \tfrac{d}{2}-k+\delta = 1-\ep,
\end{equation*}
to obtain that for all $u\in H^{\mu}(\R^d)$ with $\mu\in\R$,
\begin{align*}
\|T_{a^{(m-1)}} u\|_{H^{\mu- m +\ep}(\R^d)} \le& C \sup\limits_{|\alpha|\le d/2 +1}\sup\limits_{|\xi|=1}\|\d_{\xi}^{\alpha} a^{(m-1)}(\cdot,\xi)\|_{H^{s-3}(\R^d)} \|u\|_{H^{\mu}(\R^d)}\\
\le& C\big(\|\nabla \eta\|_{H^{s-2}(\R^d)} \big) \|u\|_{H^{\mu}(\R^d)},
\end{align*} 
where in the last line we used Definition \ref{def:Sigma}\ref{item:Sigma3}. Since $T_{b}$ is of order $-m$, it follows that $\mR=-\tilde{\mR}-T_{b}T_{a^{(m-1)}}$ is of order $-\ep$ with
\begin{equation}\label{Rest}
\|\mR u\|_{H^{\mu + \ep}(\R^d)} \le C\big( \|\nabla \eta\|_{H^{s-2}(\R^d)} \big) \|u\|_{H^{\mu}(\R^d)}.
\end{equation}
Since $(I-\mR)u=T_{b}T_{a} u$, we have that for each $N\in\mathbb{N}$,
\begin{gather}
(I-\mR^{N+1}) u= (I + \mR + \cdots + \mR^N) (I-\mR) u = (I+\mR + \cdots + \mR^N) T_{b} T_{a} u,\nonumber\\
\text{which implies } \quad u = (I+\mR + \cdots + \mR^N) T_{b} T_{a} u + \mR^{N+1} u.\label{uIR}
\end{gather}
First by Theorem \ref{thm:paraL2} and (\ref{Rest}),
\begin{align}\label{IRN}
&\|(I+\mR+\cdots+\mR^{N})T_{b} \|_{H^{\mu+m}\to H^{\mu+m}}\\
\le& \|(I+\mR+\cdots+\mR^{N})\|_{H^{\mu+m}\to H^{\mu+m}} \|T_{b}\|_{H^{\mu} \to H^{\mu+m}}\le C\big( \|\nabla \eta\|_{H^{s-2}(\R^d)} \big).\nonumber
\end{align}
Moreover applying (\ref{Rest}) for $N+1$ times, we obtain:
\begin{align*}
& \|\mR^{N+1} u\|_{H^{\ep(N+1)}(\R^d)} \le C\big( \|\nabla \eta\|_{H^{s-2}(\R^d)} \big)\|\mR^N u\|_{H^{\ep N}(\R^d)}\\
\le & \big\{C\big( \|\nabla \eta\|_{H^{s-2}(\R^d)} \big) \big\}^2 \|\mR^{N-1} u\|_{H^{ \ep (N-1)}(\R^d)} \le \dotsc  \le \big\{C\big( \|\nabla \eta\|_{H^{s-2}(\R^d)} \big) \big\}^{N+1} \| u\|_{L^2(\R^d)}. 
\end{align*}
For given $\ep>0$ and $m,\,\mu\in\R$, choose $N\in\mathbb{N}$ such that $(N+1)\ep > m+\mu$. Then
\begin{align}\label{RN1}
\|\mR^{N+1} u\|_{H^{\mu+m}(\R^d)} \le \|\mR^{N+1} u\|_{H^{\ep(N+1)}(\R^d)} \le \big\{C\big( \|\nabla \eta\|_{H^{s-2}(\R^d)} \big) \big\}^{N+1} \| u\|_{L^{2}(\R^d)} 
\end{align}
Putting (\ref{IRN}) and (\ref{RN1}) into (\ref{uIR}), we obtain the desired estimate.
\end{proof}

%%%%%%%%%%
\subsection{Symmetrising symbols \texorpdfstring{$p, q$}{p, q} and \texorpdfstring{$\gamma$}{gamma}}\label{ssec:pqg}
To seek for symbols $p$ and $q$ that symmetrize the system (\ref{zak-para}), we apply the symbolic calculus established in Section \ref{ssec:Sigma} for the case $d=1$. First, we define the following equivalence relation for operators:
\begin{definition}\label{def:sim}
Set $m\in \R$, and let 
$\big\{A(t)\,\vert\, t\in[0, T]\big\}$ and $\big\{B(t)\,\vert\, t\in[0, T]\big\}$ be two families of operators which are of order $m$. We say that $A\sim B$ if and only if $A-B$ is of order $m-3/2$, and there is a strictly positive non-decreasing $x\mapsto C(x)$ such that 
\begin{equation}\label{sim}
 \|A-B\|_{H^{\mu}\to H^{\mu-m+3/2}} \le C\big(\|\teta\|_{H^{s+1/2}}\big) \ \text{ for all } \ \mu\in \R.
\end{equation}
\end{definition}
We aim to find symbols $p, q, \gamma$ such that
\begin{align}
\tag*{\textrm{(S1)}}\label{item:S1}& T_pT_\lambda\sim T_\gamma T_q, \quad T_qT_\ell\sim T_\gamma T_p, \quad T_\gamma\sim(T_\gamma)^{\ast};\\
\tag*{\textrm{(S2)}}\label{item:S2}& p=p^{(1/2)}+p^{(-1/2)} \in \Sigma^{1/2}+\Sigma^{-1/2}, \quad q=q^{(0)} \in \Sigma^0,\\ &\gamma=\gamma^{(3/2)}+\gamma^{(1/2)}\in \Sigma^{3/2}+ \Sigma^{1/2};\nonumber\\ 
\tag*{\textrm{(S3)}}\label{item:S3}& p^{(1/2)}\ge K|\xi|^{1/2}, \quad q^{(0)}\ge K, \quad \gamma^{(3/2)}\ge K|\xi|^{3/2};\\
\tag*{\textrm{(S4)}}\label{item:S4}& \overline{p(t, z, \xi)}=p(t, z, -\xi), \quad \overline{q(t, z, \xi)}=q(t, z, -\xi), \quad 
\overline{\gamma(t, z, \xi)}=\gamma(t, z, -\xi),
\end{align}
where $\overline{a}$ denotes the complex conjugate of $a$.
\begin{proposition}\label{prop:qpgamma}
Let $\lambda$, $\ell$ be the symbols given in (\ref{lambda}) and (\ref{ell}). Set:
\begin{subequations}\label{qpgamma}
\begin{gather}
\label{gamma}	 \gamma= \gamma^{(3/2)} + \gamma^{(1/2)} \vcentcolon= \sqrt{\ell^{(2)}\lambda^{(1)}}+\sqrt{\frac{\ell^{(2)}}{\lambda^{(1)}}}\frac{\textnormal{\textrm{Re}}\lambda^{(0)}}2-\frac i2(\d_\xi\cdot\d_z)\sqrt{\ell^{(2)}\lambda^{(1)}},\\
\label{q} q\vcentcolon=\frac{R^{1/3}}{\eta^{1/3} (1+|\d_z\eta|^2)^{1/4}}\exp\Big(-\int_{-\infty}^{z} \frac{(\d_z\eta)^3}{6\eta}(t,y)\,\dif y\Big),  \\
\label{p} p = p^{(1/2)} + p^{(-1/2)}, \quad \text{ where } \quad q^{(1/2)} \vcentcolon= \frac{\ell^{(2)} q}{\gamma^{(3/2)}},\\
\text{ and } \quad p^{(-1/2)}\vcentcolon=\frac1{\gamma^{(3/2)}}\Big\{ \ell^{(1)} q -\gamma^{(1/2)}p^{(1/2)}+i\d_\xi\gamma^{(3/2)}\cdot\d_z p^{(1/2)}\Big\}.\nonumber
\end{gather}
\end{subequations}
Then $p$, $q$, and $\gamma$ satisfy \textnormal{\ref{item:S1}}--\textnormal{\ref{item:S4}}. In addition, for matrix $\mathscr{M}_{\lambda,\ell}$ in (\ref{Mlaml}),
\begin{equation}\label{symmetrizer}
\textnormal{if one sets } \ \mathbf{S} \vcentcolon= 
\begin{pmatrix} 
T_{p} & 0 \\ 0 & T_{q} 
\end{pmatrix} \ \textnormal{ and } \ \mathbb{i}T_{\gamma} \vcentcolon= \begin{pmatrix}
0 & -T_{\gamma} \\ T_{\gamma} & 0
\end{pmatrix}, \ \textnormal{ then } \ \mathbf{S} \mathscr{M}_{\lambda,\ell} \sim \mathbb{i} T_{\gamma} \mathbf{S}.
\end{equation}
\end{proposition}
\begin{proof}
Following Metivier \cite{Metivier} page 90, we set 
$a\sharp b\vcentcolon=\sum_{|\alpha|<r} \d^\alpha_\xi a \d_z^\alpha b / (i^{|\alpha|}\alpha !)$ and $a^{\ast}\vcentcolon=\sum_{|\alpha|<r} \d^\alpha_\xi \d_z^\alpha \overline{a}/ (i^{|\alpha|}\alpha !)$ for symbols $a\in \Gamma^{m}_r$ and $b\in \Gamma_r^{m'}$ with some $m$, $m'\in \R$, and $r>0$. %But with some abuse of notation we identify $\sharp$ with the partial sum appearing in \ref{prop:sharp-high-order}, to capture only the high (leading) order symbols.
In view of Theorems \ref{thm:paraL2} and \ref{thm:adjprod}\ref{item:prod}, we set the relation ``$\triangleq$" as follows:
\begin{equation*}
 \text{for } \ a,\,b\in \Gamma^{m}_r(\R), \quad a\triangleq b \ \text{ if and only if } \ a-b \ \text{ is of order } \ m-2 \ \text{ in } \ \xi,
\end{equation*}
With this notation, one can check with Definition \ref{def:sim} that if $a\triangleq b$ then $T_a \sim T_b$. Note also that this is consistent with the covention introduced in Proposition \ref{prop:sharp-high-order}.

To start, we aim to find symbols $q$ and $\gamma$ such that 
\begin{equation}\label{eq:triple-TTT}
T_qT_\ell T_\lambda\sim T_\gamma T_\gamma T_q. 
\end{equation}
If (\ref{eq:triple-TTT}) is obtained, then one can verify from Propositions \ref{prop:CZ}-\ref{prop:coerc} that $T_pT_\lambda\sim T_\gamma T_q$ if and only if $T_q T_\ell\sim T_\gamma T_p$. Thus our strategy is to first determine $\gamma$, $q$ such that (\ref{eq:triple-TTT}) is satisfied. Once this is done, we construct $p$ by solving the relation $T_q T_\ell\sim T_\gamma T_p$, which then implies $T_pT_\lambda\sim T_\gamma T_q$ by (\ref{eq:triple-TTT}).
   
Since $s>\frac{5}{2}$, it holds that $\lambda \in \Gamma_{3/2}^{1}(\R)$ and $\ell \in \Gamma_{3/2}^{2}(\R)$. Thus
\begin{gather*}
\ell \sharp \lambda %=(\ell^{(2)}+\ell^{(1)})\sharp (\lambda^{(1)}+\lambda^{(0)})
=\ell^{(2)}\lambda^{(1)}+\ell^{(2)}\lambda^{(0)}+\ell^{(1)}\lambda^{(1)}+\ell^{(1)}\lambda^{(0)}+ \tfrac{1}{i}\d_{\xi}\ell \cdot \d_z\lambda,\\
\d_\xi \ell \cdot \d_z \lambda %= \d_\xi(\ell^{(2)}+\ell^{(1)})\cdot \d_z(\lambda^{(1)}+\lambda^{(0)})
=  \d_\xi\ell^{(2)}\d_z\lambda^{(1)}
+
\d_\xi\ell^{(2)}\d_z\lambda^{(0)}
+
\d_\xi\ell^{(1)}\d_z \lambda^{(1)}
+
\d_{\xi}\ell^{(1)}\cdot \d_z\lambda^{(0)}
\end{gather*}
Ignoring the low order symbols in accordance with the relation ``$\triangleq$", we have
\begin{equation}\label{llambda}
\ell\sharp\lambda
\triangleq
\ell^{(2)}\lambda^{(1)}+\ell^{(1)}\lambda^{(1)}+\ell^{(2)}\lambda^{(0)}
+
\frac{1}{i} \d_\xi\ell^{(2)}\d_z\lambda^{(1)}.  
\end{equation}
A similar calculation shows that for symbol of the form $\gamma= \gamma^{(3/2)}+\gamma^{(1/2)}$, one has
\begin{equation}\label{gammagamma}
\gamma\sharp\gamma\triangleq |\gamma^{(3/2)}|^2+2\gamma^{(1/2)}\gamma^{(3/2)}+\frac1i\d_\xi \gamma^{(3/2)}\cdot \d_z\gamma^{(3/2)}.   
\end{equation}
With the above relations, the left hand side of \eqref{eq:triple-TTT} has symbol equivalent to:
\begin{align}\label{qllam}
q\sharp(\ell\sharp\lambda)
=&
q^{(0)}(\ell\sharp\lambda)+\frac1i\d_\xi q^{(0)}\cdot \d_z(\ell\sharp\lambda) +
q^{(-1)}(\ell\sharp\lambda)+\frac{1}{i}\d_\xi q^{(-1)}\cdot\d_z(\ell\sharp \lambda)\\
\triangleq&
q^{(0)}(\ell\sharp \lambda)
+
\frac{1}{i}\d_\xi q^{(0)}\cdot \d_z\big(\ell^{(2)}\lambda^{(1)}\big)
+
q^{(-1)}\ell^{(2)}\lambda^{(1)}.\nonumber
\end{align}
On the other hand, the symbol for the right hand side of \eqref{eq:triple-TTT} is equivalent to 
\begin{align}\label{gamma2q}
(\gamma\sharp\gamma)\sharp q
\triangleq&
\big\{ |\gamma^{(3/2)}|^2+2\gamma^{(1/2)}\gamma^{(3/2)}+\frac{1}{i}\d_\xi \gamma^{(3/2)}\cdot \d_z\gamma^{(3/2)}\big\}\sharp \big\{ q^{(0)}+q^{(-1)}\big\}\\
\triangleq&
(\gamma\sharp \gamma)q^{(0)}+\frac1i\d_\xi(\gamma\sharp \gamma)\cdot \d_z q^{(0)} +|\gamma^{(3/2)}|^2q^{(-1)}.\nonumber
\end{align}
Thus for (\ref{eq:triple-TTT}) to hold,  we equate $q\sharp(\ell\sharp\lambda)\triangleq(\gamma\sharp\gamma)\sharp q$. From this, we first impose that the third order terms in $\xi$ must match, meaning $q^{(0)}\ell^{(2)}\lambda^{(1)}=|\gamma^{(3/2)}|^2 q^{(0)}$, which yields
\begin{equation}\label{eq:gamma32-symbol}
\gamma^{(3/2)}\vcentcolon=\sqrt{\ell^{(2)}\lambda^{(1)}}.
\end{equation}
Note that $\textrm{Re}\gamma^{(3/2)}= \gamma^{(3/2)}$ by the expressions given in (\ref{lambda}) and (\ref{ell}). To ensure the condition $T_{\gamma}\sim (T_{\gamma})^{\ast}$ in \ref{item:S1} is satisfied, we impose $\gamma^{\ast}\triangleq \gamma$. Combining this with \eqref{eq:gamma32-symbol} and using Proposition \ref{prop:sharp-high-order}, we see that $\gamma$ must satisfy the equation
\begin{equation*}
\gamma^{\ast}
\triangleq
\gamma^{(3/2)}+\overline{\gamma^{(1/2)}}+\frac{1}{i}\d_z\d_\xi\gamma^{(3/2)} 
\triangleq 
\gamma^{(3/2)}+\gamma^{(1/2)} = \gamma. 
\end{equation*}
Therefore, if we set $\gamma^{(3/2)}$ to take the form in (\ref{eq:gamma32-symbol}), then we must impose that
\begin{equation}\label{eq:gamma12-symbol-im}
\textrm{Im} \gamma^{(1/2)}\vcentcolon=-\frac{1}{2}\d_z\d_\xi\gamma^{(3/2)},
\end{equation}
where $\textrm{Im}f$ denotes the imaginary part of $f$. To find the remaining unknown $\textrm{Re}\gamma^{(1/2)}$, $q^{(0)}$, and $q^{(-1)}$, we equate the second order terms in $q\sharp(\ell\sharp\lambda)\triangleq(\gamma\sharp\gamma)\sharp q$. Substituting (\ref{qllam})--(\ref{gamma2q}) into this relation, and using $\gamma^{(3/2)} = \sqrt{\ell^{(2)}\lambda^{(1)}}$, we have:
\begin{align*}
q^{(0)}(\ell\sharp\lambda-\gamma\sharp\gamma)
\triangleq &
\frac{1}{i}\d_\xi(\gamma\sharp \gamma) \d_z q^{(0)}
\!+
|\gamma^{(3/2)}|^2q^{(-1)}\! - \frac{1}{i}\d_\xi q^{(0)} \d_z\big(\ell^{(2)}\lambda^{(1)}\big)
\!-
q^{(-1)}\ell^{(2)}\lambda^{(1)}\\
=&
\frac{1}{i} \Big\{ \d_\xi(\gamma\sharp \gamma) \d_z q^{(0)} - \d_\xi q^{(0)} \d_z\big(\ell^{(2)}\lambda^{(1)}\big) \Big\}\\
\triangleq& \frac{1}{i} \Big\{ \d_\xi\big(\ell^{(2)}\lambda^{(1)}\big) \d_z q^{(0)}  - \d_\xi q^{(0)} \d_z(\ell^{(2)}\lambda^{(1)})
\Big\}=\vcentcolon
\frac{1}{i} \pbl\ell^{(2)}\lambda^{(1)}, q^{(0)}\pbr.
 \end{align*}
Here $\pbl a, b \pbr\vcentcolon= \d_\xi a \d_z b - \d_\xi b \d_z a$ is the Poisson bracket. Denoting $\tau:=\ell\sharp\lambda-\gamma\sharp\gamma$, we see that the last equation is equivalent to the following two equations: $\textrm{Re}\tau=0$ and $q \text{Im}\tau=-\pbl\ell^{(2)}\lambda^{(1)}, q^{(0)}\pbr$. 
 By (\ref{llambda})--(\ref{gammagamma}) and $\gamma^{(3/2)} = \sqrt{\ell^{(2)}\lambda^{(1)}}$, we get 
 \begin{equation*}
 \tau \triangleq \ell ^{(1)}\lambda^{(1)}+\ell^{(2)}\lambda^{(0)}+\frac{1}{i}\d_\xi\ell^{(2)}\d_z\lambda^{(1)}-2\gamma^{(1/2)}\gamma^{(3/2)}
 -\frac{1}{i} \d_\xi\gamma^{(3/2)}\d_z\gamma^{(3/2)}.
 \end{equation*}
According to (\ref{lambda}) and (\ref{ell}), the first, third, and fifth terms in $\tau$ are purely imaginary, hence 
$0=\textrm{Re}\tau=\textrm{Re}\big(\ell^{(2)}\lambda^{(0)}-2\gamma^{(1/2)}\gamma^{(3/2)}\big)$. Since $\textrm{Re}\gamma^{(3/2)}=\gamma^{(3/2)}$, we set
\begin{equation}\label{eq:gamma12-re}
\textrm{Re}\gamma^{(1/2)}\vcentcolon=\sqrt{\frac{\ell^{(2)}}{\lambda^{(1)}}}\frac{\textrm{Re}\lambda^{(0)}}2.
\end{equation}
With this construction, we see that the real part of $\gamma$ is a multiple of $|\xi|^{3/2}$ while the imaginary part is $|\xi|^{-1/2}\xi$, thus it can be verified that $\overline{\gamma(t,z,\xi)}=\gamma(t,z,-\xi)$.
 
To obtain $q^{(0)}$, we use the last equation  $q \textrm{Im}\tau=-\pbl\ell^{(2)}\lambda^{(1)}, q^{(0)}\pbr$. Substituting (\ref{llambda}) and (\ref{gammagamma}) into $\tau = \ell \sharp \lambda - \gamma \sharp \gamma$, and by (\ref{eq:gamma32-symbol})--(\ref{eq:gamma12-symbol-im}), we find that 
\begin{align*}
\textrm{Im}\tau
\triangleq&
\lambda^{(1)} \textrm{Im} \ell ^{(1)}+\ell^{(2)}\textrm{Im} \lambda^{(0)}-\d_\xi\ell^{(2)}\d_z\lambda^{(1)}-2\gamma^{(3/2)}\textrm{Im} \gamma^{(1/2)}
 + \d_\xi\gamma^{(3/2)}\d_z\gamma^{(3/2)}\\
 =&
\lambda^{(1)} \textrm{Im} \ell ^{(1)}+\ell^{(2)}\textrm{Im} \lambda^{(0)}-\d_\xi\ell^{(2)}\d_z\lambda^{(1)}+\gamma^{(3/2)}\d_z\d_\xi\gamma^{(3/2)}
 + \d_\xi\gamma^{(3/2)}\d_z\gamma^{(3/2)}\\
=& 
\lambda^{(1)} \textrm{Im} \ell ^{(1)}+\ell^{(2)}\textrm{Im} \lambda^{(0)}%+\gamma^{(3/2)}\d_z\d_\xi\gamma^{(3/2)}
 + 
\frac{1}{2}  \d_\xi\d_z(\gamma^{(3/2)})^2,
\end{align*}
where the last line follows from the fact $\lambda^{(1)}=|\xi|$ is independent of $z$. Furthermore, using the expression for $\lambda^{(0)}$, $\ell^{2}$, and $\ell^{1}$ given in (\ref{lambda}) and (\ref{ell}), we get
\begin{align*}
\lambda^{(1)} \textrm{Im} \ell ^{(1)}+\ell^{(2)}\textrm{Im} \lambda^{(0)}
%=& |\xi| \dfrac{\xi\d_z\eta (3\eta\d_z^2\eta-1-|\d_z\eta|^2)}{2\eta(1+|\d_z\eta|^2)^{5/2}} - \dfrac{|\xi|^2}{2(1+|\d_z\eta|^2)^{3/2}} \dfrac{ (\d_z\eta)^3 \sgn(\xi)}{2\eta}\\ =& \frac{\xi|\xi|\d_z\eta}{2\eta(1+|\d_z\eta|^2)^{3/2}} \left( \dfrac{(3\eta\d_z^2\eta-1-|\d_z\eta|^2)}{1+|\d_z\eta|^2}-\frac{|\d_z\eta|^2}2\right)\\
=&
\frac{\xi|\xi|\d_z\eta}{2\eta(1+|\d_z\eta|^2)^{3/2}}
\left( \dfrac{3\eta\d_z^2\eta}{1+|\d_z\eta|^2}-1-\frac{|\d_z\eta|^2}2\right).
\end{align*}
Moreover, using $(\gamma^{(3/2)})^2=\ell^{(2)}\lambda^{(1)}$, we also have 
\begin{align*}
\frac{1}{2}  \d_\xi\d_z(\gamma^{(3/2)})^2 
=
\frac{1}{2}  \d_\xi\d_z\left(\dfrac{|\xi|^3}{2(1+|\d_z\eta|^2)^{3/2}}\right) 
=
-\frac{9}{4}\xi|\xi| \dfrac{\d_z\eta\d^2_z\eta}{(1+|\d_z\eta|^2)^{5/2}}.
\end{align*}
Substituting these expression into $\textrm{Im}\tau$, we obtain that 
\begin{align*}
\textrm{Im} \tau =&
\frac{\xi|\xi|\d_z\eta}{2\eta(1+|\d_z\eta|^2)^{3/2}}
\left( \dfrac{3\eta\d_z^2\eta}{1+|\d_z\eta|^2}-1-\frac{|\d_z\eta|^2}2\right)
-
\frac{9}{4}\xi|\xi| \dfrac{\d_z\eta\d^2_z\eta}{(1+|\d_z\eta|^2)^{5/2}}\\
%=& \dfrac{\xi |\xi|}{(1+|\d_z\eta|^2)^{3/2}} \bigg\{ \frac{\d_z\eta}{2\eta} \left( \dfrac{3\eta\d_z^2\eta}{1+|\d_z\eta|^2}-1-\frac{|\d_z\eta|^2}2\right) - \dfrac{9\d_z\eta\d^2_z\eta}{4(1+|\d_z\eta|^2)} \bigg\}\\ =& \dfrac{\xi |\xi|}{(1+|\d_z\eta|^2)^{3/2}}  \bigg\{  \dfrac{3\d_z\eta \d_z^2\eta}{2(1+|\d_z\eta|^2)} -\frac{\d_z\eta}{2\eta} -\frac{(\d_z\eta)^3}{4\eta} - \dfrac{9\d_z\eta\d^2_z\eta}{4(1+|\d_z\eta|^2)} \bigg\}\\
=&\dfrac{\xi |\xi|}{(1+|\d_z\eta|^2)^{3/2}} \bigg\{ - \dfrac{3 \d_z\eta \d_z^2\eta}{4(1+|\d_z\eta|^2)}
-\frac{\d_z\eta}{2\eta}
-\frac{(\d_z\eta)^3}{4\eta} \bigg\}
\end{align*}
Thus if we impose that $q^{(0)}$ depends only on $z$, and $q^{(-1)}=0$, then from the expressions (\ref{lambda}), (\ref{ell}), and the equation $q \textrm{Im} \tau = -\pbl \ell^{(2)}\lambda^{(1)}, q^{(0)} \pbr$, it follows that $q^{(0)}$ solves 
\begin{align*}
\frac{\d_z q^{(0)}}{q^{(0)}}=\dfrac{\textrm{Im} \tau}{\d_\xi(\ell^{(2)}\lambda^{(1)})} = -\dfrac{2}{3}\bigg\{ \dfrac{3 \d_z\eta \d_z^2\eta}{4(1+|\d_z\eta|^2)}
+\frac{\d_z\eta}{2\eta}
+\frac{(\d_z\eta)^3}{4\eta}\bigg\}.
\end{align*}
To make sure this differential equation holds, we set $q=q^{(0)}$ to be 
\begin{equation*}
q=q^{(0)} \vcentcolon= \frac{R^{1/3}}{\eta^{1/3} (1+|\d_z\eta|^2)^{1/4}}\exp\Big(-\int_{-\infty}^{z} \frac{(\d_z\eta)^3}{6\eta}(t,y)\,\dif y\Big).
\end{equation*}
To determine $p=p^{(1/2)}+p^{(-1/2)}$, we find it from the relation $T_qT_\ell\sim T_\gamma T_p$, which holds if $q\sharp \ell \triangleq \gamma \sharp p$. Modulo terms of order $0$ in $\xi$, we have that 
the symbol of $T_qT_\ell$ is given by
\begin{align*}
	q\sharp \ell= q(\ell^{(2)}+\ell^{(1)})+\frac{1}{i}\d_\xi q\d_z\ell \triangleq  \ell^{(2)}q + \ell^{(1)}q + \frac{1}{i} \d_\xi q \d_z \ell^{(2)},
\end{align*}
and for $T_\gamma T_p$ the symbol is given by
\begin{align*}
\gamma\sharp p
=&
(\gamma^{(3/2)}+\gamma^{(1/2)})(p^{(1/2)}+p^{(-1/2)})+\frac1i\d_\xi\gamma\d_zp\\
\triangleq&
\gamma^{(3/2)}p^{(1/2)}+\gamma^{(1/2)}p^{(1/2)}+\gamma^{(3/2)}p^{(-1/2)}+\frac1i\d_\xi\gamma^{(3/2)}\d_z p^{(1/2)}.
\end{align*}
Equating the second order terms in $q\sharp \ell \triangleq \gamma \sharp p$, we find that  
\begin{equation}
p^{(1/2)}\vcentcolon=\frac{\ell^{(2)} q}{\gamma^{(3/2)}}.
\end{equation}
Moreover, comparing the first order terms in $q\sharp \ell \triangleq \gamma \sharp p$, we also find that 
\[
\ell^{(1)}q +\frac1i \d_\xi q \d_z\ell^{(2)}
=
\gamma^{(1/2)}p^{(1/2)}+\gamma^{(3/2)}p^{(-1/2)}+\frac1i\d_\xi\gamma^{(3/2)}\d_z p^{(1/2)}.
\]
In order to satisfy this, we set $p^{(-1/2)}$ as 
\begin{equation}
p^{(-1/2)}
\vcentcolon=\frac{1}{\gamma^{(3/2)}}\Big\{ \ell^{(1)} q +\frac{1}{i} \d_\xi q\d_z\ell^{(2)}-\gamma^{(1/2)}p^{(1/2)}-\frac{1}{i}\d_\xi\gamma^{(3/2)}\d_z p^{(1/2)}\Big\}.
\end{equation}
By construction, the real part of $p^{(-1/2)}$ is a function of $|\xi|$
whereas the imaginary part is of the form $\xi|\xi|^{-3/2}$. Thus 
$\overline {p(t, z, \xi)}=p(t, z,-\xi)$.  

To finish the proof, we show the matrix operator relation (\ref{symmetrizer}). From the above construction, we have that $T_{p}T_{\lambda}\sim T_{\gamma} T_{q}$ and $T_{q} T_{\ell} \sim T_{\gamma} T_{p}$. Thus
\begin{align*}
\mathbf{S} \mathscr{M}_{\lambda,\ell} =& 
\begin{pmatrix}
	T_{p} & 0 \\
	0     & T_{q}
\end{pmatrix}
\begin{pmatrix}
	0        & -T_{\lambda}\\
	T_{\ell} & 0
\end{pmatrix} \\
=&
\begin{pmatrix}
	0 & - T_{p} T_{\lambda} \\
	T_{q} T_{\ell} & 0
\end{pmatrix} 
\sim
\begin{pmatrix}
	0 & - T_{\gamma} T_{q} \\
	T_{\gamma} T_{p} & 0
\end{pmatrix}
=
\begin{pmatrix}
	0 & - T_{\gamma}\\
	T_{\gamma} & 0
\end{pmatrix} 
\begin{pmatrix}
	T_{p} & 0\\
	0     & T_{q}  
\end{pmatrix}
= 
\mathbb{i}T_{\gamma} \mathbf{S}
\end{align*}
This concludes the proof of proposition.
\end{proof}

We also introduce the parametrix $\mathbf{Q}$ for the symmetrizer $\mathbf{S}=\big(\begin{smallmatrix} T_{p} & 0 \\ 0 & T_{q} \end{smallmatrix}\big)$, which will be useful later in Section \ref{ssec:mollified}. More precisely, we look for a $2\times 2$ matrix operator $\mathbf{Q}$ such that $\mathbf{S} \mathbf{Q} \sim I$ where $I$ denotes the $2\times 2$ identity matrix.

First, we observe that the symbol $q=(1+|\d_z\eta|^2)^{-1/2}$ does not depend on $\xi$. By Theorem \ref{thm:adjprod}\ref{item:prod}, one has $T_{q} T_{1/q} \sim 1$. Thus we set the $(2,2)$ entry of $\mathbf{Q}$ to be $T_{1/q}$. In addition, we wish to find $\wp = \wp^{(1/2)} + \wp^{(-1/2)} \in \Sigma^{-1/2}(\R) + \Sigma^{-3/2}(\R)$ such that 
\begin{equation*}
    p\sharp \wp = p^{(1/2)} \wp^{(-1/2)} + p^{(1/2)} \wp^{(-3/2)} + p^{(-1/2)} \wp^{(-1/2)} + \dfrac{1}{i} \d_\xi p^{(1/2)} \d_z \wp^{(-1/2)} = 1  
\end{equation*}
To fulfill this condition, we explicitly set
\begin{equation}\label{wp}
    \wp^{(-1/2)}\vcentcolon= \dfrac{1}{p^{(1/2)}}, \quad \wp^{(-3/2)}\vcentcolon= - \dfrac{1}{p^{(1/2)}} \Big\{ p^{(-1/2)} \wp^{(-1/2)} + \dfrac{1}{i} \d_\xi p^{(1/2)}  \d_z \wp^{(-1/2)} \Big\}.
\end{equation}
With this construction we see that the symbol $p\sharp \wp -1$ is of degree at most $-2$ in $\xi$. Thus Theorem \ref{thm:adjprod}\ref{item:prod} and Definition \ref{def:sim} implies that $T_{p} T_{\wp} \sim T_{p\sharp \wp} \sim 1$. In summary,
\begin{equation}\label{parametrix}
\text{if we define } \quad \mathbf{Q}\vcentcolon= \begin{pmatrix}
        T_{\wp} & 0 \\ 0 & T_{q} 
    \end{pmatrix} \quad \text{then} \quad \mathbf{S}\mathbf{Q} \sim I.
\end{equation}
\subsection{Symmetrized system}

\begin{lemma}\label{lemma:dtpq}
There is a non-decreasing function $C(x)>0$ such that for $\mu\in\R$,
\[
\|T_{\d_t p(t)}\|_{H^\mu\to H^{\mu-1/2}}+
\|T_{\d_tq(t)}\|_{H^\mu\to H^\mu}
\le
C(\|(\eta(t), \psi(t))\|_{H^{s+1/2}\times H^s}), \quad \text{for all } t\in[0,T].
\]
\end{lemma}
\begin{proof}
By Sobolev embedding theorem $H^{s-1}(\R)\xhookrightarrow[]{} W^{1,\infty}(\R)$ (see Definition \ref{def:holder}),
\begin{equation*}
\|\d_t \eta \|_{W^{1,\infty}(\R)} \le C \|\d_t \eta\|_{H^{s-1}(\R)} \le C\big(\|\teta\|_{H^{s+1/2}(\R)}\big)\|\psi\|_{H^{s}(\R)}. 
\end{equation*}
By (\ref{q})--(\ref{p}), it follows that
\begin{equation*}
\|\d_t q\|_{L^{\infty}} + \mM_0^1\big(\d_t p^{(1/2)}\big) \le C\big(\|(\teta,\psi)\|_{H^{s+1/2}\times H^{s}}\big),
\end{equation*}
where $\mM_0^1(\cdot)$ is the semi-norm defined in Definition \ref{def:symbols}. Theorem \ref{thm:paraL2} implies
\begin{equation*}
\big\|T_{\d_t p^{(1/2)}}\big\|_{H^{\mu}\to H^{\mu-1/2}} + \big\|T_{\d_t q}\big\|_{H^{\mu}\to H^{\mu}} \le C\big(\|(\teta,\psi)\|_{H^{s+\frac{1}{2}}\times H^{s}}\big).
\end{equation*}
Since $p\in \Sigma^{1/2}(\R)$, by Definition \ref{def:Sigma}$, p^{(-1/2)}$ takes the form:
\begin{equation*}
p^{(-1/2)}(t,z,\xi) = \sum_{|\alpha|=2} \mathcal{Z}_{\alpha}\big(\d_z\eta(t,z),\xi\big) \d_z^{\alpha} \eta,
\end{equation*}
where $\xi\mapsto \mathcal{Z}_{\alpha}(\d_z\eta, \xi)$ is homogeneous of degree $-\frac{1}{2}$. Then we have
\begin{equation*}
\d_t p^{(-1/2)} = \sum_{|\alpha|=2}\big\{ \d_z^{\alpha} \eta \d_t \mathcal{Z}_{\alpha}\big(\d_z\eta(t,z),\xi\big) + \mathcal{Z}_{\alpha}\big(\d_z\eta(t,z),\xi\big) \d_t \d_z^{\alpha} \eta \big\}=\vcentcolon p_{\ast} + p_{\ast\ast}.
\end{equation*}
Since $s>\frac{5}{2}$ and $\teta\in H^{s+\frac{1}{2}}(\R)$, we have $\|\d_z^{\alpha} \eta \|_{L^{\infty}}\le \|\teta\|_{H^{s+1/2}}$ for $|\alpha|=2$. Moreover,
\begin{equation*}
	\mM_{0}^{-1/2}\big( \d_t \mathcal{Z}_{\alpha} (\nabla \eta,\xi) \big) \le C\big(\|(\teta,\psi)\|_{H^{s+\frac{1}{2}}\times H^{s}}\big).
\end{equation*}
This implies that $p_{\ast}\in \Gamma_{0}^{-1/2}(\R)$ and $T_{p_\ast}$ is of order $-1/2$. Next, we consider the symbol $p_{\ast\ast}$. Since $\d_t\d_z^{\alpha} \eta \in H^{s-3}(\R)$ for $|\alpha|=2$, Propositions \ref{prop:Sobcomp}--\ref{prop:clprod} implies
\begin{align*}
&\|p_{\ast\ast}\|_{H^{s-3}} = \|\mathcal{Z}_{\alpha}(\d_z\eta,\xi)\d_t\d_z^{\alpha}\eta\|_{H^{s-1}}\\
\le& C \big\{ |\mathcal{Z}_{\alpha}(0,\xi)| + \| \mathcal{Z}_{\alpha}(\d_z\eta,\xi) - \mathcal{Z}_{\alpha}(0,\xi) \|_{H^{s-\frac{1}{2}}} \big\} \|\d_t \d_z^{\alpha} \eta\|_{H^{s-3}}\le C\big(\|(\teta,\psi)\|_{H^{s+\frac{1}{2}}\times H^{s}}\big).
\end{align*}
Since $\xi\mapsto \mathcal{Z}_{\alpha}(\d_z \eta, \xi)$ is homogeneous of order $-1/2$, one can apply Proposition \ref{prop:Hsym} with $r=-1/2$ to get
\begin{equation*}
	\|T_{p_{\ast\ast}}\|_{H^{\mu}\to H^{\mu-\frac{1}{2}}} \le C\sup\limits_{|\alpha|\le d/2+1}\sup\limits_{|\xi|=1}\|\d_{\xi}^{\alpha}p_{\ast\ast}(\cdot,\xi)\|_{H^{s-3}} \le C\big(\|(\teta,\psi)\|_{H^{s+\frac{1}{2}}\times H^{s}}\big).
\end{equation*}
Since $\d_t p^{(1/2)}=p_{\ast}+p_{\ast\ast}$, one has $T_{\d_t p^{(1/2)}}$ is of order $1/2$. This concludes the proof.
\end{proof}

\begin{theorem}\label{thm:symPHI}
Let $\mathbf{S}$ be the symmetrizer given in (\ref{symmetrizer}), and define $\Phi = (\Phi_1,\Phi_2)^{\top}$ as
\begin{equation}\label{PHI}
\Phi\vcentcolon= \mathbf{S} \begin{pmatrix} \eta \\ U \end{pmatrix}, \quad \textrm{i.e. } 
\quad  \Phi_1\vcentcolon=T_p\eta, \quad \Phi_2\vcentcolon=T_q U.
\end{equation}
Then $\Phi \in \mC^0\big([0, T];H^s(\R)\times H^{s}(\R)\big)$ and it solves the symmetrized system: 
\begin{gather}\label{PHIeq}
(\d_t + T_{V} \d_z ) \Phi + \mathbb{i} T_{\gamma} \Phi = \fF \vcentcolon= \begin{pmatrix}
    \fF^1 \\ \fF^2 
\end{pmatrix} \ \text{ with } \\
\mathbb{i}T_{\gamma}\vcentcolon= \begin{pmatrix}
    0 & - T_{\gamma} \\
    T_{\gamma} & 0
\end{pmatrix} \quad \text{ and } \quad \|\fF\|_{L^\infty(0, T; H^s\times H^s)}\le C\big(\|(\teta, \psi)\|_{L^\infty(0, T;H^{s+1/2}\times H^s)}\big),\nonumber
\end{gather}
where $x\mapsto C(x)$ is a strictly positive non-decreasing function.
\end{theorem}
\begin{proof} 
Applying $\mathbf{S}$ on the paralinearised system (\ref{zak-para}) from the left, we get
\begin{equation*}
(\d_t + T_{V}\d_z) \Phi + \mathbb{i}T_{\gamma} \Phi =\fF \vcentcolon= \mathbf{S}\! \begin{pmatrix}
    f^1_{\ast} \\ f^2_{\ast}
\end{pmatrix}\! - \! \big[\mathbf{S}, \d_t+T_{V}\d_z \big]\! \begin{pmatrix}
    \eta \\ U
\end{pmatrix}\! - \! \big(\mathbf{S} \mathscr{M}_{\lambda, \ell} - \mathbb{i}T_{\gamma} \mathbf{S} \big)\! \begin{pmatrix}
    \eta \\ U
\end{pmatrix}.
\end{equation*}
By Propositions \ref{prop:JohnWick-4}, \ref{prop:qpgamma}, and Theorem \ref{thm:paraL2}, we get 
\begin{equation*}
  \|(T_{p}f_{\ast}^1, T_{q}f_{\ast}^2)\|_{L^\infty(0, T; H^s\times H^s)}\le C\big(\|(\teta, \psi)\|_{L^\infty(0, T;H^{s+1/2}\times H^s)}).  
\end{equation*}
Next, the commutator is evaluated as 
\begin{align*}
\big[\mathbf{S}, \d_t + T_{V}\d_z \big]
\begin{pmatrix}
\eta \\ U
\end{pmatrix}
=
\begin{pmatrix}
[\d_t + T_V\d_z, T_p] \eta \\
[\d_t + T_V\d_z, T_q] U 
\end{pmatrix} 
= 
\begin{pmatrix}
T_{\d_t p} \eta + [T_V\d_z, T_p] \eta \\
T_{\d_t q} U+ [ T_V\d_z, T_q]U 
\end{pmatrix}.
\end{align*}
By constructions \ref{item:S1}--\ref{item:S4}, Lemma \ref{lemma:dtpq}, and Theorems \ref{thm:paraL2}--\ref{thm:bony}, we have
\begin{align*}
\| T_{\d_t q}U\|_{H^s} \le& \|T_{\d_t q}\|_{H^{s}\to H^s}\|U\|_{H^{s}} \le C\big(\|(\teta, \psi)\|_{H^{s+1/2}\times H^s}\big), \\
\| T_{\d_t p}\eta\|_{H^s} \le&
 \|T_{\d_t p}\|_{H^{s+1/2}\to H^s}\|\teta\|_{H^{s+1/2}}\le C\big(\|(\teta, \psi)\|_{H^{s+1/2}\times H^s}\big),\\
\big\|[T_{V}\d_z, T_{p}]\eta \big\|_{H^s}  \le& \big\|[T_{V}\d_z, T_{p}] \big\|_{H^{s+1/2}\to H^s}\|\eta\|_{H^{s+1/2}}\le C\big(\|(\teta, \psi)\|_{H^{s+1/2}\times H^s}\big),\\ 
\big\|[T_{V}\d_z, T_{q}]U \big\|_{H^s}  \le& \big\|[T_{V}\d_z, T_{q}] \big\|_{H^s\to H^s}\|U\|_{H^s}\le C\big(\|(\teta, \psi)\|_{H^{s+1/2}\times H^s}\big).
\end{align*}
Finally, recall the relation ``$\sim$" in Definition \ref{def:sim}. Then property \ref{item:S1} and (\ref{symmetrizer}) imply
\begin{equation*}
\big\|\mathbf{S} \mathscr{M}_{\lambda,\ell} - \mathbb{i} T_{\gamma} \mathbf{S} \big\|_{H^{\alpha}\times H^{\beta} \to H^{\beta} \times H^{\alpha-1/2} }  \le C\big( \| \teta \|_{H^{s+1/2}} \big) \ \text{ for } \ (\alpha,\beta)\in \R^2.
\end{equation*}
Setting $(\alpha,\beta)=(s+\frac{1}{2},s)$, it then follows that
\begin{equation*}
\bigg\|\big(\mathbf{S} \mathscr{M}_{\lambda, \ell} - \mathbb{i}T_{\gamma} \mathbf{S} \big)\! \begin{pmatrix}
    \eta \\ U
\end{pmatrix} \bigg\|_{H^{s}\times H^s} \le C\big( \| \teta \|_{H^{s+1/2}} \big) \|(\teta,U)\|_{H^{s+1/2}\times H^{s}}.
\end{equation*}
Combining the above estimates, we conclude that $\fF\in L^{\infty}\big(0,T; H^s\times H^s\big)$.
\end{proof}
%%%%%%%%%%%%%%%

%---------------------
%     Section
%---------------------
\section{A priori estimates for mollified system}\label{sec:mollified-system}

The main result of this section is the following 
Theorem \ref{thm:InvEst}.

\subsection{An equivalent form of the system and operator \texorpdfstring{$\mathscr L$}{L} }

The next Lemma hints that one can make a change of variables $(\eta, \psi)$ to 
$(\eta, U)$.
\begin{lemma}\label{lemma:MB-multiplied}
The system \eqref{zak-cyl1} can be written in the following equivalent form 
\begin{subequations}
\begin{gather}\label{eq:MB-multiplied}
\begin{pmatrix}
1& 0\\
-T_\mB & 1
\end{pmatrix}
(\d_t+T_V\cdot \d_z)
\begin{pmatrix}
\eta\\
\psi
\end{pmatrix}\!
+
\!\begin{pmatrix}
0 & -T_\lambda\\
T_\ell & 0
\end{pmatrix}\!
\begin{pmatrix}
1 & 0\\
-T_\mB & 1
\end{pmatrix}\!
\begin{pmatrix}
\eta\\
\psi
\end{pmatrix}
=
\begin{pmatrix}
f^1\\
f^2
\end{pmatrix} \\
\label{eq:JohnWickf1}   \text{where } \ \left\{\begin{aligned} 
f^1 \vcentcolon=&
G[\eta]\psi-\big(T_\lambda(\psi-T_\mB\eta)-T_V\d_z\eta \big), \\
f^2 \vcentcolon=&
-\frac{1}{2}|\d_z\psi|^2+\frac{1}{2}\frac{(\d_z\eta \d_z\psi +G[\eta]\psi)^2}{1+|\d_z\eta|^2}+\mH(\eta)+\dfrac{1}{2R} \\
&+T_V\d_z \psi -T_\mB T_V \d_z\eta-T_\mB G[\eta]\psi+T_\ell\eta,
\end{aligned}\right.
\end{gather}
\end{subequations}
and $\| (f^1,f^2) \|_{L^{\infty}(0,T; H^{s+1/2}\times H^{s})} \le C\big( \| (\teta,\psi) \|_{L^{\infty}(0,T; H^{s+1/2}\times H^{s})} \big)$.
\end{lemma}
\begin{proof}
By the first equation of (\ref{zak-cyl1}),
\[
-T_\mB(\d_t \eta +T_V\d_z \eta )
=
-T_\mB (G[\eta]\psi+T_V\d_z \eta)=-T_\mB G[\eta]\psi-T_\mB T_V \d_z \eta
\]
Using the above and (\ref{zak-cyl1}), we obtain
\begin{align*}
&
\begin{pmatrix}
I & 0\\
-T_\mB & I
\end{pmatrix}
\begin{pmatrix}
\d_t \eta + T_V \d_z \eta\\
\d_t \psi + T_V \d_z \psi
\end{pmatrix}
+
\begin{pmatrix}
0 & -T_\lambda\\
T_\ell & 0
\end{pmatrix}
\begin{pmatrix}
1 & 0\\
-T_\mB & 1
\end{pmatrix}
\begin{pmatrix}
\eta\\
\psi
\end{pmatrix}\\
%%%%%%%%%%%%%%%%%%%%%%
=&
\begin{pmatrix}
\d_t \eta+T_V \d_z \eta\\
-T_\mB ( \d_t \eta +T_V \d_z \eta)+\d_t\psi+T_V \d_z \psi
\end{pmatrix}
+
\begin{pmatrix}
-T_{\lambda}\psi + T_{\lambda}T_{\mB}\eta \\
T_{\ell} \eta
\end{pmatrix}
~ = ~
\begin{pmatrix}
f^1\\
f^2
\end{pmatrix},
\end{align*}
where the last line follows from the definition of 
$f^1$ and $f^2$. Finally, one can obtain the estimate of $(f^1,f^2)$ using Theorem \ref{thm:paraG} and Lemmas \ref{lemma:mean}--\ref{lemma:ParaDyn}.
\end{proof}
Using the identity $
\big(\begin{smallmatrix}
1 & 0\\
T_\mB & 1
\end{smallmatrix}\big)
\big(\begin{smallmatrix}
1 & 0\\
-T_\mB & 1
\end{smallmatrix}\big)
=
\big(\begin{smallmatrix}
1 & 0\\
0 & 1
\end{smallmatrix}\big) $, we rewrite the system \eqref{eq:MB-multiplied} as 
\begin{gather}\label{eq:the-L-system}
\left\{
\begin{aligned}
&(\d_t+T_V\d_z+\mathscr L)\begin{pmatrix}
\eta\\
\psi
\end{pmatrix}
=
\tilde{f}(\eta, \psi),\\
%%%%
%%%%
&(\eta, \psi)|_{t=0}=(\eta_0, \psi_0),
\end{aligned}
\right.\\
\label{eq:defn-of-L-and RHS}
\text{where } \ \mathscr{L}:=\!
\begin{pmatrix}
1 & 0\\
T_\mB & 1
\end{pmatrix}\!
\begin{pmatrix}
0 & -T_\lambda\\
T_\ell & 0
\end{pmatrix}\!
\begin{pmatrix}
1 & 0\\
-T_\mB & 1
\end{pmatrix}, \quad
%%%%%%%%
\tilde{f}(\eta, \psi)\vcentcolon=\!\begin{pmatrix}
1 & 0\\
T_{\mB} & 1
\end{pmatrix}\!
\begin{pmatrix}
f^1\\
f^2
\end{pmatrix}.
\end{gather}
It is convenient to introduce the following notations for matrices: 
\begin{subequations}\label{scrM}
\begin{alignat}{3}
&\mathscr{M}_{\mB}\vcentcolon= 
\begin{pmatrix} 
1 & 0\\ 
-T_\mB & 1
\end{pmatrix},
\qquad
&&\mathscr{M}_{\mB}^{-1}\vcentcolon=
\begin{pmatrix} 
1 & 0\\ 
T_\mB & 1
\end{pmatrix},
\qquad
&&\mathscr{M}_{\lambda,\ell}\vcentcolon=
\begin{pmatrix}
0 & -T_\lambda\\
T_\ell & 0
\end{pmatrix}, \\
&\mathbf{S} \vcentcolon=
\begin{pmatrix}
T_p &0\\
0& T_q
\end{pmatrix}, 
\qquad
&& \mathbb{i} T_{\gamma}\vcentcolon=
\begin{pmatrix}
0 & -T_\gamma\\
T_\gamma & 0
\end{pmatrix},
\qquad
&&\mathscr{L}\vcentcolon=
\mathscr{M}_{\mB}^{-1}\mathscr{M}_{\lambda, \ell}\mathscr{M}_\mB.
\end{alignat}
\end{subequations}

\subsection{The mollifier \texorpdfstring{$J_\epsilon$}{J\_epsilon} and the system for \texorpdfstring{$\Phi$}{Phi} with \texorpdfstring{$\mathscr{L}_\ep$}{Lep}}\label{ssec:mollified}
An important point for the choice of mollifier $J_{\ep}$ is that it must ``preserves" the structure of the symmetrized system (\ref{PHIeq}), in the sense that
\begin{equation*}
    \begin{pmatrix}
        0 & - T_{\gamma} J_{\ep}\\
        (T_{\gamma})^{\ast} J_{\ep} & 0 
    \end{pmatrix}^{\ast} \ \sim \ -\begin{pmatrix}
        0 & - T_{\gamma} J_{\ep}\\
        (T_{\gamma})^{\ast} J_{\ep} & 0 
    \end{pmatrix},
\end{equation*}
where the relation ``$\sim$" is defined in Definition \ref{def:sim}. To get this, we set for $\ep\in(0,1]$,
\begin{gather}
	J_{\ep}(t,z,\xi) = J_\ep^{(0)} + J_\ep^{(-1)} \vcentcolon=J_\ep^{(0)}(t,z,\xi) - \dfrac{i}{2} \d_z \d_\xi J_{\ep}^{(0)}(t,z,\xi) \label{Je} \quad \text{where} \\
	J_{\ep}^{(0)}(t,z,\xi) \vcentcolon= \exp\Big(\!-\!\ep \gamma^{(3/2)}(t,z,\xi)\Big) \ \text{ and } \ \gamma^{(3/2)}(t,z,\xi) = \dfrac{|\xi|^{3/2}}{\sqrt{2}(1+|\d_z\eta|^2)^{3/4}},\nonumber
\end{gather}
An important operation for a commutator $[T_{a},T_b]$ is the Poisson bracket:
\[
\pbl
a, b
\pbr
\vcentcolon=\d_\xi a\cdot\d_z b-\d_\xi b\cdot\d_z a.
\]
Then by the construction in (\ref{Je}), one can verify that
\begin{equation*}
\pbl J_{\ep}^{(0)}, \gamma^{(3/2)} \pbr=0, \quad \textrm{Im} J_{\ep}^{(-1)} = -\tfrac{1}{2} \d_z \d_\xi J_{\ep}^{(0)}, \quad J_{\ep} \in \mC^0\big( [0,T]; \Gamma_{s-1}^0(\R) \big) \text{ uniformly in } \ep.    
\end{equation*}
We note that for a fixed $\ep>0$, one has $J_{\ep} \in \mC^0\big( [0,T]; \Gamma_{s-1}^{m}(\R) \big)$ for all $m\ge 0$, however its norm in such space is no longer uniformly bounded in $\ep$. As a consequence of these,
\begin{equation}\label{JepCom}
    \big\| [J_{\ep}, T_{\gamma} ] \big\|_{H^{\mu}\to H^{\mu} } + \big\| (J_{\ep})^{\ast} - J_{\ep} \big\|_{H^{\mu}\to H^{\mu+3/2} } \le C\big( \|\d_z \eta\|_{W^{s-1,\infty}(\R)} \big),
\end{equation}
where $[A,B]\equiv AB -BA$, and $x\mapsto C(x)$ is some strictly positive non-decreasing function which is independent of $\ep$. In the language of (\ref{sim}), this means
\begin{equation*}
    J_{\ep} T_{\gamma} \sim T_{\gamma} J_{\ep} \ \text{ and } \ (J_{\ep})^{\ast} \sim J_{\ep} \ \text{ uniformly in } \ \ep\in (0,1].
\end{equation*}
Recall the parametrix operator $\mathbf{Q}$ for $\mathbf{S}$ defined in (\ref{parametrix}). Using it, we define:
\begin{equation}\label{sLep}
\mathscr{L}_\ep\vcentcolon= \sM_{\mB}^{-1} \cdot \sM_{\lambda, \ell}\cdot \mathbf{Q} \cdot (J_{\ep} I) \cdot \mathbf{S} \cdot \sM_{\mB}  =\sM_\mB^{-1}\sM_{\lambda, \ell}
\begin{pmatrix}
T_{\wp}J_\ep T_p & 0\\
0& T_{1/q}J_\ep T_q
\end{pmatrix}
\sM_\mB,
\end{equation}
where $I$ denotes the $2\times 2$ identity matrix. With this, we consider solution $(\eta, \psi)$ for the mollified system: 
\begin{equation}\label{eq:the-L-system-ep}
\left\{
\begin{aligned}
&(\d_t+T_V\d_z+\mathscr L_\ep)\begin{pmatrix}
\eta\\
\psi
\end{pmatrix}
=
\tilde{f}(J_\ep\eta, J_\ep\psi),\\
&(\eta, \psi)|_{t=0}=(\eta_0, \psi_0),
\end{aligned}
\right.
\end{equation}
where $\tilde{f}$ is given by \eqref{eq:defn-of-L-and RHS}, and $(\eta_0-R,\psi_0)\in H^{s+1/2}(\R)\times H^s(\R)$ is the initial data. Note that since $J_{\ep} \in \Gamma_{s-1}^0(\R)$ uniformly in $\ep$, by Lemma \ref{lemma:MB-multiplied} and Theorem \ref{thm:paraL2}, the following estimate holds uniformly in $\ep$:
\begin{equation}\label{tildefEst}
\|\tilde{f}(J_{\ep}\eta,J_{\ep}\psi)\|_{L^{\infty}(0,T;H^{s+1/2}\times H^{s})} %\le C\big( \|(J_{\ep}\teta,J_{\ep}\psi)\|_{L^{\infty}(0,T;H^{s+1/2}\times H^{s})}\big) 
\le C\big( \|(\teta,\psi)\|_{L^{\infty}(0,T;H^{s+1/2}\times H^{s})}\big).
\end{equation}

Next, recall the relation ``$\sim$" defined in Definition \ref{def:sim}. By definition (\ref{sLep}), the relation (\ref{symmetrizer}), and parametrix identity (\ref{parametrix}), we have that 
\begin{align*}
\mathbf{S} \sM_{\mB}\mathscr L_\ep 
%%%%
%%%%
%=& \mathbf{S}\sM_{\mB}\sM_{\mB}^{-1}\sM_{\lambda, \ell} \begin{pmatrix} T_{\wp}J_\ep T_p & 0\\ 0& T_{1/q}J_\ep T_q \end{pmatrix}\sM_{\mB} 
%%%%
%%%%
%=& \mathbf{S}\sM_{\lambda, \ell} \begin{pmatrix} T_{\wp}J_\ep T_p & 0\\ 0& T_{1/q}J_\ep T_q \end{pmatrix}\sM_{\mB}\\
%%%%
%%%%
=&
\mathbf{S}\sM_{\lambda, \ell} \mathbf{Q} (J_{\ep} I) \mathbf{S} \sM_{\mB} 
\sim
(\mathbb{i} T_{\gamma}) \mathbf{S} \mathbf{Q} (J_{\ep} I) \mathbf{S} \sM_{\mB}\\ 
%%%%
%%%%
\sim&
(\mathbb{i}T_{\gamma}) (J_{\ep}I) \mathbf{S} \sM_{\mB}
%%%%
%%%%
%\sim & (\mathbb{i}T_{\gamma}) \begin{pmatrix} J_\ep T_p & 0\\ 0& J_\ep T_q \end{pmatrix} \sM_{\mB}\\
%%%%
%%%%
= \begin{pmatrix} 0 & -T_\gamma J_\ep\\ T_\gamma J_\ep& 0 \end{pmatrix} \mathbf{S}\sM_\mB. 
\end{align*}
In summary, the above calculation shows that the following conjugation relation hold:
\begin{equation}\label{conj-ep}
\text{if we set } \quad \mathbb{i}T_{\gamma} J_{\ep}
:=
\begin{pmatrix}
0 & -T_\gamma J_\ep\\
T_\gamma J_\ep& 0
\end{pmatrix} \quad \text{ then } \quad \mathbf{S} \sM_{\mB} \mathscr{L}_{\ep} \sim (\mathbb{i}T_{\gamma} J_{\ep}) \mathbf{S} \sM_{\mB}.  
\end{equation}

\begin{remark}
Note that if we employ the relations in (\ref{Jep}) given by $T_\gamma J_\ep\sim J_\ep T_\gamma$ and $ (J_\ep)^{\ast}\sim J_\ep$, then one can check that 
\[
(\mathbb{i}T_{\gamma} J_{\ep})^{\ast} =
\begin{pmatrix}
0 & (J_{\ep})^{\ast}(T_\gamma)^{\ast} \\
-(J_\ep)^{\ast} (T_\gamma)^{\ast}  & 0
\end{pmatrix}
\sim 
\begin{pmatrix}
0 & T_\gamma J_\ep\\
- T_\gamma J_\ep & 0
\end{pmatrix} = - \mathbb{i} T_{\gamma} J_{\ep},
\]
which indicates that certain cancellation occurs in the sum of $\mathbb{i}T_{\gamma} J_{\ep}$ and its adjoint.  
\end{remark}

With the above construction, we want to show that $\Phi:=\mathbf{S}\big(\begin{smallmatrix}\eta\\U\end{smallmatrix}\big) = \mathbf{S} \sM_{\mB} (\begin{smallmatrix}\eta\\ \psi\end{smallmatrix}\big) $ solves:
\begin{equation*}
(\d_t+T_V\d_zJ_\ep + \mathbb{i}T_{\gamma} J_{\ep})\Phi=\fF_\ep, \quad \text{for some } \ \fF_{\ep} \in L^{\infty}\big(0,T; H^s(\R)\times H^s(\R)\big).
\end{equation*}
Applying $\mathbf{S} \sM_{\mB} $ on (\ref{eq:the-L-system-ep}), we have that 
\begin{equation}\label{eq:pepperment-09}
\mathbf{S} \sM_\mB 
\left\{
\d_t+T_V\d_zJ_\ep+\mathscr L_\ep
\right\}\begin{pmatrix}\eta\\\psi\end{pmatrix}
=
\mathbf{S} \sM_\mB \tilde{f}(J_\ep\eta, J_\ep\psi)
\end{equation}
Define $\fF_{1,\ep} \vcentcolon= \mathbf{S}  \sM_{\mB} \tilde{f}(J_{\ep}\eta,J_{\ep} \psi)$, then by (\ref{eq:pepperment-09}), and $\Phi = \mathbf{S} \sM_{\mB} (\begin{smallmatrix}\eta\\ \psi\end{smallmatrix}\big)$, one has
\begin{align*}
\fF_{1, \ep}
%
%=&\mathbf{S}\sM_\mB\left\{ \d_t+T_V\d_zJ_\ep+\mathscr{L}_\ep\right\}\begin{pmatrix}\eta\\\psi\end{pmatrix}\\
%
=&
\mathbf{S}\sM_\mB (\d_t+T_V\d_z J_\ep)\begin{pmatrix}\eta\\\psi\end{pmatrix}
+
\mathbf{S}\sM_\mB\mathscr L_\ep
\begin{pmatrix}\eta\\\psi\end{pmatrix}\\
%%%%
%%%%
=&
\Big\{(\d_t+T_V\d_zJ_{\ep})\mathbf{S}\sM_\mB + \big[\mathbf{S}\sM_\mB,\d_t+T_V\d_z J_{\ep}\big]
\Big\}
\begin{pmatrix}\eta\\\psi\end{pmatrix} +
\mathbf{S}\sM_{\mB}\mathscr{L}_{\ep}
\begin{pmatrix}\eta\\\psi\end{pmatrix}\\
=&
(\d_t+T_V\d_z J_{\ep})\Phi + 
\mathbf{S}\sM_{\mB}\mathscr{L}_{\ep}
\begin{pmatrix}\eta\\\psi\end{pmatrix}
+ \big[\mathbf{S}\sM_\mB,\d_t+T_V\d_z J_{\ep}\big]
\begin{pmatrix}\eta\\\psi\end{pmatrix}.
\end{align*}
Next, in light of the conjugate relation (\ref{conj-ep}) and $\Phi=\mathbf{S}\sM_{\mB} \big(\begin{smallmatrix}
    \eta \\ \psi
\end{smallmatrix}\big)$, we write
\begin{align*}
\fF_{1,\ep}
=& \d_t\Phi +T_V\d_z J_{\ep}\Phi
+
\mathbb{i}T_{\gamma} J_{\ep} \mathbf{S} \sM_{\mB} \begin{pmatrix}
    \eta \\ \psi
\end{pmatrix} \\
%%%%
%%%%
&+
\big\{ \mathbf{S} \sM_{\mB} \mathscr{L}_{\ep} - \mathbb{i}T_{\gamma} J_{\ep} \mathbf{S} \sM_{\mB} \big\} \begin{pmatrix}
    \eta \\ \psi
\end{pmatrix}
+
\big[ \mathbf{S}\sM_\mB, \d_t+T_V\d_z J_{\ep} \big]
\begin{pmatrix}\eta\\\psi\end{pmatrix}
\\
%%%%
%%%%
=&
\d_t \Phi + T_{V} \d_z J_{\ep} \Phi   + \mathbb{i} T_{\gamma} J_{\ep} \Phi\\
%%%%
%%%%
&+
\big\{ \mathbf{S} \sM_{\mB} \mathscr{L}_{\ep} - \mathbb{i}T_{\gamma} J_{\ep} \mathbf{S} \sM_{\mB} \big\} \begin{pmatrix}
    \eta \\ \psi
\end{pmatrix}
+
\big[\mathbf{S}\sM_\mB, \d_t+T_V\d_z J_{\ep}\big]
\begin{pmatrix}\eta\\\psi\end{pmatrix}.
\end{align*}
Define $\fF_{\ep} \vcentcolon= \fF_{1,\ep} + \fF_{2,\ep} + \fF_{3,\ep}$ with
\begin{equation}\label{fFep123}
\left\{
\begin{aligned}
\fF_{1, \ep} =&
\mathbf{S} \sM_{\mB}
\tilde{f}(J_\ep\eta, J_\ep\psi),\\
\fF_{2, \ep}
=&
\big\{ \mathbb{i}T_{\gamma} J_{\ep} \mathbf{S} \sM_{\mB} - \mathbf{S} \sM_{\mB} \mathscr{L}_{\ep} \big\} \begin{pmatrix}
    \eta \\ \psi
\end{pmatrix},\\
\fF_{3, \ep}
=&
\big[\d_t+T_V\d_z J_{\ep}, \mathbf{S}\sM_\mB\big]
\begin{pmatrix}\eta\\\psi\end{pmatrix}.
\end{aligned}
\right. 
\end{equation}
By Proposition \ref{prop:qpgamma}, Lemma \ref{lemma:dtpq}, estimate (\ref{tildefEst}), conjugate relation (\ref{conj-ep}), and Theorems \ref{thm:paraL2}--\ref{thm:bony}, the following estimate holds uniformly in $\ep$:
\begin{equation*}
    \| \fF_{\ep} \|_{L^{\infty}(0,T; H^s\times H^s)} \le C\big( \|(\teta,\psi)\|_{L^{\infty}(0,T;H^{s+1/2}\times H^s)} \big).
\end{equation*}
Summarizing the derivations so far, we obtained the following lemma: 
\begin{lemma} \label{lemma:Phi-ep}
Let $(\eta,\psi)\in L^{\infty}\big(0,T; H^{s+1/2}(\R)\times H^s(\R)\big)$ be a solution to (\ref{eq:the-L-system-ep}). Then the function $\Phi\vcentcolon= \mathbf{S}\sM_{\mB}\big(\begin{smallmatrix} \eta \\ \psi \end{smallmatrix}\big) \in L^{\infty}\big(0,T; H^{s}(\R)\times H^s(\R)\big)$ solves:
\begin{equation}\label{eq:Phi-ep}
\d_t\Phi +T_V\d_zJ_\ep \Phi + \mathbb{i} T_{\gamma} J_{\ep} \Phi=\fF_\ep \quad \text{where } \ \mathbb{i}T_{\gamma} J_{\ep} \vcentcolon=\begin{pmatrix}
0 & -T_\gamma J_\ep\\
T_\gamma J_\ep & 0
\end{pmatrix},  
\end{equation}
and $\fF_\ep=\fF_{1, \ep}+\fF_{2, \ep}+\fF_{3, \ep}$ is given in (\ref{fFep123}) and there is a strictly positive non-decreasing function $x\mapsto C(x)$ such that the following estimate holds uniformly in $\ep$:
\begin{equation*}
\| \fF_{\ep} \|_{L^{\infty}(0,T; H^s\times H^s)} \le C\big( \|(\teta,\psi)\|_{L^{\infty}(0,T;H^{s+1/2}\times H^s)} \big).
\end{equation*}
\end{lemma}
\begin{remark}
By construction, we have that 
\begin{equation*}
\Phi = \mathbf{S} \sM_{\mB}\begin{pmatrix}
\eta\\
\psi
\end{pmatrix} 
=
\begin{pmatrix}
T_p & 0\\
0& T_q
\end{pmatrix}
\begin{pmatrix}
\eta\\
\psi-T_\mB\eta
\end{pmatrix}
=
\begin{pmatrix}
T_p \eta\\
T_q\psi-T_qT_\mB\eta
\end{pmatrix}.
\end{equation*}
By Proposition \ref{prop:qpgamma}, $\mathbf{S} = \big(\begin{smallmatrix} T_{p} & 0 \\ 0 & T_{q} \end{smallmatrix}\big)$ is of order $\frac{1}{2}\times 0$. Thus, if $(\teta, \psi)\in H^{s+1/2}(\R)\times H^s(\R)$ then $\Phi \in H^{s}(\R) \times H^s(\R)$ and the following estimate is uniform in $\ep$:
\begin{equation}\label{PhiEtaPsi}
    \|\Phi\|_{H^{s}\times H^s} \le C\big( \|\teta\|_{H^{s-1}} \big) \|(\teta,\psi)\|_{H^{s+1/2}\times H^{s}}. 
\end{equation}
\end{remark}

\subsection{Uniform a-priori estimates}
The main result of this section is the following 
\begin{theorem}\label{thm:InvEst}
Let $s>\frac52$. Suppose $(\eta-R, \psi)\in C^1\big([0, T]; H^{s+1/2}(\R)\times H^s(\R)\big)$ is a solution to (\ref{eq:the-L-system-ep}) for some $T>0$ and $\epsilon\in (0,1]$. Define:
\[
M(T):=\sup\limits_{0\le t\le T}\|(\eta-R, \psi)(t,\cdot)\|_{H^{s+\frac{1}{2}}(\R)\times H^s(\R)}.
\]
Then there exists a strictly positive non-decreasing function $x\mapsto C(x)$ such that
\[
M(T)\le C(M_0)+ TC(M(T)),
\]
where $M_0\vcentcolon=\|(\eta_0-R, \psi_0)\|_{H^{s+\frac12}(\R)\times H^s(\R)}$.
\end{theorem}

As the first step of proving Theorem \ref{thm:InvEst}, we derive a set of uniform estimates for the lower regularity space $H^{s-1}(\R)\times H^{s-3/2}(\R)$, which are stated as follows:
\begin{lemma}\label{lem:intermediate-s-1}
Suppose $(\eta,\psi)$ solves the approximate system (\ref{eq:the-L-system-ep}). Then
\begin{align*}
&\|(\eta, \psi)\|_{L^\infty(0, T; H^{s-1}\times H^{s-3/2})}
\le 
C(M_0)+TC(M(T)), \\
&\|(\d_t\eta,\d_t\psi)\|_{L^{\infty}(0,T; H^{s-1}\times H^{s-3/2})} \le C\big(M(T)\big).
\end{align*}
\end{lemma}
\begin{proof}
Fundamental Theorem of Calculus implies that
\begin{equation}\label{etaFTC}
\|\teta(t)\|_{H^{s-1}} \le \|\teta(0)\|_{H^{s-1}}
+
\int_0^t \|\d_t\eta(\tau)\|_{H^{s-1}} \, \dif \tau \le
M_0+T\|\d_t \eta\|_{L^\infty(0, T; H^{s-1})}
\end{equation}
Next, the first equation of \eqref{eq:the-L-system-ep} gives 
\begin{equation}\label{dteta-ep}
\d_t \eta =\tilde{f}^1(J_\ep\eta, J_\ep\psi) - T_V\d_zJ_\ep\eta + T_\lambda T_{1/q}J_\ep T_q(\psi-T_{\mathcal B}\eta).
\end{equation}
The term $ \tilde{f}^1(J_\ep\eta, J_\ep\psi)$
can be estimated as the remainder term in the 
paralinearization of the 
DN operator given in Theorem \ref{thm:paraG}:
\begin{equation*}
\| \tilde{f}^1(J_\ep\eta, J_\ep\psi) \|_{L^{\infty}(0,T;H^{s-1})} \le C\big(\|J_{\ep}\teta\|_{L^{\infty}(0,T;H^{s+1/2})} \big) \|J_{\ep}\psi\|_{L^{\infty}(0,T;H^s)} %\le C\big(\|\teta\|_{L^{\infty}(0,T;H^{s+1/2})})\big) \|\psi\|_{L^{\infty}(0,T;H^s)} 
\le  C\big(M(T)\big),
\end{equation*}
where we used that $\|J_{\ep}\|_{H^{\mu}\to H^{\mu}} \le C\big(\|\teta\|_{H^{s+1/2}}\big)$ uniformly in $\ep$ for $\mu\in\R$. As for the second term, by the embedding 
$H^{s-1}(\R)$ to $L^\infty(\R)$ since $s>\frac52$, Theorem \ref{thm:paraL2} implies
\begin{align*}
\|T_V\d_zJ_\ep \eta\|_{H^{s-1}} \le &
C\|V\|_{L^\infty}\|\d_zJ_\ep\eta\|_{H^{s-1}} \le C\|V\|_{H^{s-1}}\|\teta\|_{H^{s}} \le C\big(M(T)\big),
\end{align*}
Lastly, since $\lambda \in \Sigma^1$, $q, \, 1/q \in \Sigma^0$, and $J_{\ep}\in \Gamma_{s-1}^0$ uniformly in $\ep$, Theorem \ref{thm:paraL2} implies
\begin{align*}
&\|T_\lambda T_{1/q}J_\ep T_q(\psi-T_\mB\eta)\|_{H^{s-1}}
\le
 \|T_\lambda T_{1/q}J_\ep T_q\|_{H^s\to H^{s-1}}
\|(\psi-T_\mB\eta)\|_{H^{s}}\\
\le&
\|T_\lambda T_{1/q}J_\ep T_q\|_{H^s\to H^{s-1}}
\left(
\|\psi\|_{H^s}+\|\mB\|_{H^{s-1}}\|\teta\|_{H^{s}}
\right) \le C\big(M(T)\big).
\end{align*}
Putting the above estimates into (\ref{etaFTC}) and (\ref{dteta-ep}), we get
\begin{equation*}
    \|\d_t \eta\|_{L^{\infty}(0,T;H^{s-1})} \le C\big(M(T)\big), \qquad  \|\teta\|_{L^{\infty}(0,T;H^{s-1})} \le C(M_0) + T C\big(M(T)\big).
\end{equation*}
The corresponding estimate for $\psi$ follows along the same lines.
\end{proof}

To obtain the $ H^s\times H^s$ estimate for $\Phi$, we wish to commute the system (\ref{eq:Phi-ep}) with some elliptic operator of order $s$. However, one observes from the construction of $J_{\ep}$ in (\ref{Je}) that commuting with the standard Fourier multiplier $\absm{D}^s$ does not yield estimates that are uniform in $\ep$. Also it is easy to see that $[\absm{D}^s, T_{\gamma}]$ is of order higher than $s$, hence $\absm{D}^s$ is not a good candidate for this purpose. Instead, we set symbol: 
\begin{equation}\label{symbBeta}
\beta:=\big\{\gamma^{(3/2)}\big\}^{2s/3}\in \Sigma^s(\R),
\end{equation}
where $\gamma^{(3/2)}$ is the symbol constructed in Section \ref{ssec:pqg}. By commuting the system (\ref{eq:Phi-ep}) with the paradifferential operator $T_{\beta}$, we see that the new unknown $\phi\vcentcolon= T_{\beta} \Phi$ solves equations that has a similar structure as (\ref{eq:Phi-ep}), in other words $T_{\beta}$ preserves (\ref{eq:Phi-ep}) up to some lower order commutator terms. To elaborate further on this point, one can check from Proposition \ref{prop:qpgamma}, (\ref{Je}), and (\ref{symbBeta}) that
\begin{gather*}
\d_\xi\beta \cdot \d_z\gamma^{(3/2)}=\d_\xi\gamma^{(3/2)}\cdot \d_z\beta, \qquad
\d_\xi\beta\cdot\d_z J_\ep^{(0)}=\d_\xi J^{(0)}_\ep\cdot \d_z\beta,\\
\sup\limits_{|\xi|=1}\| \d_t \beta (\cdot, \cdot, \xi) \|_{L^{\infty}(0,T;L^{\infty}(\R))} \le C\big( \|(\d_z\eta, \d_t \d_z\eta\|_{L^{\infty}(0,T;H^{s-1/2}\times H^{s-2})} \big).
\end{gather*}
Using this, we see that the following commutator estimates are uniform in $\ep$:
\begin{subequations}\label{betaCom}
\begin{align}
&[ T_\beta, T_\gamma ] \ \mbox{is of order } \ s,
&& [ T_\beta, J_\ep] \ \mbox{ is of order } \ s - \tfrac{3}{2},\\
&[T_\beta, T_V  \d_z J_\ep] \ \mbox{ is of order } \ s, &&
[T_\beta, T_{\d_t}]=-T_{\d_t\beta} \ \mbox{ is of order } \ s.
\end{align}
\end{subequations}
Computing the commutators we have that
\begin{align*}
T_\beta T_\gamma J_\ep\Phi =& T_\gamma J_\ep \phi + 
[T_\beta, T_\gamma]J_\ep\Phi+T_\gamma[T_\beta, J_\ep]\Phi,\\
T_\beta(\d_t+T_V\d_zJ_\ep)\Phi
=& \d_t\phi + T_V\d_zJ_\ep\phi+ [T_\beta, \d_t]\Phi+[T_\beta, T_V\d_zJ_\ep]\Phi.
\end{align*}
By Proposition \ref{prop:qpgamma}, (\ref{JepCom}), and (\ref{betaCom}), we have the following uniform estimates in $\ep$:
\begin{subequations}\label{opCom}
\begin{align}
&\|[T_{\beta},T_{\gamma}] J_{\ep} \Phi\|_{L^2} + \|T_{\gamma} [T_{\beta},J_{\ep}] \Phi\|_{L^2} %\\ \le& \| [T_{\beta},T_{\gamma}] \|_{H^s\to L^2}\|J_{\ep}\|_{H^s\to H^s} \|\Phi\|_{H^s} + \|T_{\gamma}\|_{H^{3/2}\to L^2} \| [T_{\beta},J_{\ep}] \|_{H^s\to H^{3/2}} \|\Phi\|_{H^s}
\le C\big( \|\teta\|_{H^{s-1}} \big) \| (\teta,\psi) \|_{H^{s+1/2}\times H^s},\\
&\|[T_{\beta},T_{V}\d_z J_{\ep}]\Phi\|_{L^2} + \|[T_{\beta},\d_t] \Phi\|_{L^2} %\le \|[T_{\beta},T_{V}\d_z J_{\ep}]\|_{H^s\to H^s} \|\Phi\|_{H^s} + \|[T_{\beta},\d_t]\|_{H^{s}\to L^2} \|\Phi\|_{H^s} 
\le C\big( \|(\teta,\psi)\|_{H^{s+1/2}\times H^{s}} \big) \| (\teta,\psi) \|_{H^{s+1/2}\times H^s},
\end{align}
\end{subequations}
Using the symbol $\beta$ constructed above and estimates (\ref{opCom}), we prove the Gr\"onwall type estimate 
for $\phi=T_\beta \Phi$, which is stated as follows:
\begin{lemma}\label{lem:L2-est-phi}
Let $\phi=T_\beta \Phi$. Then $\phi$ solves the system 
\begin{gather}\label{small-phi-eq}
(\d_t+T_V\d_zJ_\ep)\phi+ \mathbb{i} T_{\gamma} J_{\ep} \phi=\widetilde{\fF}_\ep \qquad \text{where } \\
\mathbb{i}T_{\gamma} J_{\ep} \vcentcolon= \begin{pmatrix}
0 & -T_\gamma J_\ep\\
T_\gamma J_\ep & 0
\end{pmatrix}, \qquad \widetilde{\fF}_\ep\vcentcolon=T_\beta \fF_\ep+[\d_t+T_V\d_zJ_\ep, T_\beta]\Phi+\left[ \mathbb{i} T_{\gamma} J_{\ep}
, T_\beta\right]\Phi,\nonumber
\end{gather}
and $\|\widetilde{\fF}_\ep\|_{L^\infty(0,T;L^2\times L^2)}\le C(M(T))$. Moreover for any solution $\phi$ to (\ref{small-phi-eq}), the following estimate hold uniformly in $\ep$:
\begin{equation}
\|\phi\|_{L^\infty(0,T;L^2\times L^2)} \le C(M_0)+ C(M(T)).
\end{equation}

\end{lemma}

\begin{proof}
From $(\teta, \psi)\in \mC^1([0, T]; H^{s+1/2}\times H^s)$, one has $\phi\in \mC^1([0, T]; L^2\times L^2)$ uniformly in $\ep$. Then using the equation (\ref{small-phi-eq}), and taking the adjoint of operators, one gets 
\begin{align}\label{dphidt}
&\frac{\dif}{\dif t} \|\phi\|_{L^2\times L^2}^2 =
2\langle \d_t \phi, \phi\rangle_{L^2\times L^2} =
2\left\langle
-T_V\d_zJ_\ep\phi- \mathbb{i} T_{\gamma} J_{\ep} \phi+\widetilde{\fF}_\ep, 
\phi\right\rangle_{L^2\times L^2}\\
=&
2\left\langle \widetilde{\fF}_\ep , 
\phi\right\rangle_{L^2\times L^2}
+
2\big\langle
\left\{-T_V\d_zJ_\ep- \mathbb{i}T_{\gamma} J_{\ep}
\right\}\phi, \phi
\big\rangle_{L^2\times L^2}\nonumber\\
=&
2\left\langle \widetilde{\fF}_\ep, \phi\right\rangle_{L^2\times L^2} 
-
\big\langle
\left\{ T_V\d_z J_\ep +(T_V\d_zJ_\ep)^{\ast}
+ \mathbb{i} T_{\gamma} J_{\ep} + \left(\mathbb{i}T_{\gamma} J_{\ep}\right)^{\ast}
\right\}\phi, \phi
\big\rangle_{L^2\times L^2}.\nonumber
\end{align}
By (\ref{eq:pepperment-09}), the commutator estimates (\ref{opCom}), and Proposition \ref{prop:CZ} we have 
\begin{equation}\label{ofF}
    \|\widetilde{\fF}_{\ep}\|_{L^2\times L^2} \le C\big(\|(\teta,\psi)\|_{H^{s+1/2}\times H^s}\big) = C\big(M(T)\big) \quad \text{uniformly in } \ \ep.
\end{equation}
Next, we can write $T_{V}\d_z = T_{\mathfrak{a}}$ with the symbol $\mathfrak{a}(t,z,\xi)= i \xi V(t,z) \in \Gamma_{s-3/2}^1(\R)$, since $V\in H^{s-1}(\R)\xhookrightarrow{} \mC_{\ast}^{s-3/2}(\R)$. Note that the complex conjugate gives $\overline{\mathfrak{a}}= -i\xi V = -\mathfrak{a}$. Thus by the symbolic calculus, Theorem \ref{thm:adjprod}, one can verify that
\begin{align}\label{1com}
    \| T_V\d_z J_\ep +(T_V\d_zJ_\ep)^{\ast} \|_{L^2\times L^2 \to L^2\times L^2 } \le C\big(\|(\teta,\psi)\|_{H^{s+1/2}\times H^s}\big) \le C\big(M(T)\big).
\end{align}
Moreover, from Proposition \ref{prop:qpgamma}, and the properties for $J_{\ep}$ in (\ref{JepCom}), it holds that
\begin{equation*}
    T_{\gamma}^{\ast} \sim T_{\gamma}, \qquad (J_{\ep})^{\ast} \sim J_{\ep}, \qquad T_{\gamma}J_{\ep} \sim J_{\ep} T_{\gamma},
\end{equation*}
where ``$\sim$" is the relation given in Definition \ref{def:sim}. This implies:
\begin{align*}
    (\mathbb{i}T_{\gamma} J_{\ep})^{\ast} =& \begin{pmatrix}
        0 & -T_{\gamma}J_{\ep} \\
        T_{\gamma}J_{\ep} & 0
    \end{pmatrix}^{\ast} 
    %
    %\sim 
    %
    %\begin{pmatrix}
    %0 & (T_{\gamma}J_{\ep})^{\ast} \\ -(T_{\gamma}J_{\ep})^{\ast} & 0
    %\end{pmatrix}
    %
    %\sim
    %
    %\begin{pmatrix}
    %0 & (J_{\ep})^{\ast} (T_{\gamma})^{\ast} \\ -(J_{\ep})^{\ast} (T_{\gamma})^{\ast} & 0
    %\end{pmatrix}\\
    %
    %&\sim
    %
    %\begin{pmatrix}
    %0 & J_{\ep} (T_{\gamma})^{\ast} \\ -J_{\ep} (T_{\gamma})^{\ast} & 0
    %\end{pmatrix}
    %
    %\sim
    %
    %\begin{pmatrix}
    %0 & J_{\ep} T_{\gamma}\\-J_{\ep} T_{\gamma} & 0
    %\end{pmatrix}
    %
    \sim
    \begin{pmatrix}
        0 &  T_{\gamma} J_{\ep} \\ -T_{\gamma} J_{\ep}  & 0
    \end{pmatrix}
    = -\mathbb{i}T_{\gamma} J_{\ep}.
\end{align*}
Therefore we also have the following uniform estimate in $\ep$:
\begin{align}\label{2com}
    \big\| \mathbb{i}T_{\gamma} J_{\ep} + \big(\mathbb{i}T_{\gamma} J_{\ep}\big)^{\ast} \big\|_{L^2\times L^2 \to L^2\times L^2} \le C\big(\|(\teta,\psi)\|_{H^{s+1/2}\times H^s}\big) \le C\big(M(T)\big).
\end{align}
Since $\|\phi\|_{L^2\times L^2} \le C\big(M(T)\big)\|\Phi\|_{H^s\times H^s} \le C\big(M(T)\big)$, substituting (\ref{ofF})--(\ref{2com}) into (\ref{dphidt}), we obtain 
\[
\dfrac{\dif}{\dif t} \|\phi\|_{L^2\times L^2}^2 \le C\big(M(T)\big)\|\phi\|_{L^2\times L^2}^2 + C\big(M(T)\big) \le C\big(M(T)\big).
\]
Integrating in time, we obtain that for all $t\in[0,T]$,
\[
\|\phi(t,\cdot)\|_{L^2\times L^2}
\le  \|\phi(0,\cdot)\|_{L^2\times L^2} + T C\big(M(T)\big) \le C(M_0)+ T C\big(M(T)\big).
\]
This concludes the proof.
\end{proof}

From  Lemma \ref{lem:L2-est-phi}, the estimate in terms of  $(\eta, U)$ 
takes the form 
\begin{equation}\label{eq:phi-L2-estimate}
\|T_\beta T_p\eta\|_{L^\infty(0, T; )}
+
\|T_\beta T_q U\|_{L^\infty(0, T; )}
\le 
C(M_0)+TC(M(T)).
\end{equation}
With the above estimate, we are well-prepared to finish the proof of Theorem \ref{thm:InvEst}:

%%%%%%
\begin{proof}[Proof of Theorem \ref{thm:InvEst}]
Applying Propositions \ref{prop:CZ}--\ref{prop:coerc} with the fact that $\beta \in \Sigma^{s}(\R)$ and $p\in \Sigma^{1/2}(\R)$, we have 
\begin{align}\label{etaRecover}
\|\teta\|_{H^{s+1/2}}
%\le& \mathcal{K}\big(\|\teta\|_{H^{s-1}}\big) \left\{ \|T_{p} \teta\|_{H^{s}} + \|\teta\|_{L^2} \right\},\\
\le& \mathcal{K}\big(\|\teta\|_{H^{s-1}}\big)
\left\{ \|T_{\beta } T_{p} \teta\|_{L^2} + \| T_{p} \teta \|_{L^2} 
+
\|\teta\|_{L^2}\right\} \\
\le &
\mathcal{K}\big(\|\teta\|_{H^{s-1}}\big)
\left\{ \|T_{\beta } T_{p} \teta\|_{L^2} + \| \teta \|_{H^{1/2}} \right\},\nonumber
\end{align}
for some strictly positive non-decreasing smooth function $x\mapsto\mathcal{K}(x)$. To estimate the term $\mathcal{K}(\|\teta\|_{H^{s-1}})$, we set $\mathcal{V}(t) \vcentcolon= \mathcal{K}\big(\|\teta(t,\cdot)\|_{H^{s-1}}\big) $. By Fundamental Theorem of Calculus, and Lemma \ref{lem:intermediate-s-1}, we have for all $t\in [0,T]$
\begin{align}\label{UT}
&\mathcal{V}(t) \le \mathcal{V}(0) +
\int_0^t \mathcal{K}^{\,\prime}\big(\|\teta\|_{H^{s-1}}\big) \|\d_t\eta\|_{H^{s-1}} \, \dif \tau
%\\ \le& \mathcal{K}(M_0) +C\big(M(T)\big) \int_{0}^{T} \,\dif t 
\le C(M_0)+TC\big(M(T)\big).
\end{align}
Lemma \ref{lem:intermediate-s-1} also implies $\|\teta\|_{H^{1/2}}\le \|\teta\|_{H^{s-1}}\le C(M_0) + T C\big( M(T) \big)$ since $s>\frac{5}{2}$. Thus applying this estimate and (\ref{eq:phi-L2-estimate}), (\ref{UT}) on the inequality (\ref{etaRecover}), we have
\begin{equation}\label{eta-CTCM}
    \sup\limits_{0\le t \le T}\|\teta(t,\cdot)\|_{H^{s+1/2}} \le C(M_0) + T C\big( M(T) \big).
\end{equation}
Next, we obtain the corresponding estimate for $\psi$. Similarly to the case for $\teta$, we apply Propositions \ref{prop:CZ}--\ref{prop:coerc} with $\beta\in \Sigma^s$ and $q\in \Sigma^0$ to obtain that
\begin{align}\label{psiRecover}
 \|\psi\|_{H^s} %\le \mathcal{K}(\|\teta\|_{H^{s-1}}) \left\{ \|T_q \psi \|_{H^s} + \|\psi\|_{L^2} \right\} 
\le& \mathcal{K}(\|\teta\|_{H^{s-1}}) \left\{ \| T_{\beta} T_q \psi \|_{L^2} + \|T_{q} \psi\|_{L^2} + \|\psi\|_{L^2} \right\}\\
\le&  \mathcal{K}(\|\teta\|_{H^{s-1}}) \left\{ \| T_{\beta} T_q \psi \|_{L^2} + \|\psi\|_{L^2} \right\}.\nonumber
\end{align}
Using the expression $U \vcentcolon = \psi - T_\mB\eta$, and estimates (\ref{eq:phi-L2-estimate}), (\ref{eta-CTCM}), we have
\begin{align}\label{betaqpsi}
\|T_\beta T_q \psi\|_{L^2} 
%\le& \|T_\beta T_q U\|_{L^2} + \|T_\beta T_q T_\mB \eta\|_{L^2}\\
\le& \|T_\beta T_q U\|_{L^2} + \|T_\beta T_q T_\mB \teta \|_{L^2}\\
\le& C(M_0) + T C\big( M(T) \big) 
+  \|T_\beta T_q T_\mB \teta \|_{L^2}.\nonumber
\end{align}
Next, we claim that $\|T_{\beta}T_{q} T_{\mB} \teta\|_{L^2} \le C(M_0) + T C\big( M(T)\big)$. Applying the estimate (\ref{eta-CTCM}), Proposition \ref{prop:CZ}, and Theorem \ref{thm:paraPEst} with $d=1$ and $m=3-s$, we obtain
\begin{align*}
&\|T_{\beta}T_{q} T_{\mB} \teta\|_{L^2} \le \|T_{\beta}\|_{H^{s}\to L^2} \| T_{q} \|_{H^{s}\to H^{s}} \| T_{\mB} \teta \|_{H^{s}} \le \mathcal{K}\big(\|\teta\|_{H^{s-1}}\big) \| T_{\mB} \teta \|_{H^{2s-5/2}}\\
\le & \mathcal{K}\big(\|\teta\|_{H^{s-1}}\big)  \|\mB\|_{H^{s-5/2}} \|\teta\|_{H^{s+1/2}} \le \widetilde{\mathcal{K}}\big( \|(\teta,\psi)\|_{H^{s-1}\times H^{s-3/2}} \big) \big\{ C(M_0) + T C\big( M(T) \big) \big\},
\end{align*}
where $x\mapsto \widetilde{\mathcal{K}}(x)$ is a strictly positive non-decreasing smooth function. Set the function $\widetilde{\mathcal{V}}(t)\vcentcolon= \widetilde{\mathcal{K}}\big( \|(\teta,\psi)(t,\cdot)\|_{H^{s-1}\times H^{s-3/2}} \big)$. Then by the same argument as (\ref{UT}), we have using Fundamental Theorem of Calculus and Lemma \ref{lem:intermediate-s-1} that for all $t\in[0,T]$,
\begin{align*}
\widetilde{\mathcal{V}}(t) \le& \widetilde{\mathcal{V}}(0) + \int_{0}^{t} \widetilde{\mathcal{K}}^{\,\prime}\big( \|(\teta,\psi)(t,\cdot)\|_{H^{s-1}\times H^{s-3/2}} \big) \|(\d_t \eta, \d_t \psi)\|_{H^{s-1}\times H^{s-3/2}} \, \dif \tau  \\
\le& C(M_0) + T C\big(M(T)\big).
\end{align*}
Substituting this into the previous inequality, we obtain that
\begin{equation}\label{betaqBeta}
\|T_{\beta}T_{q} T_{\mB} \teta(t)\|_{L^2} \le \widetilde{\mathcal{V}}(t) \big\{ C(M_0) + T C\big( M(T) \big) \big\} \le C(M_0) + T C\big( M(T) \big). 
\end{equation}
Putting (\ref{betaqpsi})--(\ref{betaqBeta}) into (\ref{psiRecover}), we get
\begin{align*}
\|\psi\|_{H^s} \le &
\mathcal{K}(\|\teta\|_{H^{s-1}}) \left\{ \| T_{\beta} T_q \psi \|_{L^2} + \|\psi\|_{L^2} \right\} \le C(M_0) + T C\big( M(T) \big).
\end{align*}
This and (\ref{eta-CTCM}) imply the desired estimate of Theorem \ref{thm:InvEst}.
%\todo{The actual estimate is $C_{\ast}\big( C(M_0) + T C(M(T)) \big)\big\{C(M_0) + T C(M(T))\big\}$ where $C_{\ast}(x)$ is the function for Sobolev estimate for DN Lemma \ref{lemma:GSob}.}
\end{proof}

It also follows from Lemma \ref{lem:L2-est-phi} and (\ref{UT}) that 
\begin{align*}
 \| \Phi(t,\cdot) \|_{H^s\times H^s} \le&  \mathcal{K}\big(\|\teta\|_{H^{s-1}}\big) \Big\{ \| T_{\beta} \Phi(t,\cdot) \|_{H^s\times H^s} + \| \Phi(t,\cdot) \|_{L^2\times L^2} \Big\}\\
 \le& C(M_0) + T C\big(M(T)\big).
\end{align*}
In addition, taking $H^{s-3/2}(\R)$ norm on the equation (\ref{eq:Phi-ep}), and using the above estimate we also obtain that
\begin{equation*}
   \sup\limits_{0\le t\le T} \| \d_t \Phi(t,\cdot) \|_{H^{s-3/2}\times H^{s-3/2}}  \le C(M_0) + T C\big(M(T)\big).
\end{equation*}
In summary, we have the following: 
\begin{corollary}\label{corol:PhiUniEst}
Suppose $\Phi$ solves (\ref{eq:Phi-ep}) in $(t,z)\in [0,T]\times\R$, then
\begin{equation*}
    \sup\limits_{0\le t\le T} \Big\{ \| \Phi(t,\cdot) \|_{H^s\times H^s} + \| \d_t \Phi(t,\cdot) \|_{H^{s-3/2}\times H^{s-3/2}} \Big\} \le C(M_0) + T C\big(M(T)\big). 
\end{equation*}
\end{corollary}
Repeating the same argument as the proof of Theorem \ref{thm:InvEst}, the following can be obtained from Corollary \ref{corol:PhiUniEst}:
\begin{corollary}\label{corol:InvEst-dt}
    Suppose the assumptions in Theorem \ref{thm:InvEst} hold. Then there exists a strictly positive, non-decreasing function $x\mapsto C(x)$ such that 
    \begin{equation*}
        \sup\limits_{0\le t\le T} \big\{ \| (\d_t \eta, \d_t \psi) (t,\cdot) \|_{H^{s-1}\times H^{s-3/2}} + \|\d_t^2 \eta(t,\cdot)\|_{H^{s-5/2}} \big\} \le C(M_0) + T C(M(T)).
    \end{equation*}
\end{corollary}
%\todo{We need this corollary for the compactness to get solution in the $\ep\to 0$ limit}

\subsection{Generalised uniform estimate}
For $\ep\in[0,1]$, we denote the operator $$\mathfrak{E}(\ep,\eta,\psi)\vcentcolon=(\d_t + T_{V}\d_z J_{\ep} + \mathscr{L}_\ep).$$ Consider $(\eta-R,\psi)\in \mC^0\big( [0,T]; H^{s+\frac{1}{2}}(\R)\times H^s(\R) \big)$ as a solution to (\ref{eq:the-L-system-ep}). Then
\begin{equation*}
\left\{
\begin{aligned}
&\mathfrak{E}(\ep,\eta,\psi)\begin{pmatrix}
    \eta\\ \psi
\end{pmatrix} = (\d_t + T_{V}\d_z J_{\ep}+\mathscr{L}_{\ep}) \begin{pmatrix}
    \eta\\ \psi
\end{pmatrix} = \tilde{f}(J_{\ep}\eta,J_{\ep}\psi),\\
&(\eta,\psi)\vert_{t=0}= (\eta_0,\psi_0).    
\end{aligned}
\right.
\end{equation*}
For a given vector function $F=(F^1,F^2)^{\top}$, let $(\eta_{\ast}\,,\psi_{\ast})$ be the solution to
\begin{equation*}
\left\{
\begin{aligned}
&\mathfrak{E}(\ep,\eta,\psi)\begin{pmatrix}
    \eta_{\ast}\\ \psi_{\ast}
\end{pmatrix}  = F,\\
&(\eta_{\ast}\,,\psi_{\ast})\vert_{t=0}= (\eta_{\ast}^0\,,\psi_{\ast}^0).    
\end{aligned}
\right.
\end{equation*}
In addition, for $\sigma\ge 0$, $T\ge 0$, and two real-valued functions $u_1$, $u_2$, we shall also denote
\begin{equation*}
    \| (u_1,u_2) \|_{\mathrm{X}^{\sigma}(T)} \vcentcolon= \|(u_1,u_2)\|_{L^{\infty}(0,T;H^{\sigma+1/2}\times H^\sigma)}.
\end{equation*}
The following generalised uniform estimates will be useful for proving the uniqueness and short time stability of solution, which is the main subject of Section \ref{ssec:uniq}.
\begin{proposition}\label{prop:2pair}
Let $s>\frac{5}{2}$ and $0\le \sigma\le s$. Then there is 
a constant $\tilde{C}>0$ and a positive non-decreasing function $x\mapsto C(x)$ such that, if $\ep\in[0, 1]$, $T\in (0,1]$ and
\begin{gather*}
(\eta\, , \psi)\in \mC^0\big([0, T];H^{s+\frac{1}{2}}(\R)\times H^s(\R)\big),\\
(\eta_{\ast}\, , \psi_{\ast})\in \mC^1\big([0, T];H^{\sigma+\frac{1}{2}}(\R)\times H^\sigma(\R)\big),\\
F=(F^1,F^2)^{\top}\in L^\infty\big(0, T;H^{\sigma+\frac{1}{2}}(\R)\times H^\sigma(\R)\big),
\end{gather*}
are solutions to the systems:
\[
\mathfrak{E}(\ep, \eta, \psi)\begin{pmatrix}
\eta\\ \psi
\end{pmatrix}
=\tilde{f}(J_\ep\eta, J_\ep\psi), \qquad
\mathfrak{E}(\ep, \eta, \psi)\begin{pmatrix}
\eta_{\ast}\\ \psi_{\ast}
\end{pmatrix}
=F.
\]
Then the following estimate holds:
\begin{align*}
\|(\eta_{\ast}, \psi_{\ast})\|_{X^\sigma(T)}
\le&
\tilde C\|(\eta_{\ast}^0, \psi_{\ast}^0)\|_{H^{\sigma+\frac{1}{2}}\times H^\sigma}
\\ &+
TC(\|(\teta, \psi)\|_{X^\sigma(T)})
\left\{
\|(\teta_{\ast}, \psi_{\ast})\|_{X^\sigma(T)}+\|F\|_{X^\sigma(T)}
\right\}.    
\end{align*}
\end{proposition}
The proof is omitted since it is the same as Proposition 5.4 in \cite{ABZ}. Note that if $(\eta_{\ast}\,,\psi_{\ast})=(\eta\, ,\psi)$, then we recover the uniform estimates given in Theorem \ref{thm:InvEst}. 
%%%%%%%%%%%%%%%

%---------------------
%     Section
%---------------------
\section{Cauchy problem}\label{sec:cauchy}
The main results of this section are  
Theorem \ref{thm:exist}, where we prove the existence, 
Lemma \ref{lemma:cit} on the  continuity in time, and,  
finally, the uniqueness and the stability with respect to initial data of the solution $(\eta,\psi)$ in 
Theorem \ref{thm:uniq}.
%---------------------
%     Subsection
%---------------------
\subsection{Existence}\label{ssec:exist}
We show that a solution to the Zakharov system (\ref{000-intro-2}) exists. The main strategy is to first obtain approximate solutions to the mollified system (\ref{eq:the-L-system-ep}). Then we apply compactness argument using the uniform a-priori estimate proven in Theorem \ref{thm:InvEst}. The main theorem of this section is the following:  
\begin{theorem}\label{thm:exist}
Suppose $(\eta_0-R,\psi_0)\in H^{s+\frac{1}{2}}(\R)\times H^{s}(\R)$ with $\eta_0 \ge c_0>0$. Then there exists $T_0>0$ and some function $(\eta,\psi)$ satisfying:
\begin{gather*}
(\eta-R,\psi)\in L^{\infty}\big(0,T_0;H^{s+\frac{1}{2}} \times H^{s}(\R)\big),  \qquad (\d_t\eta,\d_t\psi) \in L^{\infty}\big(0,T_0;H^{s-1} \times H^{s-\frac{3}{2}}(\R)\big),
\end{gather*}
such that $(\eta,\psi)$ solves (\ref{000-intro-2}) for $(t,z)\in [0,T_0]\times \R$ almost everywhere. 
\end{theorem}

\begin{lemma}\label{lemma:exist-ep}
Suppose $(\eta_0-R,\psi_0)\in H^{s+\frac{1}{2}}(\R)\times H^{s}(\R)$. Then for any $\ep>0$, there exists $T_{\ep}>0$ which depends on $\ep$ and $(\eta_0,\psi_0)$ such that, the Cauchy problem:
\begin{equation}\label{ode-ep}
\left\{\begin{aligned}
&\d_t + T_{V}\d_z J_{\ep} + \mathscr{L}_{\ep} \begin{pmatrix}
\eta \\ \psi
\end{pmatrix} = \tilde{f}(J_{\ep}\eta,J_{\ep}\psi),\\
&(\eta,\psi)\vert_{t=0} = (\eta_0,\psi_0).
\end{aligned}\right.
\end{equation}
has a unique maximal solution $(\eta_\ep-R,\psi_\ep) \in \mC^0\big([0,T_{\ep}); H^{s+\frac{1}{2}}(\R)\times H^s(\R)\big)$.
\end{lemma}
\begin{proof}
Denote $\mathfrak{A}= (\eta,\psi)^{\top}$. Then (\ref{ode-ep}) can be written in the form $\d_t \mathfrak{A} = \mathscr{F}_{\ep}(\mathfrak{A})$, with initial condition $\mathfrak{A}\vert_{t=0}=\mathfrak{A}_0\equiv (\eta_0,\psi_0)^{\top} $. Hence it is an ODE with values in the Banach space $X^s(\R)\equiv H^{s+1/2}(\R)\times H^{s}(\R) $. Since $J_{\ep}$ is a smoothing operator, using the shape-differentiability of $G[\eta](\psi)$ given by Theorem \ref{thm:shape}, it can be verified that the function $\mathscr{F}_{\ep}(\cdot)\vcentcolon X^{s}(\R) \to X^{s}(\R)$ is in $\mC^1$. Therefore, by Cauchy-Lipschitz theorem, we obtain the desired result.
\end{proof}

\begin{lemma}\label{lemma:unifEst}
For $(\eta_0-R,\psi_0)\in H^{s+\frac{1}{2}}(\R)\times H^{s}(\R)$ with $\eta_0 \ge c_0>0$, there exists $T_0>0$ such that $T_{\ep}\ge T_0$ for all $\ep\in (0,1]$. Moreover, there exists a constant $C_0>0$ depending only on $\|(\eta_0-R,\psi_0)\|_{H^{s+1/2}\times H^{s}}$ such that
\begin{equation*}
 \sup\limits_{\ep\in (0,1]}\sup\limits_{0\le t \le T_0}\| (\eta_{\ep}-R,\psi_{\ep})(t,\cdot) \|_{H^{s+\frac{1}{2}}(\R)\times H^{s}(\R)} \le C_0.   
\end{equation*}
\end{lemma}
\begin{proof}
Denote $X^s(\R)\equiv H^{s+1/2}(\R)\times H^{s}(\R) $. Fix $\ep\in (0,1]$. For $0\le T\le T_{\ep}$, we define 
\begin{equation*}
    M_{\ep}(T)\vcentcolon=\sup\limits_{0\le t \le T}\| (\eta_{\ep}-R,\psi_{\ep})(t,\cdot) \|_{X^{s}(\R)}.
\end{equation*}
By construction in Lemma \ref{lemma:exist-ep}, $(\eta_{\ep}-R,\psi_{\ep})\in \mC^1\big([0,T_{\ep}); X^s(\R)\big)$, hence $T\mapsto M_{\ep}(T)$ is continuous in $[0,T_{\ep})$. Moreover, applying Theorem \ref{thm:InvEst}, we get that there exists a continuous, monotone increasing, strictly positive function $x\mapsto C(x)$ such that
\begin{equation}\label{MepT}
    M_{\ep}(T) \le C(M_0) + T C\big( M_{\ep}(T) \big) \qquad \text{for all } \ T\in [0,T_{\ep}],
\end{equation}
where $M_0 \equiv M_{\ep}(0) = \|\eta_0-R,\psi_0\|_{X^s(\R)}$. Next, we set $M_1 \vcentcolon= 2 C(M_0)$ and choose a time $0<T_0 \le \frac{1}{2}C(M_0)/C(M_1)$. Then $C(M_0)+T_0 C(M_1) \le \frac{3}{4} M_1 < M_1$. We claim:
\begin{equation}\label{MClaim}
    M_{\ep}(T) < M_1 \quad \text{for all } \ T\in \mathfrak{I}=[0,\min\{T_0,T_{\ep}\} ).
\end{equation}
Suppose by contradiction, $\exists T_{\ast} \in \mathfrak{I}\backslash\{0\}$ with $M_{\ep}(T_{\ast}) \ge M_1$. Since $M_{\ep}(0)=M_0 < M_1$, and $T\mapsto M_{\ep}(T)$ is continuous in $\mathfrak{I}$, it follows by Intermediate Value theorem that there exists $\tau\in (0,T_{\ast}]\subseteq (0,T_0]$ with $M_{\ep}(\tau)=M_1$. Then (\ref{MepT}) implies that:
\begin{equation*}
    M_1 = M_{\ep}(\tau) \le  C(M_0) + \tau C\big(M_{\ep}(\tau)\big) \le C(M_0) + T_{0} C(M_1) < M_1,
\end{equation*}
which is a contradiction, hence (\ref{MClaim}) holds true.

Now if $T_{\ep}<T_0$, then we extend the maximal time of existence into $T_{\ep}^{\ast} > T_{\ep}$ by taking $(\eta_{\ep},\psi_{\ep})\vert_{t=T_{\ep}}$ as initial data in (\ref{ode-ep}), then concatenating the resulting solution with $\{(\eta_{\ep},\psi_{\ep})\}_{0\le t \le T_{\ep}}$. Subsequently, Theorem \ref{thm:InvEst} implies that the inequality (\ref{MepT}) also holds true in the extended interval $[0,T_{\ep}^{\ast})$, and by repeating the same argument as above, one also gets $M_{\ep}(T_{\ep}^{\ast})\le M_1$. Note that the length of continuation: $\Delta T_{\ep} \vcentcolon= T_{\ep}^{\ast}-T_{\ep}$ depends only on the size of $\| (\eta_{\ep},\psi_{\ep})\vert_{t=T_{\ep}} \|_{X^s(\R)}$, which is bounded by $M_1=2C(M_0)$ due to (\ref{MClaim}). Hence, $\Delta T_{\ep}$ does not shrink for each such time extension, and one can apply this argument for finitely many times to guarantee that $T_{\ep} \ge T_0$ for each $\ep\in(0,1]$.
\end{proof}

Lemma \ref{lemma:unifEst} and Corollary \ref{corol:InvEst-dt} imply the following:
\begin{corollary}\label{corol:unifEst-dt}
    Suppose the assumptions in Lemma \ref{lemma:unifEst} hold. Then there exists $T_0>0$ and a constant $C_0>0$ depending only on $\|(\eta_0-R,\psi_0)\|_{H^{s+1/2}\times H^{s}}$ such that 
    \begin{equation*}
        \sup\limits_{\ep\in(0,1]} \sup\limits_{0\le t\le T_0} \big\{ \| (\d_t \eta_{\ep}, \d_t \psi_{\ep}) (t,\cdot) \|_{H^{s-1}\times H^{s-3/2}} + \| \d_t^2 \eta_{\ep}(t,\cdot) \|_{H^{s-5/2}} \big\} \le  C_0.
    \end{equation*}
\end{corollary}

\begin{proposition}\label{prop:compact}
Let $\{ (\eta_{\ep},\psi_\ep) \}_{0<\ep\le 1}$ be the solutions constructed in Lemma \ref{lemma:exist-ep}. Then there exists a subsequence $\ep_k \to 0$ as $k\to\infty$ and functions $(\eta,\psi)(t,z)$ satisfying:
\begin{equation*}
    (\eta-R,\psi)\in L^{\infty}\big( 0,T_0; H^{s+\frac{1}{2}}(\R)\times H^{s}(\R) \big) \cap W^{1,\infty}\big( 0,T_0; H^{s-1}(\R)\times H^{s-\frac{3}{2}}(\R) \big), 
\end{equation*}
such that the following convergence holds: as $k\to \infty$,
\begin{align*}
    &(\eta_{\ep_k},\psi_{\ep_k}) \overset{\ast}{\rightharpoonup} (\eta,\psi) && \text{ weakly-$\ast$ in } \ L^{\infty}\big(0,T_0; H^{s+\frac{1}{2}}(\R)\times H^s(\R)\big),\\
    &(\d_t \eta_{\ep_k}, \d_t \psi_{\ep_k}) \overset{\ast}{\rightharpoonup} (\d_t\eta,\d_t\psi) && \text{ weakly-$\ast$ in } \ L^{\infty}\big(0,T_0; H^{s-1}(\R)\times H^{s-\frac{3}{2}}(\R)\big).
\end{align*}
Moreover, for a given compact subset $K\subset\joinrel\subset \R$, there exists a further subsequence $\ep_k\to 0$ such that as $k\to\infty$,
\begin{align*}
    (\eta_{\ep_k},\psi_{\ep_k}, \d_z \eta_{\ep_k}, \d_z \psi_{\ep_k}, \d_t \eta_{\ep_k}) \to (\eta,\psi,\d_z\eta,\d_z\psi, \d_t \eta) \quad \text{strongly in } \ L^2([0,T_0]\times K). 
\end{align*}
\end{proposition}
\begin{proof}
The weakly-$\ast$ convergences are immediate consequences of Lemma \ref{lemma:unifEst} and Corollary \ref{corol:unifEst-dt}. For the strong convergence, we only present the proof for $\d_t \eta_{\ep_k}$, as the other cases follow in the same manner. Fix a compact subset $K\subset\joinrel\subset \R$. Then Lemma \ref{lemma:unifEst}, Corollary \ref{corol:unifEst-dt}, and $s>\frac{5}{2}$ imply that
\begin{equation*}
    \sup\limits_{\ep\in(0,1]} \big\{ \| (\d_t^2 \eta_{\ep}, \d_t \d_z \eta_{\ep}) \|_{ L^2([0,T_0]\times K) } \big\} \le C_0.
\end{equation*}
Thus $\big\{\d_t\eta_{\ep}\big\}_{0<\ep\le 1}$ is a bounded sequence in the Sobolev space $H^1([0,T_0]\times K)$. Now we cite a version of Rellich-Kondrachov compactness theorem given as follows: let $K \subset \R^2$ be a $2$-dimensional bounded set, then $H^1(K) \subset\joinrel\subset L^{q}(K) $ for all $1\le q < \infty$. %(See Part \textrm{II} of Theorem 6.3 in \textit{Sobolev Spaces} by Adam and Fournier 2003). 
Thus for fixed $K$, there exists a subsequence $\ep_k\to 0$ as $k\to 0$ such that $\d_t\eta_{\ep_k} \to f$ strongly in $L^2([0,T_0]\times K)$ for some $f \in L^2$. By the weakly-$\ast$ convergences and the Fundamental Lemma of Calculus of Variation, it can be verified that $f=\d_t\eta$ for a.e. $(t,z)\in [0,T_0]\times K$. %By the same argument, one can also obtain the strong convergence for $(\eta_{\ep_k},\psi_{\ep_k},\d_z\eta_{\ep_k},\d_z\psi_{\ep_K})\to (\eta,\psi,\d_z\eta,\d_z\psi)$. This completes the proof.  
\end{proof}
Next, for $(\eta_\ep,\psi_\ep)$, we set $v_{\ep}(z,y)$ to be the solution to:
\begin{gather}\label{vep}
\begin{aligned}
    &-\div_{(z,y)} \big( A_{\ep} \cdot \nabla_{(z,y)} v_\ep \big) = 0 && \text{ in } \ (z,y)\in \R \times(0,1],\\
    &v_{\ep}\vert_{y=1}=\psi_\ep, \qquad \d_y v_{\ep}\vert_{y=0}=0 && \text{ for } \ z\in\R,
\end{aligned}\\
\text{where } \ A_{\ep} \vcentcolon= \begin{pmatrix}
    y \eta_\ep^2 & -y^2 \eta_\ep \d_z\eta_\ep \\
    -y^2 \eta_\ep \d_z\eta_\ep & y(1+y^2|\d_z\eta_{\ep}|^2)
\end{pmatrix}.\nonumber
\end{gather}
Fix $y_0\in (0,1)$. Combining the uniform estimate in Lemma \ref{lemma:unifEst}, and the elliptic estimates in Corollaries \ref{corol:v0} and Lemma \ref{lemma:vyy}, we have 
\begin{equation}\label{vUniEst}
    \sup\limits_{\ep\in (0,1]} \sup\limits_{0\le t\le T_0} \| v_{\ep}(t,\cdot) \|_{H^{s+\frac{1}{2}}(\R\times [y_0,1])} \le \sup\limits_{\ep\in (0,1]} \sup\limits_{0\le t\le T_0} C\big(\|\teta_{\ep}\|_{H^{s+\frac{1}{2}}}\big)\| \psi_{\ep} \|_{H^{s}} \le  C_0, 
\end{equation}
for some constant $C_0>0$ depending only on the initial data. This implies the following compactness result:
\begin{proposition}\label{prop:vlimit}
Let $\{ (\eta_{\ep_k},\psi_{\ep_k}) \}_{k\in\mathbb{N}}$ be the subsequence constructed in Proposition \ref{prop:compact}, and let $(\eta,\psi)$ be its limit function. Set $v_{\ep}(z,y)$ to be the solution to (\ref{vep}), and let $v(z,y)$ be the solution to  
\begin{gather}\label{vlimit}
\left\{\begin{aligned}
    &-\div_{(z,y)} \big( A \cdot \nabla_{(z,y)} v \big) = 0 && \text{ in } \ (z,y)\in \R \times(0,1],\\
    &v\vert_{y=1}=\psi, \qquad \d_y v\vert_{y=0}=0 && \text{ for } \ z\in\R,
\end{aligned}\right.\\
\text{where } \ A \vcentcolon= \begin{pmatrix}
    y \eta^2 & -y^2 \eta \d_z\eta \\
    -y^2 \eta \d_z\eta & y(1+y^2|\d_z\eta|^2)
\end{pmatrix}.\nonumber
\end{gather}
Then for compact $K\subset \joinrel \subset \R$, there exists a further subsequence $\ep_k$ such that as $k\to\infty$,
\begin{align*}
&(\d_y^2 v_{\ep}, \d_y\d_z v_{\ep}) \overset{\ast}{\rightharpoonup} (\d_y^2 v, \d_y\d_z v) && \text{weakly-$\ast$ in } \ L^{\infty}\big(0,T_0;L^{2}(\R\times[y_0,1])\big),\\
&(v, \d_z v , \d_y v) \to (v,\d_z v, \d_y v) && \text{strongly in } \ L^{2}\big( [0,T_0]\times K\times [y_0,1] \big).
\end{align*}
\end{proposition}
\begin{proof}
    The existence of a weakly-$\ast$ and strongly convergent subsequence can be obtained from (\ref{vUniEst}) using the same procedure presented in the proof of Proposition \ref{prop:compact}. Thus it is left to show that the limit of $v_{\ep_k}\to v$ is indeed a solution to (\ref{vlimit}). Consider a test function $\varphi(z,y)\in \mC_{c}^{\infty}\big(\R\times[y_0,1) \big) $. Multiplying $\varphi$ on both sides of (\ref{vep}) then integrating by parts, we get
    \begin{equation*}
        \int_{0}^1\!\!\int_{\R}\!\! (\nabla_{(z,y)} \varphi)^{\top} \cdot A_{\ep_k} \nabla_{(z,y)} v_{\ep_k}\, \dif z \dif y = 0
    \end{equation*}
    From Proposition \ref{prop:compact} we have $(\eta_{\ep_k},\d_z\eta_{\ep_k},\nabla_{(z,y)} v_{\ep_k})\to (\eta,\d_z\eta, \nabla_{(z,y)} v)$ strongly in $L^2$. Applying these on the weak form above, we conclude that the limit $v$ is a weak solution. Finally, the boundary condition for $v$ can be recovered using the trace estimate and the fact that $\big(v_{\ep_k}, \nabla_{(z,y)} v_{\ep_k}\big) \to \big(v,\nabla_{(z,y)}v\big) $ strongly in $L^{2}([0,T_0]\times K \times [y_0,1])$.
\end{proof}

\begin{lemma}\label{lemma:DNconv}
For each $\varphi \in \mC_{c}^{\infty}(\R)$, there exists a subsequence $\{ (\eta_{\ep_k},\psi_{\ep_k}) \}_{k\in\mathbb{N}}$, and a limit function $(\eta,\psi)\in H^{s+\frac{1}{2}}(\R)\times H^{s}(\R)$ such that
\begin{equation*}
    \lim\limits_{k\to\infty}\int_{\R} \varphi G[\eta_{\ep_k}](\psi_{\ep_k}) \, \dif z = \int_{\R} \varphi G[\eta](\psi) \, \dif z \quad \textnormal{ for all } \ . 
\end{equation*}
\end{lemma}
\begin{proof}
Since $s>\frac{5}{2}$, for $(\eta_{\ep}-R,\psi_{\ep})\in H^{s+\frac{1}{2}}(\R)\times H^{s}(\R)$, we have $\eta_{\ep}, \psi_{\ep}\in \mC_{b}^{2}(\R)$ and $v_{\ep}\in \mC_{b}^2(\R \times (0,1])$. Therefore, in light of (\ref{DN-flat}), the DN operator $G[\eta](\psi)$ constructed in Section \ref{sssec:consDN} can be expressed in the classical sense as
\begin{equation}\label{DNcl}
    G[\eta_\ep](\psi_{\ep}) = \lim\limits_{y\to 1^{-}} \Big\{ \dfrac{1+y |\d_z\eta_{\ep}|^2}{\eta_{\ep}} \d_y v_{\ep} - \d_z \eta_{\ep} \d_z v_{\ep} \Big\}.
\end{equation}
Let $\varphi(z)\in \mC_{c}^{\infty}(\R)$ and $\chi(y)\in \mC_{c}^{1}(y_0,\infty)$ be two test functions such that $\chi(1)=1$. Set $K\vcentcolon=\supp(\varphi)\subset\joinrel\subset \R$. Then by Divergence theorem and (\ref{DNcl}), we have 
\begin{align*}
    &\int_{\R} \varphi G[\eta_{\ep}](\psi_\ep) \dif z = \int_{0}^1\!\! \int_{\R} \varphi \d_y \Big\{ \chi \Big( \dfrac{1+y |\d_z\eta_{\ep}|^2}{\eta_{\ep}} \d_y v_{\ep} - \d_z \eta_{\ep} \d_z v_{\ep} \Big) \Big\} \, \dif z \dif y\\
    =& \int_{0}^1\!\! \int_{\R} \varphi \chi^{\prime} \Big( \dfrac{1+y |\d_z\eta_{\ep}|^2}{\eta_{\ep}} \d_y v_{\ep} - \d_z \eta_{\ep} \d_z v_{\ep} \Big) \dif z \dif y\\ &+ \int_{0}^1\!\! \int_{\R} \varphi \chi \Big( \dfrac{|\d_z\eta_{\ep}|^2}{\eta_{\ep}} \d_y v_{\ep} + \dfrac{1+y |\d_z\eta_{\ep}|^2}{\eta_{\ep}} \d_y^2 v_{\ep} - \d_z \eta_{\ep} \d_y \d_z v_{\ep} \Big) \dif z \dif y
\end{align*}
By Propositions \ref{prop:compact}--\ref{prop:vlimit}, for $K\subset\joinrel\subset \R$, there exists a further subsequence $\ep_k$ such that
\begin{align*}
    &(\eta_{\ep_{k}}, \d_z\eta_{\ep_k})\to (\eta,\d_z\eta) && \text{strongly in } \ L^{2}\big([0,T_0]\times K\big),\\
    &(\d_z v_{\ep_k}, \d_y v_{\ep_k}) \to (\d_z v, \d_y v) && \text{strongly in } \ L^{2}\big( [0,T_0]\times K \times [y_0,1] \big),\\
    &(\d_z\d_y v_{\ep_k}, \d_y^2 v_{\ep_k}) \overset{\ast}{\rightharpoonup} (\d_z\d_y v_{\ep_k}, \d_y^2 v_{\ep_k}) && \text{weakly-$\ast$ in } \ L^{\infty}\big(0,T_0; L^2(\R\times [y_0,1] )\big).
\end{align*}
Combining the above with the uniform estimates Lemma \ref{lemma:unifEst} and (\ref{vUniEst}), we have
\begin{align*}
    &\lim\limits_{k\to \infty}\int_{\R} \varphi G[\eta_{\ep_k}](\psi_{\ep_k})\, \dif z\\
    = & \int_{0}^1\!\! \int_{\R} \varphi \chi^{\prime} \Big( \dfrac{1+y |\d_z\eta|^2}{\eta} \d_y v - \d_z \eta \d_z v \Big) \,\dif z \dif y\\ &+ \int_{0}^1\!\! \int_{\R} \varphi \chi \Big( \dfrac{|\d_z\eta|^2}{\eta} \d_y v + \dfrac{1+y |\d_z\eta|^2}{\eta} \d_y^2 v - \d_z \eta \d_y \d_z v \Big) \,\dif z \dif y\\
    =& \int_{0}^{1}\!\! \int_{\R} \!\! \varphi \d_y \Big\{ \chi \Big( \dfrac{1+y |\d_z\eta|^2}{\eta} \d_y v - \d_z \eta \d_z v \Big) \Big\}\, \dif y\dif z =
    %\\=& \int_{\R} \varphi \lim\limits_{y\to 1^- } \Big\{ \dfrac{1+y |\d_z\eta|^2}{\eta} \d_y v - \d_z \eta \d_z v \Big\} \, \dif z \\ =&
    \int_{\R} \varphi G[\eta](\psi)\, \dif z,
\end{align*}
where we used the fact that $(\eta,\psi)\in H^{s+\frac{1}{2}}\times H^s$ with $s>\frac{5}{2}$, hence DN operator $G[\eta](\psi)$ can be expressed in the classical sense as (\ref{DNcl}). This concludes the proof.
\end{proof}

\begin{proof}[Proof of Theorem \ref{thm:exist}]
Let $\varphi \in \mC_{c}^{\infty}([0,T_0]\times \R)$ be $\supp(\varphi)\subset [0,T_0]\times K$ for some compact set $K\subset\joinrel\subset \R$. Multiply the first equation of (\ref{ode-ep}) with $\varphi$ and integrate to get
\begin{equation*}
 \int_{0}^{T_0}\!\! \int_{\R} \varphi \big\{ \d_t \eta_{\ep} - G[\eta_{\ep}](\psi_{\ep}) \big\} \,\dif z \dif t = \int_{0}^{T_0}\!\! \int_{\R} \varphi \cdot \fr_{\ep} \, \dif z \dif t,
\end{equation*}
where $\fr_{\ep}$ consists of the commutator terms of $J_{\ep}$, and it can be shown with the paradifferential calculus that $\| \fr_{\ep}\|_{L^2([0,T_0]\times \R)} = \mathcal{O}(\ep) \to 0$ as $\ep \to 0$. By Proposition \ref{prop:compact} and Lemma \ref{lemma:DNconv}, for $K\subset\joinrel \subset \R$, there exists a subsequence $\ep_k\to 0$ such that 
\begin{equation*}
 \lim\limits_{k\to\infty}\int_{0}^{T_0}\!\! \int_{\R} \varphi \big\{ \d_t \eta_{\ep_k} - G[\eta_{\ep_k}](\psi_{\ep_k}) \big\} \,\dif z \dif t = \int_{0}^{T_0}\!\! \int_{\R} \varphi \big\{ \d_t \eta - G[\eta](\psi) \big\} \,\dif z \dif t,
\end{equation*}
where $(\eta,\psi)$ are the limit function obtained in Proposition \ref{prop:compact}. Therefore taking the limit $\ep_k\to 0$ and using the Fundamental Lemma of Calculus of Variation, we obtain that $\d_t \eta = G[\eta](\psi)$ for a.e. $(t,z)\in [0,T_0]\times K$. Since we also have $\d_t \eta_{\ep_k} \to \d_t \eta$ strongly in $L^2([0,T_0]\times K)$ from Proposition \ref{prop:compact}, it follows that as $k\to\infty$,
\begin{align}\label{DNstrong}
    &\| G[\eta_{\ep_k}](\psi_{\ep_k}) - G[\eta](\psi) \|_{L^2([0,T_0]\times K)}\\ \le&  \| \d_t \eta_{\ep_k} - \d_t \eta \|_{L^2([0,T_0]\times K)} + \| \fr_{\ep_k} \|_{L^2([0,T_0]\times \R)} \to 0 .\nonumber
\end{align}
Using Proposition \ref{prop:compact} and the strong convergence (\ref{DNstrong}), one can show that the limit $(\eta,\psi)$ also solves the second equation of (\ref{000-intro-2}) almost every where in $[0,T_0]\times K$. Since $K$ is an arbitrarily chosen compact subset, this concludes the proof. 
\end{proof}

%---------------------
%     Subsection      
%---------------------
\subsection{Continuity in time}
We show that the solution $(\teta,\psi)$ obtained in Section \ref{ssec:exist} is continuous in time with values in $H^{s+\frac{1}{2}}(\R) \times H^{s}(\R)$, which completes the proof of Theorem \ref{thm:main1}. The main argument used here closely follows Chapter 5 of \cite{taylor}.

To show the time continuity for $(\eta,\psi)$, it suffices to show instead that $\Phi$ is continuous in time with values in $H^{s}(\R)$, where $\Phi$ is constructed in Theorem \ref{thm:symPHI} via transformation from $(\eta,\psi)$. Thus the main lemma of this section is the following:
\begin{lemma}[Continuity in time]\label{lemma:cit}
Assume $s>\frac{5}{2}$, $(\eta_0-R, \psi_0)\in H^{s+1/2}(\R)\times H^{s}(\R)$. Suppose for some $T>0$, $\Phi\in L^{\infty}\big(0,T; H^{s}(\R)\big)$ solves the symmetrized system (\ref{PHIeq}) for a.e. $(t,z)\in[0,T]\times \R$. Then $\Phi \in \mC^0\big([0,T]; H^{s}(\R)\big)$.
\end{lemma}
Before the proof, we remark that by the weak compactness and uniform a-priori estimate for the approximate solution given in Corollary \ref{corol:PhiUniEst}, it follows that the limit solution $\Phi$ belongs to the space:
\begin{equation*}
    \Phi \in L^{\infty}\big(0,T;H^{s}(\R) \big) \cap W^{1,\infty}\big(0,T;H^{s-\frac{3}{2}}(\R)\big).
\end{equation*}
This implies that in the weak topology of $H^{s}(\R)$, the map $t\mapsto \Phi(t,\cdot)$ is continuous in $[0,T]$. Thus to establish Lemma \ref{lemma:cit}, it suffices to demonstrate that $t\mapsto \|\Phi(t,\cdot)\|_{H^{s}(\R)}^2$ is continuous. For this, we employ an argument similar to the analysis of Lemma \ref{lem:L2-est-phi}. However one cannot directly estimate $\frac{\dif}{\dif t} \| \absm{D}^s  \Phi(t,\cdot) \|_{L^{2}(\R)}^2$ since the term $(T_{V}\d_z +T_{\gamma}) \absm{D}^s \Phi$ does not belong to $L^{2}(\R)$, hence the inequality cannot be closed. To circumvent these issues, we instead evaluate $\frac{\dif}{\dif t} \| \absm{D}^s J_{\ep}  \Phi(t,\cdot) \|_{L^{2}(\R)}^2$, where recall that $J_{\ep}$ is the mollifier given by:
\begin{gather}
	J_{\ep}(t,z,\xi) = J_\ep^{(0)}(t,z,\xi) - \dfrac{i}{2} \d_z \d_\xi J_{\ep}^{(0)}(t,z,\xi) \label{Jep} \quad \text{where} \\
	J_{\ep}^{(0)}(t,z,\xi) = \exp\Big(\!-\!\ep \gamma^{(\frac{3}{2})}(t,z,\xi)\Big) \ \text{ and } \ \gamma^{(\frac{3}{2})}(t,z,\xi) = \dfrac{|\xi|^{\frac{3}{2}}}{\sqrt{2}\absm{\d_z\eta(t,z)}^{\frac{3}{2}}},\nonumber
\end{gather}
Note that since $\eta-R\in H^{s+\frac{1}{2}}(\R)$ with $s>\frac{5}{2}$, one has $J_{\ep}^{(0)} \in \Gamma^{0}_{s-1}(\R)$ and there exists a constant $C>0$ independent of $\ep$ so that $\mM_{s-1}^0\big(J_{\ep}^{(0)}\big) \le C$ uniformly in $\ep\in(0,1]$. Since $\Phi$ solves (\ref{PHIeq}), it follows that 
\begin{align*}
    \dfrac{\dif}{\dif t} \| \absm{D}^s J_{\ep} \Phi \|_{L^2}^2 =&  - 2\big\langle \absm{D}^s J_{\ep} T_{V}\d_z  \Phi, \absm{D}^s J_{\ep} \Phi  \big\rangle_{L^2} - 2\big\langle  \mathbb{i} \absm{D}^s J_{\ep}  T_{\gamma} \Phi, \absm{D}^s J_{\ep} \Phi  \big\rangle_{L^2}  \\
    & + 2\big\langle F , \absm{D}^s J_{\ep} \Phi  \big\rangle_{L^2} + 2 \big\langle \absm{D}^s (\d_t J_{\ep}) \Phi , \absm{D}^s J_{\ep} \Phi  \big\rangle_{L^2}.
\end{align*}
where $\mathbb{i} \vcentcolon= \bigl( \begin{smallmatrix}0 & -1\\ 1 & 0\end{smallmatrix}\bigr)$. Using the symbol form (\ref{Jep}), one obtains that 
\begin{align}\label{cont1}
 \| \absm{D}^s (\d_t J_{\ep}) \Phi \|_{L^2} \le& C\big(\|\d_t\eta\|_{H^{s-1}},\| \teta \|_{H^{s+1/2}}\big)\|\Phi\|_{H^s} \le C\big( \|(\teta,\psi)\|_{H^{s+\frac{1}{2}}\times H^s} \big).
\end{align}
Moreover, using the inequality: $\sup_{x\ge 0}x^{m}e^{-x} \le C(m)$ for fixed $m\ge 0$, we obtain that
\begin{align}
&\| \absm{D}^s J_{\ep} T_{V} \d_z   \Phi \|_{L^2} \le \ep^{-2/3} C\big( \|(\teta,\psi)\|_{H^{s+\frac{1}{2}}\times H^s} \big) \|\Phi\|_{H^s},\\ 
&\| \mathbb{i}  \absm{D}^s J_{\ep} T_{\gamma}  \Phi \|_{L^2} \le \ep^{-1} C\big( \|(\teta,\psi)\|_{H^{s+\frac{1}{2}}\times H^s} \big) \|\Phi\|_{H^s} .\label{cont2}
\end{align}
Since $\Phi$ is constructed from $(\eta,\psi)$ in (\ref{PHI}), we also have $\|\Phi\|_{H^s}\le C\big( \|(\teta,\psi)\|_{H^{s+\frac{1}{2}}\times H^s} \big) $. Thus combining (\ref{cont1})--(\ref{cont2}) one obtains that:
\begin{equation}
\dfrac{\dif}{\dif t} \| \absm{D}^s J_{\ep} \Phi(t,\cdot)\|_{L^2}^2 \le (1+ \ep^{-1}) C\big( \|(\teta,\psi)\|_{H^{s+1/2}\times H^{s}} \big).\label{dtPhi}
\end{equation}
Thus for each fixed $\ep>0$, the mapping $t\mapsto \| J_{\ep} \Phi (\cdot, t) \|_{H^{s}}^2$ is continuous. Our main aim is then to show the convergence $\| (I-J_{\ep}) \Phi \|_{L^{\infty}(0,T;H^{s})} \to 0$ as $\ep\to 0$, which implies that $t\mapsto \|\Phi(t,\cdot)\|_{H^{s}}^2$ is also continuous.

In order to show $\|(I-J_{\ep})\Phi\|_{L^{\infty}(0,T;H^{s})}\to 0$ as $\ep\to 0$, we cite a sharp version of Bony's theorem on para-differential operators, which can be found in Appendix B, Section 2.1.4 of \cite{MZ2005}. We note that the statements presented here is slightly different from the original theorem in \cite{MZ2005} due to the varying conventions used for the definition of para-differential operators.  
\begin{theorem}[M\'etivier, Zumbrun, 2005]\label{thm:MZ2005}
	Let $a(x,\xi)\in \Gamma_1^{k}(\R^d)$ and $b(x,\xi)\in\Gamma_1^{m}(\R^d)$. Denote $\phi(\xi)$ as the cut-off function in Definition \ref{def:parad}, and define the semi-norm:
	\begin{equation}\label{modSemiNorm}
		\mathcal{M}_{\rho}^{k,N}(a) \vcentcolon= \sup\limits_{|\alpha|\le N} \sup\limits_{\xi\in\R^d}  (1+|\xi|)^{|\alpha|-k} \big\| \d_\xi^{\alpha}\big( \phi(\xi) a(\cdot,\xi) \big) \big\|_{W^{\rho,\infty}(\R^d)},
	\end{equation}
	for $k\in\R$ and $N\in\mathbb{N}$. Then the following statements hold:
	\begin{enumerate}[label=\textnormal{(\arabic*)}]
		\item %Let $(T_{a})^{\ast}$ as the adjoint operator. Let $a^{\ast}(x,\xi)$ be the symbol defined in Theorem \ref{thm:adjprod}\ref{item:adj}. Then
		$(T_a)^{\ast}-T_{a^{\ast}}$ is an operator of order $k-1$. More precisely, for each $\mu\in\R$, there exists constant $C>0$ and $N\in\mathbb{N}$ which depend only on $\mu$, $k$, and $d$ such that
		\begin{gather*}
			\big\| (T_a)^{\ast}-T_{a^{\ast}} \big\|_{H^{\mu} \to H^{\mu-k+1}} \le C \mM_0^{k,N}(\d_x a).
		\end{gather*}
		\item The operator $T_a T_b - T_{ab}$ is of order $m+k-1$. More precisely, for $\mu\in\R$, there exists constant $C>0$ and $N\in\mathbb{N}$ which depend only on $\mu$, $k$, and $d$ such that
		\begin{align*}
			\big\| T_aT_b -T_{ab} \big\|_{H^{\mu} \to H^{\mu-k-m+1}} \! \le\! C \big\{ \mM_0^{k,N}\!(a) \mM_0^{m,N}\!(\d_x b) \!+\! \mM_0^{k,N}\!(\d_x a) \mM_0^{m,N}\!(b) \big\}. 
		\end{align*}
	\end{enumerate}   
\end{theorem}
\begin{proposition}\label{prop:MJ}
	For all $k,\,m\in\mathbb{N}\cup\{0\}$ satisfying $1\le k <s-1$, there exists a positive constant $C=C\big(k,m,\|\teta\|_{H^{s+\frac{1}{2}}}\big)>0$ independent of $\ep$ such that
	\begin{equation}\label{DJep}
		\sup\limits_{\ep\in(0,1)}  \sup\limits_{t\ge 0}\sup\limits_{(z,\xi)\in \R^2} \ep^{-\frac{2}{3}} (1+|\xi|)^{m-k}  \big|\d_\xi^m\big( \phi(\xi) \d_z^k J_{\ep}^{(0)}(t,z,\xi) \big) \big| \le C.
	\end{equation} 
	In particular, this implies that for fixed $1\le k< s-1 $ and $N\in\mathbb{N}$, one has
	\begin{equation*}
		\mM_{s-k-1}^{k}\big(\d_z^k J_{\ep}^{(0)}\big) \le \mathcal{O}\big(\ep^{\frac{2}{3}}\big),  \qquad \mM_0^{k,N}\big(\d_z^k J_\ep^{(0)}\big) \le \mathcal{O}\big(\ep^{\frac{2}{3}}\big).
	\end{equation*}
\end{proposition}
\begin{proof}
	First we note that by the construction of $\phi(\xi)$ given in Definition \ref{def:parad}, one has $\supp(\phi)\subset[\frac{1}{2}, \infty)$ and $\supp(\d_\xi^m\phi) \subset [\frac{1}{2},1]$ for all $m\in\mathbb{N}$. Therefore, to prove (\ref{DJep}), it suffices to prove instead the inequality:
	\begin{equation}\label{DJep'}
		\sup\limits_{\ep\in(0,1)}  \sup\limits_{t\ge 0}\sup\limits_{z\in \R} \sup\limits_{|\xi|\ge 1/2} \ep^{-\frac{2}{3}} |\xi|^{m-k}  \big| \d_\xi^m \d_z^k J_{\ep}^{(0)}(t,z,\xi) \big| \le C.
	\end{equation} 
	By definition (\ref{Jep}), for each $m>0$ there exists a constant $C(m)=C\big(m,\|\teta\|_{H^{s+\frac{1}{2}}}\big)>0$ independent of $\ep$ such that
	\begin{equation}\label{JepExp}
		\sup\limits_{\ep\in(0,1)} \sup\limits_{t\ge 0}\sup\limits_{z\in \R} \sup\limits_{|\xi|\ge 1/2} \big(\ep |\xi|^{\frac{3}{2}}\big)^m  J_{\ep}^{(0)}(t,z,\xi) \le C(m).
	\end{equation} 
	Denote $\d\equiv \d_\xi$ or $\d_z$. By Fa\`{a} di Bruno's formula for chain rule, we have
	\begin{equation}
		\d^\alpha J_{\ep}^{(0)} = J_{\ep}^{(0)}\sum_{\beta=1}^{\alpha} (-\ep)^{\beta} B_{\alpha,\beta}\Big( \d \gamma^{(\frac{3}{2})}, \d^2 \gamma^{(\frac{3}{2})}, \dotsc, \d^{\alpha-\beta+1} \gamma^{(\frac{3}{2})} \Big).\label{FaadiBruno}
	\end{equation}
	where $B_{\alpha,\beta}$ is the Bell's polynomial given by
	\begin{equation*}
		B_{\alpha,\beta}\big( x_1, x_2, \dotsc, x_{\alpha-\beta+1} \big) \!\!\vcentcolon=\! \sum \dfrac{\alpha!}{j_1 ! j_2 ! \cdots j_{\alpha-\beta+1}!} \Big( \dfrac{x_1}{1!}\Big)^{j_1} \Big( \dfrac{x_2}{2!} \Big)^{j_2}\!\! \cdots \Big( \dfrac{x_{\alpha-\beta+1}}{(\alpha-\beta+1)!} \Big)^{j_{\alpha-\beta+1}}\!\!,
	\end{equation*} 
	and the above summation is taken over all sequences $j_1,\dotsc,j_{\alpha-\beta+1}$ of non-negative integers satisfying the constraint conditions:
	\begin{gather*}
		j_1+j_2+\cdots+j_{\alpha-\beta+1} = \beta \quad \text{ and } \quad j_1+2j_2+3j_3+\cdots+(\alpha-\beta+1)j_{\alpha-\beta+1} = \alpha.
	\end{gather*}
	Setting $\d=\d_z$ in (\ref{FaadiBruno}), one can verify that for $1\le k <s-1$,
	\begin{equation}
		\d_z^k J_{\ep}^{(0)} =J_{\ep}^{(0)}\sum_{j=1}^{k} \big(\ep |\xi|^{\frac{3}{2}}\big)^j F\big( k,j, \d_z\eta,\dotsc,\d_z^k \eta \big) , \label{dz}
	\end{equation}
	where $F\vcentcolon \mathbb{N}^2 \times \R^{k} \to \R$ is some continuous function such that $$\|F(k,j,\d_z\eta,\dotsc,\d_z^{k}\eta)\|_{L^{\infty}(\R)} \le C(k,j) \|\teta\|_{H^{s+\frac{1}{2}}(\R)}. $$ In addition, setting $\d=\d_\xi$ in (\ref{FaadiBruno}), it also follows that
	\begin{equation}
		\d_\xi^{\alpha} J_{\ep}^{(0)} =  |\xi|^{-\alpha} J_{\ep}^{(0)} \sum_{\beta=1}^{\alpha} \big(\ep|\xi|^{\frac{3}{2}} \big)^{\beta} (1+|\d_z\eta|^2)^{-\frac{3}{4}\beta} q(\alpha,\beta),  \label{dxi}  
	\end{equation}
	where $q(\alpha,\beta)\in \mathbb{Q}$ is some rational function of $(\alpha,\beta)\in\mathbb{N}^2$. Combining (\ref{dz}) and (\ref{dxi}), it follows that there exists some $C=C\big(m,k,\|\teta\|_{H^{s+\frac{1}{2}}}\big)>0$ such that for $|\xi|\ge \frac{1}{2}$,
	\begin{align*}
		&|\d_\xi^m \d_z^k J_{\ep}^{(0)}|\! =\! \bigg|\sum_{\alpha=0}^{m} \binom{m}{\alpha} \d_\xi^{\alpha} J_{\ep}^{(0)} \sum_{j=1}^k \ep^j F \d_\xi^{m-\alpha} |\xi|^{\frac{3}{2}j}\bigg|\le
		%\\ =& J_{\ep}^{(0)} \sum_{j=1}^k \ep^j F \d_\xi^m |\xi|^{\frac{3}{2}j} +\sum_{\alpha=1}^{m} \binom{m}{\alpha} \bigg(\dfrac{J_{\ep}^{(0)}}{|\xi|^{\alpha}} \sum_{\beta=1}^{\alpha} \ep^{\beta} |\xi|^{\frac{3}{2}\beta} (1+|\d_z\eta|^2)^{-\frac{3}{4}\beta} q \bigg) \sum_{j=1}^k \ep^j F \d_\xi^{m-\alpha} |\xi|^{\frac{3}{2}j}\\ \le & C J_{\ep}^{(0)} \bigg\{ \sum_{j=1}^k \ep^j |\xi|^{-m+\frac{3}{2}j} + \sum_{\alpha=1}^m \sum_{\beta=1}^{\alpha} \sum_{j=1}^k \ep^{\beta+j} |\xi|^{-m+\frac{3}{2}(\beta+j)} \bigg\}\\ \le &  C J_{\ep}^{(0)} \bigg\{ \sum_{j=1}^k \ep^j |\xi|^{-m+\frac{3}{2}j} +  \sum_{j=1}^k \sum_{\alpha=1}^m \ep^{\alpha+j} |\xi|^{-m+\frac{3}{2}(\alpha+j)} \bigg\} \\ \le &
		C J_{\ep}^{(0)} \sum_{j=1}^k \sum_{\alpha=0}^m \ep^{\alpha+j} |\xi|^{-m+\frac{3}{2}(\alpha+j)} 
	\end{align*}
	Multiplying both sides by $|\xi|^{m-k}$, we have that
	\begin{equation*}
		|\xi|^{m-k} |\d_\xi^m \d_z^k J_{\ep}^{(0)}|\! \le\! C J_{\ep}^{(0)}\!  \sum_{\alpha=0}^m \sum_{j=1}^k \! \ep^{\alpha+j} |\xi|^{-k+\frac{3}{2}(\alpha+j)}\! =\! C J_{\ep}^{(0)}\! \sum_{\alpha=0}^m \sum_{j=1}^k\! \ep^{\frac{2}{3}j} \big( \ep |\xi|^{\frac{2}{3}} \big)^{\alpha+\frac{j}{3}} |\xi|^{j-k}. 
	\end{equation*}
	Since $|\xi|\ge \frac{1}{2}$, using (\ref{JepExp}) on the above inequality, we have
	\begin{equation*}
		|\xi|^{m-k} |\d_\xi^m \d_z^k J_{\ep}^{(0)}| \le C \sum_{j=1}^k\! \ep^{2j/3} \le C\ep^{\frac{2}{3}} . 
	\end{equation*}
	This proves the proposition.
\end{proof}

\begin{lemma}\label{lemma:JepI}
	Let $J_{\ep}$ be the symbol defined in (\ref{Jep}). Then for all $u\in H^s(\R)$, 
	\begin{equation*}
		\sup\limits_{0\le t\le T}\| (I-J_{\ep}) u \|_{H^{s}(\R)} \to 0 \quad \text{ as } \quad \ep\to 0.
	\end{equation*}
\end{lemma}
\begin{proof}
	Throughout this proof, we will denote $\|\cdot\|\equiv \|\cdot\|_{H^{s}(\R)}$ and $\langle \cdot, \cdot \rangle$ as the inner product for the space $H^{s}(\R)$. Let $\varphi(\xi)\in\mC^{\infty}(\R)$ be a cut-off function such that $\varphi(\xi)=0$ for $|\xi|\le 1$ and $\varphi(\xi)=1$ for $|\xi|\ge 2$. Consider the spectra decomposition:
	\begin{equation*}
		u=w+v \quad \text{where} \quad \hat{w}(\xi) = \big(1-\varphi(\xi)\big) \hat{u}(\xi) \ \text{ and } \ \hat{v}(\xi) = \varphi(\xi) \hat{u}(\xi).
	\end{equation*} 
	Since $\supp(1-\varphi)\subseteq [-2,2]$, one can verify that $w\in H^{\infty}(\R) = \cap_{k>0} H^k(\R) $. We claim that for each fixed $\delta \in (0,\frac{2}{3})$, there exists $C(\delta)>0$ such that
	\begin{equation}\label{IJdelta}
		\| I- J_{\ep} \|_{H^{s+\delta}\to H^{s}} \le C(\delta)\ep^{\frac{2}{3}\delta}.
	\end{equation}
	Suppose the above claim is true. Then, since $\|w\|_{H^{s+\delta}} \le C(\delta)\| v \|_{H^s}$ for a fixed $\delta$, it follows that $\| (I-J_{\ep}) w \|_{H^{s}} \le \mathcal{O}(\ep^{2/3}) \|v\|_{H^{s}}\to 0$ as $\ep\to 0$. Now, for the proof of (\ref{IJdelta}), we first note that one can easily verify $\|J_{\ep}-J_{\ep}^{(0)}\|_{H^{s}\to H^{s}} \le \mathcal{O}(\ep)$ using Proposition \ref{prop:MJ}. Thus we focus on the main difficult term: $\| I- J_{\ep}^{(0)} \|_{H^{s+\delta}\to H^s}$. By the Fundamental Theorem of Calculus over $\ep\mapsto (1-J_{\ep}^{(0)})$, we have
	\begin{align}\label{Jftc}
		1-J_{\ep}^{(0)}(t,z,\xi) = \int_0^{\ep} \gamma^{(\frac{3}{2})}(t,z,\xi) \exp\Big( -\lambda \gamma^{(\frac{3}{2})}(t,z,\xi) \Big)\, \dif \lambda.
	\end{align}
	By the definition of $\gamma^{(3/2)}$, we compute to get:
	\begin{align*}
		|\xi|^{-\delta} \gamma^{(\frac{3}{2})} %= \dfrac{|\xi|^{\frac{3}{2}-\delta}}{\sqrt{2}\absm{\d_z\eta}^{\frac{3}{2}}} 
		= \lambda^{-1+\frac{2}{3}\delta} 2^{-\frac{\delta}{2}} \absm{\d_z\eta}^{-\delta} \Big(\dfrac{\lambda |\xi|^{\frac{3}{2}}}{\sqrt{2}\absm{\d_z\eta}^{\frac{3}{2}}}\Big)^{1-\frac{2}{3}\delta} = \lambda^{-1+\frac{2}{3}\delta} 2^{-\frac{\delta}{3}} \absm{\d_z\eta}^{-\delta} \big(\lambda \gamma^{(\frac{3}{2})}\big)^{1-\frac{2}{3}\delta}.
	\end{align*}
	Substituting this into (\ref{Jftc}), one has
	\begin{align*}
		&1-J_{\ep}^{(0)}(t,z,\xi) = \ep^{\frac{2}{3}\delta}|\xi|^{\delta} \int_0^{\ep} \ep^{-\frac{2}{3}\delta}|\xi|^{-\delta}\gamma^{(\frac{3}{2})} \exp\big( -\lambda \gamma^{(\frac{3}{2})} \big)\, \dif \lambda \nonumber \\ =& \ep^{\frac{2}{3}\delta}|\xi|^{\delta} \cdot  \bigg\{\dfrac{\ep^{-2\delta/3}}{2^{\delta/3}\absm{\d_z\eta}^{\delta}} \int_{0}^{\ep} \!\!\! \lambda^{-1+\frac{2}{3}\delta}  \big( \lambda \gamma^{(\frac{3}{2})} \big)^{1-\frac{2}{3}\delta}\!\!  \exp\big( -\lambda \gamma^{(\frac{3}{2})} \big)\, \dif \lambda \bigg\} =\vcentcolon \ep^{\frac{2}{3}\delta}|\xi|^{\delta} K_{\ep,\delta}. 
	\end{align*}
	This implies the symbol $K_{\ep,\delta}(t,z,D)$ constructed above satisfies the operator relation:
	\begin{align}
		I-J_{\ep}^{(0)}(t,z,D)  = \ep^{\frac{2}{3}\delta} K_{\ep,\delta}(t,z,D) \circ |D|^{\delta}. \label{Ked}
	\end{align}
	Since for $m>0$, $x^m e^{-x} \le C(m)$ uniformly in $x\ge 0$, there exists $C(\delta)>0$ such that
	\begin{equation*}
		\sup\limits_{\ep>0} \sup\limits_{0\le t \le T} \sup\limits_{(z,\xi)\in\R^2} K_{\ep,\delta}(t,z,\xi) \le C(\delta).
	\end{equation*}
	In particular, this estimate implies that $\mM_{0}^{0}\big( K_{\ep,\delta} \big)\le C(\delta)$ uniformly in $\ep$. 
	Applying the proof of Theorem 4.3.1 in Section 4.3 of \cite{Metivier} on the symbol $K_{\ep,\delta}(t,z,D)$, one has
	\begin{align*}
		\|K_{\ep,\delta}(t,z,D)\|_{H^{s}\to H^{s}} \le C \mM_{0}^{0}\big(K_{\ep,\delta}\big) \le C(\delta).
	\end{align*}
	Using this and (\ref{Ked}), we conclude that
	\begin{align*}
		\big\| \big(I-J_{\ep}^{(0)}\big) w \big\|_{H^{s}} = \ep^{\frac{2}{3}\delta}\big\| K_{\ep}\circ |D|^{\delta} w  \big\|_{H^{s}} \le \ep^{\frac{2}{3}\delta} C(\delta)  \| |D|^{\delta} w \|_{H^{s}} \le  C(\delta) \ep^{\frac{2}{3}\delta} \| w \|_{H^{s+\delta}}.
	\end{align*}
	This proves (\ref{IJdelta}). As a result, we obtain the convergence $\|(I-J_{\ep}) w\|_{H^s}\to 0$ as $\ep\to 0$.
	
	Therefore it is left to show $\|(I-J_{\ep})v\|\to 0$ as $\ep\to 0$. First, differentiating on (\ref{Jep}),
	\begin{equation*}
		\d_z\d_\xi J_{\ep}^{(0)}  =\ep^2 \d_\xi \gamma^{(\frac{3}{2})} \d_z \gamma^{(\frac{3}{2})} J_{\ep}^{(0)} -\ep \d_z\d_\xi \gamma^{(\frac{3}{2})} J_{\ep}^{(0)} .
	\end{equation*}
	Using Proposition \ref{prop:MJ}, we have $ \mM_0^{0}\big(\d_z\d_\xi J_{\ep}^{(0)}\big) \le \mathcal{O}(\ep^{2/3}) $. Therefore
	\begin{equation*}
		\| T_{\d_z \d_\xi J_{\ep}^{(0)}} \|_{H^{s}\to H^{s}} \le \mathcal{O}\big(\ep^{\frac{2}{3}}\big).
	\end{equation*}
	Thus it suffices to prove the convergence result for the leading order term $J_{\ep}^{(0)}v$. Define the symbol $p_\ep(t,z,\xi) \vcentcolon= |J_{\ep}^{(0)}(t,z,\xi)|^2 $. Then with few lines of calculations, one has
	\begin{gather}
		\|T_{J_\ep^{(0)}} v\|^2=\langle T_{p_\ep} v, v \rangle + \langle \sum_{k=1}^{3}\mathcal{R}_{\ep}^k v, v \rangle, \qquad \text{ where } \label{Tinner}\\
		\mR_{\ep}^1 = \big\{ \big(T_{J_{\ep}^{(0)}}\big)^{\ast} - T_{(J_{\ep}^{(0)})^{\ast}} \big\} T_{J_{\ep}^{(0)}}, \quad \mR_{\ep}^2=\big\{ T_{(J_{\ep}^{(0)})^{\ast}} - T_{J_{\ep}^{(0)}} \big\} T_{J_{\ep}^{(0)}}, \quad \mR_{\ep}^3 = \big( T_{J_{\ep}^{(0)}}\big)^2 - T_{p_{\ep}}.\nonumber
	\end{gather}
	Since $\mM_{0}^0\big(J_{\ep}^{(0)}\big) \le \mM_{s-1}^0\big(J_{\ep}^{(0)}\big) \le C$ uniformly in $\ep$, Proposition \ref{prop:MJ}, Theorems \ref{thm:MZ2005} and \ref{thm:paraL2} imply that there exists constant $C>0$, $N\in\mathbb{N}$ depending only on $\mu$ such that
	\begin{gather*}
		\| \mR_{\ep}^1 \|_{H^{\mu}\to H^{\mu+1}} \le C \mM_0^{0,N}\big(\d_z J_{\ep}^{(0)}\big) \mM_0^0\big( J_{\ep}^{(0)} \big) \le \mathcal{O}\big(\ep^{\frac{2}{3}}\big),\\
		\|\mR_{\ep}^3\|_{H^{\mu}\to H^{\mu+1}} \le C \mM_{0}^{0,N}\big( \d_z J_{\ep}^{(0)} \big)  \le \mathcal{O}\big(\ep^{\frac{2}{3}}\big).
	\end{gather*}
	In addition, by the symbol calculus given in Theorem \ref{thm:adjprod}\ref{item:adj}, we also have
	\begin{equation*}
		\big(J_{\ep}^{(0)}\big)^{\ast} - J_{\ep}^{(0)} = \sum_{  1 \le \alpha\le s-1 } \dfrac{1}{i^{\alpha}\alpha !} \d_\xi^{\alpha} \d_z^{\alpha} J_{\ep}^{(0)}.
	\end{equation*}
	Thus using Proposition \ref{prop:MJ} and Theorem \ref{thm:paraL2}, it holds that
	\begin{equation*}
		\|\mR^2_{\ep}\|_{H^{\mu}\to H^{\mu}} \le C \mM_0^0\big( J_{\ep}^{(0)} \big) \sum_{1\le \alpha \le s-1} \mM_0^0\big( \d_\xi^{\alpha}\d_z^{\alpha} J_{\ep}^{(0)} \big)  \le \mathcal{O}\big(\ep^{\frac{2}{3}}\big).
	\end{equation*}
	Substituting these estimates into (\ref{Tinner}), we get
	\begin{equation*}
		\lim\limits_{\ep\to 0} \|T_{J_{\ep}^{(0)}} v \|^2 = \lim\limits_{\ep\to 0} \langle T_{p_{\ep}} v, v \rangle.
	\end{equation*}
	To finish the proof, we claim that
	\begin{equation}\label{peClaim}
		\|v\|^2 \le \liminf\limits_{\ep \to 0} \langle T_{p_{\ep}} v, v \rangle \le \limsup\limits_{\ep \to 0}  \langle T_{p_{\ep}} v, v \rangle \le \|v\|^2. 
	\end{equation}
	Since $\teta\in H^{s+\frac{1}{2}}(\R)$, there exists constant $C=C\big(\|\teta\|_{H^{s+1/2}}\big)>0$ such that we have the bounds $\frac{1}{2}C^{-1}|\xi|^{\frac{3}{2}} \le \gamma^{(\frac{3}{2})} \le \frac{1}{2} C |\xi|^{\frac{3}{2}}$. In light of this, we define the Fourier multipliers
	\begin{equation*}
		a_{\ep}(\xi) \vcentcolon= \exp\big( - C \ep |\xi|^{\frac{3}{2}} \big), \qquad b_{\ep}(\xi) \vcentcolon= \exp\big( -C^{-1} \ep |\xi|^{\frac{3}{2}} \big).
	\end{equation*}
	With this construction, we have $a_{\ep} \le p_{\ep} \le b_{\ep}$. Next, for each $\delta\in (0,1)$, we define:
	\begin{equation*}
		q_{\ep,\delta}(z,\xi)\vcentcolon=\sqrt{p_{\ep}(z,\xi)-a_{\ep}(\xi)+\delta}.
	\end{equation*}
	Since $p_{\ep}\in \Gamma_{s-1}^0(\R)$, we also have $q_{\ep,\delta}\in\Gamma_{s-1}^0(\R)$ for fixed $\delta\in(0,1)$. Taking the first derivatives on $q_{\ep,\delta}$, one sees that
	\begin{align*}
		\d_\xi q_{\ep} = -\dfrac{1}{2 q_{\ep}}\big( \d_\xi p_{\ep} - \d_\xi a_{\ep} \big), \qquad \d_z q_{\ep} = -\dfrac{1}{2 q_{\ep}} \d_z p_{\ep}.
	\end{align*}
	Thus in light of Proposition \ref{prop:MJ}, it can be verified that
	\begin{equation}\label{Mq}
		\mM_{0}^{k,N}\big( \d_{z}^k q_{\ep,\delta} \big) \le \dfrac{1}{\delta^{k+N+1/2}} \mathcal{O}\big(\ep^{\frac{2}{3}}\big)
	\end{equation}
	By a similar calculation as (\ref{Tinner}), we get
	\begin{gather}
		\langle T_{p_{\ep}} v, v \rangle = \langle T_{q_{\ep,\delta}^2} v, v \rangle + \langle T_{(a_{\ep}-\delta)} v, v \rangle = \| T_{q_{\ep,\delta}}v \|^2 + \langle T_{(a_{\ep}-\delta)} v, v \rangle + \langle \mR_{\ep,\delta} v, v \rangle \label{epdelta}\\
		\text{where } \ \mR_{\ep,\delta}%\vcentcolon=  \big(T_{q_{\ep,\delta}}\big)^{\ast} \big\{ \big(T_{q_{\ep,\delta}}\big)^{\ast} - T_{q_{\ep,\delta}} \big\}
		\vcentcolon= \big(T_{q_{\ep,\delta}}\big)^{\ast} \big\{ \big(T_{q_{\ep,\delta}}\big)^{\ast} - T_{q_{\ep,\delta}^{\ast}} \big\} + \big(T_{q_{\ep,\delta}}\big)^{\ast} \big\{ T_{q_{\ep,\delta}^{\ast}} - T_{q_{\ep,\delta}} \big\}.\nonumber
	\end{gather}
	Fix $\delta\in (0,1)$, applying a similar argument as the previous derivation, it follows from Theorem \ref{thm:MZ2005}, Proposition \ref{prop:MJ}, and (\ref{Mq}) that
	\begin{equation}\label{Red}
		\|\mR_{\ep,\delta}\|_{H^{\mu}\to H^{\mu+1}} \le \dfrac{1}{\delta^K} \mathcal{O}\big(\ep^{\frac{2}{3}}\big),
	\end{equation}
	for some fixed integer $K=K(\mu)$. Since $a_{\ep}(z,\xi)\to 1$ as $\ep\to 0$, applying Dominated convergence on (\ref{epdelta}), and using (\ref{Red}), we get
	\begin{equation}\label{Tliminf}
		\liminf\limits_{\ep\to 0} \langle T_{p_{\ep}} v, v \rangle \ge \langle T_{(1-\delta)} v, v \rangle = \langle (1-\delta) v, v \rangle = (1-\delta) \|v\|^2,
	\end{equation} 
	where in the last line we used the fact that $\phi(\xi)\hat{v}(\xi)=\hat{v}(\xi)$, where $\phi$ is the cut-off function introduced in Definition \ref{def:parad}. Thus for constant $c\neq 0$, it follows from the definition that $T_{c} v = c v$. Since (\ref{Tliminf}) holds for arbitrary small $\delta\in(0,1)$, we conclude that $\liminf_{\ep\to 0} \langle T_{p_{\ep}} v, v \rangle \ge \|v\|^2$. Applying the same argument with $p_{\ep}$ and $b_{\ep}$, one can also show $\limsup_{\ep\to 0} \langle T_{p_{\ep}} v, v \rangle \le \|v\|^2$. This completes the proof for Lemma \ref{lemma:JepI}, hence also the main result of this section, Lemma \ref{lemma:cit}.
\end{proof}

%---------------------
%     Subsection
%---------------------
\subsection{Uniqueness and stability}\label{ssec:uniq}
We show the uniqueness and the continuity with respect to initial data of the solution $(\eta,\psi)$, by deriving the contraction inequality. The main result of this section is the following: 
\begin{theorem}\label{thm:uniq}
Assume $s>4$, and let $(\eta^0_j-R,\psi^0_j)\in H^{s+\frac{1}{2}}(\R)\times H^{s}(\R)$ for $j=1,\,2$ be two initial data. Suppose $\{(\eta_j-R,\psi_j)\}_{j=1,\,2}\in H^{s+\frac{1}{2}}(\R)\times H^{s}(\R)$ are two solutions to the system (\ref{zak-uniq}), with some time of existence $T_0>0$. Then there exists a constant $C_0>0$ depending only on initial data such that:
\begin{align*}
    \sup\limits_{0\le t\le T_0}\big\| (\eta_1-\eta_2,\psi_1-\psi_2)(t,\cdot) \big\|_{H^{s-1}\times H^{s-\frac{3}{2}}} \le C_0 \big\| (\eta^0_1-\eta^0_2,\psi^0_1-\psi^0_2) \big\|_{H^{s-1}\times H^{s-\frac{3}{2}}}.
\end{align*}
In particular, if $(\eta^0_1,\psi^0_1)=(\eta^0_2,\psi^0_2)$, then $(\eta_1,\psi_1)=(\eta_2,\psi_2)$ for a.e. $(t,z)\in[0,T_0]\times\R$. Moreover, if $(\eta^0_{\delta},\psi^0_{\delta}) \to (\eta^0,\psi^0)$ in $H^{s-1}(\R)\times H^{s-\frac{3}{2}}(\R)$ as $\delta\to 0$, then $(\eta_{\delta},\psi_{\delta}) \to (\eta,\psi)$ in $\mC^{0}\big([0,T];H^{s-1}(\R)\times H^{s-\frac{3}{2}}(\R)\big)$ as $\delta\to 0$.  
\end{theorem}
\begin{remark}
The regularity requirement $s>4$ comes from Lemmas \ref{lemma:Deltaf1-2} and \ref{lemma:Deltaf2-2}. Specifically, this restriction is due to the regularities of the remainder terms $\fR_G(\teta,\psi)$ and $\fR_D(\teta,\psi)$ for the paralinearization in Theorem \ref{thm:paraG} and Lemma \ref{lemma:ParaDyn}.   
It is worthwhile to point out that our estimate is in agreement with the contraction estimate 
of \cite{ABZ-inv}, Theorem 5.2.
\end{remark}
To start, we list some equations, terms, and notations used throughout this section. $(\eta,\psi)$ solves the Zakharov's system (\ref{000-intro-2}) if and only if
\begin{equation}\label{zak-uniq}
(\d_t + T_{V} \d_z + \mathscr{L})\begin{pmatrix}
\eta\\ \psi
\end{pmatrix} = \tilde{f}(\eta,\psi) \vcentcolon= \mathscr{M}_{\mB}^{-1}\begin{pmatrix}
f^1\\f^2
\end{pmatrix},
\end{equation}
where $\mathscr{L}\vcentcolon= \mathscr{M}_{\mB}^{-1} \mathscr{M}_{\lambda,\ell} \mathscr{M}_{\mB}$, with matrices $\mathscr{M}_{\mB}$, $\mathscr{M}_{\lambda,\ell}$ defined in (\ref{scrM}). Moreover, 
\begin{subequations}\label{f12}
\begin{align}
&f^{1}(\eta,\psi) = G[\eta](\psi) - \big\{ T_{\lambda}(\psi-T_{\mB}\eta) - T_{V}\d_z\eta \big\}\label{f1}\\
&\begin{aligned}
f^{2}(\eta,\psi)=&-\dfrac{1}{2}|\d_z\psi|^2 + \dfrac{1}{2}\dfrac{|\d_z\eta\d_z\psi + G[\eta](\psi)|^2}{1+|\d_z\eta|^2} + \mH(\eta)+\dfrac{1}{2R}\\
&+ T_{V}\d_z\psi - T_{\mB} T_{V}\d_z\eta -T_{\mB} G[\eta](\psi) + T_{\ell}\eta.
\end{aligned}\label{f2}
\end{align}
\end{subequations}
Recall from (\ref{BV}), (\ref{lambda}), (\ref{ell}), $\mB$, $V$ can be considered as the functional of $(\eta,\psi)$: 
\begin{equation}\label{funcBV}
\mB(\eta,\psi) = \dfrac{\d_z\eta \d_z\psi + G[\eta](\psi)}{1+|\d_z\eta|^2}, \qquad V(\eta,\psi) = \d_z\psi - \mB(\eta,\psi) \d_z\eta,
\end{equation}
while $\lambda=\sum_{k=0}^1\lambda^{(k)}$ and $\ell=\sum_{k=1}^{2}\ell^{(k)}$ can be considered as the functional of $\eta$:
\begin{subequations}\label{funclamell}
\begin{align}
	\lambda^{(1)}(\xi)\vcentcolon=&|\xi|,  &&\lambda^{(0)}(\eta,\xi) \vcentcolon= - \dfrac{1+2|\d_z\eta|^2 + i (\d_z\eta)^3 \sgn(\xi)}{2\eta}.\\
	\ell^{(2)}(\eta,\xi)\vcentcolon=& \dfrac{|\xi|^2}{2(1+|\d_z\eta|^2)^{3/2}},  &&\ell^{(1)}(\eta,\xi) \vcentcolon= \dfrac{i\xi\d_z\eta (3\eta\d_z^2\eta-1-|\d_z\eta|^2)}{2\eta(1+|\d_z\eta|^2)^{5/2}}.
\end{align}
\end{subequations}
Suppose there exists two solutions $(\eta_j,\psi_j)$ to (\ref{zak-uniq}) with $j=1,2$. We denote
\begin{equation*}
	\mB_j\equiv\mB(\eta_j,\psi_j),\quad V_j\equiv V(\eta_j,\psi_j),\quad \lambda_j\equiv(\xi,\eta_j), \quad \ell_{j}\equiv\ell(\xi,\eta_j)
\end{equation*}
Similarly, we set $\mathscr{L}_j$ to be the operator:
\begin{equation*}
\mathscr{L}_j \vcentcolon=  \begin{pmatrix}
I & 0 \\
T_{\mB_j} & I
\end{pmatrix} \begin{pmatrix}
0 & -T_{\lambda_j}\\
T_{\ell_j} & 0 
\end{pmatrix} \begin{pmatrix}
I & 0 \\
-T_{\mB_j} & I
\end{pmatrix} \quad \text{for } \ j=1,\,2. 
\end{equation*}
We also denote the norms $M_1\equiv M_{1}(T)$ and $M_{2}\equiv M_{2}(T)$ as:
\begin{equation}\label{Mjs}
   M_{j} \equiv  M_{j}(T) \vcentcolon= \sup\limits_{0\le t\le T} \|(\teta_j,\psi_j)(t,\cdot)\|_{H^{s+1/2}(\R)\times H^s(\R)} \qquad \text{for } \ j =1,\,2.
\end{equation}
In addition, we denote the difference of two functions as:
\begin{equation*}
\Delta a \equiv a_1 - a_2 \quad \text{ for } \quad a=\eta,\,\, \psi,\, \mB,\,\, V,\, \lambda,\,\, \ell,\,\, \mathscr{L}.
\end{equation*}
\subsubsection{Estimates for shape differential}
To obtain estimates for $(\Delta \eta, \Delta \psi)$, we first define the functional differential: given a functional $F(\eta,\psi)$, and perturbation functions $\delta\eta\in H^{\sigma}(\R)$ for some $\sigma\in\R$, we set
\begin{align}\label{funcdiff}
\delta F(\eta,\psi) \equiv \dif_{\eta} F(\eta,\psi) \cdot \delta \eta  \vcentcolon= \lim\limits_{\ep\to 0} \dfrac{1}{\ep} \big\{ F(\eta+\ep\delta\eta,\psi) - F(\eta,\psi) \big\}.
\end{align}
With this definition, we obtain the differentials $\delta\mB$, $\delta V$ as follows:
\begin{proposition}\label{prop:deltaBV}
Let $\mB(\eta,\psi)$ and $V(\eta,\psi)$ be the functional given in (\ref{funcBV}). For $s>\frac{5}{2}$, if $(\eta,\psi)\in H^{s+\frac{1}{2}}(\R)\times H^{s}(\R)$ and $\delta\eta\in \mC_{c}^{\infty}(\R)$, then 
\begin{equation*}
	\delta \mB(\eta,\psi) = \dfrac{(1-|\d_z\eta|^2)\d_z\psi  - 2 \d_z \eta G[\eta](\psi)}{(1+|\d_z\eta|^2)^2}\d_z\delta \eta + \dfrac{G[\eta]\big(\mB\delta\eta\big)-\d_z(V\delta\eta)- \mB \eta^{-1} \delta \eta}{1+|\d_z\eta|^2}.
\end{equation*}
In addition, if $\delta\eta\in H^{\sigma}(\R)$ for $\sigma \in [\frac{1}{2}, s-1]$, then $(\delta\mB,\delta V)(\eta,\psi) \in H^{\sigma-1}(\R)$, and
\begin{equation*}
\|\delta\mB(\eta,\psi)\|_{H^{\sigma-1}(\R)} + \|\delta V(\eta,\psi)\|_{H^{\sigma-1}(\R)}  \le C\big(\|(\teta,\psi)\|_{H^{s+\frac{1}{2}}\times H^{s}(\R)}\big) \|\delta\eta\|_{H^\sigma(\R)}.
\end{equation*}
\end{proposition}
\begin{proof}
Applying chain rule for differential (\ref{funcdiff}) and Theorem \ref{thm:shape}, one has
\begin{align*}
	\delta \mB =& -\dfrac{2\d_z\eta \d_z\delta\eta}{(1+|\d_z\eta|^2)^2}\{\d_z\eta \d_z\psi + G[\eta](\psi)\} + \dfrac{\d_z\psi \d_z \delta \eta + \dif_{\eta}G[\eta](\psi)\cdot \delta \eta}{1+|\d_z\eta|^2}\\
	=&\dfrac{(1-|\d_z\eta|^2)\d_z\psi  - 2 \d_z \eta G[\eta](\psi)}{(1+|\d_z\eta|^2)^2}\d_z\delta \eta + \dfrac{G[\eta]\big(\mB\delta\eta\big)-\d_z(V\delta\eta)- \mB \eta^{-1} \delta \eta}{1+|\d_z\eta|^2}
\end{align*}
Since $\mB\in H^{s-1}(\R)$ and $\delta\eta\in H^{\sigma}(\R)$ for $\frac{1}{2} \le \sigma\le s-1$, Proposition \ref{prop:clprod} implies that $\mB\delta\eta\in H^{\sigma}(\R)$. Therefore we have from Lemma \ref{lemma:GSob} that
\begin{align*}
	\|G[\eta](\mB\delta\eta)\|_{H^{\sigma-1}} \le C\big( \|\teta\|_{H^{s+\frac{1}{2}}} \big) \|\mB\delta\eta\|_{H^{\sigma}} \le C\big(\|(\teta,\psi)\|_{H^{s+\frac{1}{2}}\times H^s}\big) \|\delta\eta\|_{H^{\sigma}}.
\end{align*}
Also, $V=\d_z\psi - \mB \d_z\eta\in H^{s-1}(\R)$, hence $\d_z(V\delta\eta) \in H^{\sigma-1}(\R)$. From these, we get
\begin{equation*}
	\|\delta\mB\|_{H^{\sigma-1}(\R)}\le C\big( \|(\teta,\psi)\|_{H^{s+\frac{1}{2}}\times H^{s}(\R)} \big) \|\delta\eta\|_{H^{\sigma}(\R)}.
\end{equation*}
Similarly, since $\delta V = \delta (\d_z \psi - \mB \d_z \eta) = -\delta\mB \d_z\eta - \mB \d_z\delta\eta$, we have by Proposition \ref{prop:clprod},
\begin{equation*}
	\|\delta V\|_{H^{\sigma-1}} \le \|\delta\mB \d_z \eta\|_{H^{\sigma-1}} + \|\mB \d_z\delta\eta\|_{H^{\sigma-1}} \le C\big( \|(\teta,\psi)\|_{H^{s+\frac{1}{2}}\times H^{s}} \big) \|\delta\eta\|_{H^{\sigma}}.
\end{equation*}
This concludes the proof.
\end{proof}
\begin{proposition}\label{prop:deltalamell}
Let $\lambda$ and $\ell$ be the symbols given in (\ref{funclamell}). If $(\eta,\psi)\in (H^{s+\frac{1}{2}}\times H^{s})(\R)$ and $\delta\eta\in H^{\tau}(\R)$ with $s>\frac{5}{2}$ and $\frac{5}{2}-s \le \tau \le s-\frac{1}{2}$, then for each $\xi\in\R\backslash\{0\}$,
\begin{equation*}
\|\delta\lambda(\eta,\xi)\|_{H^{\tau-1}(\R)} + \sum_{k=0}^{1}|\xi|^{k-2}\|\delta\ell^{(2-k)}(\eta,\xi)\|_{H^{\tau-1-k}(\R)} \le  C\big(\|\teta\|_{H^{s+\frac{1}{2}}(\R)}\big)\|\delta\eta\|_{H^{\tau}(\R)}.  
\end{equation*}
In particular, if $\delta\eta\in H^{s-1}(\R)$, then for each $t\ge 0$, 
\begin{equation*}
 (z,\xi)\mapsto \delta\lambda\big(\eta(t,z),\xi\big) \in \mathring{\Gamma}_{s-\frac{5}{2}}^{0}(\R), \quad (z,\xi)\mapsto\delta\ell^{(2)}(\eta(t,z),\xi) \in \mathring{\Gamma}^{2}_{s-\frac{5}{2}}(\R).
\end{equation*}
\end{proposition}
\begin{proof}
Since $\lambda=\lambda^{(1)}+\lambda^{(0)}$ with $\lambda^{(1)}=|\xi|$, applying the product rule for functional differential (\ref{funcdiff}), one has 
\begin{align*}
\delta \lambda %= \delta(\lambda^{(1)}+\lambda^{(0)}) 
= \delta \lambda^{(0)} %= - \delta \Big\{ \dfrac{1+2|\d_z\eta|^2 + i (\d_z\eta)^3\sgn(\xi)}{2\eta} \Big\}\\
=& \dfrac{1+2|\d_z\eta|^2+i(\d_z\eta)^3\sgn(\xi)}{2\eta^2}\delta\eta - \dfrac{4\d_z\eta+3i|\d_z\eta|^2\sgn(\xi)}{2\eta}\d_z\delta\eta
\end{align*}
If $\delta\eta \in H^{\tau}(\R)$ for $\frac{3}{2}-s\le \tau\le s-1$, then it follows from Propositions \ref{prop:Sobcomp}--\ref{prop:clprod} that
\begin{equation*}
\|\delta\lambda (\eta,\xi)\|_{H^{\tau-1}} \le C\big(\|\teta\|_{H^{s+\frac{1}{2}}}\big)\|\delta\eta\|_{H^{\tau}}.
\end{equation*}
Moreover, if $\delta\eta\in H^{s-1}(\R)$, then $z\mapsto \delta\lambda(\eta(t,z),\xi)\in H^{s-2}(\R)\xhookrightarrow[]{}W^{s-\frac{5}{2},\infty}(\R)$, where $W^{k,\infty}(\R)$ is the space stated in Definition \ref{def:holder}. Thus Definition \ref{def:symbols} implies that
\begin{equation*}
(z,\xi)\mapsto \delta\lambda(\eta(t,z),\xi)\in \mathring{\Gamma}_{s-\frac{5}{2}}^{0}(\R).
\end{equation*}
Next, using the product rule for the symbol $\ell=\ell^{(2)}+\ell^{(1)}$ given in (\ref{funclamell}), we have
\begin{align}
&\delta \ell^{(2)}(\eta,\xi) =
-|\xi|^2 \dfrac{6\d_z\eta \d_z\delta\eta}{(1+|\d_z\eta|^2)^{5/2}},\nonumber\\
&\begin{aligned}
\delta \ell^{(1)}(\eta,\xi) =&
i\xi\dfrac{3\d_z\eta}{2(1+|\d_z\eta|^2)^{5/2}}\d_z^2 \delta\eta
+i\xi\dfrac{\d_z\eta}{2\eta^2 (1+|\d_z\eta|^2)^{3/2}} \delta \eta \\
&+i\xi\dfrac{3\eta\d_z^2\eta(1-4|\d_z\eta|^2)+2|\d_z\eta|^4+|\d_z\eta|^2-1}{2\eta(1+|\d_z\eta|^2)^{7/2}}\d_z\delta\eta
\end{aligned}\label{deltaEll1}
\end{align}
Thus, applying Propositions \ref{prop:Sobcomp}--\ref{prop:clprod}, one has for $\frac{5}{2}-s\le \tau \le s-\frac{1}{2}$ that
\begin{gather*}
|\xi|^{-2}\|\delta\ell^{(2)}(\eta,\xi)\|_{H^{\tau-1}} + |\xi|^{-1}\|\delta\ell^{(1)}(\eta,\xi)\|_{H^{\tau-2}} \le C\big(\|\teta\|_{H^{s+\frac{1}{2}}}\big)\|\delta\eta\|_{H^{\tau}}.
\end{gather*}
Finally, since $z \mapsto \delta \ell^{(2)}(\eta(t,z),\xi) \in H^{s-2}(\R)$, by the exact same argument as before we have $(z,\xi)\mapsto \delta\ell^{(2)}(\eta(t,z),\xi) \in \mathring{\Gamma}^{2}_{s-5/2}(\R)$, which concludes the proof. 
\end{proof}
\begin{proposition}\label{prop:Deltas}
	Let $s>\frac{5}{2}$, and $(\eta_j,\psi_j)\in H^{s+\frac{1}{2}}\times H^{s}(\R)$ for $j=1,\,2$ be two solutions satisfying (\ref{zak-uniq}). For $\alpha\in\mathbb{N}$, $\frac{1}{2}\le \sigma\le s-1$, and $\frac{5}{2}-s\le \tau \le s-\frac{1}{2}$, one has
	\begin{gather*}
		\|\Delta V \|_{H^{\sigma-1}} + \|\Delta \mB\|_{H^{\sigma-1}} \le C(M_1,M_2) \|(\Delta\eta,\Delta\psi)\|_{H^{\sigma}\times H^{\sigma}},\\
		\sup_{|\xi|=1} \Big\{ \big\| \d_\xi^{\alpha} \Delta\lambda(\cdot,\xi)\big\|_{H^{\tau-1}} + \sum_{k=0}^{1}\big\| \d_\xi^{\alpha} \Delta\ell^{(2-k)}(\cdot,\xi)\big\|_{H^{\tau-1-k}} \Big\} \le C(M_1,M_2) \|\Delta\eta\|_{H^{\tau}}.
	\end{gather*}
	In addition, $\xi\mapsto \Delta\ell^{(1)}(z,\xi)$ is homogeneous of order $1$, and if $\Delta\eta\in H^{s-1}(\R)$, then 
	\begin{equation*}
		\Delta\lambda(z,\xi) \in \mathring{\Gamma}_{s-\frac{5}{2}}^{0}(\R), \qquad \Delta\ell^{(2)}(z,\xi) \in \mathring{\Gamma}^{2}_{s-\frac{5}{2}}(\R).
	\end{equation*}
\end{proposition}
\begin{proof}
	For $\theta\in[0,1]$, we define $\eta_{(\theta)} \vcentcolon= (1-\theta)\eta_2 + \theta \eta_1$. Then $\teta_{(\theta)}\in H^{s+\frac{1}{2}}(\R)$, where $\tilde{\eta}_{(\theta)}\vcentcolon= \eta_{(\theta)}-R$. In addition, $\d_\theta \eta_{(\theta)} = \Delta \eta$. Since $\psi\mapsto G[\eta](\psi)$ is linear, we have that $\psi\mapsto \mB(\eta,\psi)$ is also linear. Therefore it follows by chain rule that
	\begin{align}\label{DeltaB}
		\Delta \mB =& \mB(\eta_1,\psi_1) - \mB(\eta_1,\psi_2) + \mB(\eta_1,\psi_2) - \mB(\eta_2,\psi_2)\\
		=& \mB(\eta_1,\Delta\psi)  + \int_{0}^{1} \dif_\eta \mB\big(\eta_{(\theta)},\psi_2\big)\cdot \Delta\eta\, \dif \theta.\nonumber
	\end{align} 
	Applying Proposition \ref{prop:deltaBV} with $\delta\eta\equiv \Delta\eta$ and $\frac{1}{2} \le \sigma \le s-1$, one has
	\begin{align*}
		\|\dif_\eta\mB(\eta_{(\theta)},\psi_2)\cdot \Delta\eta\|_{H^{\sigma-1}(\R)} \le C(M_1,M_2)\|\Delta\eta\|_{H^{\sigma}(\R)} \qquad \text{for } \ 2-s \le \sigma\le s-1.
	\end{align*}
	Moreover, using Lemma \ref{lemma:GSob} and Propositions \ref{prop:Sobcomp}--\ref{prop:clprod}, we have for $\frac{1}{2}\le k \le s$,
	\begin{align*}
		&\|\mB(\eta_1,\Delta\psi)\|_{H^{k-1}} = \Big\| \dfrac{\d_z\eta_1 \d_z \Delta\psi + G[\eta_1](\Delta\psi)}{1+|\d_z\eta_1|^2} \Big\|_{H^{k-1}}\\ 
		\le& C\big(\|\d_z\eta_1\|_{H^{s-\frac{1}{2}}}\big)\|\d_z\Delta\psi\|_{H^{k-1}} + C\big(\|\teta_1\|_{H^{s+\frac{1}{2}}}\big)\|\Delta\psi\|_{H^{k}} \le C\big(\|\teta_1\|_{H^{s+\frac{1}{2}}}\big)\|\Delta\psi\|_{H^{k}}
	\end{align*}
	Putting these estimates in (\ref{DeltaB}), we obtain that for $\frac{1}{2}\le \sigma\le s-1$
	\begin{align}\label{DeltaB1}
		\|\Delta B\|_{H^{\sigma-1}} \le C\big(\|\teta_1\|_{H^{s+\frac{1}{2}}}\big)\|\Delta\psi\|_{H^{\sigma}} + C(M_1,M_2)\|\Delta\eta\|_{H^{\sigma}}.
	\end{align}
	This proves the first assertion. Next, by the definition (\ref{funcBV}), one gets
	\begin{equation*}
		\Delta V = V_1 - V_2 = \d_z \Delta \psi - \Delta\mB \d_z\eta_1 - \mB_2 \d_z \Delta\eta.
	\end{equation*}
	By Propositions \ref{prop:Sobcomp}--\ref{prop:clprod} and (\ref{DeltaB1}), we obtain
	\begin{align*}
		\|\Delta V\|_{H^{\sigma-1}} %\le& \|\d_z\Delta\psi\|_{H^{\sigma-1}} + C\|\Delta\mB\|_{H^{\sigma-1}}\|\d_z\eta_1\|_{H^{s-\frac{1}{2}}} + \|\mB_2\|_{H^{s-1}}\|\d_z\Delta\eta\|_{H^{\sigma-1}}\\
		\le & C(M_1,M_2)\|(\Delta\eta,\Delta\psi)\|_{H^{\sigma}\times H^{\sigma}}.
	\end{align*}
	For the estimates of $\Delta\lambda$ and $\Delta \ell$, it follows from chain rule for (\ref{funcdiff}) that
	\begin{equation*}
		\Delta\lambda(z,\xi) = \int_{0}^{1}\!\! \dif_{\eta}\lambda\big( \eta_{(\theta)}, \xi \big)\cdot \Delta \eta \, \dif \theta, \qquad \Delta\ell(z,\xi) = \int_{0}^{1}\!\! \dif_{\eta}\ell\big( \eta_{(\theta)}, \xi \big)\cdot \Delta \eta \, \dif \theta.
	\end{equation*}
	Applying Proposition \ref{prop:deltalamell} with $\delta\eta\equiv \Delta\eta$, we obtain the desired estimate.
\end{proof}
\begin{corollary}\label{corol:CDelta}
	Let $s>\frac{5}{2}$, and $\{(\eta_j,\psi_j)\}_{j=1,2}$ be $2$ solutions satisfying (\ref{zak-uniq}). Then 
	\begin{gather*}
		\|T_{\Delta V}\d_z \eta_2\|_{H^{s-1}(\R)} + \|T_{\Delta V }\d_z \psi_2\|_{H^{s-\frac{3}{2}}(\R)}  \le C(M_1,M_2) \|(\Delta\eta, \Delta \psi)\|_{(H^{s-1}\times H^{s-\frac{3}{2}})(\R)},\\
		\|T_{\Delta \lambda} \psi_2\|_{H^{s-1}(\R)} + \|T_{\Delta \ell} \eta_2\|_{H^{s-\frac{3}{2}}(\R)} \le C(M_2) \|\Delta\eta\|_{H^{s-1}(\R)}.
	\end{gather*}
\end{corollary}
\begin{proof}
	By Proposition \ref{prop:Deltas} with $\sigma=s-\frac{3}{2}$, we have
	\begin{equation*}
		\|\Delta V\|_{H^{s-\frac{5}{2}}} \le C(M_1,M_2)\|(\Delta\eta,\Delta\psi)\|_{H^{s-\frac{3}{2}}\times H^{s-\frac{3}{2}}}\le C(M_1,M_2)\|(\Delta\eta,\Delta\psi)\|_{H^{s-1}\times H^{s-\frac{3}{2}}}.
	\end{equation*}
	Since $H^{s-\frac{5}{2}}(\R)=H^{\frac{1}{2}-(3-s)}(\R)$, it follows from Theorem \ref{thm:paraPEst} and $s>\frac{5}{2}$ that 
	\begin{align*}
		\|T_{\Delta V}\d_z\eta_2\|_{H^{s-1}} \le&  \|T_{\Delta V}\d_z\eta_2\|_{H^{2s-\frac{7}{2}}}  \le C\|\Delta V\|_{H^{s-\frac{5}{2}}}\|\d_z\eta_2\|_{H^{s-\frac{1}{2}}}\\
		\le& C(M_1,M_2)\|(\Delta\eta,\Delta\psi)\|_{H^{s-1}\times H^{s-\frac{3}{2}}},\\
		\|T_{\Delta V}\d_z\psi_2\|_{H^{s-\frac{3}{2}}} \le& \|T_{\Delta V}\d_z\psi_2\|_{H^{2s-4}} \le C \|\Delta V\|_{H^{s-\frac{5}{2}}}\|\d_z\psi_2\|_{H^{s-1}}\\
		\le& C(M_1,M_2)\|(\Delta\eta,\Delta\psi)\|_{H^{s-1}\times H^{s-\frac{3}{2}}}.
	\end{align*}
	Next, by Proposition \ref{prop:Deltas} and Theorem \ref{thm:paraL2}, one has
	\begin{align*}
		\|T_{\Delta\lambda} \psi_2\|_{H^{s-1}} \le& C\|\Delta\lambda(\cdot,\xi)\|_{H^{s-2}}\|\psi_2\|_{H^{s-1}} \le C(M_1,M_2)\|\Delta\eta\|_{H^{s-1}},\\
		\|T_{\Delta\ell^{(2)}}\eta_2\|_{H^{s-\frac{3}{2}}} \le& \sup_{|\xi|=1}\|\Delta\ell^{(2)}(\cdot,\xi)\|_{H^{s-2}}\|\teta_2\|_{H^{s+\frac{1}{2}}} \le C(M_1,M_2)\|\Delta\eta\|_{H^{s-1}}.
	\end{align*}
	Finally, since $z\mapsto \ell^{(1)}(z,\xi)\in H^{s-3}(\R)$ by Proposition \ref{prop:Deltas}, and $\xi\mapsto \Delta\ell^{(1)}(z,\xi)$ is homogenous of order $1$ as indicated by (\ref{deltaEll1}), it follows from Proposition \ref{prop:Hsymbol} that
	\begin{align*}
		\|T_{\Delta\ell^{(1)}} \teta\|_{H^{s-\frac{3}{2}}(\R)} \le& C\sup _{\left| \alpha \right| < {3}/{2}}\sup _{\left| \xi \right|= 1}\left| \xi \right| ^{\left| \alpha \right| -1}\left\| \partial ^{\alpha }_{\xi }\Delta\ell^{(1)}(\cdot,\xi) \right\| _{H^{s-3}(\R)}\|\teta\|_{H^{s+\frac{1}{2}}(\R)}\nonumber\\
		\le & C\big(\|\teta\|_{H^{s+\frac{1}{2}}(\R)}\big)\|\Delta\eta\|_{H^{s-1}(\R)}.
	\end{align*}
	This concludes the proof.
\end{proof}

%%%%%%%%%%%%%%%%%%%%%%%%%%%%%%%%%%%%%%%%%%%%%%%%
%%%%%%%%%%%%%%%%%%%%%%%%%%%%%%%%%%%%%%%%%%%%%%%%
\subsubsection{Kinematic equation}
Next, we obtain estimates for the non-linearity term appearing in the right hand side of kinematic equation, which is $f^1(\eta,\psi)\vcentcolon= G[\eta](\psi)+ T_{V}\d_z\eta - T_{\lambda}U$. 
\begin{lemma}\label{lemma:df1}
Assume $s>\frac{5}{2}$. If $(\teta,\psi)\in (H^{s+\frac{1}{2}}\times H^{s})(\R)$ and $\delta\eta\in H^{s-1}(\R)$, then there exists a positive monotone increasing function $x\mapsto C(x)$ such that 
\begin{equation*}
\|\dif_\eta f^{1}(\eta,\psi)\cdot \delta \eta\|_{H^{s-1}(\R)} \le C\big( \|(\teta,\psi)\|_{H^{s+\frac{1}{2}}\times H^{s}(\R)} \big) \|\delta\eta\|_{H^{s-1}(\R)}.
\end{equation*}
\end{lemma}
\begin{proof}
Since $\delta U\! =\!\delta(\psi-T_{\mB}\eta)\! =\! - T_{\delta B} \eta - T_{\mB}\delta \eta$, Theorem \ref{thm:shape} implies that
\begin{align}
\dif_{\eta} f^{1}(\eta,\psi)\cdot \delta\eta %=& \dif_{\eta}G[\eta](\psi)\cdot \delta\eta + T_{\delta V} \d_z \eta + T_{V} \d_z \delta \eta - T_{\delta\lambda} U - T_{\lambda} \delta U,\nonumber\\
%=& -G[\eta](\mB\delta\eta) - \d_z(V\delta\eta) - \dfrac{\mB}{\eta}\delta\eta + T_{\delta V} \d_z\eta + T_{V}\d_z\delta\eta\nonumber\\ &- T_{\delta\lambda} U + T_{\lambda} T_{\delta\mB} \eta + T_{\lambda} T_{\mB} \delta \eta\nonumber\\
=& \big\{T_{\lambda^{(1)}} T_{\mB} \delta\eta -G[\eta](\mB\delta\eta) - \d_z V \delta\eta\big\} - \big\{ V \d_z\delta\eta - T_{V}\d_z\delta\eta \big\} \nonumber\\
&  -T_{\delta\lambda} U + \big\{T_{\lambda} T_{\delta\mB} \eta + T_{\delta V}\d_z\eta\big\} + \big\{T_{\lambda^{(0)}}T_{\mB}\delta\eta-\dfrac{\mB}{\eta}\delta\eta\big\}  =\vcentcolon \sum_{i=1}^{5}I_i.
\end{align}
Setting $\sigma=s-1$, and using the first order paralinearization, Lemma \ref{lemma:1stOP}, we obtain
\begin{gather}
G[\eta](\mB\delta\eta) = T_{\lambda^{(1)}}(\mB\delta\eta) + \fR_{\sigma}(\teta,\mB\delta\eta), \qquad G[\eta](\mB) = T_{\lambda^{(1)}} \mB + \fR_{\sigma}(\teta,\mB) \label{GuniqPara}\\
\text{where } \left\{\begin{aligned}
&\|\fR_{\sigma}(\teta,\mB\delta\eta)\|_{H^{s-1}(\R)} \le C\big(\|\teta\|_{H^{s+\frac{1}{2}}(\R)}\big) \|\d_z \psi \|_{H^{s-1}(\R)} \|\delta\eta\|_{H^{s-1}(\R)},\\
&\|\fR_{\sigma}(\teta,\mB)\|_{H^{s-1}(\R)} \le C\big(\|\teta\|_{H^{s+\frac{1}{2}}(\R)}\big) \|\d_z \psi \|_{H^{s-1}(\R)}.
\end{aligned}\right.\nonumber
\end{gather}
In addition, by the cancellation Lemma \ref{lemma:cancel}, we have
\begin{equation}\label{fr}
G[\eta](\mB) + \d_z V = \fr \quad \text{ for some } \ \|\fr\|_{H^{s-1}(\R)}\le C\big(\|\teta\|_{H^{s+\frac{1}{2}}(\R)}\big)\|\d_z\psi\|_{H^{s-1}(\R)}.
\end{equation}
Combining this with (\ref{GuniqPara}), we get 
\begin{equation}\label{dzV}
	\d_z V = -G[\eta](\mB) +\fr = - T_{\lambda^{(1)}} \mB - \fR_{\sigma}(\teta,\mB) + \fr.
\end{equation}
Substituting (\ref{GuniqPara}) and (\ref{dzV}) into $I_1$, we obtain
\begin{align}\label{UniqI1}
I_1 \vcentcolon=& T_{\lambda^{(1)}} T_{\mB} \delta\eta -G[\eta](\mB\delta\eta) - \d_z V \delta\eta\\
%=& T_{\lambda^{(1)}} T_{\mB} \delta \eta - T_{\lambda^{(1)}}(\mB\delta\eta) - \fR_{\sigma}(\teta,\mB\delta\eta) + \big\{ T_{\lambda^{(1)}} \mB + \fR_{\sigma}(\teta,\mB) - \fr \big\}\delta\eta\nonumber\\
%=& \delta\eta T_{\lambda^{(1)}}\mB + T_{\lambda^{(1)}}T_{\mB}\delta\eta - T_{\lambda^{(1)}}(\mB\delta\eta) + \big\{ \fR_{\sigma}(\teta,\mB) - \fr \big\} \delta\eta -\fR_{\sigma}(\teta,\mB\delta\eta)\nonumber\\
%=& (\delta\eta-T_{\delta\eta})T_{\lambda^{(1)}}\mB + T_{\delta\eta} T_{\lambda^{(1)}} \mB + T_{\lambda^{(1)}} T_{\mB} \delta \eta - T_{\lambda^{(1)}}(\mB\delta\eta)\nonumber\\ &+\big\{\fR_{\sigma}(\teta,\mB)-\fr\big\}\delta\eta - \fR_{\sigma}(\teta,\mB\delta\eta)\nonumber\\
%=& (\delta\eta-T_{\delta\eta})T_{\lambda^{(1)}}\mB + \big[T_{\delta\eta},T_{\lambda^{(1)}}\big] \mB + T_{\lambda^{(1)}}T_{\delta\eta} \mB + T_{\lambda^{(1)}} T_{\mB} \delta \eta - T_{\lambda^{(1)}}(\mB\delta\eta)\nonumber\\&+\big\{\fR_{\sigma}(\teta,\mB)-\fr\big\}\delta\eta - \fR_{\sigma}(\teta,\mB\delta\eta)\nonumber\\
=&\big\{T_{\lambda^{(1)}}T_{\delta\eta} \mB + T_{\lambda^{(1)}} T_{\mB} \delta \eta - T_{\lambda^{(1)}}(\mB\delta\eta)\big\} +(\delta\eta-T_{\delta\eta})T_{\lambda^{(1)}}\mB\nonumber\\
& +\big[T_{\delta\eta},T_{\lambda^{(1)}}\big] \mB +\big\{\fR_{\sigma}(\teta,\mB)-\fr\big\}\delta\eta - \fR_{\sigma}(\teta,\mB\delta\eta)=\vcentcolon \sum_{i=1}^4 I_{1}^{(i)}.\nonumber
\end{align}
Since $\mB$, $\delta\eta\in H^{s-1}(\R^d)$, Bony's Theorem \ref{thm:bony} and $s>\frac{5}{2}$ implies that
\begin{equation*}
T_{\delta\eta} \mB + T_{\mB}\delta\eta - \mB\delta\eta = - \fp(\mB,\delta\eta) \in H^{2s-\frac{5}{2}}(\R) \subset H^{s}(\R).
\end{equation*}
Since $\lambda^{(1)}=|\xi|$, applying Theorem \ref{thm:paraL2}, one has
\begin{align*}
\|I_1^{(1)}\|_{H^{s-1}(\R)} =& \big\|T_{\lambda^{(1)}}\big( T_{\delta\eta} \mB + T_{\mB} \delta\eta - \mB \delta\eta \big)\big\|_{H^{s-1}(\R)}\\
\le& C \| T_{\delta\eta} \mB + T_{\mB} \delta\eta - \mB \delta\eta \|_{H^{s}(\R)} \le C\big(\|(\teta,\psi)\|_{H^{s+\frac{1}{2}}\times H^s(\R)}\big) \|\delta\eta\|_{H^{s-1}(\R)}.
\end{align*} 
By Theorem \ref{thm:paraL2}, we also have $T_{\lambda^{(1)}}\mB \in H^{s-2}(\R)$. Applying Proposition \ref{prop:halfBony} with $m=s-1$ and $k=s-2$, it follows that
\begin{align*}
\|I_{1}^{(2)}\|_{H^{s-1}(\R)} =& \|(\delta\eta-T_{\delta\eta})T_{\lambda^{(1)}}\mB\|_{H^{s-1}(\R)}\\
 \le& C \|\delta\eta\|_{H^{s-1}(\R)} \|T_{\lambda^{(1)}}\mB\|_{H^{s-2}(\R)} \le \big(\|(\teta,\psi)\|_{H^{s+\frac{1}{2}}\times H^s(\R)}\big) \|\delta\eta\|_{H^{s-1}(\R)}.
\end{align*}
To estimate $I_1^{(3)}$, we note that $\delta\eta\in H^{s-1}(\R)\xhookrightarrow[]{} \mC_{\ast}^{s-\frac{3}{2}}=W^{s-\frac{3}{2},\infty}(\R)$, where $W^{k,\infty}(\R)$ is the space in Definition \ref{def:holder}. Thus $\delta\eta\in \Gamma_{s-3/2}^{0}(\R)\subseteq \Gamma_{1}^{0}(\R)$. It follows that
\begin{equation*}
\lambda^{(1)}\sharp\delta\eta= \sum_{|\alpha|<s-3/2} \dfrac{1}{i^{|\alpha|}\alpha !}\d_\xi^{\alpha}\lambda^{(1)}\d_z^{\alpha}\delta\eta = \lambda^{(1)}\delta\eta - i \sgn(\xi)\d_z\delta\eta,  \qquad \delta\eta \sharp \lambda^{(1)} = \lambda^{(1)} \delta\eta.
\end{equation*}
Since $\d_z\delta\eta\in H^{s-2}(\R)\xhookrightarrow[]{}L^{\infty}(\R)$, if we set the symbol $\nu(z,\xi)\vcentcolon= i\sgn(\xi)\d_z\delta\eta$, then $\nu\in \Gamma_{0}^{0}(\R)$, and $\delta\eta\sharp\lambda^{(1)}-\lambda^{(1)}\sharp\delta\eta=\nu$ Using these calculations, we have
\begin{align*}
\big[ T_{\delta\eta}, T_{\lambda^{(1)}} \big] %=& T_{\delta\eta}T_{\lambda^{(1)}} - T_{\lambda^{(1)}} T_{\delta\eta}  \\
=& T_{\delta\eta}T_{\lambda^{(1)}} - T_{\delta\eta\sharp\lambda^{(1)}} + T_{\delta\eta\sharp\lambda^{(1)}} - T_{\lambda^{(1)}\sharp\delta\eta} + T_{\lambda^{(1)}\sharp\delta\eta} - T_{\lambda^{(1)}} T_{\delta\eta}\\
=& \fq[\delta\eta,\lambda^{(1)}] + T_{\nu} - \fq[\lambda^{(1)},\delta\eta].
\end{align*}
Since $\delta\eta\in \Gamma_{s-3/2}^0(\R)$ and $\lambda^{(1)}=|\xi|\in \Gamma_{\infty}^{1}(\R)\subset \Gamma^{1}_{s-3/2}(\R)$, Theorem \ref{thm:adjprod}\ref{item:prod} implies that $\fq[\delta\eta,\lambda^{(1)}]$ and $\fq[\lambda^{(1)},\delta\eta]$ are operators of order $\frac{5}{2}-s$. Since $\nu\in \Gamma_0^0(\R)$, it follows from Theorem \ref{thm:paraL2} that $T_{\nu}$ has order $0$. Therefore
\begin{align*}
\|I_{1}^{(3)}\|_{H^{s-1}} =& \big\|\big[T_{\delta\eta},T_{\lambda^{(1)}}\big]\mB\big\|_{H^{s-1}} \le \|\fq[\delta\eta,\lambda^{(1)}]\mB\|_{H^{s-1}} + \|T_{\nu}\mB\|_{H^{s-1}} + \|\fq[\lambda^{(1)},\delta\eta]\mB\|_{H^{s-1}}\\ \le& C\big(\|(\teta,\psi)\|_{H^{s+\frac{1}{2}}\times H^s}\big) \|\delta\eta\|_{H^{s-1}}.
\end{align*}
By (\ref{GuniqPara}), (\ref{fr}), and Proposition \ref{prop:clprod}, we have
\begin{align*}
\|I_{1}^{(4)}\|_{H^{s-1}} = \|\big\{\fR_{\sigma}(\teta,\mB)-\fr\big\}\delta\eta - \fR_{\sigma}(\teta,\mB\delta\eta)\|_{H^{s-1}}\le C\big(\|(\teta,\psi)\|_{H^{s+\frac{1}{2}}\times H^s}\big) \|\delta\eta\|_{H^{s-1}}
\end{align*}
Substituting the estimates for $I_1^{(1)}$--$I_1^{(4)}$ into (\ref{UniqI1}), we obtain that
\begin{equation*}
\|I_1\|_{H^{s-1}(\R)} \le C\big(\|(\teta,\psi)\|_{H^{s+\frac{1}{2}}\times H^s(\R)}\big) \|\delta\eta\|_{H^{s-1}(\R)}.
\end{equation*}
Next, we estimate $I_2$. Using Proposition \ref{prop:halfBony} with $m=s-1$ and $k=s-2$, we get
\begin{align*}
\|I_2\|_{H^{s-1}(\R)} =& \|(V-T_{V})\d_z\delta\eta\|_{H^{s-1}(\R)} \le C \|V\|_{H^{s-1}(\R)}\|\d_z\delta\eta\|_{H^{s-2}(\R)}\\
\le& C\big(\|(\teta,\psi)\|_{H^{s+\frac{1}{2}}\times H^s(\R)}\big) \|\delta\eta\|_{H^{s-1}(\R)}.
\end{align*}
For $I_3$, using Proposition \ref{prop:deltalamell} and $\delta\eta\in H^{s-1}(\R)$, one has $z\mapsto\delta\lambda\big(\eta(t,z),\xi\big)\in H^{s-2}(\R)$. Since $H^{s-2}(\R)\xhookrightarrow[]{}L^{\infty}(\R)$, it follows that $\delta\lambda \in \Gamma_{0}^0(\R)$. Thus by Theorem \ref{thm:paraL2},
\begin{align*}
\|I_3\|_{H^{s-1}(\R)} =& \|T_{\delta\lambda} U\|_{H^{s-1}(\R)} \le C \mM_0^0(\delta\lambda) \|\psi-T_{\mB}\eta\|_{H^{s-1}(\R)}\\
\le & C \|\d_z\delta\eta\|_{L^{\infty}(\R)}\big\{ \|\psi\|_{H^{s-1}(\R)} + \|\mB\|_{H^{s-1}(\R)}\|\teta\|_{H^{s+\frac{1}{2}}(\R)} \big\}\\
\le& C\big(\|(\teta,\psi)\|_{H^{s+\frac{1}{2}}\times H^s(\R)}\big) \|\delta\eta\|_{H^{s-1}(\R)}.
\end{align*}
For $I_4$, it follows from Proposition \ref{prop:deltaBV} that
\begin{equation}\label{dmB}
	\|\delta\mB\|_{H^{s-2}(\R)} + \|\delta V\|_{H^{s-2}(\R)}\le C\big( \|(\teta,\psi)\|_{H^{s+\frac{1}{2}}\times H^{s}(\R)} \big) \|\delta\eta\|_{H^{s-1}(\R)}.
\end{equation}
Since $H^{s-2}(\R)\xhookrightarrow[]{}L^{\infty}(\R)$, (\ref{dmB}) implies $\delta\mB$, $\delta V\in \Gamma_0^0(\R)$. Thus by Theorem \ref{thm:paraL2},
\begin{align*}
\|T_{\delta\mB}\eta\|_{H^{s}} \le& C\|\delta\mB\|_{H^{s-2}} \|\teta\|_{H^{s}}\le C\big( \|(\teta,\psi)\|_{H^{s+\frac{1}{2}}\times H^{s}} \big) \|\delta\eta\|_{H^{s-1}},\\
\|T_{\delta V} \d_z\eta\|_{H^{s-1}} \le& C \|\delta V\|_{H^{s-2}}\|\d_z\eta\|_{H^{s-1}} \le C\big( \|(\teta,\psi)\|_{H^{s+\frac{1}{2}}\times H^{s}} \big) \|\delta\eta\|_{H^{s-1}}.
\end{align*}
It follows from Theorem \ref{thm:paraG} that $\lambda\in \Gamma_0^1(\R)$, thus applying Theorem \ref{thm:paraL2}, one has
\begin{align*}
\|I_4\|_{H^{s-1}} =& \|T_{\lambda} T_{\delta B}\eta + T_{\delta V} \d_z \eta \|_{H^{s-1}} \le \|T_{\lambda}T_{\delta B}\eta\|_{H^{s-1}} + \|T_{\delta V} \d_z \eta\|_{H^{s-1}}\\
\le & \|T_{\delta B}\eta\|_{H^{s}} + \|T_{\delta V} \d_z \eta\|_{H^{s-1}} \le C\big( \|(\teta,\psi)\|_{H^{s+\frac{1}{2}}\times H^{s}} \big) \|\delta\eta\|_{H^{s-1}}.
\end{align*}
Finally, we estimate $I_5$. Since $\lambda^{(0)}\in \Gamma^0_0(\R)$, it follows from Theorem \ref{thm:paraL2} that
\begin{align*}
\|I_5\|_{H^{s-1}} = \| T_{\lambda^{(0)}}T_{\mB}\delta\eta - \mB \eta^{-1} \delta\eta \|_{H^{s-1}} \le C\big( \|(\teta,\psi)\|_{H^{s+\frac{1}{2}}\times H^{s}} \big) \|\delta\eta\|_{H^{s-1}}.
\end{align*} 
Substituting the estimates for $I_1$--$I_5$ in (\ref{UniqI1}), we conclude the proof.
\end{proof}

Next, for the estimate of $f^{(1)}(\eta_2,\psi_1)-f^{(1)}(\eta_2,\psi_2)$, we define:
\begin{equation}\label{rhd}
\begin{aligned}
\rhd\mB \vcentcolon=& \mB(\eta_2,\psi_1)-\mB(\eta_2,\psi_2) = \mB(\eta_2,\Delta\psi),\\
\rhd V \vcentcolon=& V(\eta_2,\psi_1)-V(\eta_2,\psi_2) %= \d_z\Delta \psi - \mB(\eta_2,\Delta\psi)\d_z\eta_2 
= V(\eta_2,\Delta\psi). 
\end{aligned}
\end{equation}

\begin{proposition}\label{prop:rhd}
Suppose $s>4$ and $(\teta_j,\psi_j)\in (H^{s+\frac{1}{2}}\times H^{s})(\R)$ for $j=1,\,2$. Then
\begin{equation}\label{rhdBV}
\|\rhd \mB\|_{H^{s-\frac{5}{2}}} + \|\rhd V\|_{H^{s-\frac{5}{2}}} \le C\big(\|\teta_2\|_{H^{s-1}}\big)\|\Delta\psi\|_{H^{s-\frac{3}{2}}} .	
\end{equation}
\end{proposition}
\begin{proof}
Set $\sigma\vcentcolon= s-\frac{3}{2}$. Then $\sigma>\frac{5}{2}$, and $ (\teta_2,\Delta \psi) \in H^{\sigma+\frac{1}{2}}(\R)\times H^{\sigma}(\R)$ satisfy the assumption of Lemma \ref{lemma:GSob} in the places of $s$ and $(\teta,\psi) \in H^{s+\frac{1}{2}}(\R)\times H^{s}(\R)$. Therefore
\begin{equation*}
    \| G[\eta_2](\Delta \psi) \|_{H^{\sigma-1}(\R)} \le C\big( \|\teta\|_{H^{\sigma+\frac{1}{2}}(\R)} \big) \|\Delta \psi\|_{H^{\sigma}(\R)} = C\big( \|\teta\|_{H^{s-1}(\R)} \big) \|\Delta \psi\|_{H^{s-\frac{3}{2}}(\R)}.
\end{equation*}
Combining the above estimate with (\ref{funcBV}) and (\ref{rhd}), the result follows.
\end{proof}

\begin{lemma}\label{lemma:Deltaf1-2}
Suppose $s>4$ and $(\teta_j,\psi_j)\in (H^{s+\frac{1}{2}}\times H^{s})(\R)$ for $j=1,\,2$. Then there exists a positive monotone increasing function $x\mapsto C(x)$ such that 
\begin{equation*}
\|f^{1}(\eta_2,\psi_1)-f^{1}(\eta_2,\psi_2)\|_{H^{s-1}(\R)} \le C\big( \| \teta_2 \|_{H^{s-1}(\R)} \big) \|\Delta\psi\|_{H^{s-\frac{3}{2}}(\R)}.
\end{equation*}
\end{lemma}
\begin{proof}
Using the definition (\ref{f1}), it follows that
\begin{align}\label{f1:temp1}
f^{1}(\eta_2,\psi_1) - f^{1}(\eta_2,\psi_2) = G[\eta_2](\Delta\psi) + T_{\rhd V} \d_z\eta_2 - T_{\lambda_2} (\Delta\psi - T_{\rhd \mB}\eta_2)
\end{align}
If we set $\sigma\vcentcolon= s-\frac{3}{2}$, then $s>4$ implies that $\sigma >\frac{5}{2}$. Hence $\sigma$ and $(\eta_2,\Delta\psi) \in H^{\sigma+1/2}\times H^{\sigma}$ satisfy the assumptions of Theorem \ref{thm:paraG}, in place of $s$ and $(\teta,\psi)\in H^{s+1/2}\times H^s$. Thus we can apply Theorem \ref{thm:paraG} with (\ref{rhd}) to obtain the following paralinearization:
\begin{gather*}
G[\eta_2](\Delta\psi) = T_{\lambda_2} (\Delta\psi - T_{\rhd \mB} \eta_2) - T_{\rhd V}\d_z \eta_2 + \fR(\eta_2,\Delta\psi),\\
 \text{where } \quad \|\fR(\eta_2,\Delta\psi)\|_{H^{\sigma+\frac{1}{2}}}\le C\big(\|\teta_2\|_{H^{\sigma+\frac{1}{2}}}\big)\|\Delta\psi\|_{H^{\sigma}}.
\end{gather*}
Rewriting in $\sigma=s-\frac{3}{2}$, we obtain the estimate:
\begin{align*}
\|f^{1}(\eta_2,\psi_1) - f^{1}(\eta_2,\psi_2)\|_{H^{s-1}}=&\big\|G[\eta_2](\Delta\psi) + T_{\rhd V} \d_z\eta_2 - T_{\lambda_2} (\Delta\psi - T_{\rhd \mB}\eta_2)\big\|_{H^{s-1}}\nonumber\\
 =& \|\fR(\eta_2,\Delta\psi)\|_{H^{s-1}}\le C\big(\|\teta_2\|_{H^{s-1}}\big)\|\Delta\psi\|_{H^{s-\frac{3}{2}}}.
\end{align*}
This concludes the proof.
\end{proof}
\begin{corollary}\label{corol:Deltaf1}
Assume $s>4$. If $(\teta_j,\psi_j)\in (H^{s+\frac{1}{2}}\times H^{s})(\R)$ for $j=1,\,2$, then there exists a positive monotone increasing function $x\mapsto C(x)$ such that 
\begin{equation*}
	\|f^{1}(\eta_1,\psi_1)-f^{1}(\eta_2,\psi_2)\|_{H^{s-1}(\R)} \le C(M_1,M_2)\|(\Delta\eta,\Delta\psi)\|_{H^{s-1}\times H^{s-3/2}(\R)}.
\end{equation*}
\end{corollary}
\begin{proof}
Then by chain rule for functional differential (\ref{funcdiff}), it follows that
\begin{align*}
	f^{1}(\eta_1,\psi_1)-f^{1}(\eta_2,\psi_2) %=& f^{1}(\eta_1,\psi_1) - f^{1}(\eta_2,\psi_1) + f^{1}(\eta_2,\psi_1) - f^{1}(\eta_2,\psi_2)\\
	=&\int_{0}^1\!\!\dif_{\eta}f^{1}\big(\eta_{(\theta)},\psi_1\big)\cdot \Delta\eta\, \dif \theta + f^{1}(\eta_2,\psi_1) - f^{1}(\eta_2,\psi_2).
\end{align*}
By Lemma \ref{lemma:df1}, one has
\begin{equation*}
	\Big\| \int_{0}^1\!\!\dif_{\eta}f^{1}\big(\eta_{(\theta)},\psi_1\big)\cdot \Delta\eta\, \dif \theta \Big\|_{H^{s-1}} \le C\big( M_1, M_2 \big) \|\Delta\eta\|_{H^{s-1}}.
\end{equation*}
Combining this with Lemma \ref{lemma:Deltaf1-2}, we obtain the desired result.
\end{proof}

%%%%%%%%%%%%%%%%%%%%%%%%%%%%%%%%%%%%%%%%%%%%%
%%%%%%%%%%%%%%%%%%%%%%%%%%%%%%%%%%%%%%%%%%%%%
\subsubsection{Dynamic equation}
Next, we consider the non-linearity term appearing in the right hand side of dynamic equation, which is $f^2(\eta,\psi)$ given in (\ref{f2}). Before we do so, the following proposition regarding the mean curvature $\mH(\eta)$ is obtained:
\begin{proposition}\label{prop:deltaH}
Suppose $s>\frac{5}{2}$. Let $\mH(\eta)$ be the mean curvature given in (\ref{H}), and $\ell(\eta,\xi)$ be the symbol in (\ref{funclamell}). If $\teta\in H^{s+\frac{1}{2}}(\R)$ and $\delta\eta\in H^{s-1}(\R)$, then there exists a positive monotone increasing function $x\mapsto C(x)$ such that
\begin{equation*}
\| \delta \mH(\eta) + T_{\ell} \delta\eta + T_{\delta \ell} \teta \|_{H^{s-\frac{3}{2}}(\R)} \le C\big(\|\teta\|_{H^{s+\frac{1}{2}}(\R)}\big) \|\delta\eta\|_{H^{s-1}(\R)}.
\end{equation*}
\end{proposition}
\begin{proof}
By definition (\ref{H}) and product rule for (\ref{funcdiff}), we get
\begin{gather*}
	\delta \mH(\eta) = \fh_2 \d_z^2 \delta\eta - \fh_1 \d_z\delta\eta + \fh_0 \delta\eta,\\
	\text{where } \ \fh_2\vcentcolon= \dfrac{1}{2(1+|\d_z\eta|^2)^{3/2}}, \ \ \fh_1\vcentcolon= \dfrac{\d_z\eta(3\eta\d_z^2\eta-1-|\d_z\eta|^2)}{2\eta(1+|\d_z\eta|^2)^{5/2}}, \ \ \fh_0\vcentcolon= \dfrac{1}{2\eta^2 \sqrt{1+|\d_z\eta|^2}}. 
\end{gather*}
Moreover, by Definition \ref{def:parad}, and (\ref{ell}) in Lemma \ref{lemma:mean}, it follows that
\begin{equation}\label{Hpara1}
	T_{\fh_2} \d_z^2 \delta\eta  - T_{\fh_1}\d_z\delta\eta + T_{\fh_0}\delta\eta = -T_{\ell^{(2)}}\delta\eta - T_{\ell^{(1)}}\delta\eta + T_{\ell^{(0)}}\delta\eta = -T_{\ell} \delta\eta + T_{\ell^{(0)}}\delta\eta.
\end{equation}
Using the paralinearization rule (\ref{brackets}) and (\ref{Hpara1}), we obtain that
\begin{gather}
	\begin{aligned}
		\delta \mH(\eta) %=& T_{\fh_2} \d_z^2 \delta\eta + T_{\d_z^2 \delta\eta} \fh_2 + \fp(\fh_2,\d_z^2\delta\eta) - T_{\fh_1}\d_z\delta\eta - T_{\d_z\delta\eta} \fh_1 - \fp(\fh_1,\d_z\delta\eta)\\ &+T_{\fh_0}\delta\eta + T_{\delta\eta} \fh_0 + \fp(\fh_0,\delta\eta)\\
		=& T_{\fh_2} \d_z^2 \delta\eta  - T_{\fh_1}\d_z\delta\eta + T_{\fh_0}\delta\eta + T_{\d_z^2 \delta\eta} \fh_2 - T_{\d_z\delta\eta} \fh_1 + T_{\delta\eta} \fh_0 + \fr_{\mH}\\
		=& - T_{\ell} \delta\eta + T_{\ell^{(0)}}\delta\eta + T_{\d_z^2 \delta\eta} \fh_2 - T_{\d_z\delta\eta} \fh_1 + T_{\delta\eta} \fh_0 + \fr_{\mH},
	\end{aligned} \label{deltaH} \\
	\text{where } \ \fr_{\mH}\vcentcolon= \fp(\fh_2,\d_z^2\delta\eta) - \fp(\fh_1,\d_z\delta\eta) - \fp(\fh_0,\delta\eta).\nonumber
\end{gather}
Since $\fh_2,\, \fh_0 \in H^{s-\frac{1}{2}}(\R)$, $\fh_1 \in H^{s-\frac{3}{2}}(\R)$, and $\delta\eta\in H^{s-1}(\R)$, applying Theorem \ref{thm:bony}\ref{item:err1} and using the fact that $2s-4 > s-\frac{3}{2}$, we have
\begin{equation}\label{rH}
	\|\fr_{\mH}\|_{H^{s-\frac{3}{2}}} \le \|\fr_{\mH}\|_{H^{2s-4}} \le C\big(\|\teta\|_{H^{s+\frac{1}{2}}}\big)\|\delta\eta\|_{H^{s-1}}.
\end{equation}
Since $\d_z^2 \delta\eta \in H^{\frac{1}{2}-(\frac{7}{2}-s)}(\R)$ and $2s-4>s-\frac{3}{2}$, it follows from Theorem \ref{thm:paraPEst} that
\begin{align}\label{Hpara2}
\|T_{\d_z^2\delta\eta}\fh_2\|_{H^{s-\frac{3}{2}}} \le \|T_{\d_z^2\delta\eta}\fh_2\|_{H^{2s-4}} %\le C\|\fh_2\|_{H^{s-\frac{1}{2}}}\|\d_z^2\delta\eta\|_{H^{s-3}}
\le C\big(\|\teta\|_{H^{s+\frac{1}{2}}}\big)\|\delta\eta\|_{H^{s-1}}. 
\end{align}
Since $\d_z\delta\eta \in H^{s-2}(\R)\xhookrightarrow[]{} L^{\infty}(\R)$, one has $\d_z\delta\eta\in \Gamma_{0}^{0}(\R)$, thus by Theorem \ref{thm:paraL2},
\begin{align}\label{Hpara3}
\|-T_{\d_z\delta\eta}\fh_1-T_{\delta\eta} \fh_0\|_{H^{s-\frac{3}{2}}} %\le& C \|\d_z\delta\eta\|_{L^{\infty}} \|\fh_1\|_{H^{s-\frac{3}{2}}} + \|\delta\eta\|_{L^{\infty}}\|\fh_0\|_{H^{s-\frac{3}{2}}} \\
\le& C\big(\|\teta\|_{H^{s+\frac{1}{2}}}\big) \|\delta\eta\|_{H^{s-1}(\R)}.
\end{align}
Substituting (\ref{rH})--(\ref{Hpara3}) into (\ref{deltaH}), we obtain that
\begin{align*}
\|\delta \mH(\eta) + T_{\ell} \delta\eta\|_{H^{s-\frac{3}{2}}} =& \| T_{\ell^{(0)}}\delta\eta + T_{\d_z^2 \delta\eta} \fh_2 - T_{\d_z\delta\eta} \fh_1 - T_{\delta\eta} \fh_0 + \fr_{\mH} \|_{H^{s-\frac{3}{2}}}\\
\le& C\big(\|\teta\|_{H^{s+\frac{1}{2}}}\big) \|\delta\eta\|_{H^{s-1}}.
\end{align*}
Thus it remains to prove the estimate for $T_{\delta\ell} \teta$. To do this, we first use Proposition \ref{prop:deltalamell} and Theorem \ref{thm:paraL2} to obtain that
\begin{align}\label{dell20}
\|T_{\delta\ell^{(2)}}\teta\|_{H^{s-\frac{3}{2}}}\le& C \|\d_z\delta\eta\|_{L^{\infty}}\|\d_z^2\eta\|_{H^{s-\frac{3}{2}}} \le C\big(\|\teta\|_{H^{s+\frac{1}{2}}}\big)\|\delta\eta\|_{H^{s-1}}.
\end{align} 
Since $z\mapsto \delta\ell^{(1)}(z,\xi) \in H^{s-3}(\R)$ by Proposition \ref{prop:deltalamell}, and $\xi\mapsto \delta\ell^{(1)}(z,\xi)$ is homogeneous of order $1$ as shown in (\ref{deltaEll1}), it follows from Proposition \ref{prop:Hsymbol} that
\begin{align}\label{dell1}
	\|T_{\delta\ell^{(1)}} \teta\|_{H^{s-\frac{3}{2}}(\R)} \le& C\sup _{\left| \alpha \right| < {3}/{2}}\sup _{\left| \xi \right|= 1}\left| \xi \right| ^{\left| \alpha \right| -1}\left\| \partial ^{\alpha }_{\xi }\delta\ell^{(1)}(\cdot,\xi) \right\| _{H^{s-3}(\R)}\|\teta\|_{H^{s+\frac{1}{2}}(\R)}\\
	\le & C\big(\|\teta\|_{H^{s+\frac{1}{2}}(\R)}\big)\|\delta\eta\|_{H^{s-1}(\R)}.\nonumber
\end{align}
Combining (\ref{dell20})--(\ref{dell1}), we obtain the desired estimate for $T_{\delta\ell}\teta$, which concludes the proof.
\end{proof}
\begin{lemma}\label{lemma:df2}
Assume $s>\frac{5}{2}$. If $(\teta,\psi)\in (H^{s+\frac{1}{2}}\times H^{s})(\R)$ and $\delta\eta\in H^{s-1}(\R)$, then there exists a positive monotone increasing function $x\mapsto C(x)$ such that 
\begin{equation*}
\|\dif_\eta f^{2}(\eta,\psi)\cdot \delta \eta\|_{H^{s-\frac{3}{2}}(\R)} \le C\big( \|(\teta,\psi)\|_{H^{s+\frac{1}{2}}\times H^{s}(\R)} \big) \|\delta\eta\|_{H^{s-1}(\R)}.
\end{equation*}
\end{lemma}
\begin{proof}
By (\ref{funcBV}), we have $(1+|\d_z\eta|^2)\mB = \d_z\eta \d_z\psi + G[\eta](\psi)$. Applying the product rule for functional differential (\ref{funcdiff}) on both sides of this equation, we have
\begin{align*}
2 \mB\d_z\eta \d_z\delta\eta + (1+|\d_z\eta|^2) \delta\mB =& \d_z\psi \d_z\delta\eta + \dif_{\eta}G[\eta](\psi)\cdot\delta\eta\\
=& V \d_z\delta\eta + \mB \d_z\eta \d_z\delta\eta + \dif_{\eta}G[\eta](\psi)\cdot\delta\eta,
\end{align*}
where we used $V=\d_z\psi - \mB\d_z\eta$. Multiplying both sides of the above by $\mB$, we obtain
\begin{equation}\label{f2temp1}
\delta\big\{ \dfrac{1+|\d_z\eta|^2}{2}\mB^2 \big\} = \mB^2 \d_z\eta\d_z\delta\eta + (1+|\d_z\eta|^2)\mB \delta\mB = \mB V \d_z\delta\eta + \mB \dif_{\eta}G[\eta](\psi)\cdot\delta\eta.
\end{equation}
According to (\ref{f2}) and (\ref{funcBV}), $f^{2}(\eta,\psi)$ can be rewritten as
\begin{align*}
f^{2}(\eta,\psi)=&-\dfrac{1}{2}|\d_z\psi|^2 + \dfrac{1+|\d_z\eta|^2}{2} \mB^2 + \mH(\eta)+\dfrac{1}{2R}\\ &+ T_{V}\d_z\psi - T_{\mB} T_{V}\d_z\eta -T_{\mB} G[\eta](\psi) + T_{\ell}\teta.
\end{align*}
Thus by (\ref{f2temp1}) and the product rule with functional differential (\ref{funcdiff}), one has
\begin{align}\label{f2Est}
&\dif_{\eta}f^{(2)}(\eta,\psi)\cdot\delta\eta\\ %=& \mB V \d_z\delta\eta + \mB \dif_{\eta}G[\eta](\psi)\cdot\delta\eta + \delta \mH(\eta) + T_{\delta V}\d_z\psi\nonumber\\ & - T_{\delta\mB} T_{V}\d_z\eta - T_{\mB} T_{\delta V}\d_z\eta - T_{\mB} T_{V}\d_z\delta\eta -T_{\delta\mB} G[\eta](\psi)\\&-T_{\mB}\big( \dif_{\eta} G[\eta](\psi) \cdot \delta\eta \big)+ T_{\delta\ell}\teta + T_{\ell}\delta\eta\nonumber\\
=& (\mB V - T_{\mB} T_{V})\d_z\delta\eta + (\mB - T_{\mB}) \dif_{\eta} G[\eta](\psi)\cdot \delta\eta + \big\{\delta \mH(\eta) + T_{\delta\ell} \teta + T_{\ell} \delta\eta\big\} \nonumber\\
&+ T_{\delta V} \d_z\psi - (T_{\delta\mB}T_{V} + T_{\mB}T_{\delta V})\d_z\eta - T_{\delta\mB}G[\eta](\psi)    
= \vcentcolon \sum_{j=1}^5I_j.\nonumber
\end{align}
First, we start with term $I_1=(\mB V-T_{\mB}T_{V})\d_z\delta\eta$. By Theorem \ref{thm:adjprod}\ref{item:prod}, $\fq[\mB,V]= T_{\mB} T_{V} -  T_{\mB V} $ is an operator of order $\frac{3}{2}-s$. Since $s>\frac{5}{2}$, we have $ 2s-\frac{7}{2} > s-1$. Hence the following estimates holds:
\begin{equation*}
\|\fq[\mB,V](\d_z\delta\eta)\|_{H^{s-1}} \le \|\fq[\mB,V](\d_z\delta\eta)\|_{H^{2s-\frac{7}{2}}} \le C\big(\|(\teta,\psi)\|_{H^{s+\frac{1}{2}}\times H^s}\big)\|\d_z\delta\eta\|_{H^{s-2}}. 
\end{equation*}
Using the above estimate and Proposition \ref{prop:halfBony} with $k=s-1$ and $m=s-2$, we get
\begin{align*}
&\|I_1\|_{H^{s-1}} \vcentcolon= \|(\mB V - T_{\mB} T_{V}) \d_z\delta\eta\|_{H^{s-1}} = \| (\mB V - T_{\mB V} - \fq[\mB,V]) \d_z\delta\eta \|_{H^{s-1}}\\
\le & \|(\mB V - T_{\mB V})\d_z\delta\eta\|_{H^{s-1}} + \|\fq[\mB,V] \d_z\delta\eta \|_{H^{s-1}} \le C\big(\|(\teta,\psi)\|_{H^{s+\frac{1}{2}}\times H^s}\big)\|\d_z\delta\eta\|_{H^{s-2}}.
\end{align*}
Next, it follows from Theorem \ref{thm:shape}, Lemma \ref{lemma:GSob} and Propositions \ref{prop:Sobcomp}--\ref{prop:clprod} that
\begin{align*}
\big\|\dif_{\eta}G[\eta](\psi)\cdot \delta\eta\big\|_{H^{s-2}} =& \big\|-G[\eta](\mB\delta\eta) - \d_z(V\delta\eta) - \dfrac{\mB}{\eta} \delta\eta\big\|_{H^{s-2}}\\
%\le & C\big(\|\teta\|_{H^{s+\frac{1}{2}}}\big)\|\mB \delta\eta\|_{H^{s-1}} + \|V\|_{H^{s-1}}\|\d_z\delta\eta\|_{H^{s-2}}\\ &+ \|\d_z V\|_{H^{s-2}}\|\delta\eta\|_{H^{s-1}} + \|\eta^{-1}\mB\|_{H^{s-1}}\|\delta\eta\|_{H^{s-1}}\\
\le& C\big(\|(\teta,\psi)\|_{H^{s+\frac{1}{2}}\times H^s}\big)\|\delta\eta\|_{H^{s-1}}.
\end{align*}
Therefore applying Proposition \ref{prop:halfBony} with $k=s-1$ and $m=s-2$, we obtain
\begin{align*}
\|I_2\|_{H^{s-1}} \vcentcolon=& \big\|(\mB-T_{\mB})\dif_{\eta}G[\eta](\psi)\cdot\delta\eta\big\|_{H^{s-1}} \le C\|\mB\|_{H^{s-1}}\big\|\dif_{\eta}G[\eta](\psi)\cdot\delta\eta\big\|_{H^{s-2}}\\
\le& C\big(\|(\teta,\psi)\|_{H^{s+\frac{1}{2}}\times H^s}\big)\|\delta\eta\|_{H^{s-1}}.
\end{align*}
Next, applying Proposition \ref{prop:deltaH}, we have 
\begin{equation*}
\|I_3\|_{H^{s-\frac{3}{2}}} = \|\delta \mH(\eta) + T_{\ell}\delta\eta + T_{\delta\ell} \teta\|_{H^{s-\frac{3}{2}}} \le C\big(\|\teta\|_{H^{s+\frac{1}{2}}}\big)\|\delta\eta\|_{H^{s-1}}.
\end{equation*}
Since $\mB,\, V,\, \delta \mB,\, \delta V \in H^{s-2}(\R) \xhookrightarrow[]{} L^{\infty}(\R)$, it follows from Theorem \ref{thm:paraL2} that
\begin{align*}
\|I_4\|_{H^{s-1}} \vcentcolon=& \big\|T_{\delta V} \d_z\psi - (T_{\delta\mB}T_{V} + T_{\mB}T_{\delta V})\d_z\eta\big\|_{H^{s-1}} \le C\big(\|(\teta,\psi)\|_{H^{s+\frac{1}{2}}\times H^s}\big)\|\delta\eta\|_{H^{s-1}}.
\end{align*}
By Lemma \ref{lemma:GSob}, we have that $G[\eta](\psi)\in H^{s-1}(\R)$ hence by Theorem \ref{thm:paraL2},
\begin{align*}
\|I_5\|_{H^{s-1}}\vcentcolon= \|-T_{\delta\mB} G[\eta](\eta)\|_{H^{s-1}} \le  C\big(\|(\teta,\psi)\|_{H^{s+\frac{1}{2}}\times H^s}\big)\|\delta\eta\|_{H^{s-1}}.
\end{align*}
Combining the estimates for $I_1$--$I_5$ into (\ref{f2Est}), we obtain the desired result. 
\end{proof}

\begin{proposition}\label{prop:pqdiff}
Let $\fp(\cdot,\cdot)$ and $\fq[\cdot,\cdot](\cdot)$ be the operators defined in Theorems \ref{thm:adjprod}-\ref{thm:bony}. Moreover, Let $u_1$, $u_2$, $v_1$, $v_2$, and $w_1$, $w_2$ be functions belonging to some suitable Sobolev spaces, and denote $\Delta g \vcentcolon= g_1 - g_2$ for $g=u, \, v,\, w$. Then 
\begin{align*}
    \fp(u_1,v_1) - \fp(u_2,v_2) =& \fp(\Delta u, \Delta v ) + \fp(\Delta u, v_2 ) + \fp( u_2, \Delta v),\\
    \fq[u_1,v_1](w_1) - \fq[u_2,v_2](w_2) =& \fq[\Delta u, v_1 ](w_1) + \fq[u_2, \Delta v](w_1) + \fq[u_2,v_2](\Delta w).
\end{align*}
In addition, let $\delta$ be the functional differntial operator defined in (\ref{funcdiff}), then:
\begin{align*}
    \delta \fp(u,v)  = \fp(\delta u, v) + \fp( u, \delta v ),\quad 
    \delta \fq[u,v](w)  = \fq[\delta u, v](w) + \fq[u,\delta v](w) + \fq[u,v](\delta w).
\end{align*}
\end{proposition}
\begin{proof}
    The identities follow from the fact that $(u,v,w)\mapsto \fp(u,v), \, \fq[u,v](w)$ are multilinear maps.
\end{proof}

\begin{lemma}\label{lemma:Deltaf2-2}
Assume $s>4$. If $(\teta_j,\psi_j)\in (H^{s+\frac{1}{2}}\times H^{s})(\R)$ for $j=1,\,2$, then there exists a positive monotone increasing function $x\mapsto C(x)$ such that 
\begin{equation*}
\|f^{2}(\eta_2,\psi_1)-f^{2}(\eta_2,\psi_2)\|_{H^{s-\frac{3}{2}}(\R)} \le C\big( \|(\teta_2,\psi_2)\|_{H^{s-1}(\R)\times H^{s-\frac{3}{2}}(\R)} \big)\|\Delta\psi\|_{H^{s-\frac{3}{2}}(\R)}.
\end{equation*}
\end{lemma}
\begin{proof}
By the definitions of $\mB$ and $V$ in (\ref{funcBV}), we have the identities $G[\eta](\psi) = \mB - V \d_z\eta$ and $\d_z \psi = V + \mB \d_z \eta$. Substituting these into (\ref{f2}), it holds that
\begin{align*}
    &f^{(2)}(\eta,\psi) - \big\{ \mH(\eta) + \dfrac{1}{2R} + T_{\ell} \eta \big\}\\ 
    %=& - \dfrac{1}{2}\big\{ V + \mB \d_z\eta \big\}^2 + \dfrac{1}{2}(1+|\d_z\eta|^2) \mB^2 + T_{V}(V+\mB\d_z\eta)\\ & - T_{\mB} T_{V} \d_z\eta - T_{\mB} ( \mB - V \d_z\eta ) \\  =& - \frac{1}{2}V^2 - V \mB \d_z\eta + \dfrac{1}{2}\mB^2 + T_{V} V + T_{V} (\mB \d_z\eta)\\ &- T_{\mB} T_{V} \d_z\eta - T_{\mB} \mB + T_{\mB}(V \d_z\eta) \\ =& \dfrac{1}{2} \big( \mB^2 - 2 T_{\mB} \mB \big) - \dfrac{1}{2}\big( V^2 - 2 T_{V} V \big) - V\mB \d_z \eta + T_{V}(\mB \d_z \eta) \\ & - T_{\mB} T_{V} \d_z\eta + T_{\mB} (V \d_z\eta) \\
    =& \dfrac{1}{2} \big( \mB^2 - 2 T_{\mB} \mB \big) - \dfrac{1}{2}\big( V^2 - 2 T_{V} V \big) - \big\{ V\mB \d_z \eta - T_{V}(\mB \d_z \eta) - T_{(\mB\d_z\eta)} V \big\}  \\ &   + T_{\mB} T_{\d_z\eta} V - T_{(\mB \d_z\eta)} V + T_{\mB} \big(V \d_z\eta - T_{V}\d_z\eta - T_{\d_z\eta} V \big)  \\
    =& \dfrac{1}{2}\big\{\fp(\mB,\mB) - \fp(V,V) \big\} - \fp(V,\mB\d_z\eta) + T_{\mB} \fp(V,\d_z\eta) + \fq[\mB,\d_z\eta](V).
\end{align*}
Applying Proposition \ref{prop:pqdiff} on the above expression, and using $\rhd \mB$, $\rhd V$ in (\ref{rhd}),
\begin{align*}
    &f^{(2)}(\eta_2,\psi_1) - f^{(2)}(\eta_2,\psi_2) \\
    =&  \dfrac{1}{2}\fp(\rhd \mB,\rhd\mB) + \fp(\mB_2,\rhd\mB) - \dfrac{1}{2} \fp(\rhd V,\rhd V) - \fp(V_2,\rhd V ) - \fp(\rhd V, \rhd \mB \d_z\eta_2 ) \\
    & - \fp(\rhd V, \mB_2 \d_z\eta_2) - \fp(V_2,\rhd\mB \d_z\eta_2) + T_{\rhd \mB} \fp(\rhd V, \d_z\eta_2) + T_{\rhd \mB} \fp(V_2,\d_z\eta_2)\\ 
    & + T_{\mB_2} \fp(\rhd V, \d_z\eta_2) + \fq[\rhd \mB, \d_z\eta_2](\rhd V) + \fq[\rhd \mB, \d_z\eta_2](V_2) + \fq[\mB_2,\d_z\eta_2](\rhd V)
\end{align*}
Set $\sigma\vcentcolon= s-\frac{3}{2}>\frac{5}{2}$. Then it follows by Theorems \ref{thm:paraL2}, \ref{thm:bony}, and Proposition \ref{prop:rhd} that
% \begin{align*} \| \fp(\rhd \mB, \rhd\mB) \|_{H^{\sigma}} \le \| \fp(\rhd \mB, \rhd\mB) \|_{H^{2\sigma -\frac{5}{2}}} \le C\| \rhd \mB \|_{H^{\sigma-1}}^2 \le C\big( \|\teta\|_{H^{\sigma+\frac{1}{2}}} \big) \|\Delta \psi\|_{H^{\sigma}},\\ \| T_{\rhd \mB} \fp(\rhd V, \d_z\eta_2)  \|_{H^{\sigma}} \le \| T_{\rhd \mB} \|_{H^{\sigma}\to H^{\sigma}} \| \fp(\rhd V, \d_z\eta_2) \|_{H^{\sigma}} \le C\big( \|\teta\|_{H^{\sigma+\frac{1}{2}}} \big) \|\Delta \psi\|_{H^{\sigma}}. \end{align*}
\begin{align*}
    &\|\fp(\rhd \mB, \rhd \mB)\|_{H^{\sigma}} + \|\fp(\mB_2, \rhd \mB)\|_{H^{\sigma}} + \|\fp(\rhd V, \rhd V)\|_{H^{\sigma}} + \|\fp(V_2, \rhd V )\|_{H^{\sigma}} \\
    &+ \| \fp(\rhd V, \rhd \mB \d_z\eta_2) \|_{H^{\sigma}} + \| \fp(\rhd V, \mB_2 \d_z\eta_2) \|_{H^{\sigma}} + \| \fp( V_2, \rhd \mB \d_z\eta_2) \|_{H^{\sigma}} \\
    &+ \| T_{\rhd\mB} \fp(\rhd V, \d_z\eta_2) \|_{H^{\sigma}} + \| T_{\rhd\mB} \fp( V_2, \d_z\eta_2) \|_{H^{\sigma}} + \| T_{\mB_2} \fp(\rhd V, \d_z\eta_2) \|_{H^{\sigma}}\\ 
    \le& C\big( \|(\teta_2,\psi_2)\|_{H^{\sigma+\frac{1}{2}}\times H^{\sigma}} \big) \| \Delta \psi \|_{H^{\sigma}}  = C\big( \|(\teta_2,\psi_2)\|_{H^{s-1}\times H^{s-\frac{3}{2}}} \big) \| \Delta \psi \|_{H^{s-\frac{3}{2}}}. 
\end{align*}
Moreover, by Theorems \ref{thm:paraL2} and \ref{thm:adjprod}, we also get
\begin{align*}
    &\| \fq[\rhd \mB, \d_z\eta_2 ](\rhd V) \|_{H^{\sigma}} + \| \fq[\rhd \mB, \d_z\eta_2 ](V_2) \|_{H^{\sigma}} + \| \fq[\mB_2, \d_z\eta_2 ](\rhd V) \|_{H^{\sigma}} \\
    \le& C\big( \|(\teta_2,\psi_2)\|_{H^{\sigma+\frac{1}{2}}\times H^{\sigma}} \big) \| \Delta \psi \|_{H^{\sigma}}  = C\big( \|(\teta_2,\psi_2)\|_{H^{s-1}\times H^{s-\frac{3}{2}}} \big) \| \Delta \psi \|_{H^{s-\frac{3}{2}}}. 
\end{align*}
This concludes the proof.
\end{proof}

\begin{corollary}\label{corol:Deltaf2}
	Assume $s>4$. If $(\teta_j,\psi_j)\in (H^{s+\frac{1}{2}}\times H^{s})(\R)$ for $j=1,\,2$, then there exists a positive monotone increasing function $x\mapsto C(x)$ such that 
	\begin{equation*}
		\|f^{2}(\eta_1,\psi_1)-f^{2}(\eta_2,\psi_2)\|_{H^{s-\frac{3}{2}}(\R)} \le C(M_1,M_2)\|(\Delta\eta,\Delta\psi)\|_{H^{s-1}\times H^{s-\frac{3}{2}}(\R)}.
	\end{equation*}
\end{corollary}
\begin{proof}
 The proof follows exactly the same as Corollary \ref{corol:Deltaf1} but using Lemmas \ref{lemma:df2} and \ref{lemma:Deltaf2-2} instead. 
\end{proof}

%%%%%%%%%%%%%%%%%%%%%%%%%%%%%%%%%%%%%%%%%%%%%%%%
%%%%%%%%%%%%%%%%%%%%%%%%%%%%%%%%%%%%%%%%%%%%%%%%
\subsubsection{Proof of contraction inequality}
\begin{lemma}\label{lemma:uniqDiff}
Suppose $s>4$. If $\{(\eta_j,\psi_j)\}_{j=1,2}$ are solutions to (\ref{zak-uniq}), then
\begin{equation*}
(\d_t + T_{V_1} \d_z + \mathscr{L}_1) \begin{pmatrix}
\Delta\eta\\ \Delta\psi
\end{pmatrix} = \ff(\eta_1,\eta_2,\psi_1,\psi_2), \quad \text{where } \ (\Delta\eta,\Delta\psi)\equiv(\eta_1-\eta_2,\psi_1-\psi_2),
\end{equation*}
for some remainder term $\ff(\eta_1,\eta_2,\psi_1,\psi_2)$ such that
\begin{gather*}
\|\ff(\eta_1,\eta_2,\psi_1,\psi_2)\|_{L^{\infty}(0,T;H^{s-1}\times H^{s-3/2})} \le C\big( M_1,M_2 \big) M_{\Delta}\\
\text{where } \ \ M_{j}\vcentcolon= \|(\teta_j,\psi_j)\|_{L^{\infty}(0,T;H^{s+1/2}\times H^{s})}, \qquad M_{\Delta} \vcentcolon= \|(\Delta\eta,\Delta\psi)\|_{L^{\infty}(0,T;H^{s-1}\times H^{s-3/2})}. 
\end{gather*}
\end{lemma}
\begin{proof}
Subtracting two equations (\ref{zak-uniq}) for $\{(\eta_j,\psi_j)\}_{j=1,2}$, we obtain that
\begin{gather*}
\d_t \begin{pmatrix}
\Delta\eta\\ \Delta\psi
\end{pmatrix} + T_{V_1} \d_z \begin{pmatrix}
\Delta\eta \\ \Delta\psi
\end{pmatrix} + \mathscr{L}_1\begin{pmatrix}
\Delta\eta\\ \Delta\psi
\end{pmatrix} = \ff(\eta_1,\psi_1,\eta_2,\psi_2),\\
\text{where } \ \ \ff(\eta_1,\psi_1,\eta_2,\psi_2) \vcentcolon= \tilde{f}(\eta_1,\psi_1)-\tilde{f}(\eta_2,\psi_2) -T_{\Delta V}\d_z \begin{pmatrix}
\eta_2\\ \psi_2
\end{pmatrix} - \Delta\mathscr{L} \begin{pmatrix}
\eta_2\\ \psi_2
\end{pmatrix}, 
\end{gather*}
where $\tilde{f}(\eta_j,\psi_j)$ are defined in (\ref{zak-uniq}). By Corollaries \ref{corol:CDelta}, \ref{corol:Deltaf1}, and \ref{corol:Deltaf2},
\begin{equation*}
\big\| \ff(\eta_1,\psi_1,\eta_2,\psi_2) \big\|_{H^{s-1}\times H^{s-3/2}} \le C(M_1,M_2)\|(\Delta\eta,\Delta\psi)\|_{H^{s-1}\times H^{s-3/2}}.
\end{equation*}
This concludes the proof. 
\end{proof}

\begin{proof}[Proof of Theorem \ref{thm:uniq}]
We set the notations:
\begin{equation*}
\Delta U \vcentcolon= \Delta\psi - T_{\mB_1} \Delta\eta = \psi_1-\psi_2-T_{\mB_1}(\eta_1-\eta_2), \qquad \Delta \Phi \vcentcolon= \begin{pmatrix}
T_{p_1}\Delta\eta\\
T_{q_1} \Delta U
\end{pmatrix}.
\end{equation*}
Then by Lemma \ref{lemma:uniqDiff}, it can be shown that $\Delta \Phi$ solves the system
\begin{align*}
\d_t \Delta \Phi + T_{V_1} \d_z \Delta \Phi + \begin{pmatrix}
0 & -T_{\gamma_1} \\
T_{\gamma_1} & 0
\end{pmatrix} \Delta \Phi = \Delta \fF,\\
\text{with } \ \|\Delta \fF\|_{L^{\infty}(0,T;H^{s-3/2}\times H^{s-3/2})} \le C(M_1,M_2) M_{\Delta},   	
\end{align*}
where $M_{\Delta}(T) = \|(\Delta\eta,\Delta\psi)\|_{L^{\infty}(0,T;H^{s-1}\times H^{s-3/2})}$. Thus one can apply Proposition \ref{prop:2pair} with $\ep=0$, $\sigma=s-\frac{3}{2}$, $\teta=\Delta\eta$, $\tilde{\psi}=\Delta\psi$, to obtain that $M_{\Delta}$ satisfies the estimate
\begin{equation*}
	M_{\Delta} \le T C(M_1,M_2) M_{\Delta} + C(M_1,M_2) M_{\Delta}^0.
\end{equation*}
By choosing $T>0$ such that $T C\big(M_1(T),M_2(T)\big) < 1/2$, we get $$M_{\Delta}(T)\le 2 C\big(M_1(T),M_2(T)\big) M_{\Delta}(0).$$
This is the desired contraction inequality. Since the size of chosen time interval $T$ depends only on the a priori bounds on $M_1$, $M_2$, we can iterate this result to cover the entire interval $[0,T_0]$. This concludes the proof. 
\end{proof}

%%%%%%%%%%%%%%%%%%%%%%%%%%%%%%%%%%%%%%%%%%%%%%
%%%%%%%%%%%%%%%%%%%%%%%%%%%%%%%%%%%%%%%%%%%%%%
%%%%%%%%%%%%%%%%%%%%%%%%%%%%%%%%%%%%%%%%%%%%%%
%--------------------------------------------%
%                   Appendix                 % 
%--------------------------------------------%

\appendix

%-----------------------------------%
%  Derivation of Zakharov's System  %
%-----------------------------------%
\section{Derivation of Zakharov's System for Cylindrical Jets}\label{appen:zak}

\subsection{Kinematic and dynamic equations}\label{ssec:KinDyn}
For each $t\ge 0$, let $\Omega(t)\subset \R^3$ be the moving domain where the fluid occupies. Denote the Cartesian coordinates as $\x\equiv(x,y,z)$, and let $\vv(t,\x)\vcentcolon= (\mathrm{v}^x,\mathrm{v}^y,\mathrm{v}^z)(t,\x)$ be the velocity vector field. Suppose the flow is \textbf{irrotational}, i.e. $\curl \vv =0$. Then there exists $\phi(t,\x)\vcentcolon [0,\infty)\times\R^3 \to \R$ such that $\vv=\nabla \phi \equiv (\d_x \phi, \d_y \phi, \d_z \phi)$. Moreover, we also assume the flow is \textbf{incompressible}, i.e. $\div \vv = \d_x \mathrm{v}^x + \d_y \mathrm{v}^y + \d_z \mathrm{v}^z =0$. This then implies that
\begin{equation}\label{Lapla}
-\Delta \phi \equiv -(\d_x^2 \phi + \d_y^2 \phi + \d_z^2 \phi) = 0 \qquad \text{for } \ \x\in \Omega(t) \ \text{ and } \ t\ge 0.
\end{equation}
Moreover, the free boundary $\d\Omega(t)$ is modelled by the two laws:
\begin{enumerate}[label=\textnormal{(\roman*)}]
\item The free boundary is following the particle path (Kinematic equation);
\item The Bernoulli's principle is satisfied (Dynamic equation). 
\end{enumerate}
These statements can be formulated as follows: first, let $t\mapsto \X(t)$ be a point on $\d\Omega(t)$ that moves along with the free boundary. Then postulation \textnormal{(i)} indicates that
\begin{equation*}
\dfrac{\dif \X}{\dif t}(t) = \vv(t,\X(t)) = \nabla\phi (t,\X(t)) \quad \text{for } \ t\ge 0 \ \text{ and } \ \X(0)\in \d\Omega(0).
\end{equation*}
Next, suppose that the free boundary is described by $\d\Omega(t)= \{\x\in\R^3  \,\vert\, F(t,\x)=0\}$ at time $t\ge0$, for some function $F\vcentcolon\R^3\to \R$. Since $\X(t)$ is a point on the surface, it follows that $F(t,\X(t))=0$ for all $t\ge 0$. Thus taking derivative in time, we obtain:
\begin{equation*}
\big\{ \d_t F  + \nabla \phi \cdot \nabla F \big\}\big\vert_{\x=\X(t)} = 0, \quad \text{for } \ t\ge 0.
\end{equation*}
For each point $(\x,t)\in \{ (s,\y) \,\vert\, \y\in \d\Omega(s) \}$, one can find a backward characteristic $\X(s;\x,t)\in \d\Omega(s)$ for all $0 \le s\le t$ and $\X(t;\x,t)=\x$. Thus the above equation is true for any $\x\in \d\Omega(t)$ with $t\ge0$. In summary the postulations \textnormal{(i)}--\textnormal{(ii)} can be written as
\begin{subequations}\label{PDEs}
\begin{align}
\textrm{(i)} \qquad & \Big\{ \d_t F + \nabla \phi \cdot \nabla F \Big\} \Big\vert_{\x\in\d\Omega(t)} = 0, \label{lagrange}\\
\textrm{(ii)} \qquad & \Big\{ \d_t \phi + \dfrac{1}{2} |\nabla \phi|^2 + P \Big\} \Big\vert_{\x\in\d\Omega(t)} = 0,\label{bernoulli}
\end{align}
\end{subequations}
where $\textrm{(ii)}$ is the Bernoulli's principle and $P$ is the pressure.

%%%%%%%%%%%%%%%%%%%%%%%%%%%%%%%%%%%%%%%%%%%%%%%%%%%%%%%

\subsection{Dirichlet-Neumann operator in cylindrical coordinate} 
Let $\Gamma$ be a general surface, which is described by: $\Gamma=\{ \x \,\vert\, F(\x) = 0 \}$ for some $F(\x)\vcentcolon \R^3\to\R$. Then the Dirichlet-Neumman operator(DN operator for short) associated with $F$, denoted as $G[F](\cdot)$, is a functional defined on the space of functions $\{ \psi : \Gamma \to \R \}$ by:
\begin{equation}\label{DN}
G[F](\psi) = \nabla F \cdot \nabla \phi \vert_{\x\in\Gamma} \quad \text{where $\phi$ solves } \ \left\{ \begin{aligned}
&-\Delta \phi = 0 && \text{in } \ \{F(\x)<0\},\\
&\phi(\x) = \psi(\x) && \text{at } \x\in\Gamma.
\end{aligned} \right.
\end{equation}
Note that the vector $\nabla F \vert_{\x\in\Gamma} $ is parallel to the outward unit normal of the surface $\un(\x)$. With this definition, the term in (\ref{lagrange}) involving convective derivative, $\nabla\phi\cdot \nabla F \vert_{\x\in\d\Omega}$ can be formulated into Dirichlet-Neumann operators, and (\ref{lagrange}) is rewritten as
\begin{equation*}
\d_t F + G[F](\psi) = 0 \quad \text{ on } \ \d\Omega(t) \quad \text{where } \ \psi\vcentcolon= \phi\vert_{\d\Omega(t)}.
\end{equation*}
Since the harmonic equation (\ref{Lapla}) is implicitly imposed in the construction of operator $G[F](\cdot)$, if $P\vert_{\d\Omega}$ is determined as a function of $(\eta,\psi)$, then (\ref{PDEs}) can be thought as a well-determined system of $2$ equations with $2$ unknowns, which are $F$ and $\psi=\phi\vert_{\d\Omega(t)}$, and the problem is posed on the hyper-surface: $Q_T \vcentcolon= \{ (t,\x)\in [0,T]\times\R^3 \vcentcolon \x \in \d\Omega(t) \}$.

Next, we aim to find the expression for $G[F](\cdot)$ when $F(t,\x)=0$ describes a cylindrical surface of revolution. For this, we first list some properties of the cylindrical coordinate system $(z,r,\theta)$ which will be used for the derivation of Zakharov's system. Let $z\in\R$ be the axial position, $r\in(0,\infty)$ be the radial distance, and $\theta\in[0,2\pi)$ be the azimuthal angle. It is related to Cartesian coordinate $(x,y,z)\in\R^3$ by
\begin{gather*}
x=r\cos\theta, \qquad y=r\sin\theta, \qquad z=z,\\
\hr = \hx\cos\theta + \hy \sin\theta, \qquad \hth = -\hx\sin\theta + \hy \cos\theta, \qquad \hz=\hz.
\end{gather*}
Moreover, the unit vectors $(\hz,\hr,\hth)$ satisfies the following relations
\begin{equation}
\hz\times \hr = \hth, \qquad \hr\times\hth=\hz, \qquad \hth\times\hz = \hr.\label{cross-cyl}
\end{equation}
$(\hz,\hr,\hth)$ are functions independent of $(z,r)$, and the following derivative relations hold:
\begin{equation}
\d_\theta\hz=0, \qquad \d_\theta\hr=\hth, \qquad  \d_\theta\hth=-\hr.\label{dun-cyl}
\end{equation}
For a scalar function $f(z,r,\theta)$ and a vector function $V=\hz V^z+\hr V^r + \hth V^{\theta}$, the gradient, Laplacian, and divergence in these coordinates are given by
\begin{subequations}\label{cyl-deriv}
    \begin{gather}
		\nabla f = \hz \d_z f  + \hr \d_r f  + \hth \dfrac{\d_\theta f}{r},\qquad \Delta f =  \d_z^2 f + \dfrac{\d_r( r \d_r f )}{r}  + \dfrac{\d_\theta^2 f}{r^2} ,\label{cyl-grad}\\
		\div V = \d_z V^{z} + \dfrac{\d_r ( r  V^r )}{r}  + \dfrac{\d_\theta V^\theta}{r}.\label{cyl-div}
        %\\ \curl V =& \hz \dfrac{\d_r(r V^{\theta})-\d_\theta V^{r}}{r} + \hr \dfrac{\d_\theta V^{z}-r \d_z V^{\theta}}{r}  +\hth (\d_z V^{r}-\d_r V^{z}).\label{cyl-curl}
    \end{gather}
\end{subequations}
Now suppose that the velocity potential $\phi$ is axially symmetric and the free boundary $\d\Omega(t)$ is a surface of revolution described by the equation:
\begin{equation}\label{F}
0=F(t,x,y,z) = \sqrt{x^2+y^2} - \eta(t,z) =\vcentcolon r - \eta(t,z),
\end{equation}
for some $\eta(t,z)\vcentcolon (0,\infty)\times \R \to (0,\infty)$. In this case, the surface is entirely determined by the function $\eta$. Thus, instead of $G[F](\cdot)$, we use the notation $G[\eta](\cdot)$ to denote the DN operator associated with the surface $\d\Omega(t)=\{r-\eta(t,z)=0\}$. If we define
\begin{equation}\label{Psi}
\Psi(t,z,r,\theta)\vcentcolon= \phi(t,r\cos\theta,r\sin\theta,z),
\end{equation}
then $\Psi(t,z,r,\theta_1)=\Psi(t,z,r,\theta_2)$ for all $\theta_1,\theta_2\in[0,2\pi)$ due to the symmetry. Thus we can denote $\Psi=\Psi(t,z,r)$ without loss of generality. Since the velocity potential $\phi$ solves (\ref{Lapla}), it follows from (\ref{cyl-grad}) that for each $t\ge 0$, $\Psi$ is the solution to :
\begin{subequations}\label{Psi-BVP}
\begin{align}
&-L \Psi \vcentcolon= -\d_z(r\d_z \Psi)-\d_r (r \d_r \Psi) =0 && \text{for }  \ z\in\R, \ 0<r\le \eta(t,z),\label{PsiDE}\\
&\Psi\vert_{r=\eta(t,z)}=\psi(t,z), \quad \d_r \Psi\vert_{r=0}=0 && \text{for }  \ z\in\R.\label{PsiBC}
\end{align}
\end{subequations}
where $\psi(t,z)\vcentcolon= \Psi(t,z,\eta(t,z))$. Note that the Neumann boundary condition, $\d_r\Psi\vert_{r=0} = 0$ comes from the fact that harmonic function $\phi$ is axially symmetric and analytic at $\{r=0\}$. In addition, (\ref{F})--(\ref{Psi}) imply that the velocity vector $\vv=\nabla \phi$ and the normal vector $\nabla F$ under cylindrical coordinate $(z,r)$ is given by
\begin{equation}\label{vv}
\vv = \nabla \phi = \hz \d_z \Psi + \hr \d_r \Psi, \qquad \nabla F = -\hz \d_z \eta + \hr.
\end{equation} 
Substituting (\ref{vv}) in (\ref{DN}), we conclude that for fixed $t\ge 0$, the DN operator $G[\eta](\cdot)$ associated with the surface $\d\Omega(t)=\{\x\in\R^3 \,\vert\, F(t,\x)=r-\eta(t,z)=0 \}$ is
\begin{gather}
    G[\eta](\psi) %\vcentcolon= \nabla F\cdot \nabla \phi \vert_{\x\in \d\Omega} = \sqrt{1+|\d_z \eta|^2} \big(\un \cdot \nabla \phi\big)\big\vert_{\x\in \d\Omega}
    = \big( \d_r \Psi - \d_z \eta \d_z\Psi \big)\big\vert_{r=\eta(t,z)} \quad \text{where }  \Psi  \text{ is the solution to:}\label{DN-cyl}\\
    \Biggl\{\begin{aligned} 
        &-\d_z(r\d_z \Psi)-\d_r (r \d_r \Psi) =0 && \text{for }  \ z\in\R, \ 0<r\le \eta(t,z),\\
        &\Psi\vert_{r=\eta(t,z)}=\psi(t,z), \quad \d_r \Psi\vert_{r=0}=0 && \text{for }  \ z\in\R.
    \end{aligned}\nonumber
\end{gather}

\subsection{Reformulation on the free boundary}\label{subsec:AS}
Using (\ref{F})--(\ref{DN-cyl}), the kinematic and dynamic equations (\ref{PDEs}) can be rewritten as the following system of equations on the free boundary:
\begin{subequations}\label{PDEs-cyl'}
\begin{gather}
\d_t \eta(t,z) - \big\{ \d_r \Psi - \d_z \eta \d_z \Psi \big\} \big\vert_{r=\eta(t,z)} = 0,\label{lag-cyl'}\\
\big\{ \partial_t \Psi + \dfrac{|\d_z \Psi|^2}{2}  + \dfrac{|\d_r\Psi|^2}{2}  + P \big\}\big\vert_{r=\eta(t,z)} = 0,\label{bern-cyl'}  
\end{gather}
\end{subequations}
for all $(t,z)\in[0,\infty)\times\R$. We define $$\psi(t,z) \vcentcolon= \Psi(t,z,\eta(t,z)).$$ Then we claim that (\ref{PDEs-cyl'}) can be reformulated in terms of $(\eta,\psi)$ as:  
\begin{subequations}\label{PDEs-cyl}
    \begin{gather}
		\d_t \eta - G[\eta](\psi) = 0,\label{lag-cyl}\\
		\d_t\psi + \dfrac{\big|\d_z\psi|^2}{2}-\dfrac{| \d_z \eta \d_z \psi+G[\eta](\psi)\big|^2}{2 (1+|\d_z \eta|^2)} + P\vert_{r=\eta(t,z)}=0\label{bern-cyl},  
    \end{gather}
\end{subequations}
for $(t,z)\in[0,\infty)\times\R$.
\begin{proof}
First, (\ref{lag-cyl}) is an immediate consequence of (\ref{lag-cyl'}) and (\ref{DN-cyl}). Next, we show (\ref{bern-cyl}). It can be verified with chain rule that
\begin{subequations}\label{cPhi1Temp}
	\begin{align}
		&\d_t \psi (t,z) = (\d_t \Psi + \d_t \eta \d_r \Psi )\vert_{r=\eta(t,z)},\label{cPhi1Temp1}\\
		&\d_z \psi(t,z) = (\d_z \Psi + \d_z \eta \d_r \Psi)\vert_{r=\eta(t,z)}.\label{cPhi1Temp2}
	\end{align}
\end{subequations}
Combining (\ref{cPhi1Temp2}) and (\ref{DN-cyl}), we can express $(\d_r \Psi, \d_z \Psi)\vert_{r=\eta(t,z)}$ solely in $(\eta,\psi)$ as:
\begin{subequations}\label{cPhi2Temp}
\begin{align}
&\d_r \Psi (t,z,\eta(t,z)) = \dfrac{\d_z \eta \d_z \psi + G[\eta](\psi)}{1+|\d_z \eta|^2},\label{cPhi2Temp1}\\
&\d_z \Psi(t,z,\eta(t,z)) = \dfrac{\d_z\psi - \d_z \eta G[\eta](\psi)}{1+|\d_z \eta|^2} .\label{cPhi2Temp2}
\end{align}
\end{subequations}
Substituting (\ref{cPhi1Temp1}) and (\ref{lag-cyl'}) into (\ref{bern-cyl'}), we have
\begin{equation}\label{dtphic}
	\d_t \psi + \Big( \d_z \eta \d_z \Psi \d_r \Psi + \dfrac{1}{2}|\d_z\Psi|^2 - \dfrac{1}{2}|\d_r \Psi|^2 + P \Big) \Big\vert_{r=\eta(t,z)} = 0. 
\end{equation}
Now completing the square and using (\ref{cPhi1Temp2}), we have
\begin{align*}
	\d_z \eta \d_z \Psi \d_r \Psi =& \dfrac{1}{2} \big\{ (\d_z \Psi + \d_z \eta \d_r \Psi)^2 -|\d_z \eta \d_r \Psi|^2 - |\d_z \Psi|^2 \big\} \big\vert_{r=\eta(t,z)}\\ 
	=& \dfrac{|\d_z \psi|^2}{2} - \dfrac{1}{2} \big\{ |\d_z \eta \d_r \Psi|^2 +  |\d_z \Psi|^2 \big\} \Big\vert_{r=\eta(t,z)}
\end{align*}
Substituting this into (\ref{dtphic}), we obtain
\begin{equation*}
\d_t \psi + \dfrac{|\d_z\psi|^2}{2} + \Big\{P - \dfrac{1+|\d_z \eta|^2}{2}|\d_r\Psi|^2 \Big\} \Big\vert_{r=\eta(t,z)} = 0. 
\end{equation*}
Substituting (\ref{cPhi2Temp1}) into the above equation, we conclude that the system (\ref{PDEs-cyl'}) can be rewritten solely in terms of $(\eta,\psi)(t,z)$ as (\ref{PDEs-cyl}).
\end{proof}

\subsection{Pressure and surface tension}\label{ssec:PH}
If the surface tension is taken into account, then according to the Young-Laplace law, the jump in pressure across the free boundary, $[P]= P-P_0$ is proportional to the mean curvature, denoted as $\mH(\eta)$. Thus one has
\begin{equation*}
P\vert_{r=\eta(t,z)} = P_0-\kappa \mH(\eta). 
\end{equation*} 
where $\kappa>0$ is the constant coefficient of surface tension, and $P_0\in\R$ is the constant reference pressure, chosen so that the following far-field condition holds:
\begin{equation}\label{P-ffc}
P\vert_{r=\eta(t,z)}=P_0-\kappa \mH\big(\eta(t,z)\big)\to 0 \quad \text{ as } \quad |z|\to \infty.
\end{equation}
Next, we obtain the exact expression for $\mH$ by computing the first and second fundamental forms of the surface $\d\Omega(t)=\{\x\in\R^3\,\vert\, r-\eta(t,z)=0 \}$. Let $\X(t)\in\R^3$ be the position vector of $\d\Omega(t)$. Then it can be parametrised with $(z,\theta)\in \R\times[0,2\pi)$ as:
\begin{equation*}
    \X(t,z,\theta) = \hr \eta(t,z) + \hz z = \hx \eta(t,z) \cos\theta + \hy \eta(t,z) \sin\theta + \hz z.
\end{equation*}
By the relations of cylindrical coordinates (\ref{cross-cyl})--(\ref{dun-cyl}), it follows that
\begin{equation}\label{normal}
    \normal \vcentcolon= \d_\theta \X \times \d_z \X = (\eta \d_{\theta} \hr + z \d_{\theta} \hz ) \times (\hr \d_z \eta + \hz)  = \eta(\hr - \hz\d_z \eta),
\end{equation}
and the unit normal vector $\un$ takes the form:
\begin{equation}\label{un}
    \un = \dfrac{\normal}{|\normal|} = \dfrac{\hr- \hz\d_z \eta }{\sqrt{1+|\d_z \eta|^2}}.
\end{equation}
Using the coordinate relations (\ref{dun-cyl}) once again, we have
\begin{align*}
\d_\theta \X = \hth \eta(z), \quad \d_z \X = \hr \d_z \eta + \hz, \quad \d_\theta^2 \X = - \hr \eta, \quad \d_z\d_\theta \X = \hth \d_z \eta, \quad \d_z^2 \X = \hr \d_z^2 \eta.  
\end{align*}
With the above expressions, and unit normal given by (\ref{un}), it follows that 
\begin{subequations}
\begin{gather*}
E\vcentcolon=%\mathrm{I}(\theta,\theta)=
\d_\theta \X \cdot \d_\theta \X = \eta^2, \quad F\vcentcolon=%\mathrm{I}(\theta,z)=
\d_\theta \X \cdot \d_z \X = 0, \quad G \vcentcolon= %\mathrm{I}(z,z)=
\d_z\X \cdot \d_z \X = 1+ |\d_z\eta|^2,\\
e\vcentcolon=% \mathrm{II}(\theta,\theta)=
\un\cdot \d_\theta^2\X = - \dfrac{\eta}{\sqrt{1+|\d_z\eta|^2}}, \quad 
f\vcentcolon=%\mathrm{II}(\theta,z)=
\un\cdot \d_\theta \d_z \X=0, \quad
g\vcentcolon=%\mathrm{II}(z,z)=
\un\cdot \d_z^2 \X = \dfrac{\d_z^2 \eta}{\sqrt{1+|\d_z\eta|^2}}.  
\end{gather*}
\end{subequations} 
The mean curvature $\mH =\mH(\eta)$ is then given by
\begin{equation}\label{H}
\mH(\eta) \vcentcolon= \dfrac{eG-2fF+gE}{2(EG-F^2)} %= \dfrac{\d_z^2\eta}{2(1+|\d_z\eta|^2)^{3/2}} - \dfrac{1}{2\eta \sqrt{1+|\d_z\eta|^2}}
= \d_z\Big(\dfrac{\d_z\eta}{2\sqrt{1+|\d_z\eta|^2}}\Big) - \dfrac{1}{2\eta \sqrt{1+|\d_z\eta|^2}}
\end{equation}
%---------------------------------------
% START OF COMMENT: ANOTHER DERIVATION
%---------------------------------------
\iffalse
Another way of obtaining $\mH$ is by taking the divergence of the unit normal vector $\un$. Using the divergence for cylindrical coordiante (\ref{cyl-div}), we get:
\begin{align*}
\mH(\eta)&=-\dfrac{1}{2}\div \un\Big\vert_{\x\in\d\Omega} = -\dfrac{1}{2} \Big\{ \d_z(\hz\cdot\un) + \dfrac{1}{r} \d_r\big(r (\hr\cdot \un)\big) \Big\}\Big\vert_{r=\eta(z)} \\
&=-\dfrac{1}{2} \d_z \Big( \dfrac{-\d_z \eta}{\sqrt{1+|\d_z\eta|^2}} \Big) - \dfrac{1}{2r} \d_r \Big( \dfrac{r}{\sqrt{1+|\d_z\eta|^2}} \Big)\Big\vert_{r=\eta(z)}\\
&=\d_z \Big( \dfrac{\d_z \eta}{2\sqrt{1+|\d_z\eta|^2}} \Big) - \dfrac{1}{2\eta \sqrt{1+|\d_z\eta|^2}}.
\end{align*}
\fi
%-----------------------------
%       END OF COMMENT
%-----------------------------
\subsection{Equilibrium Solution}\label{ssec:Equi}
Consider the equilibrium state $(\eta^E,\psi^E)\equiv(R,0)$ for some fixed $R>0$. In this case, we have $P\vert_{r=\eta^{E}}= P_0 -\kappa \mH(\eta^E) =P_0+\frac{\kappa}{2R}$. Setting $P_0\vcentcolon=-\frac{\kappa}{2R}$, we see that $(\eta^E,\psi^E)$ is a solution to the system (\ref{PDEs-cyl}). $(\eta,\psi)(t,z)$ is considered to be the perturbation with respect to $(\eta^E,\psi^E)$. Thus the following far-field condition is imposed:
\begin{equation}\label{ffc}
 (\eta,\psi)(t,z) \to (R,0) \quad \text{ as } \ |z|\to\infty \ \text{ for all } \ t\ge 0. 
\end{equation}
Then condition (\ref{ffc}) and (\ref{H}) imply that $\mH(\eta)\to -\frac{1}{2R}$ as $|z|\to\infty$. Therefore, to ensure condition (\ref{P-ffc}), we set $P_0\equiv -\frac{\kappa}{2R}$ as before, and the pressure at the free surface is given by:
\begin{equation}\label{P}
	P\vert_{r=\eta} = -\dfrac{\kappa}{2R} - \kappa \mH(\eta).
\end{equation}
With these considerations, the Zakharov's system (\ref{PDEs-cyl}) becomes
\begin{subequations}\label{PDEs-cyl-H}
	\begin{gather}
		\d_t \eta - G[\eta](\psi) = 0,\label{lag-cyl-H}\\
		\d_t\psi + \dfrac{\big|\d_z\psi|^2}{2}-\dfrac{| \d_z \eta \d_z \psi+G[\eta](\psi)\big|^2}{2 (1+|\d_z \eta|^2)}-\kappa \mH(\eta) - \frac{\kappa}{2R} = 0. \label{bern-cyl-H}
	\end{gather}
\end{subequations}
Without loss of generality, we will set $\kappa=1$ for the entirety of this paper. 

\subsection{Equations with gravity}\label{append:gravity}
%\todo{move to concluding remarks section}
Suppose the flow is under the influence of gravity on the $-\hat{z}$ direction with acceleration constant $g>0$, then the Zakharov's system is given by
\begin{subequations}\label{PDEs-cyl-g}
	\begin{gather}
		\d_t \eta - G[\eta](\psi) = 0,\label{lag-cyl-g}\\
		\d_t\psi + \dfrac{\big|\d_z\psi|^2}{2}-\dfrac{| \d_z \eta \d_z \psi+G[\eta](\psi)\big|^2}{2 (1+|\d_z \eta|^2)}- \mH(\eta) - \frac{1}{2R} - gz= 0. \label{bern-cyl-g}
	\end{gather}
\end{subequations}
The following proposition states that, in order to solve (\ref{PDEs-cyl-g}), it is sufficient to obtain solution for the system (\ref{PDEs-cyl-H}) without gravity. 
\begin{proposition}\label{prop:gravity}
	Suppose $(\eta,\psi)(t,z)$ is a classical solution to (\ref{PDEs-cyl-H}). Define:
	\begin{equation}\label{gtransform}
		\eta^g(t,z) \vcentcolon= \eta\big(t,z-\frac{1}{2}gt^2\big), \qquad \psi^g(t,z)\vcentcolon= \psi\big(t,z-\frac{1}{2}gt^2\big) + g t z - \frac{1}{6}g^2 t^3.
	\end{equation} 
	Then $(\eta^g,\psi^g)(t,z)$ solves (\ref{PDEs-cyl-g}).
\end{proposition}
\begin{proof}
	Denote $\zeta\equiv\zeta(t,z)\vcentcolon= z-\frac{1}{2}gt^2$. Then $\eta^g(t,z)=\eta\big(t,\zeta(t,z)\big)$. For fixed $t\ge 0$, suppose $\Psi(t,z,r)$ is a solution to the boundary value problem (\ref{Psi-BVP}). Then one can verify that $\Psi^g(t,z,r)\vcentcolon= \Psi\big(t,\zeta(t,z),r\big)+gtz -\frac{1}{6}g^2t^3$ solves the equation (\ref{PsiDE}) in the domain $\{z\in\R \ \text{and} \ r\in[0,\eta^g(t,z)]\}$. Also, one has
	\begin{align*}
		\Psi^g(t,z,r)\vert_{r=\eta^{g}(t,z)} =& \Psi\Big(t,\zeta(t,z),\eta\big(t,\zeta(t,z)\big)\Big) + gtz - \tfrac{1}{6}g^2t^3\\
		=& \psi\big(t,\zeta(t,z)\big) + gtz - \tfrac{1}{6}g^2t^3 = \psi^g(t,z).
	\end{align*}
	Therefore by the definition of DN operator (\ref{DN-cyl}), it follows that
	\begin{align}\label{DNpsig}
		&G[\eta^g](\psi^g)(t,z) = \big\{ \d_r \Psi^g(t,z,r) - \d_z\eta^g(t,z) \d_z \Psi^g(t,z,r) \big\}\big\vert_{r=\eta^g(t,z)}\\
		=& \d_r \Psi\big( t, \zeta, \eta(t,\zeta) \big) - \d_z\eta(t,\zeta) \d_z \Psi\big( t, \zeta, \eta(t,\zeta) \big) - gt \d_z\eta(t,\zeta)\nonumber\\
		=& G[\eta](\psi)\big(t,\zeta(t,z)\big) - gt \d_z \eta\big(t,\zeta(t,z)\big).\nonumber
	\end{align}
	In addition by chain rule, one has
	\begin{align}
		\d_t \psi^{g}(t,z) =& \Big\{\d_t \psi -gt \d_z\psi\Big\}\big(t,\zeta(t,z)\big) + gz - \frac{1}{2} g^2 t^2,\label{Dpsig1}\\
		\d_z \psi^{g}(t,z) =& \d_z \psi\big(t,\zeta(t,z)\big) + gt. \label{Dpsig2}
	\end{align}
	Using (\ref{DNpsig}), (\ref{Dpsig1}), and (\ref{Dpsig2}), we have that
	\begin{gather*}
		\Big\{\d_t \psi^g+\dfrac{1}{2}|\d_z \psi^g|^2\Big\}(t,z) = \Big\{\d_t\psi +\dfrac{1}{2}|\d_z\psi|^2\Big\}\big(t,\zeta(t,z)\big) + gz\\
		\Big\{\d_z\eta^g\d_z\psi^g + G[\eta^g](\psi^g)\Big\}(t,z) = \Big\{\d_z\eta\d_z\psi + G[\eta](\psi)\Big\}\big(t,\zeta(t,z)\big)
	\end{gather*}
	Since $(\eta,\psi)$ solves (\ref{PDEs-cyl-H}), it follows that
	\begin{align*}
		&\Big\{\d_t \psi^g + \dfrac{1}{2}|\d_z\psi^g|^2 - \dfrac{1}{2} \dfrac{|\d_z\eta^g\d_z\psi^g+G[\eta^g](\psi^g)|^2}{1+|\d_z\eta^g|^2} \Big\}(t,z) - gz \\
		=& \Big\{\d_t \psi + \dfrac{1}{2}|\d_z\psi|^2 - \dfrac{1}{2} \dfrac{|\d_z\eta\d_z\psi+G[\eta](\psi)|^2}{1+|\d_z\eta|^2} \Big\}\big(t,\zeta(t,z)\big)\\
		=& \mH\Big(\eta\big(t,\zeta(t,z)\big)\Big) + \frac{1}{2R} =  \mH\big(\eta^g(t,z)\big) + \frac{1}{2R}. 
	\end{align*}
	This shows that $(\eta^g,\psi^g)(t,z)$ solves the dynamic equation of (\ref{PDEs-cyl-g}). In addition, by (\ref{DNpsig}) and chain rule, one has
	\begin{equation*}
		\big\{\d_t\eta^g - G[\eta^g](\psi^g) \big\}(t,z) = \big\{\d_t\eta - gt \d_z\eta - G[\eta](\psi) + gt \d_z\eta \big\}\big(t,\zeta(t,z)\big) = 0,
	\end{equation*}
	which proves that $(\eta^g,\psi^g)(t,z)$ also solves the kinematic equation of (\ref{PDEs-cyl-g}).
\end{proof}

\subsection{Hamiltonian energy}\label{append:hamiltonian}
%\todo{move it to concluding remarks}
In this section, we aim to derive from (\ref{DN-cyl}) and (\ref{PDEs-cyl}), the conservation of energy. We define the energy $E$ by
\begin{align}\label{energy}
E(t)\vcentcolon=&\int_{\Omega(t)}\!\!\dfrac{|\vv|^2}{2}\,\dif \x - \kappa\int_{\d\Omega(t)}\!\! \Big\{\dfrac{1}{2R} + \mH(\eta) \Big\} \vv\cdot \un\, \dif S\\ 
=&\pi \int_{\R}  \psi \eta G[\eta](\psi) \, \dif z +  \pi \kappa \int_{\R} \Big\{ \eta\big(\sqrt{1+|\d_z\eta|^2}-1\big) - \dfrac{|\eta-R|^2}{2R} \Big\} \, \dif z.\nonumber
\end{align}
\begin{lemma}\label{lemma:Henergy}
	The energy $E(t)$ is conserved i.e. $\frac{\dif E}{\dif t} =0$.
\end{lemma}
\begin{proof}
If $(\eta,\psi)$ is a smooth solution to (\ref{PDEs-cyl}), then we can rewrite them as the unknown $(\eta,\phi)$ in the Cartesian coordinates $(x,y,z)\in\R^3$ by solving the harmonic equation:
\begin{subequations}
\begin{align*}
&-\Delta \phi = -(\d_x^2 \phi + \d_y^2 \phi + \d_z^2 \phi) = 0 && \text{for } \ F(t,\x)=\sqrt{x^2+y^2} - \eta(t,z) \le 0, \\
& \phi\big(t,\eta(z)\cos\theta,\eta(z)\sin\theta,z\big) = \psi(t,z)	&& \text{for } \ (t,z,\theta)\in[0,\infty)\times\R\times[0,2\pi).
\end{align*}
\end{subequations} 
Moreover, we write $\vv=\nabla \phi$ as the velocity. Then $(\eta,\vv)$ satisfies the Euler's equations
\begin{subequations}\label{euler}
\begin{align}
&\d_t \vv + (\vv\cdot\nabla) \vv + \nabla P = 0, \quad \div \vv = 0 && \text{for } \ \x\in \Omega(t) \ \text{ and } \ t\ge0,\label{eulerE} \\
&P\vert_{\d\Omega(t)} = -\tfrac{1}{2R}- \mH(\eta) && \text{for } \ \x\in\d\Omega(t) \ \text{ and } \ t\ge0,\label{eulerP}\\
& -\d_t \eta + (\vv \cdot \nabla) (r-\eta)\vert_{r=\eta(t,z)} = 0 && \text{for } \ z\in\R \ \text{ and } \ t\ge 0,\label{eulerL}
\end{align}
\end{subequations}
where $\Omega(t)\vcentcolon= \{\x\in \R^3 \,\vert\, F(t,\x) \le 0 \}$. The condition (\ref{eulerP}) holds modulo addition by a function of time, which is illustrated by the following consideration: taking divergence on (\ref{eulerE}), pressure $P$ necessarily solves the following Neumann boundary problem:
\begin{subequations}\label{laplaP}
\begin{align}
&-\Delta P = \textrm{tr}(\nabla \vv \cdot \nabla \vv) = |\nabla^2 \phi|^2 && \text{for } \ x\in \Omega(t),\\
& \d_{n} P\vert_{ \d\Omega(t)} = \un \cdot ( \d_t \vv + (\vv\cdot \nabla ) \vv )\vert_{ \d\Omega(t)} && \text{for } \ t\ge 0.
\end{align}
\end{subequations}
Observe (\ref{eulerE}), (\ref{eulerL}), and (\ref{laplaP}) still hold if one redefine $P\mapsto P + c(t)$, for any function of time $c(t)$. By Reynolds transport theorem, (\ref{eulerE}), and Divergence theorem, one has
\begin{align}\label{ddt}
&\dfrac{\dif}{\dif t} \int_{\Omega(t)}\!\! \dfrac{|\vv|^2}{2}\, \dif \x + \int_{\d\Omega(t)}\!\!\! P \vv\cdot \un\, \dif S \\
=& \int_{\Omega(t)}\!\! \vv \cdot \d_t \vv \, \dif \x + \int_{\d\Omega(t)}\!\! \dfrac{|\vv|^2}{2} \vv \cdot \un \, \dif S + \int_{\d\Omega(t)}\!\!\! P \vv\cdot \un\, \dif S \nonumber\\
=& - \int_{\Omega(t)} \vv \cdot \{ (\vv\cdot\nabla) \vv + \nabla P \}\, \dif\x + \int_{\d\Omega(t)}\!\! \dfrac{|\vv|^2}{2} \vv\cdot \un \, \dif S + \int_{\d\Omega(t)}\!\!\! P \vv\cdot \un\, \dif S %\nonumber\\ =& - \int_{\Omega(t)} \div \Big( \big\{\dfrac{|\vv|^2}{2} + P\big\} \vv \Big)\, \dif\x + \int_{\d\Omega(t)}\!\! \dfrac{|\vv|^2}{2} \vv\cdot \un \, \dif S +\int_{\d\Omega(t)}\!\!\! P \vv\cdot \un\, \dif S 
=0,\nonumber
\end{align}
where $\un$ is the unit normal vector of $\d\Omega(t)$. By (\ref{normal}), the surface element of $\d\Omega(t)$ in the coordinate $(z,r,\theta)$ is given by:
\begin{equation}\label{se}
\dif S = \big| \d_\theta \X \times \d_z \X \big| \dif \theta \dif z = \eta \sqrt{1+|\d_z\eta|^2} \dif \theta \dif z.
\end{equation}
According to the derivative rule for cylindrical coordinate, it follows that 
\begin{equation}\label{dreta}
	\nabla (r-\eta(t,z)) = \hr - \hz \d_z \eta = \un \sqrt{1+|\d_z\eta|^2}, \quad \text{where } \  \un = \dfrac{\hr-\hz \d_z\eta}{\sqrt{1+|\d_z\eta|^2}}.
\end{equation}
By the expression of $P\vert_{r=\eta}$ given in (\ref{P}), we get 
\begin{align}\label{Pint}
	\int_{\d\Omega(t)} P \vv\cdot \un\, \dif S = -\dfrac{\kappa}{2R} \int_{\d\Omega(t)} \vv\cdot \un \dif S - \kappa \int_{\d\Omega(t)} \mH(\eta) \vv\cdot \un\, \dif S  =\vcentcolon \textrm{(I)} + \textrm{(II)}.
\end{align}
Using (\ref{dreta}), (\ref{se}), (\ref{eulerL}), we obtain that
\begin{align*}
	\textrm{(I)} =& -\dfrac{\kappa}{2R} \int_{\d\Omega} \vv \cdot \un\, \dif S = -\dfrac{\kappa}{2R} \int_{0}^{2\pi}\!\!\!\!\int_{\R} \dfrac{\vv \cdot \nabla(r-\eta)}{\sqrt{1+|\d_z\eta|^2}} \eta \sqrt{1+|\d_z\eta|^2} \, \dif z \dif \theta\\
	=& -\dfrac{\pi \kappa}{R} \int_{\R} \eta \d_t\eta \, \dif z = - \pi \kappa \dfrac{\dif }{\dif t} \int_{\R} \dfrac{\eta^2-R^2}{2R} \, \dif z.  
\end{align*}
Similarly, we also get
\begin{align*}
	\textrm{(II)}=&-\int_{\d\Omega(t)}\!\!\! \kappa \mH(\eta) \vv\cdot \un\, \dif S = -\int_{0}^{2\pi}\!\!\!\! \int_{\R} \kappa \mH(\eta) \dfrac{\vv \cdot \nabla (r-\eta) }{\sqrt{1+|\d_z\eta|^2}} \eta \sqrt{1+|\d_z\eta|^2}\,\dif z \dif \theta \\
	=& - 2\pi \kappa\! \int_{\R} \mH(\eta) \eta\d_t \eta \, \dif z = \pi \kappa\! \int_{\R}\! \Big\{ -\d_z\Big(\dfrac{\d_z\eta}{\sqrt{1+|\d_z\eta|^2}}\Big) + \dfrac{1}{\eta\sqrt{1+|\d_z\eta|^2}} \Big\} \eta\d_t \eta \, \dif z.
\end{align*}
Integrating by parts over $z\in\R$ and using the fact that $\d_z\eta \to 0$ as $z\to\infty$, one has
\begin{align*}
    \textrm{(II)} =& \pi\kappa \int_{\R}\Big(\dfrac{\d_z\eta\{\d_z\eta \d_t \eta + \eta \d_t \d_z \eta\}}{\sqrt{1+|\d_z\eta|^2}} + \dfrac{\d_t \eta}{\sqrt{1+|\d_z\eta|^2}} \Big) \, \dif z\nonumber\\
	=& \pi\kappa\int_{\R} \Big(\d_t \eta \sqrt{1+|\d_z\eta|^2} + \dfrac{\eta \d_t \d_z \eta}{\sqrt{1+|\d_z\eta|^2}} \Big)\, \dif z = \pi\kappa \dfrac{\dif }{\dif t}\int_{\R}\!\! \big\{\eta\sqrt{1+|\d_z\eta|^2} - R\big\}\, \dif z.
\end{align*}
Combining $\textrm{(I)}$ and $\textrm{(II)}$ in (\ref{Pint}), it follows that
 \begin{align}\label{Pvn}
	\int_{\d\Omega(t)} P \vv\cdot \un\, \dif S =& \pi\kappa \dfrac{\dif }{\dif t} \int_{\R} \Big\{ \eta \sqrt{1+|\d_z\eta|^2} - R - \dfrac{\eta^2 - R^2}{2R} \Big\}\, \dif z\\
	=& \pi\kappa \int_{\R} \Big\{ \eta \big( \sqrt{1+|\d_z\eta|^2} -1 \big) - \dfrac{|\eta-R|^2}{2R} \Big\}.\nonumber
\end{align}
In addition, using $-\Delta \phi=0$, (\ref{un}) and the definition of $G[\eta](\psi)$, (\ref{DN-cyl}), we have that
\begin{align}\label{v2}
&\int_{\Omega(t)} \!\! \dfrac{|\vv|^2}{2}\, \dif \x = \dfrac{1}{2} \int_{\Omega(t)}\!\! |\nabla\phi|^2\, \dif \x = \dfrac{1}{2}  \int_{\Omega(t)}\!\! \div( \phi \nabla \phi )\, \dif \x = \dfrac{1}{2} \int_{\d\Omega(t)} \phi \un\cdot\nabla \phi\, \dif S\\
&= \dfrac{1}{2} \int_{\R}\! \int_{0}^{2\pi} \!\!\!\!\!  \dfrac{\psi G[\eta](\psi)}{\sqrt{1+|\d_z\eta|^2}} \cdot \eta \sqrt{1+|\d_z\eta|^2} \dif \theta \dif z = \pi \int_{\R} \psi \eta G[\eta](\psi)\, \dif z.\nonumber
\end{align}
Putting (\ref{Pvn}) and (\ref{v2}) into (\ref{ddt}) proves the lemma.
\end{proof}

\section{Sobolev Inequalities}\label{append:sobo}

\begin{proposition}\label{prop:bochner}
Let $d\ge 1$, and let $I\subseteq \R$ be a connected interval. Suppose the function $u(z,y)\vcentcolon \R^d \times [0,1] \to \R$ is such that
\begin{equation*}
u \in \mathrm{X}_2^{s+\frac{1}{2}}(I;\R^d), \qquad \d_y f \in \mathrm{X}_2^{s-\frac{1}{2}}(I;\R^d), \quad \text{where}  \quad \mathrm{X}_p^{k}(I;\R^d) \vcentcolon= L^p_{y}\big(I;H_z^{k}(\R^d)\big). 
\end{equation*}
Then $u\in \mC_y^{0}\big(I;H_z^s(\R^d)\big)$ and $ u \in L^{\infty}_y\big(I;H_z^s(\R^d)\big)$, and there exists a constant $C>0$ independent of $u$ such that
\begin{equation*}
\sup\limits_{y\in I} \| u(\cdot, y) \|_{H^s(\R^d)} \le C \| u \|_{\mathrm{X}_2^{s-\frac{1}{2}}(I;\R^d)}^{\frac{1}{2}} \| \d_y u \|_{\mathrm{X}_2^{s-\frac{1}{2}}(I;\R^d)}^{\frac{1}{2}}.
\end{equation*}
\end{proposition}

\begin{proposition}[Sobolev Embedding for composite]\label{prop:Sobcomp}
Assume $u\in L^{\infty}\cap H^{s}(\R^d)$ is real valued for $s\ge0$. If $F\in\mC^{\infty}(\R)$ and $F(0)=0$, then $F(u)\in L^{\infty}\cap H^{s}(\R^d)$, and
\begin{equation*}
\|F(u)\|_{H^{s}(\R^d)} \le \sup\limits_{x\in\R} |F^{\prime}(x)| \| u \|_{H^{s}(\R^d)}.
\end{equation*}
\end{proposition}

\begin{proposition}[Sobolev Embedding for classical product]\label{prop:clprod}
Assume $u_1\in H^{\alpha}(\R^d)$ and $u_2\in H^{\beta}(\R^d)$ for $\alpha+\beta\ge 0$. If $ s\le \min\{\alpha,\beta\}$ and $s\le \alpha+\beta - \frac{d}{2}$, then $u_1u_2\in H^{s}(\R^d)$ and there exists a constant $C>0$ independent of $u_1$, $u_2$ such that
\begin{equation*}
\|u_1 u_2\|_{H^{s}(\R^d)} \le C \|u_1\|_{H^{\alpha}(\R^d)} \|u_2\|_{H^{\beta}(\R^d)}.
\end{equation*}
\end{proposition}

%-------------------------------------
% Appendix: Paradifferential Calculus
%-------------------------------------

\section{Paradifferential Calculus}\label{append:para}

\begin{definition}[H\"older Spaces]\label{def:holder}
For $d\ge 1$ and $k\in\mathbb{N}\cup\{0\}$, one denotes $W^{k,\infty}(\R^d)$ as the Sobolev space of $L^{\infty}(\R^d)$ functions whose derivatives of order $k$ belongs to $L^{\infty}(\R^d)$. For $k\in (0,\infty)\backslash\mathbb{N}$, $W^{k,\infty}(\R^d)$ denotes the space of bounded functions whose derivatives of order $\lfloor k \rfloor$ are H\"older continuous with exponent $k-\lfloor k \rfloor$, where $\lfloor \cdot \rfloor$ is the floor function.
\end{definition}

\begin{definition}[Symbols]\label{def:symbols}
Given $\theta\ge 0$ and $m\in\R$. $\Gamma_{\theta}^m(\R^d)$ denotes the space of locally bounded functions $a(x,\xi)\vcentcolon \R^d\times (\R^d\backslash \{ 0 \}) \to \mathbb{C}$, such that
\begin{itemize}
\item for each fixed $x\in\R^d$, $\xi\mapsto a(x,\xi)\in \mC^{\infty}(\R^d\backslash \{ 0 \})$\textnormal{;}
\item for each $\alpha\in\mathbb{N}^d$ and $\xi\in \R^d\backslash \{ 0 \}$, the function $x\mapsto \d_\xi^{\alpha} a(x,\xi)$ is in $W^{\theta,\infty}(\R^d)$ and there exists constant $C_{\alpha}>0$ such that
\begin{equation*}
\|\d_{\xi}^{\alpha} a(\cdot,\xi)\|_{W^{\theta,\infty}(\R^d)} \le C_{\alpha} (1+ |\xi| )^{m-|\alpha|}.
\end{equation*}
\item For such $a\in\Gamma_{\theta}^m(\R^d)$, we define the semi-norm $\mM_{\theta}^m(a)$ as:
\begin{equation*}
\mM_{\theta}^m(a)\vcentcolon=\sup\limits_{|\alpha|\le \theta+ 1 +\frac{d}{2} }\,\sup\limits_{|\xi|\ge \frac{1}{2}} (1+ |\xi| )^{|\alpha|-m} \|\d_{\xi}^{\alpha} a(\cdot,\xi)\|_{W^{\theta,\infty}(\R^d)} < \infty .
\end{equation*}
\end{itemize}
\end{definition}

\begin{definition}[Homogeneous and Principal Symbols]\label{def:osymbols} \
\begin{enumerate}[label=\textnormal{(\roman*)},ref=\textnormal{(\roman*)}]
\item\label{item:osymbols1} $\mathring{\Gamma}_{\theta}^m(\R^d)$ denotes the subspace of $\Gamma_{\theta}^m(\R^d)$ which consists of symbols $a(x,\xi)$ that are homogeneous of degree $m$ with respect to $\xi$. 
\item\label{item:osymbols2} If the symbol $a(x,\xi)\in \Gamma_{\theta}^{m}(\R^d)$ admits the series expansion,
\begin{equation*}
a= \sum_{0\le j \le \theta} a^{(m-j)} \qquad (j\in\mathbb{N}), \qquad \text{ where } \ a^{(m-j)}\in \mathring{\Gamma}_{\theta-j}^{m-j}(\R^d), 
\end{equation*}
then $a^{(m)}$ is said to be the principal symbol of $a$.
\end{enumerate}
\end{definition}

\begin{proposition}
	If $g(\xi)$ is homogeneous of order $m$ then $f(\xi)\vcentcolon=|\xi|^{|\alpha|-m} |\d_{\xi}^{\alpha} g(\xi)|$
	is homogeneous of order $0$ for all $\alpha\in \mathbb N^d$. In particular 
	\[
	|\xi|^{|\alpha|-m}|\d_\xi^\alpha g(\xi)|=f(\xi)=f\big(\frac\xi{|\xi|}\big)=\big|\d_\xi^\alpha g\big(\frac\xi{|\xi|}\big)\big|, \quad \text{for all } \ \xi\in \R^d\backslash\{0\}.
	\]
\end{proposition}
\begin{proof}
	Follows from Euler's equation for the homogeneous functions.
\end{proof}

\begin{corollary}\label{corol:homogenous}
If $g(\xi)$ is homogeneous of degree $m$ then 
\[
\sup_{|\xi|\ge 1/2}|\xi|^{|\alpha|-m}|\d_\xi^\alpha g(\xi)|
=
\sup_{|\xi|=1}|\d_\xi^\alpha g(\xi)|, \quad \forall\alpha\in\mathbb N^d.
\]
In particular, if $a\in \mathring{\Gamma}^m_{\theta}(\R^d)$, then 
\begin{equation*}
	\mM_{\theta}^r(a) \ \sim \sup\limits_{|\alpha|\le \theta+ 1 +\frac{d}{2} }\,\sup\limits_{|\xi|=1}\|\d_{\xi}^{\alpha} a(\cdot,\xi)\|_{W^{\theta,\infty}(\R^d)}
\end{equation*}
\end{corollary}

\begin{definition}[Paradifferential Operator]\label{def:parad}
Let $a\in\Gamma_{\theta}^m(\R^d)$ be a symbol. Then the para-differential operator $T_a$ associated with $a$, acting on $u$ is defined by
\begin{equation*}
\wh{T_a u}(\xi) \vcentcolon= \dfrac{1}{(2\pi)^{d/2}} \int_{\R^d} \chi(\xi-\zeta,\zeta) \wh{a}(\xi-\zeta,\zeta) \phi(\zeta) \wh{u}(\zeta) \dif \zeta,
\end{equation*}
where $\wh{a}(\theta,\zeta)$ denotes the Fourier transform of $a$ with respect to the first variable:
\begin{equation*}
\wh{a}(\theta,\zeta) = \dfrac{1}{(2\pi)^{d/2}} \int_{\R^d} a(x,\zeta) e^{-ix\cdot \theta}\, \dif x\qquad \text{for } \ \theta\in\R^d, 
\end{equation*}
and $\chi$, $\phi$ are two fixed $\mC^{\infty}$ functions such that
\begin{itemize}
\item $\phi(\zeta)=0$ for $|\zeta|\le \frac{1}{2}$, and $\phi(\zeta)=1$ for $|\zeta|\ge 1$;
\item $\chi(\theta,\zeta)$ is homogeneous of degree $0$ and there exists $0<\ep_1<\ep_2<1$ small such that $\chi(\theta,\zeta)=1$ if $|\theta|\le \ep_1 |\zeta|$, and $\chi(\theta,\zeta)=0$ if $|\theta|\ge \ep_2 |\zeta|$. In addition, for all $(\alpha,\beta)\in \mathbb{N}^d\times \mathbb{N}^d$, there exists constant $C_{\alpha,\beta}>0$ such that
\begin{equation}\label{chiProp}
\big|\d_\theta^{\alpha} \d_\zeta^{\beta} \chi(\theta,\zeta)\big| \le C_{\alpha,\beta} (1+|\zeta|)^{-|\alpha|-|\beta|} \quad \text{for } \ \theta\in \R^d \ \text{ and } \ |\zeta|\ge \tfrac{1}{2}. 
\end{equation}
\end{itemize}
\end{definition}

\begin{remark}
Definition \ref{def:parad} can be also understood as follows. For a given $a\in\Gamma_{\theta}^m(\R^d)$. We set the symbol $a_{\chi}(x,\xi)$ as
\begin{equation*}
 a_{\chi}(x,\xi) \vcentcolon= \dfrac{1}{(2\pi)^{d/2}}\int_{\R^d} \chi(\zeta,\xi) \wh{a}(\zeta,\xi) e^{i x \cdot \zeta}\,\dif \zeta. \quad \text{i.e. } \  \wh{a}_{\chi}(\zeta,\xi) = \chi(\zeta,\xi) \wh{a}(\zeta,\xi).
\end{equation*}
Then we set the operator $p_{a}(x,D)$ as $p_{a}(x,\xi)\vcentcolon= \phi(\xi) a_{\chi}(x,\xi)$, in other words
\begin{equation*}
p_a(x,D) u = \dfrac{1}{(2\pi)^{d/2}} \int_{\R^d} a_{\chi}(x,\zeta) \phi(\zeta) \wh{u}(\zeta) e^{ix\cdot \zeta}\, \dif \zeta.
\end{equation*} 
We claim that $p_a(x,D)=T_{a}$. To see this we take the Fourier transform to get:
\begin{align*}
\wh{p(x,D)u}(\xi) =& \dfrac{1}{(2\pi)^{d}} \int_{\R^d}\!\!\int_{\R^d}\!\! a_{\chi}(x,\zeta)\phi(\zeta) \wh{u}(\zeta) e^{ix\cdot (\zeta-\xi)} \dif \zeta \dif x\\
=& \dfrac{1}{(2\pi)^{d/2}} \int_{\R^d}\!\!\phi(\zeta) \wh{u}(\zeta) \Big( \dfrac{1}{(2\pi)^{d/2}} \int_{\R^d}\!\! a_{\chi}(x,\zeta) e^{-ix\cdot(\xi-\zeta)}\, \dif x\Big) \, \dif \zeta\\
%=&\dfrac{1}{(2\pi)^{d/2}} \int_{\R^d}\!\!\phi(\zeta) \wh{u}(\zeta) \Big( \dfrac{1}{(2\pi)^{d}} \int_{\R^d}\!\!\int_{\R^d} \chi(\sigma,\zeta) \wh{a}(\sigma,\zeta) e^{ix\cdot(\zeta-\xi+\sigma)}\, \dif\sigma \dif x\Big) \, \dif \zeta\\ =& \dfrac{1}{(2\pi)^{d/2}} \int_{\R^d}\!\!\phi(\zeta) \wh{u}(\zeta)  \int_{\R^d}\!\! \chi(\sigma,\zeta) \wh{a}(\sigma,\zeta) \Big( \dfrac{1}{(2\pi)^{d}}  \int_{\R^d} e^{ix\cdot(\zeta-\xi+\sigma)}\, \dif x \Big) \dif\sigma \dif \zeta\\ =& \dfrac{1}{(2\pi)^{d/2}} \int_{\R^d}\!\!\phi(\zeta) \wh{u}(\zeta)  \int_{\R^d}\!\! \chi(\sigma,\zeta) \wh{a}(\sigma,\zeta) \delta(\zeta-\xi+\sigma) \dif\sigma \dif \zeta\\ =& \dfrac{1}{(2\pi)^{d/2}} \int_{\R^d}\!\!\phi(\zeta) \wh{u}(\zeta) \chi(\xi-\zeta,\zeta) \wh{a}(\xi-\zeta,\zeta) \, \dif \zeta\\
=&\dfrac{1}{(2\pi)^{d/2}} \int_{\R^d}\!\!\phi(\zeta) \wh{u}(\zeta) \wh{a}_{\chi}(\xi-\zeta,\zeta) \dif \zeta\\
=& \dfrac{1}{(2\pi)^{d/2}} \int_{\R^d}\!\!\phi(\zeta) \wh{u}(\zeta) \chi(\xi-\zeta,\zeta) \wh{a}(\xi-\zeta,\zeta) \dif \zeta = \wh{T_a u}(\xi),
\end{align*}
where in the third line we used the definition of $a_{\chi}$.
\end{remark}
\begin{remark}\label{rem:paraC}
Let $c\in \R\backslash\{0\}$ be a constant. Then $\wh{c}(\xi)= (2\pi)^{d/2} c \delta(\xi)$ is the Dirac-Delta distribution centred at $\xi=0$. This leads to the following assertions:
\begin{enumerate}[label=\textnormal{(\roman*)},ref=\textnormal{(\roman*)}]
\item\label{item:paraC1} Since $\phi(0)=0$, by Definition \ref{def:parad}, it follows that for $a \in \Gamma_{\theta}^{m}(\R^d) $,
\begin{equation*}
	\wh{T_{a} c}(\xi) = \dfrac{1}{(2\pi)^{d/2}} \int_{\R^d}\!\! \chi(\xi-\zeta,\zeta) \wh{a}(\xi-\zeta,\zeta) \phi(\zeta) c (2\pi)^{d/2}\delta(\zeta)\,\dif \zeta = 0.
\end{equation*}
This implies that $T_a c = 0$ for all $a\in\Gamma_{\theta}^{m}$ and $c\in\R\backslash\{0\}$.
\item\label{item:paraC2} Since $\chi(0,\zeta)=1$, it follows that for $f\in H^{k}(\R)$,
\begin{equation*}
	\wh{T_{c}f}(\xi) = \dfrac{1}{(2\pi)^{d/2}} \int_{\R^d}\!\! \chi(\xi-\zeta,\zeta) c(2\pi)^{d/2} \delta(\xi-\zeta) \phi(\zeta) \wh{f}(\zeta) \, \dif \zeta = c \phi(\xi) \wh{f}(\xi).
\end{equation*}
This implies that $T_{c}f(x) = c \phi(D) f (x)$. From Definition \ref{def:parad}, $1-\phi(\zeta)=0$ if $|\zeta|\ge 1$ and $0 \le 1-\phi(\zeta) \le 1$ for all $\zeta\in\R^{d}$. Thus $(c-T_{c})f = c(1-\phi(D)) f \in H^{\infty}(\R^d)$. In particular if $f\in H^{k}(\R^d)$ for some $k\in\R$, then for all $l\ge k$
\begin{align*}
	\|(c-T_{c})f\|_{H^{l}} =& c \big\| \absm{\cdot}^{l} \big(1-\phi(\cdot)\big) \wh{f}(\cdot) \big\|_{L^2} = c \big\|\absm{\cdot}^{l-k}\big(1-\phi(\cdot)\big) \absm{\cdot}^k \wh{f}(\cdot)\big\|_{L^2}\\
	\le & c \sqrt{2^{l-k}} \|\absm{\cdot}^k \wh{f}(\cdot)\|_{L^2} = c \sqrt{2^{l-k}} \|f\|_{H^k}. 
\end{align*}    
\end{enumerate}
\end{remark}
\begin{proposition}\label{prop:paraLeib}
Let $a\in \Gamma_{\theta}^m(\R^d)$ for $\theta\ge 0$ and $m\ge 0$. Then for each $j=1,\dotsc,d$, we have $\d_{x^{j}}(T_a u) = T_{b^{j}}u+T_a(\d_{x^j}u)$, where $b^{j}(x,\xi)=\d_{x^{j}}a(x,\xi)$.
\end{proposition}
\begin{proof}
Using the Definition \ref{def:parad}, we have
\begin{align*}
\wh{T_a(\d_{x^j}u)}(\xi) =& \dfrac{1}{(2\pi)^{d/2}} \int_{\R^d}\!\! \chi(\xi-\zeta,\zeta) \wh{a}(\xi-\zeta,\zeta) \phi(\zeta) (i\zeta^{j}) \wh{u}(\zeta)\, \dif \zeta\\
=& -\dfrac{1}{(2\pi)^{d/2}} \int_{\R^d}\!\! \chi(\xi-\zeta,\zeta) \cdot i(\xi^{j}-\zeta^{j}) \wh{a}(\xi-\zeta,\zeta) \phi(\zeta)  \wh{u}(\zeta)\, \dif \zeta\\
 &+ \dfrac{i\xi^{j}}{(2\pi)^{d/2}} \int_{\R^d}\!\! \chi(\xi-\zeta,\zeta) \wh{a}(\xi-\zeta,\zeta) \phi(\zeta)  \wh{u}(\zeta)\, \dif \zeta\\
=& -\dfrac{1}{(2\pi)^{d/2}} \int_{\R^d}\!\! \chi(\xi-\zeta,\zeta) (\wh{\d_{x^{j}}a})(\xi-\zeta,\zeta) \phi(\zeta)  \wh{u}(\zeta)\, \dif \zeta\\
&+ \dfrac{i\xi^{j}}{(2\pi)^{d/2}} \int_{\R^d}\!\! \chi(\xi-\zeta,\zeta) \wh{a}(\xi-\zeta,\zeta) \phi(\zeta)  \wh{u}(\zeta)\, \dif \zeta\\
=& - \wh{T_{b^j} u}(\xi) + \wh{\d_{x^j}(T_{a} u)}(\xi)
\end{align*}
This proves the proposition.
\end{proof}

\begin{definition}\label{def:order}
Let $m\in\R$. An operator $T$ is said to be of order $m$ if, for all $k\in\R$, $\|T\|_{H^{k}(\R^d)\to H^{k-m}(\R^d)}<\infty$.
\end{definition}

\begin{theorem}\label{thm:paraL2}
Let $m\in\R$. If $a\in \Gamma_0^m(\R^d)$, then $T_a$ is of order $m$. Moreover for all $k\in\R$ there exists a constant $C(m,k)>0$ such that
\begin{equation*}
\|T_a\|_{H^{k}(\R^d)\to H^{k-m}(\R^d)} \le C(m,k)\mM_0^m(a).
\end{equation*}
\end{theorem}
For homogeneous symbol $p(x,\xi)$ with low regularity in $x\in\R^d$, a similar estimate holds for $T_{p}$, but with reduced order. This is stated in the following proposition: 
\begin{proposition}[Low Regularity Symbols]\label{prop:Hsym}
Let $r\in\R$ and $k\le \frac{d}{2}$. Suppose the symbol $p( x,\xi) \vcentcolon\R^{d}\times (\R^{d}\backslash \{ 0 \}) \rightarrow \mathbb{C}$ satisfies the conditions:
\begin{itemize}
\item for $x\in \R^d$, $\xi\mapsto p(x,\xi)\in \mC^{\infty}\big(\R^d\backslash\{0\}\big)$, and $\xi\mapsto p(x,\xi)$ is homogeneous of order $r\in\R$, meaning $p( x,\lambda \xi) =\lambda^r p ( x,\xi )\ $ for all $\lambda  > 0$ and $( x,\xi )\in \R^{d}\times \big(\R^{d}\backslash \{ 0 \} \big)$,
\item $x\mapsto \d^{\alpha}_{\xi}p ( x,\xi ) \in H^{k}(\mathbb{R} ^{d})$ for all $\alpha \in \mathbb{N} ^{d}$ and $ \xi \in \mathbb{R} ^{d}\backslash \{ 0\}$.
\end{itemize}
For any $\delta\in(0,1]$, set $ l\vcentcolon= \frac{d}{2}-k +\delta > 0$. Then there exists a constant $C=C(r,k,d)>0$ independent of $\delta\in(0,1]$ such that for all $u\in H^{\mu}(\R^d)$ with $\mu\in\R$, one has
\begin{align*}
\|T_p u\|_{H^{\mu-r-l}(\R^d)} \le C\sup _{|\alpha|\le {d}/{2}+1}\sup _{|\xi|= 1}\left\| \partial ^{\alpha }_{\xi }p\left(\cdot,  \xi \right) \right\| _{H^{k}(\R^d)}\|u\|_{H^\mu(\R^d)}.
\end{align*}
\end{proposition}
\begin{proof}
Set $\delta>0$ and $l\vcentcolon=\frac{d}{2}-k +\delta >0$. We define $\xi\mapsto q ( x,\xi )$ to be the symbol of order $r-l$ constructed by: 
\begin{gather*}
\wh{q}(\zeta ,\xi) := | \xi |^{-l} \chi_{1}(\zeta ,\xi) \phi_{1}(\xi) \wh{p} ( \zeta, \xi ),\\
\begin{aligned}
\text{where } \ \ \bullet\ \ & \widehat {f}(\zeta, \xi) :=\dfrac {1}{(2\pi) ^{d/2}}\int _{\mathbb{R} ^{d}}f(x,\xi) e^{-ix\zeta}dx, \\
\bullet\ \ & \phi _{1}(\xi) =1 \quad \mbox{for} \ \xi  \in \supp\phi, \quad \phi_1(\xi)=0\quad \mbox{for}\ |\xi|\le 1/3,\\
\bullet\ \ & \chi _{1}(\zeta,\xi) =1 \quad \mbox{for} \ (\zeta ,\xi) \in \supp\chi, \quad \chi_1(\zeta, \xi)=0\quad \mbox{for}\ |\zeta|\ge|\xi|.
\end{aligned}
\end{gather*}
Here $\phi(\xi)$ and $\chi(\zeta,\xi)$ are cut-off functions in Definition \ref{def:parad}. From the construction of $p$ and $q$, it follows that $T_p=T_q|D_x|^{l}$. Indeed, by definition of paradifferential operator
\begin{align*}
\big(\widehat{T_p u}\big)(\zeta)
=& \dfrac {1}{( 2\pi)^{d/2}}\int _{\mathbb{R} ^{d}}\chi( \zeta -\xi ,\xi ) \widehat{p}(\zeta -\xi ,\xi) \phi(\xi) \widehat{u} (\xi)\,\dif\xi \\
=& \dfrac {1}{(2\pi)^{d/2}} \int_{\mathbb{R}^{d}} \chi(\zeta -\xi ,\xi)\dfrac{ |\xi|^{l} \wh{q}(\zeta-\xi ,\xi)}{\chi_{1}(\zeta-\xi,\xi)\phi _{1}(\xi) } \phi(\xi) \wh{u}(\xi) \,\dif\xi \\
=& \dfrac {1}{(2\pi)^{d/2}}\int _{\mathbb{R} ^{d}} \chi (\zeta -\xi ,\xi) \wh{q}( \zeta-\xi ,\xi) \phi (\xi) |\xi|^{l}\widehat {u}(\xi) \,\dif \xi \\
=&\dfrac {1}{(2\pi)^{d/2}}\int _{\mathbb{R}^{d}} \chi (\zeta -\xi,\xi) \wh{q}(\zeta-\xi,\xi) \phi(\xi)  \wh{(|D_x|^{l}{u})}(\xi)\, \dif\xi = \big(\wh{T_q |D_x|^{l}u}\big)(\zeta).
\end{align*}
Next, we claim that for $\alpha\in \mathbb N^d $, there exists $C_{\alpha}>0$ independent of $\delta\in(0,1]$ such that
\begin{equation}\label{hsymb1}
|\d_\xi^\alpha\widehat q(\zeta, \xi)| \le C \absm{\zeta}^{-l}\sum_{|\beta|\le |\alpha|} |\xi|^{|\beta|-|\alpha|}\big|\partial_\xi^\beta\wh{p}( \zeta , \xi ) \big|.
\end{equation}
To start, we see that $|\d^{\alpha }_{\xi }\wh{q}(\zeta,\xi)|$ is bounded by linear combination of terms of the form: 
\[
\Big| \dfrac {\d^{\beta_{1}}_{\xi}\chi _{1}(\zeta ,\xi) \d_{\xi}^{\beta_2}\phi_{1}(\xi) }{ |\xi|^{l+\beta _{3}}}\d^{\beta_4}_{\xi }\widehat {p}(\zeta,\xi) \Big| , \quad \text{with } \ |\beta_1|+|\beta_2|+|\beta_3|+|\beta_4| = |\alpha|.
\]
In order to further estimate the above expression, we consider two cases: 
\paragraph{Case 1:} $| \zeta |  < 1$. In this case, one has 
$\sqrt {1+|\zeta|^2 } < \sqrt {2}$, hence $\frac{1}{\sqrt{2}}\le \absm{\zeta}^{-1}$. Since $\supp(\widehat q(\zeta, \cdot))\subseteq \supp(\phi_1(\cdot))$ by construction, it follows that if $\xi\in \supp(\widehat q(\zeta, \cdot))$ then $|\xi|\ge1/3$. Using this and $\delta\le 1$, we have
\begin{equation}
|\xi|^{-l}\le 3^l \le (3\sqrt{2})^l  \absm{\zeta}^{-l} \le (3\sqrt{2})^{\frac{d}{2}-k+1} \absm{\zeta}^{-l} \quad \text{if } \ |\zeta|< 1 \text{ and } (\zeta,\xi)\in \supp(\wh{q}).\label{zetaxi1}
\end{equation}
\paragraph{Case 2:} $|\zeta|\ge 1$. Then $\sqrt{1+\zeta^2}\le \sqrt{2} |\zeta|$ which implies that $|\zeta|^{-1}\le \sqrt{2} \absm{\zeta}^{-1}$. Since $\supp(\widehat q(\cdot, \cdot))\subseteq \supp(\chi_1(\cdot, \cdot))$ by construction, one has $|\zeta|/|\xi|<1$ for $(\zeta, \xi)\in \supp(\widehat q)$. This implies that
\begin{equation}
|\xi|^{-l}\le |\zeta|^{-l}\le(\sqrt 2)^{l} \absm{\zeta}^{-l} \le (\sqrt{2})^{\frac{d}{2}-k+1}\absm{\zeta}^{-l} \quad \text{if } \ |\zeta|\ge 1 \text{ and } (\zeta,\xi)\in \supp(\wh{q}).\label{zetaxi2}
\end{equation}
Combining (\ref{zetaxi1}) and (\ref{zetaxi2}), and using the property of cut-off function $\chi(\zeta,\xi)$ from (\ref{chiProp}), we have for all $\sum_{i=1}^4|\beta_i|= |\alpha|$,
\begin{align*}
&\Big|\dfrac {\d^{\beta _{1}}_{\xi}\chi _{1}(\zeta ,\xi) \d_{\xi}^{\beta_2}\phi_{1}(\xi) }{|\xi|^{l+\beta _{3}}}\d^{\beta_4}_{\xi}\wh{p}(\zeta,\xi)\Big| \le C_{\alpha} \absm{\zeta}^{-l} |\xi|^{|\beta_4|-|\alpha|}\big|\d_\xi^{\beta_4}\wh{p}(\zeta,\xi) \big|,
\end{align*}
where $C_{\alpha}>0$ does not depend on $\delta>0$ due to (\ref{zetaxi1})--(\ref{zetaxi2}). This implies that
\begin{align*}
|\d_\xi^\alpha\widehat q(\zeta, \xi)| \le C_{\alpha} \absm{\zeta}^{-l}\sum_{|\beta|\le |\alpha|} |\xi|^{|\beta|-|\alpha|}\big|\d_\xi^{\beta}\wh{p}(\zeta,\xi)\big|.
\end{align*}
This is the inequality \eqref{hsymb1} as desired. Multiplying both sides of \eqref{hsymb1} by $|\xi|^{|\alpha|-r}\absm{\zeta}^{k+l}$ and integrating over $\zeta\in \R^d$ we have for all $\xi\in \R^{d}\backslash\{0\}$ and $\alpha\in\mathbb{N}^d$,
\begin{align*}
&|\xi|^{2|\alpha|-2r}\|\d_\xi^\alpha q(\cdot, \xi)\|^2_{H^{k+l}(\R^d)} = \int_{\R^d} \big\{|\xi|^{|\alpha|-r}\absm{\zeta}^{k+l}\d_\xi^\alpha \wh{q}(\zeta, \xi)\big\}^2\, \dif \zeta\\
\le& C_{\alpha}\sum_{|\beta|\le |\alpha|} \int_{\R^d} \big\{ |\xi|^{|\beta|-r} \absm{\zeta}^{k} \d_\xi^\beta \wh{p}(\zeta, \xi)\big\}^2\, \dif \zeta \le C_{\alpha}\sum_{|\beta|\le |\alpha|} |\xi|^{2|\beta|-2r} \big\|\d_\xi^\beta p(\cdot, \xi)\big\|^2_{H^{k}(\R^d)}.
\end{align*}
By construction, $k+l=k+\frac{d}{2}-k +\delta = \frac{d}{2} + \delta >\frac{d}{2}$. Thus Sobolev embedding implies that $H^{k+l}(\R^d)\xhookrightarrow[]{}L^{\infty}(\R^d)$. Hence, for each $|\xi|\ge 1/2$ and $\alpha\in\mathbb N^d$,
\begin{equation}\label{hsymb2}
|\xi|^{2|\alpha|-2r} \big\|\d_\xi^{\alpha} q(\cdot,\xi)\big\|_{L^{\infty}(\R^d)}^2 \le  C_{\alpha}\!\sum_{|\beta|\le |\alpha|} |\xi|^{2|\beta|-2r}
\big\|\d_\xi^{\beta} p(\cdot, \xi)\big\|_{H^{k}(\R^d)}^2. 
\end{equation}
Also, since $\wh{q}(\zeta, \xi)=|\xi|^{-l} \chi_1(\zeta, \xi)\phi_1(\xi) \wh{p}(\zeta,\xi)$, one has $\xi\to q(x, \xi)$ is of order $r-l$ and 
\[
q\in \Gamma_0^{r-l}(\R^d)\subset \Gamma_0^{r}(\R^d) \qquad \text{since } \ l=\frac{d}{2}-k+\delta >0.
\]
By Theorem \ref{thm:paraL2}, it follows that for all $v\in H^{\sigma}(\R^d)$ with some $\sigma\in\R$,
\begin{equation}\label{hsymb3}
\|T_q v\|_{H^{\sigma-r}(\R^d)}\le \mM_0^r(q)\|v\|_{H^\sigma(\R^d)},
\end{equation}
where $\mM_0^r(\cdot)$ is the semi-norm of $\Gamma_0^r(\R^d)$. By Definition \ref{def:symbols}, and (\ref{hsymb2}), we have 
\begin{align*}
&|\mM_0^r(q)|^2 \le
\sup _{|\alpha| \le {d}/{2}+1}\sup _{|\xi| \ge 1/2} ( 1+|\xi|)^{2|\alpha|-2r}\big\| \d^{\alpha}_{\xi}q(\cdot,  \xi) \big\| _{L^\infty(\R^d)}^2\\
\le& \sup _{|\alpha| \le {d}/{2}+1}\sup _{|\xi| \geq 1/2} 3^{2|\alpha|-2r} |\xi|^{2|\alpha|-2r}\big\| \d^{\alpha }_{\xi} q(\cdot, \xi) \big\| _{L^\infty(\R^d)}^2\\
\le& \sup_{|\alpha| \le {d}/{2}+1} \sup _{|\xi| \geq 1/2} C_{\alpha} \sum_{|\beta|\le |\alpha|} |\xi|^{2|\beta| -2r}\big\| \d^{\beta}_{\xi} p(\cdot,  \xi) \big\|_{H^{k}(\R^d)}^2\\
\le& C\sup _{|\alpha| \le {d}/{2}+1} \sup _{|\xi| \geq 1/2} |\xi|^{2|\alpha|-2r}\big\| \d^{\alpha}_{\xi} p(\cdot,  \xi) \big\|_{H^{k}(\R^d)}^2.
\end{align*}
Since $\xi\mapsto p(x,\xi)$ is homogeneous of degree $r$. Using Corollary \ref{corol:homogenous}, one has 
\begin{align}\label{hsymb4}
\mM_0^r(q) \le& C\sup _{|\alpha|\le d/2+1}\sup _{|\xi| \ge 1/2} |\xi|^{|\alpha|-r}\big\| \d^{\alpha}_{\xi } p(\cdot, \xi) \big\|_{H^{k}(\R^d)}\\ 
=& C\sup _{|\alpha| \leq d/2+1}\sup _{|\xi|= 1} \big\| \d^{\alpha}_{\xi}p(\cdot,  \xi) \big\| _{H^{k}(\R^d)}.\nonumber
\end{align}
For $u\in H^{\mu}(\R^d)$, we set $v\vcentcolon=|D_x|^{l}u$ and $\sigma\vcentcolon=\mu-l$. Then (\ref{hsymb3}) and (\ref{hsymb4}) implies
\begin{align*}
&\|T_p u\|_{H^{\mu-r-l}(\R^d)}=\big\|T_q|D_x|^{l} u\big\|_{H^{\sigma-r}(\R^d)}=
\|T_q v\|_{H^{\sigma-r}(\R^d)} \le \mM_0^r(q)\|v\|_{H^\sigma(\R^d)}\\
=& \mM_0^r(q)\big\||D_x|^l u\big\|_{H^{\mu-l}(\R^d)} \le  C\sup _{ | \alpha| \leq {d}/{2}+1}\sup _{ |\xi|= 1} \big\| \d^{\alpha}_{\xi }p(\cdot,  \xi) \big\| _{H^{k}(\R^d)}\|u\|_{H^\mu(\R^d)}.
\end{align*}
This concludes the proof.
\end{proof}

As a consequence of Proposition \ref{prop:Hsym}, we have the following estimate. 
\begin{proposition}\label{prop:Hsymbol}
Assume $s > 2+\frac{d}{2}$ and $p( x,\xi) \vcentcolon\R^{d}\times \left(\R^{d}\backslash \left\{ 0\right\} \right) \rightarrow \mathbb{C}$ is such that 
\begin{itemize}
\item For $x\in \R^d$, $\xi\mapsto p(x,\xi)\in \mC^{\infty}\big(\R^d\backslash\{0\}\big)$, and $\xi\mapsto p(x,\xi)$ is homogeneous of order $r\in\R$, meaning $p( x,\lambda \xi) =\lambda^r p ( x,\xi )\ $ for all $\lambda  > 0$ and $( x,\xi )\in \R^{d}\times \big(\R^{d}\backslash \{ 0 \} \big)$,
\item $x\rightarrow \partial ^{\alpha }_{\xi }p\left( x,\xi \right) \in H^{s-3}\left( \mathbb{R} ^{d}\right)\ $ for all $\alpha \in \mathbb{N} ^{d}$ and $ \xi \in \mathbb{R} ^{d}\backslash \left\{ 0\right\}$.
\end{itemize}
Then for all $u\in H^{\mu}(\R^d)$ with $\mu\in\R$, one has
\begin{align*}
\|T_p u\|_{H^{\mu-r-1}(\R^d)} \le C\sup _{\left| \alpha \right| \leq {d}/{2}+1}\sup _{\left| \xi \right|= 1}\left\| \partial ^{\alpha }_{\xi }p\left(\cdot,  \xi \right) \right\| _{H^{s-3}(\R^d)}\|u\|_{H^\mu(\R^d)}.
\end{align*}
\end{proposition}
\begin{proof}
If $s-3>\frac{d}{2}$ then $p\in \Gamma_{0}^{r}(\R^d)$ by the embedding $H^{s-3}(\R^d)\xhookrightarrow[]{}L^{\infty}(\R^d)$, hence Theorem \ref{thm:paraL2} implies the result. If $s-3\le \frac{d}{2}$, then the desired estimate follows by applying Proposition \ref{prop:Hsym} with the parameters:
\begin{equation*}
k\vcentcolon= s-3, \qquad \delta\vcentcolon=s-2-\tfrac{d}{2}>0, \qquad l\vcentcolon=\tfrac{d}{2}-k+\delta = 1.
\end{equation*}
This concludes the proof.
\end{proof}

\begin{theorem}\label{thm:paraPEst}
	Let $m>0$. If $a\in H^{\frac{d}{2}-m}(\R^d)$ and $u\in H^{\mu}(\R^d)$, then it follows that $T_a u \in H^{\mu-m}(\R^d)$. Moreover there exists $C>0$ independent of $a$, $u$ such that
	\begin{equation*}
		\|T_{a} u\|_{H^{\mu-m}(\R^d)}\le C \|a\|_{H^{\frac{d}{2}-m}(\R^d)}\|u\|_{H^\mu(\R^d)}.
	\end{equation*} 
\end{theorem}
\begin{proof}
Since $a=a(x)$, $\xi\mapsto a(x)$ is homogeneous of order $0$. Fix $\ep\in(0,1)$. Applying Proposition \ref{prop:Hsym} with $k\vcentcolon=\frac{d}{2} - m$, $\delta \vcentcolon= \ep$, and $l \vcentcolon= \frac{d}{2} - k + \delta = m+\ep$, we obtain
\begin{equation*}
	\|T_{a} u\|_{H^{\mu-m-\ep}(\R^d)}=\|T_{a} u\|_{H^{\mu-l}(\R^d)} \le C \|a\|_{H^k(\R^d)} \|u\|_{H^{\mu}(\R^d)} =  C \|a\|_{H^{\frac{d}{2}-m}(\R^d)} \|u\|_{H^{\mu}(\R^d)},
\end{equation*}
where the constant $C=C(k,m,d)=C(m,d)>0$ is independent of $\ep\in(0,1)$. Since this inequality holds for all $\ep\in (0,1)$, by Monotone convergence theorem, one has
\begin{equation*}
\|T_{a} u\|_{H^{\mu-m}(\R^d)}=\lim\limits_{\ep\to 0^{+}}\|T_{a} u\|_{H^{\mu-m-\ep}(\R^d)}\le C \|a\|_{H^{\frac{d}{2}-m}(\R^d)} \|u\|_{H^{\mu}(\R^d)},
\end{equation*}
which concludes the proof.
\end{proof}

\begin{theorem}[Adjoint and Composition]\label{thm:adjprod}
Let $m\in\R$ and $\theta>0$. 
\begin{enumerate}[label=\textnormal{\textrm{(\roman*)}},ref=\textnormal{\textrm{(\roman*)}}]
\item\label{item:adj} If $a\in\Gamma_{\theta}^{m}(\R^d)$, and $(T_a)^{\ast}$ is the adjoint operator of $T_a$, then the operator $(T_a)^{\ast}-T_{a^{\ast}}$ is of order $m-\theta$ where $a^{\ast}$ is a symbol defined by
\begin{equation*}
	a^{\ast}=\sum_{|\alpha|<\theta} \dfrac{1}{i^{|\alpha|}\alpha!} \d_\xi^{\alpha} \d_x^{\alpha} \bar{a}, \quad  \text{and $\bar{a}$ is the complex conjugate of $a$.}
\end{equation*}
Moreover, for all $s\in\R$, there exists a constant $C>0$ such that
\begin{equation*}
\|(T_a)^{\ast}-T_{a^{\ast}}\|_{H^s\to H^{s-m+\theta}} \le C \mM_{\theta}^m(a).
\end{equation*}
\item\label{item:prod} If $a\in\Gamma_{\theta}^{m}(\R^d)$ and $b\in\Gamma_{\theta}^k(\R^d)$, then the operator $\fq[a,b]\vcentcolon=T_aT_b-T_{a\sharp b}$ is of order $m+k-\theta$ where $a\sharp b$ is a symbol defined by
\begin{equation*}
	a\sharp b = \sum_{|\alpha|<\theta} \dfrac{1}{i^{|\alpha|}\alpha!}\d_\xi^{\alpha} a \d_x^{\alpha} b.
\end{equation*}
Moreover, for all $s\in\R$, there exists a constant $C>0$ such that
\begin{equation*}
\|\fq[a,b]\|_{H^s\to H^{s-m-k+\theta}}%\equiv\|T_aT_b - T_{a\sharpb} \|_{H^s\to H^{s-m-k+\theta}}
 \le C \mM_{\theta}^m(a) \mM_{\theta}^{k}(b).
\end{equation*}
\end{enumerate}
\end{theorem}
\begin{remark}
If $a=a(x)$ is a symbol of order $m=0$, then $T_au$ is called a paraproduct. If $k>\frac{d}{2}$, then Sobolev embedding theorem states that $H^{k}(\R^d)\xhookrightarrow[]{} \mC_{\ast}^{k-\frac{d}{2}}(\R^d)$, where $\mC_{\ast}^{m}(\R^d)$ denotes the Zygmund class of order $m\in\R$. If $k-\frac{d}{2}$ is not an integer, then the Zygmund class coincides with the H\"older class, i.e. $\mC_{\ast}^{k-\frac{d}{2}}(\R^d)=W^{k-\frac{d}{2},\infty}(\R)$. Using this and Theorem \ref{thm:adjprod}, one can also verify that
\begin{enumerate}[label=\textnormal{\textrm{(\roman*)}}]
\item if $f\in H^{k}(\R^d)$ with $k-\frac{d}{2}\in (0,\infty)\backslash \mathbb{N}$, then $\|(T_{f})^{\ast}-T_{\bar{f}}\|_{H^s\to H^{s-m}}<\infty$ for all $s\in\R$, where $m=\frac{d}{2}-k$.
\item if $f\in H^{k}(\R^d)$ and $g\in H^{l}(\R^d)$ with $k-\frac{d}{2}$, $l-\frac{d}{2}\in (0,\infty)\backslash \mathbb{N}$, then $f\sharp g = f\cdot g$, hence for all $s\in\R$, $\|T_{fg}-T_fT_g\|_{H^s\to H^{s-m}}<\infty$ where $m=\frac{d}{2}-\min\{ k, l \}<0$.
\end{enumerate} 
\end{remark}

\begin{theorem}[J.-M. Bony, 1981]\label{thm:bony}
Let $k>\frac{d}{2}$ and $l>\frac{d}{2}$. Then
\begin{enumerate}[label=\textnormal{\textrm{(\roman*)}},ref=\textnormal{\textrm{(\roman*)}}]
\item\label{item:err1} If $F\in\mC^{\infty}$ and $f\in H^{k}(\R^d)$ then $\fC[F](f)\vcentcolon=	F(f) - F(0) - T_{F^{\prime}(f)}f \in H^{2k-\frac{d}{2}}(\R^d)$. Moreover, there exists a constant $C_{F}>0$ depending only on $F$ and $k$ such that
\begin{equation*}
\|\fC[F](f)\|_{H^{2k-\frac{d}{2}}(\R^d)} %\equiv \|F(f) - F(0) - T_{F^{\prime}(f)}f\|_{H^{2k-\frac{d}{2}}(\R^d)} 
\le C_{F} \|f\|_{H^{k}(\R)}.
\end{equation*}
\item\label{item:err2} If $f\in H^{k}(\R^d)$ and $g\in H^{l}(\R^d)$ then $\fp(f,g)\vcentcolon= fg-T_f g - T_g f \in H^{k+l-\frac{d}{2}}(\R^d)$
\begin{equation*}
\|\fp(f,g)\|_{H^{k+l-\frac{d}{2}}(\R^d)}%\equiv\|fg-T_f g- T_g f\|_{H^{k+l-\frac{d}{2}}(\R^d)} 
\le C \|f\|_{H^{k}(\R^d)} \|g\|_{H^{l}(\R^d)},
\end{equation*}   
where $C>0$ is a positive constant independent of $a$ and $b$.
\end{enumerate}
\end{theorem}
\begin{proposition}\label{prop:halfBony}
	Let $d\ge1$. If $f\in H^{m}(\R^{d})$ for some $m>1+\frac{d}{2}$, and $g\in H^{k}(\R^d)$ with $k\in[0,m-1]\backslash\{\frac{d}{2}\}$. , then there exists a constant $C=C(k,m,d)>0$ independent of $f$, $g$ such that
	\begin{equation*}
		\|fg - T_{f} g\|_{H^{k+1}(\R^d)} \le C \|f\|_{H^{m}(\R^d)} \|g\|_{H^{k}(\R^d)}.
	\end{equation*}
\end{proposition}
\begin{proof}
	We write $fg-T_{f}g = fg - T_{f}g - T_{g}f + T_{g}f = \fp(f,g) + T_{g}f$. It follows from Bony's Theorem \ref{thm:bony} and $k+m-\frac{d}{2}> k+1$ that
	\begin{equation*}
		\|\fp(f,g)\|_{H^{k+1}(\R^d)}	\le \|\fp(f,g)\|_{H^{k+m-\frac{d}{2}}(\R^d)} \le C \|f\|_{H^m(\R^d)} \|g\|_{H^{k}(\R^d)}.
	\end{equation*}
	Thus we are left with estimating $T_g f$. For this, we divide the proof into two cases.
	
	\textbf{Case 1:} $k>\frac{d}{2}$. Then by the Sobolev Embedding Theorem, $g\in H^{k}(\R^d)\xhookrightarrow[]{} L^{\infty}(\R^d)$, which implies $g\in \Gamma_0^0(\R^d)$. Thus applying Theorem \ref{thm:paraL2} and using $k+1\le m$, one has
	\begin{equation*}
		\|T_{g} f\|_{H^{k+1}(\R^d)} \le C\|g\|_{L^{\infty}(\R^d)} \|f\|_{H^{k+1}(\R^d)} \le C \|g\|_{H^{k}(\R^d)} \|f\|_{H^m(\R^d)}.
	\end{equation*}
	
	\textbf{Case 2:} $0\le k < \frac{d}{2}$. Then $g\in H^{k}(\R^d) = H^{\frac{d}{2}-(\frac{d}{2}-k)}(\R^d)$. Applying Theorem \ref{thm:paraPEst}, and using $m>1+\frac{d}{2}$, it follows that
	\begin{equation*}
		\|T_{g}f\|_{H^{k+1}(\R^d)} \le C\|g\|_{H^{k}(\R^d)}\|f\|_{H^{k+1+(\frac{d}{2}-k)}(\R^d)} \le C \|g\|_{H^{k}(\R^d)}\|f\|_{H^{m}(\R^d)}. 
	\end{equation*}
	This concludes the proof.
\end{proof}
The commutator estimates presented below can be found in Appendix B.2.1 of \cite{lannes2}. First we define the following class of Fourier multiplier symbols:
\begin{definition}\label{def:FM}
A Fourier multiplier $\sigma(D)$ is said to belong to the symbol class $\Sym^k(\R^d)$ for some $k\in\R$ if $\xi\mapsto \sigma(\xi)\in\mathbb{C}$ is smooth and satisfies
\begin{equation*}
\sup\limits_{\xi\in\R^d} \absm{\xi}^{|\beta|-s}|\d_\xi^{\beta}\sigma(\xi)|<\infty \qquad \text{for all } \ (\xi,\beta)\in \R^d\times \mathbb{N}^d.
\end{equation*}
Moreover, one also defines the following semi-norm for $\sigma\in\Sym^k$:
\begin{equation*}
\mathcal{N}^k(\sigma) \vcentcolon= \sup\Big\{ \absm{\xi}^{|\beta|-s} |\d_\xi^\beta \sigma(\xi)| \ \Big\vert\ \beta\in\mathbb{N} \ \text{ and } \ |\beta|\le 2+d+\lceil \tfrac{d}{2} \rceil \Big\},
\end{equation*}
where $\lceil \alpha \rceil$ denotes the ceiling function for $\alpha\in R$.
\end{definition}
\begin{proposition}\label{prop:commu}
Let $t_0>\frac{d}{2}$, $k\ge 0$ and $\sigma\in\Sym^k$. If $ 0 \le k \le t_0+1$ and $f\in H^{t_0+1}(\R^d)$, then for all $g\in H^{k-1}(\R^d)$,
\begin{equation}
	\big\| [\sigma(D),f] g \big\|_{L^2(\R^d)} \le C \mathcal{N}^k(\sigma) \|\nabla f\|_{H^{t_0}(\R^d)} \|g\|_{H^{k-1}(\R^d)},
\end{equation} 
for some constant $C>0$ independent of $\sigma$, $f$, and $g$.
\end{proposition}
If the symbol takes the special form of $\sigma(\xi)=\absm{\xi}^k$, then the following improvements on Proposition \ref{prop:commu} holds
\begin{proposition}\label{prop:commuLam}
Let $t_0>\frac{d}{2}$. Assume $-\dfrac{d}{2}<k\le t_0+\delta$. Then for all $f\in H^{t_0+\delta}(\R^d)$ and $g\in H^{k-\delta}(\R^d)$, one has
\begin{equation*}
\|[\absm{D}^k,f]g\|_{L^2(\R^d)} \le C \|f\|_{H^{t_0+\delta}(\R^d)} \|g\|_{H^{k-\delta}(\R^d)},
\end{equation*}
for some constant $C>0$ independent of $f$, and $g$.
\end{proposition}
%%%%%%%%%%%%%%%%%%%%%%%%%%%%%%%%%%%%%%%%%%%%%%%%%%%%%
%%%%%%%%%%%%%%%%%%%%%%%%%%%%%%%%%%%%%%%%%%%%%%%%%%%%%

\end{document}